\documentclass[10pt]{amsart}

\usepackage{color}
\usepackage{graphicx, tikz}
\usepackage{xcolor}
\usepackage{cancel}
\usepackage{amssymb,amsfonts,amsthm,amsmath,calligra, enumerate, stmaryrd}
\usepackage{extarrows}
\usepackage[utf8]{inputenc}
\usepackage{accents}
\usepackage{slashed}
\usepackage{yfonts}
\usepackage{mathrsfs,pifont}
\usepackage{mathtools}

\usepackage[bookmarksnumbered=true, bookmarksopen=true]{hyperref}
\usepackage[margin=1in]{geometry}
\usepackage[all]{xy}

\theoremstyle{definition}

\newtheorem{Theorem}{Theorem}[section]
\newtheorem{Proposition}[Theorem]{Proposition}
\newtheorem{Lemma}[Theorem]{Lemma}
\newtheorem{Corollary}[Theorem]{Corollary}
\newtheorem{Definition}[Theorem]{Definition}
\newtheorem{Construction}[Theorem]{Construction}

\newtheorem{Remark}[Theorem]{Remark}
\newtheorem{Example}[Theorem]{Example}

\newtheorem{Assumption}[Theorem]{Assumption}
\newtheorem{Setting}[Theorem]{Setting}

\numberwithin{equation}{section}

\def\ben{\begin{eqnarray*}}
\def\een{\end{eqnarray*}}

\newcommand{\HallEnv}{\mathrm{HallEnv}}

\newcommand{\coh}{\mathrm{coh}}

\newcommand{\oO}{\mathcal{O}}

\newcommand{\calH}{\mathcal{H}}

\newcommand{\calF}{\mathcal{F}}

\newcommand{\der}{\mathrm{der}}

\newcommand{\bA}{\mathbb{A}}

\newcommand{\A}{\mathbb{A}}

\newcommand{\C}{\mathbb{C}}
\newcommand{\GL}{\mathrm{GL}}

\newcommand{\Crit}{\mathrm{Crit}}

\newcommand{\Z}{\mathbb{Z}}

\newcommand{\bC}{\mathbb{C}}
\newcommand{\bD}{\mathbb{D}}
\newcommand{\bG}{\mathbb{G}}

\newcommand{\bN}{\mathbb{N}}
\newcommand{\bP}{\mathbb{P}}
\newcommand{\bQ}{\mathbb{Q}}
\newcommand{\bR}{\mathbb{R}}

\newcommand{\bZ}{\mathbb{Z}}

\newcommand{\ba}{\mathbf{a}}

\newcommand{\be}{\mathbf{e}}

\newcommand{\br}{\mathbf{r}}

\newcommand{\bu}{\mathbf{u}}
\newcommand{\bv}{\mathbf{v}}
\newcommand{\bw}{\mathbf{w}}
\newcommand{\bd}{\mathbf{d}}

\newcommand{\bPsi}{\mathbf{\Psi}}
\newcommand{\dR}{\mathbf{R}}

\newcommand{\cE}{\mathcal{E}}
\newcommand{\eE}{\mathcal{E}}

\newcommand{\cH}{\mathcal{H}}

\newcommand{\cK}{\mathcal{K}}
\newcommand{\cL}{\mathcal{L}}
\newcommand{\cM}{\mathcal{M}}
\newcommand{\cN}{\mathcal{N}}
\newcommand{\cO}{\mathcal{O}}

\newcommand{\cV}{\mathcal{V}}

\newcommand{\fC}{\mathfrak{C}}
\newcommand{\fg}{\mathfrak{g}}

\newcommand{\fM}{\mathfrak{M}}
\newcommand{\fN}{\mathfrak{N}}
\newcommand{\fR}{\mathfrak{R}}
\newcommand{\fp}{\mathfrak{p}}
\newcommand{\fq}{\mathfrak{q}}

\newcommand{\fL}{\mathfrak{L}}
\newcommand{\fX}{\mathfrak{X}}

\newcommand{\BM}{\mathrm{BM}}
\newcommand{\CH}{\mathrm{CH}}
\newcommand{\loc}{\mathrm{loc}}

\newcommand{\vir}{\mathrm{vir}}

\newcommand{\sg}{\mathsf{g}}
\newcommand{\sA}{\mathsf{A}}

\newcommand{\sS}{\mathsf{S}}
\newcommand{\sT}{\mathsf{T}}

\newcommand{\sF}{\mathsf{F}}

\newcommand{\sW}{\mathsf{W}}
\newcommand{\sw}{\mathsf{w}}

\newcommand{\JH}{\mathsf{JH}}

\newcommand{\Attr}{\operatorname{Attr}}
\newcommand{\Aut}{\operatorname{Aut}}

\newcommand{\can}{\operatorname{can}}

\newcommand{\cochar}{\operatorname{cochar}}
\newcommand{\codim}{\operatorname{codim}}

\newcommand{\diag}{\operatorname{diag}}

\newcommand{\End}{\operatorname{End}}

\newcommand{\Fix}{\operatorname{Fix}}
\newcommand{\Frac}{\operatorname{Frac}}
\newcommand{\Gr}{\operatorname{Gr}}
\newcommand{\Hilb}{\operatorname{Hilb}}
\newcommand{\Hom}{\operatorname{Hom}}
\newcommand{\Ext}{\operatorname{Ext}}
\newcommand{\Id}{\operatorname{Id}}
\newcommand{\im}{\operatorname{im}}

\newcommand{\Lie}{\operatorname{Lie}}

\newcommand{\Pic}{\operatorname{Pic}}
\newcommand{\pr}{\operatorname{pr}}

\newcommand{\pt}{\operatorname{pt}}

\newcommand{\rk}{\operatorname{rk}}

\newcommand{\Span}{\operatorname{Span}}
\newcommand{\Spec}{\operatorname{Spec}}
\newcommand{\Stab}{\operatorname{Stab}}

\newcommand{\tr}{\operatorname{tr}}
\newcommand{\Coh}{\operatorname{Coh}}
\newcommand{\Perf}{\operatorname{Perf}}
\newcommand{\D}{\operatorname{D}}
\newcommand{\Supp}{\operatorname{Supp}}
\newcommand{\cyc}{\operatorname{cyc}}

\newcommand{\wt}{\operatorname{weight}}
\newcommand{\IC}{\operatorname{IC}}
\newcommand{\id}{\operatorname{id}}
\newcommand{\Perv}{\operatorname{Perv}}

\newcommand{\In}{\operatorname{in}}
\newcommand{\Out}{\operatorname{out}}

\newcommand{\scrU}{\mathscr{U}}

\makeatletter
\newcommand{\ostar}{\mathbin{\mathpalette\make@circled\star}}
\newcommand{\make@circled}[2]{%
  \ooalign{$\m@th#1\smallbigcirc{#1}$\cr\hidewidth$\m@th#1#2$\hidewidth\cr}%
}
\newcommand{\smallbigcirc}[1]{%
  \vcenter{\hbox{\scalebox{0.77778}{$\m@th#1\bigcirc$}}}%
}
\makeatother

\newcommand{\yc}[1]{\textcolor{blue}{  #1 }}

\newcommand{\yl}[1]{\textcolor{blue}{$[$ Yalong: #1 $]$}}

\title{Stable envelopes for critical loci}

\author{Yalong Cao}
\address{Morningside Center of Mathematics, Institute of Mathematics \& State Key Laboratory of Mathematical Sciences, Academy of Mathematics and Systems Sciences, Chinese Academy of Sciences, Beijing, China}
\email{yalongcao@amss.ac.cn}
\author{Andrei Okounkov} 
\address{Department of Mathematics, Columbia University, New York, U.S.A.}
\email{okounkov@math.columbia.edu} 
\author{Yehao Zhou}
\address{\parbox{\linewidth}{Center for Mathematics and Interdisciplinary Sciences, Fudan University, Shanghai 200433, China\\
Shanghai Institute for Mathematics and Interdisciplinary Sciences (SIMIS), Shanghai 200433, China}}
\email{yehao.zhou@simis.cn}
\author{Zijun Zhou}
\address{School of Mathematical Sciences, Shanghai Jiao Tong University, Shanghai, China}
\email{zijun.zhou@sjtu.edu.cn}


\begin{document}

\begin{abstract}

This is the first in a sequence of papers devoted to stable envelopes in critical cohomology and critical $K$-theory for symmetric GIT quotients with potentials 
and related geometries,
and their applications to geometric representation theory and enumerative geometry.
In this paper, we construct critical stable envelopes and establish their general properties, including compatibility with dimensional reductions, 
specializations, Hall products, and other geometric constructions. In particular, for tripled quivers with canonical cubic potentials, the critical stable envelopes
reproduce those on Nakajima quiver varieties. These set up foundations for applications in subsequent papers.

\end{abstract}

\maketitle

\setcounter{tocdepth}{1}
\tableofcontents

\setlength{\parskip}{1ex}

\allowdisplaybreaks

\vspace{1cm}

\section{Introduction}

\subsection{Overview}

Stable envelopes, introduced in \cite{MO} in the setting of ordinary
equivariant cohomology, have found numerous applications in geometric
representation theory and enumerative geometry. By one of possible
definitions, geometric representation theory studies algebras
generated by geometric correspondences. Enumerative geometry is full
of important and natural correspondences such as those formed by pairs
of points of some algebraic variety $X$ that lie, perhaps in a virtual
sense, on a rational curve of a given degree. In either setting, one
fixes some cohomology theory $\mathbf{h}(X)$ in which the correspondences act. This may be ordinary equivariant cohomology, equivariant $K$-theory, or something more general, like the critical cohomology or critical $K$-theory that we use in this paper. 

One considers an enumerative problem solved if the enumerative correspondences are expressed in terms of an understandable geometric action of an understandable algebraic structure. Of course, understandable does not mean simple. The typical algebraic structures one meets in the context of stable envelopes are certain quantum loop groups and their Yangian analogs, where the underlying Lie algebra is, by itself, typically an infinitely generated Borcherds-Kac-Moody (BKM) Lie algebra. But such is the intrinsic complexity of the problem, and one's aim should be to set up an adequate geometric and algebraic framework for managing it. 

Stable envelopes are defined in the presence of an action of a
connected reductive group. In the most basic case of a torus
$\sA$, they give a certain canonical correspondences between the
fixed locus $X^\sA$ and the ambient space $X$. These are improved
versions of attracting (also known as stable) manifolds which, in
contrast to attracting manifolds, are proper over $X$, behave well in families, etc. They add certain correction terms to attracting manifolds and, in this sense, they complete or envelop them, whence the name. 

Stable envelopes depend on additional choices, such as the choice of attracting/repelling directions. This proves to be an important feature since different choices of attracting directions are related by what turns out to be the braiding, or the $R$-matrix --- the cornerstone concept in quantum group theory. 
Thus, out of stable envelopes, one constructs the  correspondences by
which a quantum loop group or a Yangian acts. Stable envelopes are equally natural in enumerative contexts, where their properness forms the basis for many computations based, ultimately, on equivariant rigidity. One very important 
example of this is the geometric identification of the quantum
Knizhnik-Zamolodchikov connection with certain shift operators. We
will say more about it below. 

So far, stable envelopes have been constructed and used in ordinary equivariant cohomology, equivariant $K$-theory, and 
equivariant elliptic cohomology \cite{MO, Oko, AO2}. Although such generality encompasses a broad range of contexts and applications, there are even more potential applications that it misses. 

For instance, a major motivation for \cite{MO} came from the work of Nekrasov and Shatashvili on quantum integrable structures in 2-dimensional and 
(2+1)-dimensional supersymmetric QFTs \cite{NS}. For mathematicians, these
structures appear in the computation of  virtual indices of Dirac
operators on the moduli spaces of maps, or quasimaps, to be more precise,
$f\colon C \to X$, where the Riemann surface $C$ is the space part of the
(2+1)-dimensional spacetime and $X$ is (a component of) the moduli
spaces of vacua for the QFT in question. These $X$ are sometimes
smooth, for instance the Nakajima quiver varieties \cite{Nak1, Nak2}, naturally arise in
this way. The theory of  \cite{MO, Oko, OS} applies to all Nakajima varieties and identifies the Nekrasov-Shatashvili quantum integrable system with the familiar Baxter-style quantum integrable system for the corresponding quantum loop group or Yangian. 

Typically, however, these moduli spaces of vacua are \textit{not} smooth. Rather, they are \textit{critical loci} $\Crit(\sw)$ of a regular function $\sw$ on a smooth ambient space or stack $X$. The enumerative setup for counting maps to such targets has been recently studied, see~e.g.~\cite{FK, CZ, CTZ, KP},\,etc., and the natural home for these computations is the \textit{critical cohomology} or \textit{critical} $K$-\textit{theory} of the pair $(X,\sw)$. These critical theories provide a flexible and versatile generalizations of ordinary cohomology and $K$-theory. In particular, their versatility manifests itself in there being a very general setup in which enumerative questions about quasimaps $f\colon C\to \Crit(\sw)$ can be asked. One of the goals of this project is to answer these enumerative questions using 
\textit{critical stable envelopes} and the corresponding quantum groups. 

A very good example to have in mind is the example of 
the Hilbert scheme $\Hilb(\C^3,n)$  of $n$ points in the affine 3-space. While $\Hilb(\C^2,n)$ is a smooth Nakajima variety, its 3-dimensional counterpart is 
a singular variety of unknown dimension and unknown number of
irreducible components. It is, however, naturally a critical locus,
stemming from the well-known  fact that three matrices $Y_1,Y_2,Y_3$
commute if and only if all partial 
derivatives of the function $\sw(Y) = \tr Y_1 [Y_2,Y_3]$ vanish. 
The computation of the quantum cohomology of 
$\Hilb(\C^2,n)$ in \cite{OP1} was a major precursor to the computation of the quantum cohomology of all Nakajima varieties. Because of the similarly important role of 
$\Hilb(\C^3,n)$, we spell out its quantum cohomology in 
\S \ref{sect intro on qc of hilbC3}. 

While the overall structure of theory developed here bears a definite
resemblance to the corresponding theories for other cohomology
theories, there is a large number of conceptual and technical points
in which it goes significantly beyond the existing frontier of
knowledge. We will touch on many such points in this
introduction. In this brief overview, it suffices to mention how much
broader is the supply of the algebraic structures that our methods
produce. In facts, it broadens in at least \textit{three directions}: the
underlying Lie algebra can be a \textit{super Lie algebra}, its BKM Cartan
matrix may be symmetrizable, and \textit{not} symmetric, and finally, the
quantum loop group or the Yangian may be \textit{shifted}. Geometrically,
the appearance of shift means that we can relax the self-duality
assumptions on
$X$ that are normally made in the construction of stable envelopes.
Notice, however, that a certain sign of this shift (antidominant) is
necessary in our setup.

\subsection{Stable envelopes}

It is fitting to start the discussion of stable envelopes with the
abelian case. Let $X$ be a smooth quasi-projective algebraic variety
with an action of a torus $\sA$. Then the fixed locus $X^\sA$ is
smooth and there is a smooth locally closed subvariety
\begin{equation} \label{Attr intro}
\Attr_\fC = \{(x_1,x_0)\,|\, \lim_{t\to 0} \sigma(t) \cdot x_1 = x_0\} \subseteq  X \times
X^{\sA},
\end{equation}
where $\sigma$ is a generic cocharacter of $\sA$. The attracting
(also known as stable) 
manifold \eqref{Attr intro} depends only a certain chamber $\fC \subseteq 
\Lie \sA$ containing $\sigma$, whence the notation. If \eqref{Attr
  intro} is closed, it can be used as a 
correspondence mapping equivariant cohomology $\mathbf{h}_\sT(X^\sA)$ of $X^\sA$ to
$\mathbf{h}_\sT(X)$. Here equivariance includes $\sA$ and whatever other
group actions preserve the setup, which we may assume to be a torus $\sT
\supset \sA$ for simplicity.  
In interesting cases, however, \eqref{Attr intro} is not closed, and this is where 
stable envelopes come in, completing $\Attr_\fC$ to a correspondence $\Stab=\Stab_\fC$ that is proper over $X$.

The attracting manifold \eqref{Attr intro} has multiple components,
indexed by the components $F_i$ of the fixed locus 
$$X^\sA= \bigsqcup F_i. $$ 
These are partially ordered by $F_j \cap \overline{\Attr(F_i)} \ne
\varnothing$, and the construction of stable envelopes may be
performed inductively. If $F_{\min}$ is a minimal element in the partial
order, then $X' = \Attr(F_{\min})$ is closed  and the corresponding
component of \eqref{Attr intro} will be one of the
components of the eventual $\Stab$. By induction, we may 
assume that $\Stab'$ been constructed for $X \setminus
X'$. To extend $\Stab'$ to all of $X$, we need to fix its lift 
with respect to the restriction map $\mathbf{h}_\sT(X) \to
\mathbf{h}_\sT(X\setminus X')$. Fixing a lift with respect to map of
algebras means solving a certain interpolation problem in the
classical language of algebraic geometry. The stable envelope fixes
the unique (when it exists) solution of the interpolation problem by
constraining the degree of
$$
\Stab\Big|_{F_j \times F_i} \in \mathbf{h}_{\sT} (F_j \times F_i) =
\mathbf{h}_{\sT/\sA} (F_j \times F_i) \otimes \mathbf{h}_{\sA} (\pt)
$$
in the $\mathbf{h}_{\sA} (\pt)$-factor. 
For ordinary cohomology or critical cohomology, we have $\mathbf{h}_{\sA} (\pt) = \Z[\Lie \sA]$ and
we require
\begin{equation}
\deg_\sA \Stab\Big|_{F_j \times F_i} < \deg_\sA \Attr\Big|_{F_j \times F_j}
\,,\,\,\, \mathrm{for} \,\,\, j\ne i, \label{degree bound intro}
\end{equation}
where $\deg_\sA$ is the usual degree of a polynomial. In equivariant
$K$-theory, ordinary or critical, 
we have $\mathbf{h}_{\sA} (\pt) = \Z[\sA]$, and $\deg_\sA$ is the
Newton polytope of a Laurent polynomial, considered modulo shifts. The
comparison \eqref{degree bound intro} is replaced by the inclusion of
Newton polytopes after a shift, see Definition \ref{def of stab k}. As usual, 
through this shift parameter, $K$-theoretic stable envelopes acquire dependence on a generic 
fractional line bundle on $X$, called the \emph{slope} of the 
stable envelope.

While the degree bounds easily imply
uniqueness, the existence of stable envelopes is a much more delicate
business which requires additional assumptions on $X$ (this was done through the above induction method 
for Nakajima varieties in \cite{Oko-ind}, it is not clear how to extend it to
critical theories). The uniqueness
of stable envelopes implies their $\sT$-equivariance. The corresponding additional equivariant 
parameters are important and become the parameters of quantum groups
and their modules in geometric realization.

Beyond the above Bialynicki-Birula-type stratification by $\Attr(F_i)$, the
same logic may be applied to other group-invariant
stratifications. In particular, for the instability stratifications
considered in the GIT context, the corresponding extension maps
are called the \textit{nonabelian stable envelopes}, see \cite{AO1} and 
\S \ref{sect on sym qui var}, \ref{intro ex of stab} below. They may be phrased as a canonical extension of
cohomology classes from the stable locus to the whole quotient
stack. For symmetric quiver varieties, which is the generality on which we focus
in this paper, nonabelian stable envelopes may, in fact, be reduced to
the abelian ones, see \S \ref{intro ex of stab}.

One can also contemplate generalizing GIT
stratifications to instability stratifications of more
general stacks, but this is not something we do in this paper.

In the paper, we normalize stable envelopes to equal attracting
manifolds modulo lower terms. While normalizations are not
fundamental, they do enter formulas and can either
simplify or complicate them. Because we work in
broader generality, we depart from the conventions of
\cite{Oko}, where stable envelopes were normalized using a choice of
polarization. This leads to a change of
slopes in the triangle lemma \eqref{triangle intro}, and to certain
dynamical shifts and signs 
in the Yang-Baxter equation \eqref{Yang Baxter equation};  see more on both of these points below.

\subsection{Symmetric quiver varieties}\label{sect on sym qui var}

Quiver varieties provide a very concrete and useful local or global
description of various moduli spaces in algebraic geometry and
mathematical physics, encompassing a very wide class of varieties of
theoretical and applied interest. 
While there is no fundamental reason to limit one's attention to quiver varieties, and although most results below work for more general GIT quotients,
for concreteness and with applications in mind,  in this introduction, we focus  on critical stable envelopes for quiver varieties with potentials. 

By definition, a \textit{quiver variety} is a GIT quotient
$X=R(Q)/\!\!/G$ of a certain linear representation $R(Q)$ by a product
$G = \prod GL(V_i)$ of general linear groups. Concretely,
\begin{align*}
R(Q):=R(Q,\mathbf{v},\mathbf{d}_\textup{in},\mathbf{d}_\textup{out})&:=\bigoplus_{i,j}
                                                                \Hom(V_i,V_j\otimes
                                                                Q_{ij})\oplus \bigoplus_{i} \left(\Hom(D_{i,\textup{in}},V_i) \oplus \Hom(V_i,D_{i,\textup{out}})\right). 
\end{align*}
Here the group $G$ acts trivially in all vector spaces other than
$\{V_i\}$. The dimensions of $Q_{ij}$ are fixed and constitute the
adjacency matrix of the quiver $Q$, that is, the number of
arrows between the $i$th and the $j$th vertex. 
The gauge dimension vector $\mathbf{v} = (\dim V_i)$, and framing dimension vectors
 $$\mathbf{d}_\textup{in}= (\dim D_{i,\textup{in}}), \quad \mathbf{d}_\textup{out}= (\dim D_{i,\textup{out}})$$ are allowed to
vary. In geometric representation theory, a given value of
$\mathbf{v}$ corresponds to a weight space in a quantum group module,
the highest weight of which is determined by
$\mathbf{d}_\textup{in/out}$. One may be tempted to compare the quiver
$Q$ to the Dynkin diagram of the quantum group, but the actual relation between the structure of the quiver and the structure of the resulting quantum group is considerably more subtle. In the ordinary cohomology of Nakajima varieties, it has been the subject of a conjecture made by one of us and established recently in full generality by B.~Davison and T.~Botta \cite{BD}. We make a similar conjecture in the critical case in \cite[\S 4.3]{COZZ1}, see also \S \ref{subsection on shifted qg}.

We consider a torus $\sT$ which acts on the spaces $Q_{ij}$, 
$D_{i,\textup{in/out}}$ and thus commutes with $G$ and 
acts on $X$. Let $$\sw\colon R(Q)\to \mathbb{C}$$ be a $(G\times \sT)$-invariant function  on $R(Q)$. 
By abuse of notation, $\sw$ also denotes the descent of the function to $X$.
We assume that the quiver variety is \emph{symmetric}, 
which means that 
\begin{equation}\label{intro equ on sd}\dim Q_{ij} = \dim Q_{ji}, \quad 
\dim D_{i,\textup{in}} = \dim D_{i,\textup{out}}. \end{equation} 
Let $\sA \subseteq  \sT$ be a subtorus such that $R(Q)$ is a self-dual representation of $G\times \sA$.
The latter assumption of \eqref{intro equ on sd} on the symmetry of framing admits a
relaxation to be $\dim D_{i,\textup{in}}\geqslant \dim D_{i,\textup{out}}$, see below, \S \ref{sect on more exist on costab}, 
\S \ref{sect on hall comp with hall} and \S \ref{sec: pstab vs stab}.

We denote \textit{critical cohomology} to be
$$
H^\sT_i (X,\sw) = H^{-i}_\sT(X,\varphi_\sw \, \omega_X) \,, 
$$
where $\omega_X$ is the dualizing sheaf of $X$ and 
$\varphi_\sw$ is the vanishing cycle functor. 

\textit{Critical stable envelopes}
give canonical
$H_\sT(\pt)$-linear maps
\begin{equation}
\Stab_\fC\colon H^\sT(X^\sA,\sw) \to H^\sT(X,\sw)\,, \label{Stab intro}
\end{equation}
and similarly in critical $K$-theory (i.e.~the Grothendieck group of matrix factorization category). 

There are also \textit{(critical) nonabelian stable envelopes} 
(Definition \ref{na stab coh_sym quot}): 
$$
\textup{naStab}\colon H^\sT(X,\sw) \to H^{G \times \sT}(R(Q),\sw)\,, 
$$
which extend critical cohomology classes from the stable locus $X\subseteq  R(Q)/G$ to the whole stack, and similarly in critical $K$-theory. 
We remark that in the cohomological case, there is an interpretation of nonabelian stable envelopes 
using BPS cohomology of Davison and Meinhardt \cite{DM}, see Proposition \ref{prop:na stab and bps coho}.




\subsection{Existence of critical stable envelopes}\label{intro ex of stab}
Uniqueness of $\Stab_\fC$ being a simple 
corollary of the definition, it is its 
existence that requires proof. 
\begin{Theorem}\label{intro thm on stab ex}
$($Theorems \ref{exi of coh stab on symm quiver var}, \ref{thm: hall k_sym quot}, Propositions \ref{uniqueness of coh stab}, \ref{uniqueness of K stab}, \ref{corr induce stab}$)$
Let $X$ be a symmetric quiver variety, $\sA\subseteq \sT$ be tori acting on $X$ such that $\sA$-action on $X$ is self-dual. 
Fix an arbitrary chamber $\fC$, and a $\sT$-invariant function $\sw\colon X\to \C$. 
Then stable envelopes exist and are unique in critical cohomology and critical $K$-theory. 
\end{Theorem}
This theorem works for more general symmetric GIT quotients with potentials. 
The self-dual condition of the $\sA$-action on $X$ is necessary, see Remark \ref{rmk on sd con} and Proposition \ref{prop on 3 ex} for counterexamples when it is not satisfied.
The above theorem is one of the main results of this paper and we sketch its proof.
We introduce the concept of critical stable envelope \textit{correspondences} in Definitions \ref{stab corr_coh}, \ref{stab corr_k} and show that they induce critical stable envelopes (Lemma \ref{can map induce stab}).
It is easy to see that the stable envelope 
correspondence for $\sw=0$ induces a stable envelope correspondence for any $\sT$-invariant function $\sw$ (Lemma \ref{can map induce stab}). This is yet another 
manifestation of the technical versatility of critical cohomology theories. 

In \cite{MO}, the following simple construction of 
cohomological stable envelopes for Nakajima 
varieties and other equivariant symplectic resolutions was given. 
One puts $X$ is a generic one-parameter deformation family, all other
fibers
in which are affine algebraic varieties. In the total space of the family, one takes the closure of the attracting manifold, and restricts this correspondence to the central fiber $X$. This gives the stable envelope correspondence between $X^\sA$ and 
$X$. 

Remarkably, a similar description works in \textit{cohomology} for \textit{symmetric quiver varieties} (or more generally \textit{symmetric GIT quotients}, see Definition \ref{def of sym var}), even \textit{without} the need for a generic deformation. Namely, a symmetric quiver variety $X$ (or more generally a symmetric GIT quotient) turns out to be as good as the generic deformation family, in the sense that the closure of attracting manifold for $X$ is the stable envelope correspondence for $\sw=0$. This is proven in Theorem \ref{construct coh stab corr on symm quiver var}. For comparison with the Nakajima varieties, one may note that the generic deformation of a Nakajima variety is obtained by giving generic $G$-invariant values to the moment map. Therefore, this whole deformation family is naturally embedded in the ambient quiver variety $X$ (which is just a GIT quotient, not a symplectic reduction). 
We remark that this explicit description as the closure of attracting manifold is particularly useful in the manipulation of stable envelopes, as shown in later sections.

The validity of the above approach is limited to critical cohomology. To prove the existence of $K$-\textit{theoretic} \textit{critical stable envelopes}, 
we first extend elements in $K^\sT(X^\sA,\sw)$ to the
ambient stack using the $K$-theoretic nonabelian stable
envelope defined using window subcategories  (Definition \ref{def of na stab k_window}). Then, on the stack, use the attracting correspondence,
which is proper over its target on
the stack. And, finally, we restrict to the stable locus, see \S \ref{sect on ex kstab}, \S \ref{app on k symm var} for details.

There may be several ongoing manuscript projects devoted to categorical stable envelopes, including \cite{HMO}. When their results become available, one will be able to construct $K$-theoretic critical stable envelopes as the decategorification of these results.

\subsection{Stable envelopes v.s.\ Hall envelopes}\label{intro ex of stab2}

It is important to stress the following points about the above
construction of the $K$-theoretic critical stable envelopes (same also applies to the cohomological case). 

On symmetric quiver varieties, nonabelian stable envelopes admit a description (Proposition \ref{prop: compare na stab_sym quiver}) originated from the work of \cite{AO1},~i.e.~ 
adding extra framings on quivers and using (abelian) stable envelopes 
for one dimensional tori (Proposition \ref{prop on stable corr}).
One can
apply the same algorithm to an \emph{arbitrary} quiver variety and
a $(G\times \sT )$-invariant function $\sw$ on $R(Q)$. This will produce an
extension map
$$
\bPsi_H\colon H^\sT(X,\sw) \to H^{G\times \sT } (R(Q), \sw)\,, \quad \bPsi_K\colon K^\sT(X,\sw) \to K^{G\times \sT } (R(Q), \sw) \,, 
$$
which we call the \emph{interpolation map} (see Definition \ref{def of nona stab}). In general, it does not
satisfy the definition of the nonabelian stable envelope (unless the quiver is symmetric, see Proposition \ref{prop:Psi} for the characterization property), whence the
need for a separate name.

The interpolation map can always be composed with the attracting
correspondence of the stack, and then restricted back to the
stable locus. This will produce a map
$$
\HallEnv\colon H^\sT(X^\sA,\sw) \to H^\sT(X,\sw), \quad \HallEnv\colon K^\sT(X^\sA,\sw) \to K^\sT(X,\sw)\,, 
$$
which we call the \textit{Hall envelope}, because its core ingredient is the
attracting, or Hall, correspondence on the stack. Again, the reason for
introducing a separate name is that, in general, the Hall envelope
fails the definition of the stable envelope, see \S \ref{sec: pstab vs stab}. 

In \S \ref{sec:na stab},  we compare and contrast
stable envelopes and Hall envelopes that generalize them, summarized below. 
\begin{Theorem}
$($Theorems \ref{thm: hall coh_sym quot}, \ref{thm: hall k_sym quot}$)$
As in the setting of Theorem \ref{intro thm on stab ex}, stable envelopes are compatible with the
Hall correspondences on the stack,~i.e.~the following diagrams commute:  
\begin{equation}
  \label{Hall compatible intro}
  \xymatrix{
    H^{G\times \sT } (R(Q)^\sA, \sw)  \ar[rr]^{\textup{Hall}} && H^{G\times \sT } (R(Q), \sw)
    \\
    H^\sT(X^\sA,\sw) \ar[u]^{\textup{naStab}} \ar[rr]^{\Stab} &&
    H^\sT(X,\sw), \ar[u]_{\textup{naStab}} 
    }
    \quad \quad
    \xymatrix{
    K^{G\times \sT} (R(Q)^\sA, \sw)  \ar[rr]^{\textup{Hall}} && K^{G\times \sT} (R(Q), \sw)
    \\
    K^\sT(X^\sA,\sw) \ar[u]^{\textup{naStab}} \ar[rr]^{\Stab} &&
    K^\sT(X,\sw). \ar[u]_{\textup{naStab}} 
    }
\end{equation}
Here in the top-left corners, the torus $\sA$ acts on $R(Q)$ both via
its embedding into $\sT$ and a certain homomorphism $\phi\colon \sA \to G$,
and the top arrows are Hall/attracting correspondences on the ambient
stack. 

In particular, stable envelopes equal to Hall envelopes in this case. 
\end{Theorem}
We remark that via dimensional reduction mentioned below, the above theorem reproduce a result of Botta \cite{Bot} and Botta-Davison \cite{BD} on Nakajima quiver varieties (Remark \ref{rmk on cpr with b}). When potentials are zero, the theorem has applications to explicit calculations 
of stable envelopes, see \S \ref{sect on explicit form_coh} and \S \ref{sect on explicit form_k}.

A more general diagram, in which the bottom arrows are
replaced by the Hall envelopes and the vertical arrows are replaced by
the interpolation maps does not need to commute. But we have the following remarkable converse.
\begin{Theorem}
$($Theorem \ref{thm: AFSQV}$)$
Let $Q$ be a symmetric quiver and $\bv,\bd_{\In},\bd_{\Out}\in \bN^{Q_0}$ be dimension vectors such that $\bd_{\In,i}\geqslant \bd_{\Out,i}$ for all $i\in Q_0$. 
Let $\sA\subseteq \sT$ be subtori of the flavour group $\sF$ such that $\sA$-action on $X$ is pseudo-self-dual. 
Then Hall envelopes are compatible with
Hall correspondences on the stack,~i.e.~the following diagrams commute:  
\begin{equation}
  \label{Hall compatible intro 2}
  \xymatrix{
    H^{G\times \sT} (R(Q)^\sA, \sw)  \ar[rr]^{\textup{Hall}} && H^{G\times \sT} (R(Q), \sw)
    \\
    H^\sT(X^\sA,\sw) \ar[u]^{\bPsi_{H}} \ar[rr]^{\HallEnv} &&
    H^\sT(X,\sw), \ar[u]_{\bPsi_H} 
    }
    \quad \quad
    \xymatrix{
    K^{G\times \sT} (R(Q)^\sA, \sw)  \ar[rr]^{\textup{Hall}} && K^{G\times \sT} (R(Q), \sw)
    \\
    K^\sT(X^\sA,\sw) \ar[u]^{\bPsi_{K}} \ar[rr]^{\HallEnv} &&
    K^\sT(X,\sw). \ar[u]_{\bPsi_{K}} 
    }
\end{equation}
\end{Theorem}
This compatibility implies triangle lemma (Lemma \ref{tri lem for compatible prestab}) which we introduce right below, and provides foundations for 
the study of geometric $R$-matrices of \textit{shifted} quantum groups.

We also remark that the condition $\bd_{\In,i}\geqslant \bd_{\Out,i}$ is crucial for the above result to hold 
(see Example \ref{ex on fail of tri lem} and Proposition \ref{prop on 3 ex}).





\subsection{Properties of stable envelopes}
The most fundamental property of stable envelopes, 
known as the \textit{triangle lemma}, relates stable envelopes for a torus
$\sA$, a subtorus $\sA'\subseteq  \sA$, and the quotient torus
$\sA/\sA'$ that acts on $X^{\sA'}$. It says that the following diagram 
\begin{equation}\label{triangle intro}
\xymatrix{
H^\sT(X^\sA,\sw) \ar[rr]^{\Stab} \ar[dr]_{\Stab} & & H^\sT(X,\sw) \,, \\
 & H^\sT(X^{\sA'},\sw) \ar[ur]_{\Stab} &
}
\end{equation}
commutes (see Theorem \ref{tri lem for coh stab}), where a consistent choice of attracting directions is understood. 
There is a parallel statement in critical $K$-theory, see
\S \ref{sec on tri lem}. A certain slope shift occurs in our
$K$-theoretic triangle lemma because of the way we normalize stable
envelopes. While very basic, the triangle lemma is the real reason
many stable envelope constructions work. For example, it is the reason
the \textit{Yang-Baxter equation} \eqref{Yang Baxter equation} holds for $R$-matrices constructed using stable envelopes. 

Other general properties of critical stable envelopes include their
compatibility with several standard constructions in critical cohomology and 
$K$-theory. One example is the compatibility with the Hall
correspondence stated in \eqref{Hall compatible intro}. 

Another example is the \textit{dimensional reduction} or, more
generally, \textit{deformed dimensional reduction}. This refers to the following
situation. Suppose that $X$ is the total space of an equivariant
vector bundle over a base $Y$, and suppose that the function $\sw$ has the form 
$$
\sw_X = \langle s, \, \cdot \, \rangle + \sw_Y \,.
$$
Here $s$ is a section of the dual bundle, which defines a
fiber-wise linear function $\langle s, \, \cdot \, \rangle$, and
$\sw_{Y}$ is a function pulled back from the base $Y$. The two extreme
cases here is when $s$ is regular or when $\sw_{Y}=0$, in either of which cases there
is an isomorphism between $H^\sT(X,\sw_{X})$ and
$H^\sT\left(\{s=0\},\sw_{Y}\right)$ (Theorem \ref{thm: def dim red reg sec}, \cite{Dav}). To interpolate between these two
extremes,
we consider the notion of a \textit{compatible pair} of dimensional reduction
data (Definition \ref{compatible dim red data}), which always induces an isomorphism between the corresponding
critical theories (Proposition \ref{prop: coh compatible map is iso}). Since the $\sA$-fixed part of any compatible pair
is again compatible (Lemma \ref{lem on cpt on afix}), it makes sense to ask whether critical stable
envelopes commute with dimensional reduction. And, indeed,
this is what we verify below, both in cohomological and $K$-theoretic case. 
\begin{Theorem}
$($Theorems \ref{dim red and stab_coh}, \ref{dim red and stab_K}, Example \ref{ex: dim red}, Remark \ref{rmk on dim red is}$)$
Critical stable envelopes are compatible with interpolations between compatible dimensional reduction data. 
As a consequence, they are compatible with deformed dimensional reductions, and in particular, critical stable envelopes on tripled quivers with canonical cubic potentials reproduce stable envelopes of Nakajima quiver varieties \cite{MO, Oko}.
\end{Theorem}

In \S \ref{sect on stab and vb}, we investigate the relation between stable
envelopes for $X$ and a total space of a $\sT$-equivariant bundle $E$ over
$X$. In the situation when $E=E_{+} \oplus E_{-}$, where the bundles
$E_{\pm}$ are attracting (resp.~repelling) for a given chamber
$\fC$, we show that stable envelopes for the base and the total
space exist synchronously, and are related in a simple fashion.
This is used to relax the 
$\mathbf{d}_{\textup{in}}=\mathbf{d}_{\textup{out}}$ part in our
symmetry condition for the quiver $Q$. For the stability condition in
which the maps from the framing and the quiver maps generate the
spaces $V_i$, we show it is enough to assume
$\mathbf{d}_{\textup{in}}\geqslant \mathbf{d}_{\textup{out}}$\,.
The sign here is correlated with the sign of the shift in the quantum
group. It means that our approach builds antidominantly shifted quantum loop groups, see below. 

\subsection{Specializations}
\textit{Specialization maps} play a very important role in both geometric
representation theory and its enumerative applications. Working, for
concreteness, in equivariant $K$-theory, the setting for the
specialization map is the following. Let $U$ be a smooth affine
algebraic curve with a
point $0\in U$. Let $u\colon X \to U$ be
a $G$-equivariant map, where $G$ acts trivially on $U$. The specialization map
$$
K^G(u\ne 0) \to K^G(u=0)
$$
takes the $K$-theory of the generic fibre of $u$ to the $K$-theory of
special fibre. In geometric representation theory, many
computations are done by showing that the action of correspondences
commutes with specialization, see \cite{CG}.

In particular, noncritical stable
envelopes commute with specialization, which is another way to say
that they behave well in families. Since the torus $\sA$ does not act
on $U$, this is immediate from the degree characterization of stable
envelopes. This argument also applies in critical theories, as soon as
a specialization map is available. 

Critical $K$-theory is a quotient of coherent $K$-theory classes by perfect $K$-theory classes,
which the specialization map \emph{fails} to preserve, in general.
Therefore, one should not expect a useable all-purpose specialization map
in critical $K$-theory and further assumptions need to be made.

In this paper, we focus on the case when the space is fixed, while the
potential $\sw(x,u)$ varies. 
Importantly, we assume that $\sw(x,u)$ can be
\emph{scaled} to $\sw(x,0)$, which means we
consider a function $\sw(x,u)$ on
the product $X \times \bA^1_u$, such that there exist an action of $\C^*_u = \{u \ne 0\}$ on
$X$ that commutes with $\sT$ and leaves $\sw(x,u)$ invariant (Setting \ref{setting of def of potentials}). 
With additional assumption that the action of $\C^*_u$ on $X$ is such that
$\lim_{u\to \infty} u\cdot x$ exists for every $x\in X$, we define a \textit{specialization map} in critical cohomology (Definition \ref{def of sp crit_coh}): 
$$\mathsf{sp}\colon H^\sT(X,\sw(x,1))\to H^\sT(X,\sw(x,0)).$$
There are ample examples satisfying the assumption, see Example \ref{ex of sp for crit_coh}.

To define the specialization map in critical $K$-theory, we make further assumption about the action of $\C^*_u$ on the ambient stack, see Section
\ref{sec: sp for crit k existence by dim red} for details. 

When specialization maps exist, they are compatible with \textit{canonical maps} \eqref{can_coh}, \eqref{equ on kgps} 
and are \textit{functorial} with respect to \textit{proper pushforward} and \textit{lci pullback} 
(Propositions \ref{prop on sp comm with stab coho}, \ref{prop on sp comm with stab}). Importantly, we have
\begin{Theorem}
$($Propositions \ref{prop: stab comp with sp crit}, \ref{prop on sp comm with stab}$)$
Critical stable envelopes are compatible with specialization maps.
\end{Theorem}
This provides powerful tool to relate quantum group modules for different potentials, see \cite{COZZ1}.

\subsection{Quantum groups}\label{s QG intro}
By definition, \textit{quantum groups} are Hopf algebra deformations of
$\scrU(\fg)$, or other Hopf algebras appearing in classical Lie theory.
Here $\fg$ is a Lie algebra and $\scrU(\fg)$ denotes its universal
enveloping algebra. Quantum loop algebras $\scrU_t(\widehat{\fg})$ are Hopf algebra
deformations of $\scrU(\widehat{\fg})$, where $\widehat{\fg}=\fg[u^{\pm 1}]$ is the Lie
algebra of Laurent polynomials with values in $\fg$, with point-wise
commutator. This deformation is further required to preserve the
$GL(1)$-action on $\scrU(\widehat{\fg})$ that scales $u$. Precomposing
with these automorphisms, we get a $1$-parameter family of modules
$V(u)$ from any $\scrU_t(\widehat{\fg})$-module $V=V(1)$.

In geometry, the deformation parameters $t$ come from
$\mathbf{h}_{\sT/\sA}(\pt)$. In the setting of \cite{MO}, the torus
$\sT/\sA$ was often 1-dimensional, with a coordinate $\hbar$ given by the
weight of the symplectic form. In the setting of this paper, $\sT/\sA$ needs to preserve potential functions. 

A crucial feature of the quantum deformation is the loss of the
\textit{cocommutativity of the coproduct}. In other words, the permutation of
factors is no longer an module map for a tensor product and,
moreover, $V_1 \otimes V_2 \not \cong V_1 \otimes_{\textup{opp}} V_2$
as a $\scrU(\widehat{\fg})$-module, in
general. Here $\otimes_{\textup{opp}}$ is the opposite coproduct. It turns out, however, that there exists a map
$$
R_{V_1,V_2}(u_1/u_2)\colon V_1(u_1)\otimes V_2(u_2) \to  V_1(u_1) \otimes_{\textup{opp}}
V_2(u_2),
$$
which is a \emph{rational} function of $u_1/u_2$ and a module isomorphism
for generic $u_1/u_2$. This is called the $R$-\textit{matrix} and is closely
related the braiding in the tensor category of quantum group modules.
The braid relation $(12)(23)(12)=(23)(12)(23)$ manifests itself as
the \textit{Yang-Baxter equation}
\begin{equation}
  \label{Yang Baxter equation}
  R_{12}(u_1/u_2)  R_{13}(u_1/u_3)  R_{23}(u_2/u_3)  =
   R_{23}(u_2/u_3)   R_{13}(u_1/u_3) R_{12}(u_1/u_2) \,, 
 \end{equation}
 or one of its modifications, such as the \textit{dynamical} Yang-Baxter
 equation, e.g.~\cite{COZZ2} or the nondynamical, but \textit{slope-shifted}
 equation \eqref{Yang Baxter equation s} below. In \eqref{Yang Baxter equation}, $R_{ij}$ is a shorthand for $R_{V_i,V_j}$ acting in the
 corresponding factors of the triple tensor product.

One efficient way to construct the Hopf algebra $\scrU_t(\widehat{\fg})$
is to start from constructing a suitable tensor category of its
modules. Following the approach of \cite{Resh90}, such category may be constructed from
a collection of operators $R_{V_i,V_j}(u)$ satisfying the Yang-Baxter
equation. The quantum group operators appear in this scenario as matrix elements
of $R$-matrices. Namely, for any vector and covector in $V_1$, the
corresponding matrix element of $R_{V_1,V_2}(u)$ is a rational
function of $u$ with values in endomorphisms of $V_2$. The
coefficients of its series expansion around $u=0,\infty$ define
individual quantum group operators. These satisfy commutation and
cocommutation relations as a consequence of the Yang-Baxter
equation.

There is an additive analog of this story, in which one
deforms $\scrU(\fg[u])$, the group
operation in \eqref{Yang Baxter equation} is replaced by $u_i - u_j$,
and the result is the Yangian $Y(\fg)$. 

To connect this to geometry,  we introduce a countable union
$$
X(\mathbf{d}) = \bigsqcup_\mathbf{v} R(Q,\mathbf{v}, \mathbf{d}) /\!\!/
G(\mathbf{v})
$$
of quiver varieties over all dimension vectors $\mathbf{v}$, with
the framing dimensions $\mathbf{d} = \mathbf{d}_{\textup{in/out}}$ and
the quiver $Q$ fixed. Clearly, direct sum of quiver representations embeds
\begin{equation}
X(\mathbf{d})  \times X(\mathbf{d'}) \xrightarrow{\quad \oplus \quad}
X(\mathbf{d}+\mathbf{d'}), \label{direct sum intro}
\end{equation}
as a fixed locus for a $1$-dimensional torus $\sA$ acting with
weight one in the unprimed framing spaces. 
In the situation when
stable envelopes exist, e.g.\ for symmetric quivers, we \emph{declare} the
stable envelope map to be a morphism in our future module category.
Since stable envelopes are isomorphisms after localization, this
immediately gives a rational $R$-matrix $R_{\mathbf{d},
  \mathbf{d'}}(a)$, the spectral parameter in which is the generator
of equivariant cohomology of $\pt/\sA$. The triangle lemma implies
a form of the YB equation, where the group law in the YB equation is naturally
identified with the group law of the cohomology theory, that is, the
algebraic group defined by the Hopf algebra $\mathbf{h}_{GL(1)}(\pt)$. Thus we
get a Yangian in equivariant (critical) cohomology, a quantum loop group in
equivariant (critical) $K$-theory (and an elliptic quantum group in elliptic cohomology which, however, will remain outside the confines of this
paper).

More precisely, when we talk about equivariant $K$-theory, there is a
choice between the algebraic and topological $K$-theory. While either choice
has its advantages and its scope of applications, in this paper we opt
for algebraic equivariant $K$-theory. In algebraic theory, the natural map
$K(X_1) \otimes K(X_2) \to K(X_1 \times X_2)$ is, in general, very
far from an isomorphism. It is, however, an isomorphism in examples of maximal
interest to us in this paper. Otherwise, the right thing to do is
taking $K(X_1 \times X_2)$ to be tensor product object in our module 
category.

While the lists of explicit generators and relations for quantum loop groups  
may be enormous and somewhat uninspiring, the corresponding categories of
modules are often known or expected to be generated by a nice set of
morphisms of geometric origin. This highlights both theoretical and
practical advantages of the above hands-off way to construct
quantum group actions.

\subsection{Shifted quantum supergroups}\label{subsection on shifted qg}

While efficient, the above construction of the quantum group does not immediately give
a good control over the size of the resulting quantum group. In the
setting of \cite{MO}, it was shown that there is an Reshetikhin type Yangian such that 
\begin{equation}
  \label{gr Yangian intro}
  \mathrm{gr}\ \textup{Yangian} \cong \scrU(\fg[u]) 
\end{equation}
where $\fg$ is the Lie algebra spanned by the coefficients of the
classical $R$-matrix $\pmb r$, which appears as the $u^{-1}$ coefficient in
\begin{equation}
  \label{R in u intro}
  R(u) = 1 + \frac{\pmb r}{u} + O(u^{-2})\,, \quad u \to \infty
  \,. 
\end{equation}
The filtration in the left-hand side of \eqref{gr Yangian intro} is by
how far down the $u\to \infty$ expansion a given element appears,
counting from $u^{-1}$. In particular,  $\scrU(\fg)$ is the 0th term in
this filtration. About the Lie algebra $\fg$ itself, it was conjectured by one of us that the graded multiplicities of its roots are given by the Kac polynomials of $Q$ (generalizing the famous conjecture
of Kac concerning the constant term of the Kac polynomial). This conjecture was recently proven by Botta and Davison in \cite{BD} and Schiffmann and Vasserot in \cite{SV}. This was later generalized to
a conjectural isomorphism between the positive half of $\fg$ and the Davison-Meinhardt BPS Lie algebra, which was also proven in \cite{BD}.


In our current, more general setting, we prove the following result.
\begin{Theorem}
$($\cite[\S 4]{COZZ1}$)$
Given a symmetric quiver  $Q$ with potential $\sw$, and $\mu\in \bZ_{\leqslant 0}^{Q_0}$, there is a Reshetikhin type $\mu$-shifted Yangian $Y_{\mu}(Q,\sw)$.
When $\mu=0$, it has a filtration whose associated graded satisfies 
\begin{equation}
  \label{gr Yangian mu intro}
  \mathrm{gr}\ Y_{0}(Q,\sw) \cong \scrU(\fg_{Q,\sw}[u]) \,, 
\end{equation}
where $\fg_{Q,\sw}$ is a Lie superalgebra such that the 
$$
\textup{parity of } \fg_\alpha = \sum \alpha_i \cdot (\dim
Q_{ii} +1)  \mod 2 \,. 
$$
Here $\fg_\alpha$ is a  root subspaces in 
\begin{equation}
\fg_{Q,\sw}= \fg_0 \oplus \bigoplus_{\alpha\gtrless 0}
\fg_{\alpha} \,. \label{root g intro}
\end{equation}
The Cartan subalgebra $\fg_0$ in \eqref{root g intro} has rank twice the
number of nodes and records the dimension vectors $\mathbf{v}$ and
$\mathbf{d}$, the latter half being central. The root subalgebras
$\fg_{\alpha}$ are finite-dimensional and act by changing $\mathbf{v} \mapsto \mathbf{v} +
\alpha$.
\end{Theorem}
In \cite[\S 4]{COZZ1}, we also prove other basic structural results about \eqref{root g intro}, including an identification of $\fg_{Q,\sw}$ with 
MO Lie algebra \eqref{gr Yangian intro} for tripled quivers with canonical cubic potentials. 
We conjecture an isomorphism between the 
positive half of $\fg_{Q,\sw}$ and the corresponding BPS Lie
superalgebra, generalizing the previous conjecture.

In the context of \cite{MO}, the leading term $1$ in \eqref{R in u intro} appears because stable envelopes are normalized using a
polarization. In our more general setting, the leading asymptotics of the
$R$-matrix comes from the diagonal terms of stable envelopes, that is,
from the attracting $\Attr_+$ and repelling $\Attr_-$ manifolds of the
fixed locus $X^\sA$. In the situation of \eqref{direct sum intro},
assume that the quiver is self-dual with the exception of
$\mathbf{d}'_{\textup{in}} \ne \mathbf{d}'_{\textup{out}}$, the torus
$\sA$ acts with weight $1$ in the unprimed framing spaces, and
the attracting chamber $\fC$ is $\{u>0\} \subseteq  \Lie \sA$. Then
\begin{alignat*}{3} 
  \Attr_+ & = \big(&&\cdots \cdots \big) && \oplus \bigoplus_i
  \Hom(D'_\textup{in}, V_i)\,, \\
  \Attr_-& = \big(&&\cdots \cdots \big)^\vee && \oplus \bigoplus_i
  \Hom(V_i, D'_\textup{out})\,,
\end{alignat*}
where dual denotes the $\sA$-equivariant dual for vector bundles.
Since $\Attr_{\mp}$ is the normal bundle to the total space of
$\Attr_{\pm}$, the diagonal terms in stable envelopes are related by
\begin{equation}
  \label{Stab/Stab intro}
 \frac{\Stab_{+} \big|_{\diag X^\sA }} {\Stab_{-} \big|_{\diag X^\sA }
 } = (-1)^{\rk \Attr_-} u^{\mathbf{v}\cdot \mu} \cdot (1+O(u^{-1}))\,, \quad \mu =
 \mathbf{d}'_{\textup{out}} - \mathbf{d}'_{\textup{in}} \,. 
\end{equation}
Both the sign and the monomial in $u$ are important here.

More precisely, in \cite{MO} one factors out the weight of the
symplectic form $\hbar$ from the classical $R$-matrix $\pmb r$,
making its cohomological degree vanish. In our setting, the classical
$R$-matrix may not be divisible by a single equivariant variable, and we
leave to have cohomological degree $2$ in the fully symmetric
case. This means makes the commutation relation in $\fg$ depend on
equivariant variables.

The monomial in \eqref{Stab/Stab intro} 
results in the Yangian being a \emph{shifted Yangian}, and similarly for the
quantum loop group. Recall that for the stability condition in
which the maps from the framing and the quiver maps generate the
spaces $V_i$ our stable envelops exists when 
$\mathbf{d}_{\textup{in}}\geqslant \mathbf{d}_{\textup{out}}$. This
corresponds to the \textit{antidominant shift} for quantum groups (i.e.~$\mu\in \bZ_{\leqslant 0}^{Q_0}$).
We expect this to be the exact generality in which all of our
constructions work.

In equivariant $K$-theory, one has to look at  both the $u\to 0$ and
the $u\to \infty$ asymptotics of $R(u)$, where $u$ takes values in torus
$\sA$ itself. The determinants of attracting/repelling bundles now
enter the asymptotics. These are line bundles on the fixed loci and
their appearance leads to slope shifts in the triangle lemma
\eqref{triangle intro} and to the dynamical shifts in the Yang-Baxter
equation. We show in \cite{COZZ2} that these dynamical shifts can be gauged away almost
completely, producing the following generalization of the equation \eqref{Yang Baxter equation}:
\begin{equation}
  \label{Yang Baxter equation s}
  R^s_{12}(u_1/u_2)  R^{s-\frac12\mu_2}_{13}\!(u_1/u_3)  R^s_{23}(u_2/u_3)  =
   R^{s-\frac12\mu_1}_{23}\!(u_2/u_3)   R^s_{13}(u_1/u_3) R^{s-\frac12\mu_3}_{12}(u_1/u_2).
 \end{equation}
Here $s$ is an
 arbitrary slope and $\mu_i$ is the quantum group shift $\mathbf{d}_{\textup{out}} -
 \mathbf{d}_{\textup{in}}$ in the $i$th factor.

 As another illustration of the power and flexibility of \textit{critical
theories}, one can construct geometrically Lie algebras with a
symmetrizable, but \textit{not} symmetric, Cartan matrices, see \cite[\S 4.3]{COZZ1}.

\subsection{Quantum Knizhnik-Zamolodchikov equations}

To any solution of the Yang-Baxter equation one can associate a
canonical flat difference (or $q$-difference, in the quantum loop
group case) connection, namely
the \textit{quantum Knizhnik-Zamolodchikov} connection of I.~Frenkel and
N.~Reshetikhin \cite{FR}.  The corresponding commuting difference
operators are best represented pictorially as in the following
illustration of the $(u_1,u_2,u_3,u_4) \mapsto
(u_1,u_2+\varepsilon,u_3,u_4)$
shift: 
\begin{equation}
  \label{qKZ picture}
 \raisebox{-2cm}{ \begin{tikzpicture}
  \tikzset{
    partial ellipse/.style args={#1:#2:#3}{
        insert path={+ (#1:#3) arc (#1:#2:#3)}
    }
}
 \draw[thick] (0,0) [partial ellipse=215:360:2.5 and 0.5];
 \draw[dotted] (0,0) [partial ellipse=0:180:2.5 and 0.5];
 \draw[thick] (0,0) [partial ellipse=180:210:2.5 and 0.5];
 \draw[thick] (0,3) [partial ellipse=215:360:2.5 and 0.5];
 \draw[thick] (0,3) [partial ellipse=0:210:2.5 and 0.5];
\draw[thick] (-2.5,0)--(-2.5,3);
\draw[thick] (2.5,0)--(2.5,3);
\draw[very  thick,blue,-latex] (0.5,-0.49)--(0.5,3 -0.49);
\draw[very  thick,blue,-latex] (1.5,-0.4)--(1.5,3 -0.4);
\draw[very thick,blue,-latex] (-1.5,-0.4)--(-1.5,3 -0.4);
\draw[very thick] (-2.0478,-0.2867)--(-2.0478,3-0.2867);
\draw[very thick] (-2.165,-0.25)--(-2.165,3-0.25);
\node at (-2.05,2.95) {\scalebox{0.9}{$Z$}}; 
\draw[very  thick,red]  (-0.5,-0.49) .. controls (-0.5,0.5) and (2.5,0) .. (2.5,1);
\draw[dotted,thick,red] (2.5,1) .. controls (2.5,2) and (-2.5,3) .. (-2.5,2);
\draw[very thick,red,-latex]  (-2.5,2) .. controls (-2.5,1) and (-.5,
2-.49) .. (-0.5,3-.49);
\node[blue] at (0.5,-0.49-0.3) {\scalebox{0.9}{${V_3(u_3)}$}}; 
\node[red] at (-0.5,1-0.49-0.3) {\scalebox{0.9}{${V_2(u_2)}$}}; 
\node[red] at (-0.5,3-0.49+0.33) {\scalebox{0.9}{${V_2(u_2+\varepsilon)}$}}; 
\node[blue] at (1.5,-0.4-0.3) {\scalebox{0.9}{${V_4(u_4)}$}}; 
\node[blue] at (-1.5,-0.4-0.3) {\scalebox{0.9}{${V_1(u_1)}$}};
\draw[very thick,blue,-latex] (4.5,0)--(3.5,1); 
\draw[very thick,red,-latex] (3.5,0)--(4.5,1);
\node [anchor=west] at (4.5,0.5) {$=R_{V_2,V_i}(u_2 - u_i)$};
\draw[very thick] (4.4,1.5)--(3.5,2.4); 
\draw[very thick] (4.5,1.6)--(3.6,2.5); 
\draw[very thick,red,-latex] (3.5,1.5)--(4.5,2.5);
\node [anchor=west] at (4.5,2) {$=1\otimes Z \otimes 1 \otimes 1$};
\end{tikzpicture}}
\end{equation}
The qKZ operators act in the tensor product $V_1(u_1) \otimes \cdots
\otimes V_n(u_n)$ of quantum groups modules, where the parameters
$u_i$'s are treated as variables. In down-to-earth terms, the difference
operators act on function of $(u_1,\dots,u_n)$ with values in
$V_1\otimes \cdots\otimes V_n$. As we follow the $k$th strand around
the cylinder in \eqref{qKZ picture}, where $(n,k)=(4,2)$, we apply
$R$-matrices in the corresponding tensor factors, as well as an operator
$Z$ that implements quasi-periodicity around the cylinder. The
operator $Z$ needs to satisfy
$$
[Z \otimes Z, R ] = 0
$$
and is usually chosen in the form
\begin{equation}
Z =z \cdot
(\textup{shift $u\mapsto u+\varepsilon$}) \,, \quad z \in
\exp(\fg_0)\,, \label{Z qKZ intro}
\end{equation}
where $\fg_0$ is the Cartan subalgebra in \eqref{root g intro}. In the
multiplicative, $K$-theoretic situation the shift operator becomes $u
\mapsto qu$. Here $\varepsilon$ and $q$ are free parameters, the
geometric meaning of which will be explained momentarily.

The fundamental link between the geometric representation
theory and enumerative geometry is provided by the following geometric
interpretation of the qKZ connection. If one is interested in counting
maps, or quasimaps $f\colon C \to X$, one can broaden the setup and the
range of available tools by counting sections, or
quasisections, of nontrivial $X$ bundles over $C$. This becomes
especially constraining in the situation when the bundle is
equivariant with respect to the action of $\sT$ and the
automorphisms of $C$. 

Specifically, given any cocharacter of $\sT$: 
\begin{equation}\label{equ on cochar}\sigma\colon \C^*\to \sT,  \end{equation}
we can use it as a clutching function to construct
$(\sT\times \bC^*_q)$-equivariant bundle over $C=\bP^1$, where
$\bC^*_q=\Aut(\bP^1,0,\infty)$. When $\sigma$ is a cocharacter of
$\sA$ and the quiver variety $X$ is symmetric, the $\sigma$-twisted
quasimaps have the same virtual dimension and the same self-duality
properties of their obstruction theory as the untwisted ones. We count
them \emph{relative} to the evaluation at the two fixed points
$0,\infty \in C$. For fixed $\deg f$, this gives an operator from
$H^{\sT\times \bC^*_q}(X,\sw)$ to
itself, and similarly in $K$-theory. Note, however, that because the count
is twisted, $\bC^*_q$ acts by $\sigma$ in the target of this map,
while acting trivially in its source.

Summing up these operators with weight $z^{\deg f}$ gives a flat
difference connection on $\Lie \sA$ or $\sA$ in cohomology or
$K$-theory, respectively. These difference operators are known as the
\emph{shift operators}, see \cite{MO,Oko}. The corresponding shifts are by
$\sigma(\varepsilon)$ and $\sigma(q)$, respectively, where
$\varepsilon \in \Lie GL(1)_q$.

In \cite{COZZ2}, we consider the shift operator in critical theories, here say in critical $K$-theory:
\begin{equation}\label{equ on shift oper}
\sS_\sigma \in \End K^{\sT \times \bC^*_q} (X, \sw)_\loc  [\![z]\!],  
\end{equation}
where $\bC^*_q$ is the torus scaling the distinguished $\mathbb{P}^1$ in parametrized rational curves. 
We identify the shift
operator connection with the qKZ connection for any torus $\sA$
generated by \textit{minuscule} cocharacters (i.e.~$\C[X]$ is generated by elements of degree $0$ or $\pm 1$).

\begin{Theorem}
$($\cite{COZZ2}$)$
Let $(X,\sw)$ be a symmetric quiver variety with potential, $\sigma$ be a minuscule cocharacter, and 
$s$ be a slope in certain area\,\footnote{
$\left\{\text{sufficiently small neighbourhood of } (\det P_X)^{-1/2} \right\} \cap (-C_{\mathrm{amp}}(X) ) \subseteq  \Pic_\sA (X) \otimes_\bZ \bR$, 
where $P_X$ is a partial polarization of $X$ and $C_{\mathrm{amp}}(X)$ is the ample cone of $X$, see \cite{COZZ2} for details.}. Then the conjugation of $\sS_\sigma$ is given by 
$z^{\deg f}$ (up to some locally constant function).

Moreover there is a capping operator 
$J\in  \End K^{\sT \times \bC^*_q} (X^\sigma, \sw)_\loc [\![q]\!][\![z]\!]$ that solves the qKZ equation: 
$$\Psi|_{a \mapsto q^\sigma a} = z^{\deg f_\sigma} R^{\mathsf{s}}_{\sigma} \cdot \Psi, 
$$
where $R^{\mathsf{s}}_{\sigma}$ is certain (normalized) $R$-matrix.
\end{Theorem}

This has dimensional reduction to Nakajima quiver varieties, generalizing  results of \cite{MO, Oko}.
Here minuscule is an important
techical condition, which is satisfied by tori used in
\eqref{direct sum intro}, that is, tori $\sA$ that act on framing spaces
preserving a direct sum decomposition. As a small technical detail, the
identification of two connections requires a multiplicative shift of
the variables $z$, which was called \emph{modified} quantum product
in \cite{MO}.

Note that the curve counting monomials $z^{\deg f_{\sigma}}$ become the weights
of action of $z$ in \eqref{Z qKZ intro}. From the quantum groups point
of view, this is a minor diagonal part of the qKZ equation, the main complexity
of which is contained in the $R$-matrix. From the enumerative
perspective, this  looks very surprising, since a monomial in $z$
appears where one generally expects an infinite series. The geometric
explanation for this is that only \emph{constant} maps contribute to
properly formulated twisted quasimaps counts. Note that a constant
twisted quasimap takes $C$ to one of the components of $X^\sigma$, and
these maps have \emph{different} degrees for different
components. Whence the appearance of monomials with different exponents
in the qKZ operators.

Vanishing of contribution of nonconstant maps is
shown using, fundamentally, equivariant rigidity. To make the rigidity
argument work, all key ingredients of the construction of the
stable envelopes are required. The properness of stable envelopes is used to prove that
the count is a polynomial on $\Lie \sA$ (or $\sA$ itself, in $K$-theory),
while the degree bounds imply that degree of this polynomial is
negative (or that its Newton polytope contains no lattice points, in
$K$-theory).
\subsection{ Quantum critical cohomology}

Fundamental structures in enumerative geometry include the quantum
differential equation in cohomology, and its $q$-difference analog in
$K$-theory. These commute with shift operators, including the
qKZ connection discussed above, which very strongly contrains them.
Using these constraints, the quantum differential equations for
Nakajima varieties were identified with the Casimir connection \cite{TL}
for the corresponding Yangian in \cite{MO}, while the $K$-theoretic
$q$-difference connection was identified with the \emph{dynamical
  connection} for the corresponding quantum loop group in \cite{OS}.
When $\fg$ is finite-dimensional, the dynamical connection is the
lattice part of the dynamical affine Weyl group of Etingof and
Varchenko \cite{EV}. In general, the corresponding connection was
constructed in \cite{OS}.

The linear operator in the quantum differential equation is the
operator of  \textit{modified quantum multiplication} $c_1(\lambda)\,\widetilde{\star}$ 
by the first Chern class of the bundle $\prod (\det V_i)^{\lambda_i}$,
where the modification is a certain sign shift of the variables $z$. 
The following general formula for the operator $c_1(\lambda)\,\widetilde{\star}\,  \cdot$ on
Nakajima varieties was proven in \cite{MO}:
\begin{equation}
  \label{quantum product intro}
   c_1(\lambda)\,\widetilde{\star}\, \cdot = c_1(\lambda)\cup \cdot - \sum_{\theta \cdot \alpha
     >0} (\lambda,\alpha) \frac{z^\alpha}{1-z^\alpha} \,
   \textup{Casimir}_\alpha + \mathrm{scalar}\,,
 \end{equation}
 where
 $$
 \textup{Casimir}_\alpha  = \textup{multiplication}
 (\pmb r_\alpha)\,, \quad \pmb r_\alpha \in \fg_\alpha \otimes
 \fg_{-\alpha}  \,. 
 $$
 Here $\theta$ is the stability parameter, $\pmb r_\alpha$ is projection of $\pmb r$ on the
 corresponding root subspaces, and the multiplication map takes
 $\fg \otimes \fg$ to $\scrU(\fg) \subseteq  Y(\fg)$. The dots in
 \eqref{quantum product intro} stand for a scalar, which is uniquely
 fixed by the requirement that the purely quantum part of
 $c_1(\lambda)\,\widetilde{\star}$ annihilates the identity $1\in H^\sT(X)$.
 For comparison with \cite{MO}, note that our definition of
 $\pmb r$ includes the factor $\hbar$, and which makes
 the the cohomological degree of $\textup{Casimir}_\alpha$ equal 2.

 Building on the geometric identification of the qKZ connection in critical cohomology, we prove
\begin{Theorem}
$($\cite[\S 5]{COZZ1}$)$
 \eqref{quantum product intro} holds for critical cohomology on any symmetric quiver variety $X$ with potential when the specialization map to the cohomology of $X$ is injective. 
\end{Theorem}

 The $\mathbf{d}_{\textup{in}} > \mathbf{d}_{\textup{out}}$ counts 
 may be accessed from the fully symmetric counts by a certain limit
 transition. This is analogous to how the Toda equations, which
 describe the quantum cohomology of the flag varieties $G/B$ \cite{GK, Kim} can be
 obtained by a limit transition from the Calogero-style equations
 describing the quantum cohomology of $T^*(G/B)$ \cite{BMO}.
 This is well illustrated by the following example.

 \subsection{Quantum cohomology of $\Hilb(\C^3,n)$}\label{sect intro on qc of hilbC3}

The Hilbert scheme $\Hilb(\C^3,n)$ of $n$-points on $\C^3$ has a canonical presentation as the critical locus of the cubic function 
$$\sw=\tr(x[y,z])\colon X=\left(\Hom(\C^n,\C^n)^3\oplus \Hom(\C,\C^n)\right)/\!\!/ \mathrm{GL}(n)\to \C. $$
Let $\bC^*_{q_1},\bC^*_{q_2},\bC^*_{q_3}$ be the tori that scale the loops $x,y,z$ with weights $-1$ respectively. 
Set $$\sT=\ker\left(\bC^*_{q_1}\times \bC^*_{q_2}\times \bC^*_{q_3}\to\bC^*\right),\quad (t_1,t_2,t_3)\mapsto t_1t_2t_3. $$ 
Let $\hbar_i$ be the equivariant parameter for $\bC^*_{q_i}$, then $\bC[\mathsf t]=\bC[\hbar_1,\hbar_2,\hbar_3]/(\hbar_1+\hbar_2+\hbar_3)$.

Although $X$ is not symmetric, one can symmetrize it by adding a path from the gauge node 
to the framing node, and introduce an equivariant parameter $\hbar$. By taking certain limit of $\hbar$, we obtain the formula of 
quantum multiplication by divisors for $\Hilb(\C^3,n)$.
 
\begin{Theorem}
$($\cite[\S 10]{COZZ1}$)$
For an equivariant line bundle $L$ on $X$, we have 
$$c_1(L)\, \widetilde{\star}_{d}\, \cdot =\sigma_3\deg(L)\, d\cdot J_{-d}\, J_{d}+\mathrm{scalar}\,, \quad \forall\,\,d>0.$$
Here $\sigma_3=\hbar_1\hbar_2\hbar_3$, $J_i=\frac{\sigma_3^{-i}}{(i-1)!} \text{ad}^{i-1} _{-f_1} f_0$, $J_{-i}=-\frac{\sigma_3^{-i}}{(i-1)!} \text{ad}^{i-1} _{e_1} e_0$ for $i>0$, and $e_i$, $f_i$ are parts of the generators of shifted Yangian $Y_{-1}(\widehat{\mathfrak{gl}}_1)$, see \cite[\S 10.1]{COZZ1}, which acts on the critical cohomology. The scalar vanishes for $d>1$.
\end{Theorem}

\subsection*{Related works and future directions} 

As already mentioned, many current trends in the field may be traced to the influential work of Nekrasov and Shatashvili \cite{NS}. Among their predictions was the identification of the operators of quantum multiplications with the image of Baxter/Bethe subalgebras of certain quantum groups in specific representations.
Concretely, Nekrasov and Shatashvili computed the spectra of quantum multiplication operators and saw they satisfy Bethe-type equations. 

While the mathematical foundations of enumerative 
theory of critical loci were yet to be laid down, and the relevant quantum 
groups and their representation were, with a few exceptions, yet to be constructed (using precisely the stable envelopes, including our results), the Nekrasov-Shatashvili computation of the spectra of quantum operators works very generally. In particular, it applies to the enumerative problems studied in the present work. Looking at the Bethe equations, Nekrasov and Shatashvili predicted the appearance of what Hernandez and Jimbo called asymptotic, or prefundamental,  representations of quantum groups \cite{HJ}. In such representations, Drinfeld currents act by rational functions with unbalanced numerator and denominator, which is by now well understood to be a real signature of a shifted quantum group. 

The asymptotic representations of Hernandez and Jimbo arise via a limit procedure which, among its many incarnations in mathematical physics, is related to taking the length of a spin chain to infinity. As a parallel construction in enumerative geometry, one may want to approximate maps, or quasimaps, from a curve to some target $X$ by finite jets. Moduli of regular maps $\Spec \C[t]/t^N \to X$, where $X=X(\bd)$ is a Nakajima variety, is naturally an open set in $\Crit(\bw)$, where $\bw$ deforms the canonical cubic potential for the Nakajima variety $X(N\bd)$ by quadratic terms. With different stability condition for $\Crit(\bw)$, one may approximate different quasimap moduli spaces. The corresponding limit procedure for critical cohomology or critical $K$-theory groups is parallel to how Hernandez and Jimbo approximate the prefundamental representations by Kirillov-Reshetikhin modules. At the time of the Spring 2018 MSRI program, this point of view was adopted in an unpublished work of Nakajima and Okounkov \cite{NO}. Several technical difficulties encountered by them were subsequently overcome in special cases in the PhD thesis of Henry Liu \cite{Liu}, which broke the ground in geometric construction of shifted quantum group actions in critical theories. In particular, critical $R$-matrices and asymptotic modules for shifted quantum groups make their appearance in \cite{Liu}, including applications to one-leg DT and PT counts. 

In this paper, we improve on this asymptotic approach in several aspects. On the one hand, we work directly with the critical cohomology or $K$-theory, without the need to push forward the computations to any ambient space or specialization. In representation theoretic terms, this means that we study both the limit objects and the Kirillov-Reshetikhin-type modules that approximate them abstractly, and not as submodules of some ambient modules like the tensor-product module corresponding to $X(N\bd)$. On the other hand, we can work directly with the critical locus description of quasimaps moduli spaces, so the whole machinery of approximation is no longer necessary to construct representations of the relevant shifted quantum groups.  Finally, we clarify both positive and negative results about specialization in critical theories, the absence of which stood in the way of applying several conventional geometric representation theory arguments in the critical context. 

Varagnolo and Vasserot introduced critical convolution algebras in \cite{VV2}, and they constructed maps from (shifted) quantum loop groups of a simple Lie algebra $\mathfrak{g}$ to $K$-theoretic critical convolution algebras of graded triple quiver varieties of the Dynkin quiver associated to $\mathfrak{g}$ (see \cite{VV2} for simply laced case and \cite{VV3} for non-simply laced case). In \textit{loc.\,cit.}, they showed that the prefundamental modules and Kirillov-Reshetikhin modules can be realized as critical $K$-theories of graded triple quiver varieties. Their construction of shifted quantum loop groups actions fits into our framework in the sense that those actions factor through  maps to the Reshetikhin type shifted quantum loop groups.

Interactions of geometric representation theory and enumerative geometry constitute a very broad field of study, in which the setting of the more traditional, noncritical theories occupies an important, but relatively small corner. There are therefore compelling theoretical and applied reasons for extending \emph{all} of the existing quantum 
group machinery to the setting of \textit{critical theories} and \textit{shifted} quantum groups. 
Among other things, this should include the identification of the quantum difference equations, extending the results of \cite{Oko}, 
a correspondence between relative and descendent insertions extending those found in \cite{AO1}, as well as 
a relation between vertex functions and nonabelian elliptic stable envelopes 
\cite{AO1}. While in some directions the path of this process may appear somewhat predictable, other directions present the researchers with genuinely new features and puzzles. 

Focusing on the latter, it appears both interesting and challenging to pinpoint the precise relation between representation theory and enumerative invariants in the situation when the shift is not antidominant. Both sides of the story here are 
well-defined in general and are directly related in the zero or antidominant shift case. 
We expect them to remain connected in general, but now in a more subtle way. 
Certainly, a connection of this form should contain some very interesting 
enumerative and representation-theoretic information. 

We similarly expect the shifted quantum group story to lead to many subtleties in the 
elliptic stable envelopes situation. This is because, traditionally, computation with elliptic 
objects tend to rely very heavily on self-duality features, and also because there is no simple way to get rid of unwanted variables in the formulas with elliptic functions
by making them go to $0$ or $\infty$.



\subsection*{Acknowledgments}
This work benefits from helpful discussions and communications from many people, including
Mina Aganagic, Daniel Halpern-Leistner, Tasuki Kinjo, Ryosuke Kodera, Yixuan Li, Yuan Miao, Hiraku Nakajima,
Andrei Negu\c{t}, Tudor P\u{a}durariu, Spencer Tamagni, Yukinobu Toda, Yaping Yang, Gufang Zhao, Tianqing Zhu. A.O. would like to thank SIMIS for hospitality. We would like to thank Kavli IPMU for bringing us together.

\subsection*{Statements and Declarations}
We have no conflicts of interest to disclose.

\vspace{1cm}

\section{Critical cohomology and \texorpdfstring{$K$}{K}-theory}

In this section, we recall the notions of critical cohomology, critical $K$-theory and their functorial properties. 
We discuss their excisions when closed subvarieties are  full attracting subvarieties with respect to tori actions.

\subsection{Attracting subvarieties}

We work under the following setting. 
\begin{Setting}\label{setting of stab}
Let $X$ be a smooth quasi-projective variety over $\bC$ with a torus $\sT$-action and $\sw\colon X \to \bC$ be a $\sT$-invariant regular function (called \emph{superpotential})\footnote{Note that by \cite[Cor.~2]{Sum}, every point of $X$ is contained in a $\sT$-invariant affine neighbourhood.}.
Let $\sA \subseteq  \sT$ be a subtorus. Define $\Fix_\sA(X)$ to be the set of connected components of the torus fixed locus $X^\sA$.

\end{Setting}

A cocharacter $\sigma \colon \bC^* \to \sA$ of $\sA$ is \emph{generic} if $X^{\sigma} = X^\sA$,~i.e.~the torus fixed loci coincide. 
There is a wall-and-chamber structure on $\Lie(\sA)_\bR$, such that $\sigma$ is generic if and only if it lies in a chamber. 

\begin{Definition}\label{def of roots and chambers}
The \textit{torus roots} are the set of $\sA$-weights $\{\alpha\}$ occurring in the normal bundle to $X^\sA$. 
A connected component of the complement of union of (finite) root hyperplanes is called a \textit{chamber}, i.e.
\begin{align*}
    \Lie(\sA)_\bR\bigg{\backslash} \bigcup_{\alpha\in \text{roots}}\alpha^{\perp}=\bigsqcup_j \fC_j,
\end{align*}
where $\fC_j$ are chambers.
\end{Definition}

For an algebraic variety $M$ with an $\sA$-action, and a cocharacter $\sigma$, let $S$ be a subset of $M^\sigma$, the \emph{attracting set} of $S$ in $M$ is 
$$
\Attr_{\sigma} (S)_M := \{ x\in M \mid \lim_{t\to 0} \sigma (t) \cdot x \in S\}.
$$
In the Setting \ref{setting of stab} and let $\fC$ be a chamber, we define
\begin{align*}
    \Attr_{\fC} (S)_M :=\Attr_{\xi} (S)_M
\end{align*}
for a subset $S\subseteq  X^\sA$ and $\xi\in \fC$. Note that the definition does not depend on the choice of $\xi$. When there is no ambiguity, we also write $\Attr_{\sigma} (S)=\Attr_{\sigma} (S)_M$ and $\Attr_{\fC} (S)=\Attr_{\fC} (S)_M$. Let $F \in \Fix_\sA(X)$ be a connected component, then 
$\Attr_\fC (F)$ is a locally closed subscheme in $X$, and admits an \textit{affine fibration} $p\colon \Attr_\fC (F) \to F$ by the result of Bialynicki--Birula \cite{BB}. 

Consider a partial order on $\Fix_\sA(X)$ which is the transitive closure of the following relation:
\begin{align}\label{flow order}
    F_i \preceq F_j \quad \text{if} \quad F_j \cap  \overline{\Attr_\fC (F_i)}\neq \emptyset.
\end{align}
The \emph{full attracting set} is defined as
$$
\Attr_\fC^f (F) := \bigcup_{F \preceq F'} \Attr_\fC (F'). 
$$
We denote by $\Attr_\fC^f$ the smallest $\sA$-invariant closed subset of $X\times X^\sA$ such that $\Attr_\fC^f$ contains the diagonal $\Delta\subseteq  X^\sA\times X^\sA$ and 
\begin{align*}
    (x',y)\in \Attr_\fC^f\text{ and }\lim_{t\to 0} \sigma (t) \cdot x=x' \text{ implies }(x,y)\in \Attr_\fC^f.
\end{align*}
It follows from definition that $\Attr_\fC^f$ is a subset of $\bigcup_{F \in \Fix_\sA(X)} \Attr_\fC^f (F) \times F \subseteq  X \times X^\sA$.


\subsection{Critical cohomology}\label{sect on crit coho}


Let $\D^b_c(X)$ be the bounded derived category of constructible sheaves of $\bC$-vector spaces. 

\subsubsection{Definition and canonical map}

There is a functor of \emph{vanishing cycles}:
$$
\varphi_\sw \colon \D^b_c (X) \to \D^b_c (\sw^{-1} (0)). 
$$
\begin{equation}\label{eqn:van}
\varphi_\sw(F):=\dR \Gamma_{\mathrm{Re}(\sw)\geqslant  0}(F)|_{\sw^{-1}(0)}. 
\end{equation}
Here we use an equivalent definition due to \cite[Ex.~VIII 13]{KS}. 

Let $\bD_X$ be the Verdier duality functor on $\D^b_c(X)$, with  dualizing sheaf
$$\omega_X := \bD_X \bQ_X \cong \bQ_X [2\dim X]. $$
The vanishing cycle functor commutes with Verdier duality 
$$\varphi_\sw \bD_X = \bD_X \varphi_\sw. $$ 

\begin{Definition}
The \emph{critical cohomology} for $(X, \sw)$ is defined as
\begin{equation}\label{equ on crit coho}
H(X, \sw) := H (X, \varphi_\sw \omega_X) \cong H_c (X, \varphi_\sw \bQ_X)^\vee. 
\end{equation}
\end{Definition}
Let $\psi_\sw\colon \D^b_c (X) \to \D^b_c (\sw^{-1} (0))$ be the functor of nearby cycles, and let $i\colon \sw^{-1}(0) \hookrightarrow X$ be the closed embedding. 
There is a distinguished triangle of functors, referred to as the \emph{Milnor triangle}:
$$
\xymatrix{
\psi_\sw \ar[r] & \varphi_\sw \ar[r] & i^* \ar[r] & 
}
$$
Its Verdier dual is
\begin{equation}\label{equ on mil tri}
\xymatrix{
i^! \ar[r] & \varphi_\sw \ar[r] & \psi_\sw \ar[r] & 
}
\end{equation}
Apply the map $i^! \to \varphi_\sw$ to $\omega_X$, and recall the Borel--Moore homology of $\sw^{-1} (0)$ is
$$
H^\BM (\sw^{-1} (0) ) = H^{} \left(\sw^{-1} (0), \omega_{\sw^{-1} (0)}\right) = H^{} (\sw^{-1} (0), i^! \omega_X). 
$$
We obtain a \emph{canonical map} (or \emph{cospecialization map}) from the Borel--Moore homology of $\sw^{-1} (0)$ to the critical cohomology:
\begin{align}\label{can_coh}
\can\colon H^\BM (\sw^{-1} (0) ) \to H (X, \sw).
\end{align}
More generally, let $i\colon Z\hookrightarrow X$ be the embedding of a closed subset, then there is a natural transformation 
$$\varphi_{\sw|_Z}\circ i^!\to i^!\circ \varphi_{\sw}. $$
Applying to $\omega_X$, we get a natural map 
\begin{equation}\label{equ on cano map mod}\varphi_{\sw|_Z}\omega_Z\to i^!\varphi_{\sw}\omega_X. \end{equation}
We define \textit{critical cohomology with support} on $Z$ by 
\begin{equation}\label{equ on def of rel coh}
    H(X,\sw)_Z:=H\left(Z,i^!\varphi_\sw\omega_X\right).
\end{equation}
Then \eqref{equ on cano map mod} gives a canonical map
\begin{align}\label{can with supp_coh}
    \can\colon H(Z,\sw|_Z)\to H(X,\sw)_Z.
\end{align}

\subsubsection{Euler class operator}\label{sect on eu cla op}

Let $\pi\colon E\to X$ be a  
vector bundle of rank $r$, with zero-section $z\colon X\to E$. Let $\sw$ be a regular function on $X$ and we extend it to a function on $E$ by $\sw_E:=\sw\circ \pi$, then 
$\sw=\sw_E\circ z$.

Consider the composition: 
\begin{equation}\label{eqn:Euler0}
    z_*\mathbb{Q}_X\to \mathbb{Q}_{E}[2r]\to z_*\mathbb{Q}_{X}[2r], 
\end{equation}
where the second map is by adjunction and the first map is its Verdier dual. 
Applying $\varphi_{\sw_E}$ to \eqref{eqn:Euler0}:
\[z_*\varphi_\sw\mathbb{Q}_X\cong \varphi_{\sw_E} z_*\mathbb{Q}_X\to \varphi_{\sw_E}\mathbb{Q}_{E}[2r]\to \varphi_{\sw_E}z_*\mathbb{Q}_{X}[2r]\cong z_*\varphi_\sw\mathbb{Q}_X[2r], \]
and taking global section, we obtain the \textit{Euler class operator} on critical cohomology: 
\begin{equation}\label{eqn:Euler} e(E)\cdot(-)\colon H(X,\sw)\to H(X,\sw).\end{equation}
Here we use proper pushforward commutes with vanishing cycle functor. 
It is straightforward to check that the Euler class operator commutes with the canonical map \eqref{can_coh}. 

\subsubsection{Functorial properties}

The critical cohomology has natural \emph{functorial properties}.
Let $f\colon X\to Y$ be a map between smooth varieties, and $\sw\colon Y \to \bC$ be a regular function. 

\begin{enumerate}[(i)]
\setlength{\parskip}{1ex}

\item Since both $X$ and $Y$ are smooth, the map $f$ is a locally complete intersection (l.c.i.).
There is an l.c.i.\,pullback for the critical cohomology: 
$$
f^*\colon H_* (Y, \sw) \to H_{* + 2 \dim f} (X, \sw\circ f), 
$$
which is compatible with the canonical map. 

\item If $f$ is proper, there is a proper pushforward:
$$
f_* \colon H (X, \sw\circ f) \to H (Y, \sw), 
$$
which is compatible with the canonical map.

\item Given a closed embedding $i \colon Z \hookrightarrow X$ between smooth varieties, by \cite[Prop.~2.16]{Dav}, we have 
\begin{equation}\label{equ on euler of normal}i^*i_*(-)=e(N_{Z/X})\cdot(-), \end{equation}
where $e(N_{Z/X})\cdot$ is the Euler class operator \eqref{eqn:Euler}.

\item Given an affine fibration $\pi\colon \widetilde{X}\to X$ with the induced potential function $\tilde{\sw}:= \sw\circ \pi$, the pullback map 
\begin{equation}\label{pb is iso on coho}
\pi^*\colon H_{}(X, \sw) \to H_{}(\widetilde{X}, \sw\circ \pi)
\end{equation}
is an isomorphism \cite[Eqn.~(37)]{Dav}.

\end{enumerate}

\subsubsection{Excisions}

Next we explain the \textit{excision} for critical cohomologies. 
Let $j \colon U\hookrightarrow X$ be the inclusion of an open subset, and $i \colon Z \hookrightarrow X$ be the closed embedding of its complement. 
There is an excision triangle:
\begin{equation} \label{eqn-exc}
\xymatrix{
i_{*} i^! \ar[r] & \Id \ar[r] & j_* j^*.   
} 
\end{equation}
When applied to $\omega_X$, it gives the excision long exact sequence for Borel--Moore homology. 
In critical cohomology, we obtain two long exact sequences.

\begin{enumerate}[a)]
\setlength{\parskip}{1ex}
\item Apply (\ref{eqn-exc}) to $\omega_X$ and then apply $\varphi_\sw$ from the left. 
We get
$$
\xymatrix{
\varphi_\sw i_{*} \omega_Z \ar[r] & \varphi_\sw \omega_X  \ar[r] & \varphi_\sw j_* \omega_U  
} 
$$
Using $\varphi_\sw i_{*} \cong i_{*} \varphi_{\sw |_Z}$ (e.g.~\cite[Ex. VIII.15]{KS}), and then applying $R\Gamma$, we have
$$
\xymatrix{
 \ar[r] & H (Z, \sw |_Z ) \ar[r] & H (X, \sw) \ar[r] &  H_{} (X, \varphi_\sw j_* \omega_U ) \ar[r] & 
}
$$

\item Apply (\ref{eqn-exc}) to $\varphi_\sw \omega_X$.
We get
$$
\xymatrix{
i_{*} i^! \varphi_\sw \omega_X  \ar[r] & \varphi_\sw \omega_X  \ar[r] & j_* j^* \varphi_\sw \omega_X  } 
$$
Using $j^* \varphi_\sw \cong \varphi_{\sw |_U} j^*$ (see \cite[(23)]{Dav}) and then apply $R\Gamma$, we have
$$
\xymatrix{
 \ar[r] & H^{} (X, \sw)_Z \ar[r] & H (X, \sw) \ar[r] & H (U, \sw |_U) \ar[r] & 
}
$$
\end{enumerate}
Consider the Cartesian diagram
$$
\xymatrix{
\sw|_Z^{-1} (0) \ar@{^(->}[r] \ar@{^(->}[d] & \sw^{-1} (0) \ar@{^(->}[d]^i & \sw |_U^{-1} (0) \ar@{_(->}[l] \ar@{^(->}[d] \\
Z \ar@{^(->}[r]^i & X & U \ar@{_(->}[l]_j . 
}
$$
One can obtain a commutative diagram of (co)homology groups:
$$
\xymatrix{
H^\BM (\sw |_Z^{-1} (0) ) \ar[r] \ar[d]^{\can} & H^\BM (\sw^{-1} (0) ) \ar[r] \ar[d]^{\can} & H^\BM (\sw |_U^{-1} (0) ) \ar[r] \ar[d] & \\
H (Z, \sw |_Z ) \ar[r] \ar[d]^{\can} & H (X, \sw) \ar[r] \ar@{=}[d] &  H^{} (X, \varphi_\sw j_* \omega_U ) \ar[r] \ar[d] &  \\
H^{} (X, \sw)_Z \ar[r] & H (X, \sw) \ar[r] & H (U, \sw |_U) \ar[r] & 
}
$$
where the right vertical map in the upper right corner is given by applying the Milnor triangle \eqref{equ on mil tri} to $j_*\omega_{U}$.

\begin{Proposition}\label{excision for crit coh}
Let $F$ be a connected component of $X^\sA$ and $Z=\Attr^f(F)$ with immersion $i\colon Z\to X$. 
Then the map $$H (Z, \sw |_Z )\to H^{} (X, \sw)_Z$$ in the above diagram is an isomorphism. In particular, there is a  long exact sequence
$$
\xymatrix{
 \ar[r]  & H (Z, \sw |_Z ) \ar[r] & H (X, \sw) \ar[r] & H (U, \sw |_U) \ar[r] &. 
}
$$
\end{Proposition}

\begin{proof}
We first prove the case when $F$ is maximal, i.e. $Z:=\Attr^f(F)=\Attr(F)$. The limit map $p:\Attr(F)\to F$ is an affine fibration by a result of Bialynicki-Birula \cite{BB}. Since $\sw$ is $\sT$-invariant, so we have $\sw|_Z=p^{-1}(\sw|_F)$; thus 
$$\varphi_{\sw|_Z}\omega_Z\cong p^! \varphi_{\sw|_F}\omega_F. $$ 
It follows that
\begin{align*}
    p_*\varphi_{\sw|_Z}\omega_Z\cong p_*p^! \varphi_{\sw|_F}\omega_F\cong p_*p^* \varphi_{\sw|_F}\omega_F[2\dim_p]\cong \varphi_{\sw|_F}\omega_F[2\dim_p]\cong \varphi_{\sw|_F}p_*\omega_Z.
\end{align*}
As hyperbolic localization commutes with vanishing cycle functor \cite{Ric}: 
$$p_*i^!\varphi_\sw\xrightarrow{\cong} \varphi_{\sw|_F} p_*i^!,$$
therefore
\begin{align*}
    p_*i^!\varphi_\sw\omega_X\cong \varphi_{\sw|_F} p_*i^! \omega_X\cong \varphi_{\sw|_F}p_*\omega_Z. 
\end{align*}
Thus the natural map $p_*\varphi_{\sw|_Z}\omega_Z\to p_*i^!\varphi_\sw\omega_X$ is an isomorphism, which implies that 
\begin{align}\label{equ on hl} 
    H (Z, \sw |_Z ) \cong H^{}(F, p_*\varphi_{\sw|_Z}\omega_Z)\cong H^{}(F, p_*i^!\varphi_\sw\omega_X)\cong H^{} (Z, i^! \varphi_\sw \omega_X).
\end{align}
This proves the lemma in the case when $F$ is maximal.

For general $F$, $Z=\Attr^f(F)$ admits a filtration $\emptyset=Z_0\subseteq  Z_1\subseteq  Z_2\subseteq \cdots\subseteq  Z_n=Z$ by closed subvarieties $Z_i$, 
such that $Z_i\setminus Z_{i-1}=\Attr(F_i)$ for some union of fixed components $F_i$ and $F_i$ is maximal in $X\setminus Z_{i-1}$. We prove for general $Z$ by induction on $n$, and the case $n=1$ is what we have shown in the above.

Let $i_1\colon Z_1\hookrightarrow Z$ and $j_1\colon Z\setminus Z_1\hookrightarrow X\setminus Z_1$ be the embeddings, and denote $\tilde i_1:=i\circ i_1$. Then we have a commutative diagram of long exact sequences:
$$
\xymatrix{
\ar[r] & H^{} (Z_1, i_1^!\varphi_{\sw |_Z} \omega_Z) \ar[r] \ar[d] & H (Z, \sw|_Z) \ar[r] \ar[d] &  H (Z\setminus Z_1, \sw|_{Z\setminus Z_1} ) \ar[r] \ar[d] &  \\
\ar[r] & H^{} (Z_1, \tilde i_1^! \varphi_\sw \omega_X) \ar[r] & H^{} (Z, i^! \varphi_\sw \omega_X) \ar[r] & H^{}  (Z\setminus Z_1,  j_1^!\varphi_{\sw |_{X\setminus Z_1}}\omega_{X\setminus Z_1}) \ar[r] & 
}
$$
The right vertical arrow is an isomorphism be induction hypothesis, so it remains to show that the left vertical arrow is an isomorphism. Consider the commutative diagram: 
\begin{equation}\label{diag on pf of exc}
\xymatrix{
H (Z_1, \sw |_{Z_1} ) \ar[r] \ar[dr] & H^{} (Z_1, i_1^!\varphi_{\sw |_Z} \omega_Z) \ar[d] &  \\
 & H^{} (Z_1, \tilde i_1^! \varphi_\sw \omega_X).  & 
}
\end{equation}
We have shown in the above that $H (Z_1, \sw |_{Z_1} )\cong H^{} (Z_1, \tilde i_1^! \varphi_\sw \omega_X)$. 
As $Z_1$ is the attracting set of a maximal component of $Z^\sA$, we can replace $(Z,X)$ in \eqref{equ on hl} by $(Z_1,Z)$  and obtain 
$$H (Z_1, \sw |_{Z_1} ) \cong H^{} (Z_1, i_1^!\varphi_{\sw |_Z} \omega_Z). $$ 
It follows that the vetical map in \eqref{diag on pf of exc} 
is an isomorphism. This concludes the proof.
\end{proof}

\begin{Remark}\label{rmk: excision for crit coh}
As long as $Z$ is the form of $\bigcup_{F} \Attr(F)$ for a collection of fixed components $\{F\}\subseteq  \Fix_\sA(X)$, the above proof works and the map $H (Z, \sw |_Z )\to H^{} (X, \sw)_Z$ is an isomorphism for such $Z$.
\end{Remark}

\begin{Definition}
We say that a class in $H (X, \sw)$ is \emph{supported} on a closed subvariety $Z$, if its image under
$$
H (X, \sw) \to H (U, \sw |_U)
$$
vanishes.
\end{Definition}

\subsubsection{Torus equivariance}

Let $\sT$ be a torus acting on $(X,\sw)$ as in Setting \ref{setting of stab}, 
all constructions above can be generalized to this setting. 
For a subtorus $\sA \subseteq  \sT$, denote $i_\sA\colon X^\sA \to X$ to be the closed embedding of torus fixed locus. 
We have proper pushforward and Gysin pullback
\begin{equation}\label{equ on pp on coh}
i_{\sA, *}\colon H^\sT (X^\sA, \sw^\sA) \to  H^\sT (X, \sw) , \quad i_\sA^* \colon H^\sT (X, \sw) \to H^\sT (X^\sA, \sw^\sA).
\end{equation}
By \eqref{equ on euler of normal}, the composition $i_\sA^* i_{\sA, *}$ is the multiplication by the Euler class $e^\sT (N_{X^\sA / X})$. 
Therefore $i_{\sA, *}$ is \textit{injective}. And $i_\sA^*$ is \textit{surjective} after localization,~i.e.
$$i_\sA^* \colon H^\sT (X, \sw)_\loc \twoheadrightarrow H^\sT (X^\sA, \sw^\sA)_\loc, $$ 
where $(-)_\loc := (-) \otimes_{H^*_\sT (\pt)} \Frac(H^*_\sT (\pt))$ denotes the localized theory.

\begin{Remark}
(1) Proposition \ref{excision for crit coh} holds equivariantly as long as the extra torus preserves the potential $\sw$ and commutes with $\sA$ action. 
(2) If the extra torus in (1) contains $\sA$,~i.e.~as the $\sT$ in Setting \ref{setting of stab},
one can show the equivariant version of the map in Proposition \ref{excision for crit coh} is injective. 
\end{Remark}

\subsection{Critical \texorpdfstring{$K$}{K}-theory}

\subsubsection{Definition}

Let $(X,\sw,\sT)$ be as in Setting \ref{setting of stab}. 
Consider the dg-category $\mathrm{Fact}_{\mathrm{coh}}(X,\sw,\sT)$ of coherent factorizations 
of $\sw$ \cite[Def.~3.1]{BFK}, whose objects are pairs 
$$\eE_{-1}\xrightarrow{d_{0}} \eE_0, \quad \eE_0\xrightarrow{d_{-1}}  \eE_{-1} $$
of morphisms in $\mathrm{Coh}_{\sT}(X)$
such that $$d_{-1}\circ d_{0}=d_{0}\circ d_{-1}=\sw.$$
Let $\mathrm{HFact}_{\mathrm{coh}}(X,\sw,\sT)$ be the homotopy category of $\mathrm{Fact}_{\mathrm{coh}}(X,\sw,\sT)$, which is a triangulated 
category, and $\mathrm{Acy}_{\mathrm{coh}}$ denote the minimal thick triangulated subcategory containing totalizations of short exact sequences
of coherent factorizations. Define the \textit{triangulated category of coherent factorizations} to be the following Verdier quotient \cite{Orl}, \cite[Def.~3.9]{BFK}
$$\mathrm{MF}_{\mathrm{coh}}(X,\sw,\sT):=\mathrm{HFact}_{\mathrm{coh}}(X,\sw,\sT)/\mathrm{Acy}_{\mathrm{coh}}. $$
Without causing confusions, we also denote it to be $\mathrm{MF}_{}(X,\sw,\sT)$, or sometimes write it as $\mathrm{MF}_{}([X/\sT],\sw)$.
\begin{Definition}
The \textit{critical $K$-theory} of $(X,\sw,\sT)$ is the Grothendieck group of $\mathrm{MF}_{\mathrm{ }}(X,\sw,\sT)$:
$$K^{\sT}(X,\sw):=K_0^{\sT}(X,\sw):=K_0(\mathrm{MF}_{\mathrm{ }}(X,\sw,\sT)).$$
\end{Definition}

\subsubsection{Canonical map}

Let $i\colon Z(\sw)\subseteq  X$ be the inclusion of the zero locus of the function $\sw$. 
There is an equivalence of categories (\cite[Rmk.~3.65]{BFK}, \cite[Thm.~3.6]{Hir}) from the $\sT$-equivariant singularity category of $Z(\sw)$:
$$\D^b\Coh_{\sT}(Z(\sw))/\Perf_{\sT}(Z(\sw))\xrightarrow{\cong} \mathrm{MF}_{\mathrm{ }}(X,\sw,\sT), $$
which induces an exact sequence of the Grothendieck groups (e.g.~\cite[Lem.~1.10]{PS}):
\begin{equation}\label{equ on kgps}\to K(\Perf_{\sT}(Z(\sw)))\to K(\D^b\Coh_{\sT}(Z(\sw))) \xrightarrow{\mathrm{can}}  K^{\sT}(X,\sw) \to 0. \end{equation}
More generally, let $i\colon Z\hookrightarrow X$ be the embedding of a $\sT$-invariant closed subset, there exists a \textit{matrix factorization category} $\mathrm{MF}_{\mathrm{ }}(X,\sw,\sT)_{Z}$
\textit{with support} on $Z$, whose Grothendieck group is denoted by 
$$K^{\sT}(X,\sw)_Z. $$
One has a surjective \textit{canonical map} (e.g.~\cite[Prop.~2.6]{VV2}):
\begin{equation}\label{equ on can map k with sp}
\can\colon K^{\sT}(Z\cap Z(\sw))\twoheadrightarrow K^{\sT}(X,\sw)_Z,  \end{equation}
where $K^{\sT}(Z\cap Z(\sw)):=K_0(\D^b\Coh_{\sT}(Z\cap Z(\sw)))$.

\subsubsection{Euler class operator}\label{sect on eu cla op k}

Let $\pi\colon E\to X$ be a vector bundle of rank $r$ and $\sw$ be a regular function on $X$. 
We define the  \textit{K-theoretic Euler class operator} on critical $K$-theory: 
\begin{equation}\label{eqn:Euler k} e_K(E)\cdot(-)\colon K(X,\sw)\to K(X,\sw),\end{equation}
$$[(\eE_{-1},\eE_0,d_{-1}, d_{0})] \mapsto 
[\left(\wedge^*E^\vee\otimes\eE_{-1},\wedge^*E^\vee\otimes\eE_0,d_{-1}, d_{0}\right)].$$
It is straightforward to check that the Euler class operator commutes with the canonical map \eqref{equ on can map k with sp}.

\subsubsection{Functorial properties}

Matrix factorization categories have functorial properties, e.g.~\textit{pullbacks} by \textit{flat} morphisms and \textit{pushforwards}
by \textit{proper morphisms} (\cite{BFK}, \cite[\S 2.2.2]{VV2}), critical $K$-groups also have such properties. For regular embeddings, one 
can also define \textit{Gysin pullbacks} for critical $K$-groups. 

In particular, for the inclusion $i_{\sA} \colon X^\sA \hookrightarrow X$ of $\sA$-fixed locus, 
there are proper pushforward and Gysin pullback:
\begin{equation}\label{equ on pp on k}
i_{\sA, *}\colon K^\sT (X^\sA, \sw^\sA) \to  K^\sT (X, \sw), \quad i_\sA^* \colon K^\sT (X, \sw) \to K^\sT (X^\sA, \sw^\sA).
\end{equation}
By \cite[Prop.~2.8]{VV2}\footnote{When working with torus \textit{invariant} functions, the flatness condition on functions imposed in \cite[Prop.~2.8]{VV2} (which uses \cite[Lem.~2.4]{VV2}) can be dropped off by following the argument of \cite[Lem.~2.4.7]{Toda1}.}, we have 
\begin{equation}\label{equ on k euler of normal}i_\sA^*\circ i_{\sA, *}(-)=e^\sT_K(N_{X^\sA/X})\cdot(-), \end{equation}
where $e^\sT_K (-)\cdot$ is the equivariant version of the $K$-theoretic Euler class operator \eqref{equ on can map k with sp}. 

This moreover induces an isomorphism
\begin{equation}\label{prop on loc of crit k}K^\sT (X^\sA, \sw^\sA)_{\loc}  \xrightarrow{\cong}  K^\sT (X, \sw)_\loc \end{equation}
of localized critical $K$-theories, 
where $(-)_\loc := (-) \otimes_{K_\sT (\pt)} \Frac(K_\sT (\pt))$.

Given a $\sT$-equivariant vector bundle $\pi\colon \widetilde{X}\to X$, it is straightforward to show the pullback map 
\begin{equation}\label{pb is iso on k}
\pi^*\colon K^{\sT}(X, \sw) \to K^{\sT}(\widetilde{X}, \sw\circ \pi)
\end{equation}
is an isomorphism.

\subsubsection{Excisions}\label{subsubsec: excision}
Let $j \colon U\hookrightarrow X$ be the inclusion of an open subset, and $i \colon Z \hookrightarrow X$ be the closed embedding of its complement. 
There is an excision sequence
\begin{equation} \label{eqn-exc K}
\xymatrix{
K^\sT(Z,\sw) \ar[r]^{i_*} & K^\sT(X,\sw) \ar[r]^{j^*} & K^\sT(U,\sw) .
} 
\end{equation}
It is obvious that $j^*\circ i_*=0$. The map $j^*$ is surjective because it fits into a commutative diagram
(below the canonical map is given as \eqref{equ on kgps})
$$
\xymatrix{
 K^\sT(\sw^{-1}(0))  \ar@{^{}->>}[r]^{j^* \quad} \ar@{^{}->>}[d]^{\mathrm{can}} & K^\sT(U\cap \sw^{-1}(0)) \ar@{^{}->>}[d]^{\mathrm{can}}  \\
K^\sT(X,\sw) \ar[r]^{j^* \quad} & K^\sT(U,\sw).
} 
$$
The following is a useful result showing when $\ker(j^*)=\im(i_*)$ holds.

\begin{Proposition}\label{excision for crit K}
Let $F$ be a connected component of $X^\sA$ and $Z=\Attr^f(F)$ with immersion $i\colon Z\to X$. 
Then $\ker(j^*)=\im(i_*)$ in the sequence \eqref{eqn-exc K}.
\end{Proposition}

\begin{proof}
The case when $F$ is maximal,~i.e.~$\Attr^f(F)=\Attr(F)$ is essentially proven in \cite[Thm.~2.5]{P}. Strictly speaking, \cite[Thm.~2.5]{P} assumes $X$ to be affine, but the argument works for quasi-projective $X$. For the convenience of readers, we explain the details here. Let $p\colon Z=\Attr(F)\to F$ be the attraction morphism. By \cite[Thm.~3.35, Cor.~3.28]{Hal}, there is a semi-orthogonal decomposition (SOD) for the derived category of coherent sheaves: 
\begin{align*}
    \left\langle \cdots,i_*\D^b([Z/\sT])_{-2},i_*\D^b([Z/\sT])_{-1},\mathbb{G},i_*\D^b([Z/\sT])_{0},i_*\D^b([Z/\sT])_{1},\cdots\right\rangle=\D^b([X/\sT]),
\end{align*}
such that 
$$\left\langle \cdots,\D^b([Z/\sT])_{-2},\D^b([Z/\sT])_{-1},\D^b([Z/\sT])_{0},\D^b([Z/\sT])_{1},\cdots\right\rangle$$ is an SOD of $\D^b([Z/\sT])$, 
where the weight is with respect to a cocharacter in the chamber $\fC$. 
Moreover, under the restriction to the open locus $\D^b([X/\sT])\twoheadrightarrow \D^b([(X\setminus Z)/\sT])$, we have an equivalence of categories
\begin{align*}
    \mathbb{G}\cong \D^b([(X\setminus Z)/\sT]).
\end{align*}
This induces an SOD for the matrix factorization category (e.g.~\cite[Prop.~2.1]{P}): 
\begin{align*}
\left\langle \cdots,i_*\mathrm{MF}([Z/\sT],\sw)_{-2},i_*\mathrm{MF}([Z/\sT],\sw)_{-1},\mathbb{W},i_*\mathrm{MF}([Z/\sT],\sw)_{0},i_*\mathrm{MF}([Z/\sT],\sw)_{1},\cdots\right\rangle =\mathrm{MF}([X/G],\sw),
\end{align*}
and an equivalence of categories:  
$$\mathbb{W}\cong \mathrm{MF}([(X\setminus Z)/\sT],\sw).$$
By passing to the Grothendieck group, there is a decomposition 
\begin{align*}
    K^\sT(Z,\sw)\oplus K^\sT(X\setminus Z,\sw)=K^\sT(X,\sw),
\end{align*}
where the map $K^\sT(Z,\sw)\to K^\sT(X,\sw)$ is $i_*$. In particular we have $\ker(j^*)=\im(i_*)$.

For general $F$, $Z=\Attr^f(F)$ admits a filtration 
$$\emptyset=Z_0\subseteq  Z_1\subseteq  Z_2\subseteq \cdots\subseteq  Z_n=Z$$ by closed subvarieties $Z_i$, such that $Z_i\setminus Z_{i-1}=\Attr(F_i)$ for some union of fixed components $F_i$ and $F_i$ is maximal in $X\setminus Z_{i-1}$. We prove for general $Z$ by induction on $n$. Let $i_1\colon Z_1\hookrightarrow Z$ and $j_1\colon Z\setminus Z_1\hookrightarrow X\setminus Z_1$ be the embeddings, and denote $\tilde i_1:=i\circ i_1$. By the induction hypothesis, the lemma holds when replacing $(X,Z)$ by $(X\setminus Z_1,Z\setminus Z_1)$. Suppose that $\alpha\in \ker(j^*)$, then $\alpha|_{X\setminus Z_1}\in \im(j_{1*})$, i.e. $\exists\, \beta\in K^\sT(Z\setminus Z_1,\sw)$ such that $\alpha|_{X\setminus Z_1}=j_{1*}\beta$. Take $\tilde \beta\in K^\sT(Z,\sw)$ such that $\tilde \beta|_{Z\setminus Z_1}=\beta$, then we have $(\alpha-i_*\tilde\beta)|_{X\setminus Z_1}=0$. Using the $n=1$ case, we can find $\gamma\in K^\sT(Z_1,\sw)$ such that $\alpha-i_*\tilde\beta=\tilde i_{1*}\gamma$, and this implies that $\alpha=i_*(\tilde\beta+ i_{1*}\gamma)$, in particular $\alpha\in \im(i_*)$. This shows that $\ker(j^*)=\im(i_*)$.
\end{proof}

\section{Stable envelopes}

In this section, we define stable envelopes on critical loci and prove their uniqueness in both cohomology and $K$-theory. 
We show that they can be constructed from stable envelope correspondences\,\footnote{Introducing stable envelope correspondences is essential for constructing stable envelopes for critical loci, see Remark \ref{rmk on corr is need}.} in \S \ref{sect on st ev co} through a convolution construction \S \ref{sect on exi of stb}. We provide existence results for stable envelope (correspondences), one for any critical loci and with respect to one dimensional torus $\sA$ (Proposition \ref{prop on stable corr}), the other for critical loci of functions on symmetric GIT quotients without constrain on the dimension of the torus. The cohomological case is shown in 
Theorem \ref{exi of coh stab on symm quiver var} and Section \ref{app on symm var}, while the $K$-theoretic case makes use of nonabelian stable envelopes and Hall 
operations, see \S \ref{sect on ex kstab}, \S \ref{app on k symm var}.
We prove the triangle lemma in \S \ref{sec on tri lem}. 

In this section, we often further impose the following condition. 
\begin{Setting}\label{setting: proper to affine} 
We assume that there exists a proper morphism $\pi\colon X\to X_0$ to an affine variety $X_0$. 
\end{Setting}

\begin{Remark}\label{rmk on setting}
The assumption in Setting \ref{setting: proper to affine} is equivalent to that $\Gamma(X,\mathcal O_X)$ is a finite generated $\bC$-algebra and the natural morphism 
$X\to \Spec \Gamma(X,\mathcal O_X)$ is proper. In particular, we can take $X_0=\Spec \Gamma(X,\mathcal O_X)$. If we further assume that there is an algebraic group $G$ action on $X$, then $G$ naturally acts on $\Gamma(X,\mathcal O_X)$ linearly and $X\to \Spec \Gamma(X,\mathcal O_X)$ is $G$-equivariant.
\end{Remark}

\begin{Remark}\label{proper-over-affine}
The assumption in Setting \ref{setting: proper to affine} implies that $\Attr^f_\fC(F)$ is a closed subset in $X$ for all $F\in \Fix_\sA(X)$ by \cite[Lem.~3.2.7]{MO}, and $\Attr^f_\fC$ is proper over $X$ along the projection map $X\times X^\sA\to X$ by \cite[\S 3.5]{MO}.
\end{Remark}

\subsection{Cohomological stable envelopes}


In this section, we use the following shorthand to denote $\sT$-equivariant critical cohomology of $(X,\sw)$ and $(X^\sA,\sw |_{X^\sA})$:
$$H^\sT (X, \sw):=H^\sT (X, \varphi_\sw\omega_X), \quad H^\sT \left(X^\sA, \sw^{\sA}\right):=H^\sT\left(X^\sA, \varphi_{\sw|_{X^\sA}}\omega_{X^\sA}\right). $$

\subsubsection{Definitions}



\begin{Definition}\label{def of stab coho}
Fix $(X,\sw,\sT,\sA)$ as in Setting \ref{setting of stab}, a choice of a chamber $\fC$ as Definition \ref{def of roots and chambers}. The \emph{cohomological stable envelope} for $(X,\sw,\sT,\sA,\fC)$ is a map of $H^*_\sT (\pt)$-modules:
$$ \Stab_\fC\colon   H^\sT (X^\sA, \sw^\sA) \to H^\sT (X, \sw) $$
such that for any $\gamma \in H^{\sT/\sA} (X^\sA, \sw^\sA) $ (and we identify $\gamma$ with $\gamma\otimes 1\in H^{\sT/\sA} (X^\sA, \sw^\sA) \otimes \bQ[\Lie(\sA)]\cong H^{\sT} (X^\sA, \sw^\sA) $) supported on a connected component $F \subseteq  X^\sA$, $\Stab_\fC (\gamma)$ satisfies the following axioms:
\begin{enumerate}[(i)]
\setlength{\parskip}{1ex}

\item  $\Stab_\fC (\gamma)$ is supported on $\Attr^f_\fC (F)$; 

\item $\Stab_\fC (\gamma) |_F =e^\sT (N_{F / X}^-) \cdot \gamma$;

\item For any $F' \neq F$, the inequality $\deg_\sA \Stab_\fC (\gamma)\big|_{F'} < \deg_\sA e^\sT (N_{F' / X}^-)$ holds.

\end{enumerate}
Here $(-)|_{F}$ means the Gysin pullback $\left(F\hookrightarrow X \right)^*(-)$,
$N_{F / X}^{\pm}$ is the sub-bundle of $N_{F / X}$ which are positive/negative with respect to the chamber $\fC$, and $\deg_\sA$ of a class in $H^\sT(F',\sw^\sA)\cong H^{\sT/
\sA}(F',\sw^\sA)\otimes \bQ[\Lie(\sA)]$ is the $\bQ[\Lie(\sA)]$-polynomial degree, where we define the $\sA$-degree of $0$ to be $-1$.
\end{Definition}

\begin{Definition}\label{def of normalizer coho}
Fix $(X,\sw,\sT,\sA)$ as in Setting \ref{setting of stab}, a  cohomological \textit{normalizer} is an element $\epsilon\in \{\pm 1\}^{\Fix_\sA(X)}$, i.e. a sign $\epsilon_F\in \{\pm 1\}$ for each fixed component $F$. If $\Stab_\fC$ is a cohomological stable envelope for $(X,\sw,\sT,\sA,\fC)$, then we define 
$$ \Stab_{\fC,\epsilon}:= \Stab_{\fC}\cdot\: \epsilon\colon   H^\sT (X^\sA, \sw^\sA) \to H^\sT (X, \sw),$$
i.e. for any $\gamma \in H^{\sT} (F, \sw^\sA) $, $\Stab_{\fC,\epsilon}(\gamma)=\epsilon_F\cdot \Stab_{\fC}(\gamma)$.
\end{Definition}


\subsubsection{Uniqueness}

\begin{Proposition}\label{uniqueness of coh stab}
Let $\gamma \in H^{\sT/\sA} (X^\sA, \sw^\sA) $ be supported on a connected component $F \subseteq  X^\sA$, then $\Stab_\fC (\gamma)$ which satisfies the axioms (i)--(iii) in Definition \ref{def of stab coho} is unique if it exists.
\end{Proposition}

\begin{proof}
Let $\beta \in H^{\sT} (X, \sw) $ be supported on $\Attr^f_\fC(F)$ and satisfies $\deg_\sA \beta|_{F'}<\deg_\sA e^\sT (N^-_{F'/X})$ for any embedding $i\colon F'\hookrightarrow X$ of a fixed component such that $F \preceq F'$. We claim that $\beta=0$.

Pick a total ordering on the set of fixed component refining $\preceq$, and let $Z$ be the minimal fixed component such that $\Supp\beta\subseteq  \Attr^f_\fC(Z)$ and $\Supp\beta\nsubseteq \Attr^f_\fC(Z)\setminus \Attr_\fC(Z)$. Factor $i\colon Z\hookrightarrow X$ as the following composition
\begin{align*}
  i\colon  Z\overset{f_1}{\hookrightarrow}\Attr_\fC(Z)\overset{f_2}{\hookrightarrow}\Attr^f_\fC(Z) \overset{f_3}{\hookrightarrow} X.
\end{align*}
Here $f_1$ is regular and $f_2$ is open. By Proposition \ref{excision for crit coh}, $\Supp\beta\subseteq  \Attr^f_\fC(Z)$ implies that $\exists\,\alpha\in H^{\sT} (\Attr^f_\fC(Z), \sw|_{\Attr^f_\fC(Z)})$ such that $\beta=f_{3,*}\alpha$. The maps $f_2$ and $f_3$ fit into the following Cartesian diagram
\begin{equation*}
\xymatrix{
\Attr_\fC(Z) \ar[r]^{f_2} \ar[d]_{g} \ar@{}[dr]|{\Box}  &\Attr^f_\fC(Z) \ar[d]^{f_3}\\
X\setminus(\Attr^f_\fC(Z)\setminus \Attr_\fC(Z)) \ar[r]^{\quad \quad \quad\quad \quad h}  & X,
}
\end{equation*}
where $g$ is a regular closed embedding and $h$ is an open embedding. Therefore we have 
\begin{align*}
i^*\beta=f_1^*g^*h^*f_{3*}\alpha=f_1^*g^*g_*f_2^*\alpha=f_1^*e^\sT(N_{\Attr_\fC(Z)/X})\cdot f_2^*\alpha=e^\sT(N^-_{Z/X})\cdot f_1^*f_2^*\alpha, 
\end{align*}
where the third equality uses \eqref{equ on euler of normal} and the last one uses the definition of attracting set. The multiplication by $e^\sT(N^-_{Z/X})$ is injective on $H^{\sT} (Z, \sw|_Z)$, so if $f_1^*f_2^*\alpha\neq 0$ then  
$$\deg_\sA e^\sT(N^-_{Z/X})\cdot f_1^*f_2^*\alpha\geqslant \deg_\sA e^\sT (N^-_{Z/X}), $$ 
which contradicts with the condition that $\deg_\sA i^*\beta<\deg_\sA e^\sT (N^-_{Z/X})$; therefore $f_1^*f_2^*\alpha=0$. We note that 
$$\sw|_{\Attr_\fC(Z)}=\pi^{*}(\sw|_{Z}), $$ 
where $\pi:\Attr_\fC(Z)\to Z$ is the projection map which is an affine fibration. 
By \eqref{pb is iso on coho}, the pullback
$$\pi^*\colon H^\sT(Z,\sw|_Z)\to H^\sT\left(\Attr_\fC(Z),\sw|_{\Attr_\fC(Z)}\right)$$ is an isomorphism, which implies that $f_1^*\colon H^\sT\left(\Attr_\fC(Z),\sw|_{\Attr_\fC(Z)}\right)\to H^\sT(Z,\sw|_Z)$ is an isomorphism because $ \pi\circ f_1=\Id$. This forces $f_2^*\alpha=0$, i.e. $\Supp\alpha\subseteq  \Attr^f_\fC(Z)\setminus \Attr_\fC(Z)$, contradicting with the choice of $Z$; hence $\beta=0$.

Now if $\Gamma_1,\Gamma_2\in H^{\sT} (X, \sw)$ are two classes that satisfy (i)--(iii) in Definition \ref{def of stab coho}, then $\Gamma_1-\Gamma_2$ satisfies the hypothesis above; hence $\Gamma_1=\Gamma_2$.
\end{proof}

\begin{Remark}\label{weak axiom coh}
Let $\le$ be a partial order on $\Fix_\sA(X)$ that refines $\preceq$ in \eqref{flow order}, and define
\begin{align*}
    \Attr^{\le}_\fC (F):=\bigcup_{F\le F'}\Attr_\fC (F').
\end{align*}
Suppose that $\Stab'_{\fC}(\gamma)\in H^\sT(X,\sw)$ is a class which satisfies the axioms (ii) and (iii) in Definition \ref{def of stab coho} and the following modifications of (i):
\begin{enumerate}
\setlength{\parskip}{1ex}

\item[(i')]  $\Stab'_\fC (\gamma)$ is supported on $\Attr^{\le}_\fC (F_i)$.

\end{enumerate}
Suppose that $\Stab_\fC$ exists, then $\Stab'_\fC (\gamma)=\Stab_\fC (\gamma)$. This is because the same argument in the proof of Proposition \ref{uniqueness of coh stab} implies that the cohomology class that satisfies (i') as above and (ii) and (iii) as in the Definition \ref{def of stab coho} is unique.
\end{Remark}

\begin{Lemma}\label{refined partial order and closed subset}
Let $\le$ be a partial order on $\Fix_\sA(X)$ that refines $\preceq$ in \eqref{flow order}, and $Z\subseteq  X$ be an $\sA$-invariant closed subvariety. Then 
\begin{align}
    \Attr^{\le}_\fC (F)_X\cap Z=\bigcup_{F\le F'} \Attr_\fC \left(F'\cap Z\right)_Z,
\end{align}
where the subscripts $X$ and $Z$ mean taking attracting sets in $X$ and in $Z$ respectively. 
\end{Lemma}

\begin{proof}
Let us choose a $\sigma\in \fC$, then
\begin{align*}
    \Attr_\fC (F)_X\cap Z=\left\{x\in Z\:|\: \lim_{t\to 0}\sigma(t)\cdot x\in F\cap Z\right\}=\Attr_\fC (F\cap Z)_Z.
\end{align*}
It follows that 
\begin{equation*}
    \Attr^{\le}_\fC (F)_X\cap Z=\left(\bigcup_{F\le F'}\Attr_\fC (F')_X\right)\cap Z =\bigcup_{F\le F'}\Attr_\fC (F'\cap Z)_Z. \qedhere
\end{equation*}
\end{proof}

\begin{Remark}[The ample partial order]\label{partial order by line bundle}
A useful refinement of $\preceq$ is induced by an ample $\sA$-equivariant line bundle $\mathcal L$ on $X$ \cite[\S 3.2.4]{MO}. Namely we define
\begin{align}
    F< F'\Longleftrightarrow \left(\mathrm{weight}\:\mathcal L|_{F}-\mathrm{weight}\:\mathcal L|_{F'}\right)\big|_{\fC}>0.
\end{align}
The ample partial order has the nice feature that it is compatible with the restriction to closed subvarieties, that is, for an $\sA$-invariant closed subvariety $Z$, we have
\begin{align*}
    F< F'\Longrightarrow (F\cap Z) < (F'\cap Z),
\end{align*}
in particular,
\begin{align*}
\Attr^{\le}_\fC (F)_X\cap Z=\Attr^{\le}_\fC (F\cap Z)_Z.
\end{align*}
\end{Remark}

\subsection{\texorpdfstring{$K$}{K}-theoretic stable envelopes}

\subsubsection{Definition}

\begin{Definition}\label{def of stab k}
Fix $(X,\sw,\sT,\sA)$ as in Setting \ref{setting of stab}, a choice of chamber $\fC$, and an $\mathsf s\in \Pic_\sA(X)\otimes_\bZ \bR$ which will be called a \textit{slope}. The \emph{K-theoretic stable envelope} for $(X,\sw,\sT,\sA,\fC,\mathsf s)$ is a map of $K_\sT (\pt)$-modules:
$$ \Stab^{\mathsf s}_{\fC}\colon   K^\sT (X^\sA, \sw^\sA) \to K^\sT (X, \sw) $$
such that for any $\gamma \in K^{\sT/\sA} (X^\sA, \sw^\sA)$ (and we identify $\gamma$ with $\gamma\otimes 1\in K^{\sT/\sA} (X^\sA, \sw^\sA) \otimes \bQ[\sA]\cong K^{\sT} (X^\sA, \sw^\sA) $) supported on a connected component $F \in \Fix_\sA(X)$, $\Stab^{\mathsf s}_{\fC}(\gamma)$ satisfies the following axioms:
\begin{enumerate}[(i)]
\setlength{\parskip}{1ex}

\item  $\Stab^{\mathsf s}_{\fC}(\gamma)$ is supported on $\Attr^f_\fC (F)$; 

\item $\Stab^{\mathsf s}_{\fC} (\gamma) |_F =  e^\sT_K (N_{F / X}^-) \cdot \gamma$;

\item For any $F' \neq F$, there is a strict inclusion of polytopes:
$$\deg_\sA\Stab^{\mathsf s}_{\fC}(\gamma)\big|_{F'} \subsetneq  \deg_\sA e^\sT_K (N_{F' / X}^-)+\mathrm{shift}_{F'}-\mathrm{shift}_{F},$$ where $$\mathrm{shift}_{F}:=\mathrm{weight}_\sA\left(\det(N_{F/X}^-)^{1/2}\otimes\mathsf s|_F\right),$$
$\deg_\sA$ of a Laurent polynomial is the Newton polytope: 
$$\deg_\sA\sum_{\mu}f_\mu z^\mu:=\text{Convex hull}\left(\{\mu:f_\mu\neq 0\}\right)\subseteq  \mathrm{Char}(\sA)\otimes_\bZ \bR,$$ and we define the $\sA$-degree of $0$ to be the empty set.
\end{enumerate}
\end{Definition}

\begin{Definition}\label{def: generic slope}
We say a slope $\mathsf s\in \Pic_\sA(X)\otimes \bR$ is generic if for any $F\prec F'$, the $\sA$-weight difference
\begin{align*}
    \mathrm{weight}_\sA\left(\det(N_{F'/X}^-)^{1/2}\otimes\mathsf s|_{F'}\right)-\mathrm{weight}_\sA\left(\det(N_{F/X}^-)^{1/2}\otimes\mathsf s|_F\right)
\end{align*}
is nonintegral.
\end{Definition}

\begin{Remark}\label{rmk: generic slope}
Suppose that the slope $\mathsf s$ is generic, then any inclusion 
$$\deg_\sA\Stab^{\mathsf s}_{\fC}(\gamma)\big|_{F'} \subseteq  \deg_\sA e^\sT_K (N_{F' / X}^-)+\mathrm{shift}_{F'}-\mathrm{shift}_{F}$$ must be strict.
\end{Remark}

\begin{Remark}\label{rmk on Pic(X)}
Since $X$ is smooth and $\Pic(\sA)=0$, every line bundle on $X$ admits a (not necessarily unique) $\sA$-equivariant structure by \cite[Thm.~5.1.9]{CG}\,\footnote{The positive integer $n$ in \cite[Thm.~5.1.9]{CG} is determined by \cite[Prop.~5.1.17]{CG} which can be taken to be $1$ for the case $G=\sA$ since $\Pic(\sA)=0$.}. Assume $\Gamma(X,\mathcal O_X^*)=\bC^*$, then the set of $\sA$-equivariant structures on a line bundle $\mathcal L\in \Pic(X)$ form a $\mathrm{Char}(\sA)$-torsor, in this case we have a fiber sequence:
\begin{align*}
    \mathrm{Char}(\sA)\to \Pic_\sA(X)\to \Pic(X).
\end{align*}
Note that for any $\chi\in \mathrm{Char}(\sA)$, $\mathsf s$ and $\mathsf s\otimes\chi$ induce the same shifts in Definition \ref{def of stab k}(iii), so we can take $\mathsf s\in \Pic(X)\otimes_\bZ\bR$ if $\Gamma(X,\mathcal O_X^*)=\bC^*$.
\end{Remark}

\begin{Definition}\label{def of normalizer k}
Fix $(X,\sw,\sT,\sA)$ as in Setting \ref{setting of stab}, a  $K$-theoretic \textit{normalizer} is an element $\mathcal E\in \pm \Pic_\sT(X^\sA)$,~i.e.~a $K$-theory class $\mathcal E\in K^\sT(X^\sA)$ such that for every $F\in \Fix_\sA(X)$, either $\mathcal E|_F$ or $-\mathcal E|_F$ is a line bundle. If $\Stab^{\mathsf s}_\fC$ is a $K$-theoretic stable envelope for $(X,\sw,\sT,\sA,\fC,\mathsf s)$, then we define 
$$ \Stab^{\mathsf s}_{\fC,\mathcal E}:= \Stab^{\mathsf s}_\fC\cdot\: \mathcal E\colon   K^\sT (X^\sA, \sw^\sA) \to K^\sT (X, \sw),$$
i.e. for any $\gamma \in K^{\sT} (F, \sw^\sA) $, $\Stab^{\mathsf s}_{\fC,\mathcal E}(\gamma)= \Stab^{\mathsf s}_\fC(\mathcal E|_F\cdot\gamma)$.
\end{Definition}

\begin{Remark} \label{rk Oko normalizer}
The normalization condition and degree condition in Definition \ref{def of stab k} is slightly different from the ordinary definition of $K$-theoretic stable envelope in the case $\sw=0$ assuming the existence of polarization $T^{1/2}_X$ \cite[\S 9.1]{Oko}. In fact, the original definition can be restored by the following.
\begin{align*}
    &\Stab^{\mathsf s}_{\fC}\text{ in \cite[\S 9.1]{Oko}}= \Stab^{\mathsf s'}_{\fC,\mathcal E}\text{ in this paper, where }\\
    &\mathcal E=(-1)^{\rk N^{1/2,+}_{F/X}}\left(\frac{\det N^-_{F/X}}{\det N^{1/2}_{F/X}}\right)^{ 1/2},\text{ and }\mathsf s'=\mathsf s\otimes \left(\det T^{1/2}_X\right)^{-1/2}.
\end{align*}
\end{Remark}

\subsubsection{Uniqueness}

\begin{Proposition}\label{uniqueness of K stab}
Let $\mathsf s\in \Pic_\sA(X)\otimes \bR$ be a generic slope and $\gamma \in K^{\sT/\sA} (X^\sA, \sw^\sA) $ be supported on a connected component $F \subseteq  X^\sA$, then $\Stab^{\mathsf s}_\fC (\gamma)$ which satisfies the axioms (i)--(iii) in Definition \ref{def of stab k} is unique if it exists.
\end{Proposition}
\begin{proof} 
The argument is similar to the proof of Proposition \ref{uniqueness of coh stab}, and we omit the details. 
\end{proof}

\begin{Remark}\label{weak axiom K}
Let $\le$ be a partial order on $\Fix_\sA(X)$ that refines $\preceq$, and define $\Attr^{\le}_\fC (F)$ as in the Remark \ref{weak axiom coh}. Suppose that $\Stab'_{\fC}(\gamma)\in K^\sT(X,\sw)$ is a class which satisfies the axioms (ii) and (iii) in Definition \ref{def of stab coho} and the following modifications of (i):
\begin{enumerate}
\setlength{\parskip}{1ex}

\item[(i')]  $\Stab'_\fC (\gamma)$ is supported on $\Attr^{\le}_\fC (F_i)$.

\end{enumerate}
Suppose that $\Stab_\fC$ exists, then $\Stab'_\fC (\gamma)=\Stab_\fC (\gamma)$. This is because the same argument in the proof of Proposition \ref{uniqueness of K stab} implies that the class that satisfies (i') as above and (ii) and (iii) as in the Definition \ref{def of stab k} is unique.
\end{Remark}

\begin{Remark}\label{twist by line bundle}
Let $\Stab^{\mathsf s}_\fC$ be a $K$-theoretic stable envelope for $(X,\sw,\sT,\sA,\fC,\mathsf s)$ as in the Definition \ref{def of stab k}. Let $\mathcal L$ be a $\sT$-equivariant line bundle on $X$, and denote $\mathcal L^\sA:=\mathcal L|_{X^\sA}$, then we have
\begin{align*}
    \mathcal L\circ \Stab^{\mathsf s}_\fC \circ (\mathcal L^\sA)^{-1}=\Stab^{\mathsf s\otimes\mathcal L}_\fC.
\end{align*}
\end{Remark}

\subsection{Stable envelope correspondences}\label{sect on st ev co}

Stable envelopes defined in the previous section are most effectively constructed through convolutions, which requires 
stable envelope correspondences introduced in this section. 

\subsubsection{Definitions and uniqueness}

\begin{Definition}\label{stab corr_coh}
Let $(X,\sw,\sT,\sA)$ be in the Setting \ref{setting of stab}, and fix a chamber $\fC$. A \textit{cohomological stable envelope correspondence} for $(X,\sw,\sT,\sA,\fC)$ is a critical cohomology class with support \eqref{equ on def of rel coh} on $\Attr^f_\fC$: 
$$[\Stab_{\fC}]\in H^\sT(X\times X^\sA,\sw\boxminus\sw^\sA)_{\Attr^f_\fC}$$ which satisfies the following two axioms:
\begin{enumerate}[(i)]
\setlength{\parskip}{1ex}


\item For any fixed component $F \in \Fix_\sA(X)$, $[\Stab_{\fC }]|_{F\times F}  = e^\sT (N_{F / X}^-) \cdot [\Delta_F]$ for diagonal 
$\Delta_F\subseteq  F\times F$;

\item For any $F'\neq F$, the inequality $\deg_\sA [\Stab_{\fC }]\big|_{F'\times F} < \deg_\sA e^\sT (N_{F' / X}^-)$ holds.

\end{enumerate}
Here $$[\Delta_F]\in H^\sT(F\times F,\sw^\sA\boxminus\sw^\sA)_{\Delta_F}$$ is the image of $[\Delta_F]\in H^\sT(\Delta_F)$ under the canonical map \eqref{can with supp_coh}. Note also that $\Attr^f_\fC\cap\, (F\times F)=\Delta_F$.
\end{Definition}

\begin{Definition}\label{stab corr_k}
Let $(X,\sw,\sT,\sA)$ be in the Setting \ref{setting of stab}, and fix a chamber $\fC$ and an $\mathsf s\in \Pic_\sA(X)$. A \textit{$K$-theoretic stable envelope correspondence} for $(X,\sw,\sT,\sA,\fC,\mathsf s)$ is a class $$[\Stab^{\mathsf s}_{\fC}]\in K^\sT(X\times X^\sA,\sw\boxminus\sw^\sA)_{\Attr^f_\fC}$$ which satisfies the following two axioms:
\begin{enumerate}[(i)]
\setlength{\parskip}{1ex}


\item For any fixed component $F \in \Fix_\sA(X)$, $[\Stab^{\mathsf s}_{\fC}]\big|_{F\times F} = e^\sT_K (N_{F / X}^-) \cdot [\Delta_F]$;

\item For any $F'\neq F$, there is a strict inclusion of polytopes $\deg_\sA [\Stab^{\mathsf s}_{\fC}]\big|_{F'\times F} \subsetneq  \deg_\sA e^\sT_K (N_{F' / X}^-)+\mathrm{shift}_{F'}-\mathrm{shift}_{F}$.

\end{enumerate}
Here 
$$[\Delta_F]\in K^\sT(F\times F,\sw^\sA\boxminus\sw^\sA)_{\Delta_F}$$ is the image of $[\oO_{\Delta_F}]\in K^\sT(\Delta_F)$ under the canonical map defined in \eqref{equ on kgps}.
\end{Definition}

\begin{Proposition}\label{prop on un of stab corr k}
Suppose that $\sw=0$, then the class $[\Stab_{\fC}]$ (resp.~$[\Stab^{\mathsf s}_{\fC}]$) in Definition \ref{stab corr_coh} (resp.~Definition \ref{stab corr_k}) is unique if exists.
\end{Proposition}

\begin{proof}
The proof is similar to that of Proposition \ref{uniqueness of coh stab} and Proposition \ref{uniqueness of K stab} respectively, and we omit it.
\end{proof}

\subsubsection{Existence}\label{sect on ext of stab corr}

We notice that $\sw\boxminus \sw^\sA$ vanishes on $\Attr^f_\fC$ for arbitrary $\sA$-invariant function $\sw$. So the canonical maps \eqref{can with supp_coh}, \eqref{equ on can map k with sp} in this setting are:  
\begin{align*}
\can\colon H^\sT(X\times X^\sA)_{\Attr^f_\fC}\to H^\sT(X\times X^\sA,\sw\boxminus \sw^\sA)_{\Attr^f_\fC},\quad \can\colon K^\sT(X\times X^\sA)_{\Attr^f_\fC}\to K^\sT(X\times X^\sA,\sw\boxminus \sw^\sA)_{\Attr^f_\fC}.
\end{align*}
Here by definition \eqref{equ on def of rel coh}, $H^\sT(X\times X^\sA)_{\Attr^f_\fC}= H^\sT(\Attr^f_\fC)$, and 
$K^\sT(X\times X^\sA)_{\Attr^f_\fC}\cong K^\sT(\Attr^f_\fC)_{}$ by d\'evissage.

\begin{Lemma}\label{can map induce stab}
Let $[\Stab_{\fC}]$ be a cohomological stable envelope correspondence for $(X,0,\sT,\sA,\fC)$, then the image $\can([\Stab_{\fC}])$ under canonical map is a cohomological stable envelope correspondence for $(X,\sw,\sT,\sA,\fC)$. 

Similarly, let $[\Stab^{\mathsf s}_{\fC}]$ be a $K$-theoretic stable envelope correspondence for $(X,0,\sT,\sA,\fC,\mathsf s)$, then $\can([\Stab^{\mathsf s}_{\fC}])$ is a $K$-theoretic stable envelope correspondence for $(X,\sw,\sT,\sA,\fC,\mathsf s)$.
\end{Lemma}

\begin{proof}
The axioms (i) and (ii) in Definition \ref{stab corr_coh} for nontrivial potential $\sw$ follows from the corresponding axioms for trivial potential $\sw=0$ because canonical map commutes with Gysin pullback and preserves $\deg_\sA$. Namely, we have 
\begin{align*}
\left(\can([\Stab_{\fC}])\right)|_{F\times F} =\can\left([\Stab_{\fC}]|_{F\times F} \right) = \can\left( e^\sT (N_{F / X}^-) \cdot [\Delta_F]\right)= e^\sT (N_{F / X}^-) \cdot [\Delta_F],
\end{align*}
and 
\begin{align*}
\deg_\sA\left(\can([\Stab_{\fC}])\right)|_{F'\times F} =\deg_\sA\can\left([\Stab_{\fC}]|_{F'\times F} \right) < \deg_\sA e^\sT (N_{F' / X}^-).
\end{align*}
The $K$-theoretic counterpart is proven similarly.
\end{proof}

In the case when $\dim\sA=1$, there are two choices of chambers $\{+,-\}$. In this case, the stable envelope correspondence always exists, by the next proposition 
(as motivated by \cite[\S 9.2.3]{Oko}).

\begin{Proposition}\label{prop on stable corr}
Let $(X,\sw,\sT,\sA)$ be in the Setting \ref{setting of stab} and suppose that $\mathsf A\cong\bC^*$. 
\begin{itemize} 

\item Then there exists a cohomological stable envelope correspondence $[\Stab_{+}]$ for $(X,\sw,\sT,\bC^*,+)$. 

\item If we fix $\mathsf s\in \Pic_\sA(X)\otimes_\bZ \bR$, then there exists a $K$-theoretic stable envelope correspondence $[\Stab^{\mathsf s}_{+}]$ for $(X,\sw,\sT,\bC^*,+,\mathsf s)$.
\end{itemize}
\end{Proposition}

\begin{proof}
We first prove the cohomology version. According to Lemma \ref{can map induce stab}, it suffices to construct $[\Stab_{+}]$ for $(X,0,\sT,\bC^*,+)$. 
$\Attr^f:=\Attr^f_+$ admits a filtration\,\footnote{To prove the existence of such filtration, let us define $Z_i$ inductively by setting $Z_1=\Attr(\Delta)$ and 
$Z_i:=Z_{i-1}\cup \Attr(\overline Z_{i-1}\setminus Z_{i-1})$. 
Note that $\overline Z_{i-1}\setminus Z_{i-1}$ are unions of torus fixed locus, which are distinct for different $i$. As the number of torus fixed components is finite, 
the induction process terminates at some $Z_n$ (and $Z_n$ is closed in $X\times X^\sA$). Then $Z_n$ contains $\Attr^f$ by the definition of $\Attr^f$. $\Attr^f$ contains $Z_1$ by definition; thus it contains $\overline{Z}_1$ since it is closed, then it follows that $\Attr^f$ contains $Z_2$, and by induction $\Attr^f$ contains $Z_n$. Thus $\Attr^f=Z_n$.} $$\emptyset=Z_0\subseteq  Z_1\subseteq  Z_2\subseteq \cdots\subseteq  Z_n=\Attr^f$$ by open subvarieties $Z_i$, such that $Z_1=\Attr(\Delta)$ where $\Delta\subseteq  X^\sigma\times X^\sigma$ is the diagonal, and
$$Z_i=Z_{i-1}\cup \Attr(\overline Z_{i-1}\setminus Z_{i-1}),  $$ 
where $\Attr(-)$ is the attracting set in $X\times X^\sigma$. 
We write 
\begin{align*}
    \overline Z_{i-1}\setminus Z_{i-1}=\bigcup_{F_j\in \Fix_\sigma(X)}(G_{j_i}\times F_j)\cap \Attr^f, 
\end{align*}
where $G_{j_i}$ is a union of torus fixed components, and $\bigcup_{F_j\in \Fix_\sigma(X)}(G_{j_i}\times F_j)$ is the union of $F'\times F$ for all $F'\succ F$ such that $(F'\times F)\cap (\overline Z_{i-1}\setminus Z_{i-1})\neq \emptyset$. 
 Then by construction 
$$Z_i\setminus Z_{i-1}=\bigcup_{F_j\in \Fix_\sigma(X)}\Attr\left((G_{j_i}\times F_j)\cap \Attr^f\right). $$ 
We claim that there exists $[\Stab]_i$ in equivariant Chow homology $A^\sT(Z_i)$ which satisfies the following two axioms:
\begin{enumerate}[(i)]
\setlength{\parskip}{1ex}

\item For any fixed component $F \in \Fix_\sA(X)$, $[\Stab]_i\big|_{F\times F}  = e^\sT (N_{F / X}^-) \cdot [\Delta_F]$;

\item For any $F' \succ F$ such that $(F'\times F)\cap Z_{i}\neq \emptyset$, $\deg_\sigma [\Stab]_i \big|_{F'\times F} < \deg_\sigma e^\sT (N_{F' / X}^-)$.

\end{enumerate}
Then $[\Stab_{+ }]$ is the image of $[\Stab]_n$ under the cycle map $$A^\sT(\Attr^f)\to H^\sT(\Attr^f)$$ sending algebraic cycle to Borel-Moore homology class. 
Define $$[\Stab]_1:=\sum_{F\in \Fix_\sigma(X)}[\Attr(F\times F)], $$ 
then $[\Stab]_1$ obviously satisfies axiom (i). For $i\geqslant 2$ and assume that we have constructed $[\Stab]_{i-1}\in A^\sT(Z_{i-1})$, then we use the surjective map $A^\sT(Z_i)\twoheadrightarrow A^\sT(Z_{i-1})$ to find a preimage of $[\Stab]_{i-1}$ in $A^\sT(Z_i)$, denoted by $[\Stab]^\sim_{i-1}$. Write $Z_i\setminus Z_{i-1}=\Attr(M_i\cap \Attr^f)$, 
where $$M_i=\bigcup_{F_j\in \Fix_\sigma(X)}G_{j_i}\times F_j, $$ 
then define $$\beta:=[\Stab]^\sim_{i-1}\big|_{M_i}\in A^\sT\left(M_i\cap \Attr^f\right).$$
Let $\alpha$ be the unique class in $A^\sT\left(M_i\cap \Attr^f\right)$ such that 
$$\deg_\sigma\left(\beta-e^\sT(N_{X^\sigma/X}^-)\cdot \alpha\right)<\deg_\sigma e^\sT(N_{X^\sigma/X}^-). $$
Here the existence uses the fact that $\sA$ is one dimensional, so one can use Euclidean algorithm to find $\alpha$. 
The uniqueness follows from the fact that given another $\alpha'$ satisfying above inequality, then the difference of the LHS also 
has degree less than the RHS,~i.e.
$$\deg_\sigma \left((\alpha'-\alpha)\cdot e^\sT(N_{X^\sigma/X}^-)\right)< \deg_\sigma e^\sT(N_{X^\sigma/X}^-), $$
which implies that $\alpha'=\alpha$.

Let $p_i\colon Z_i\setminus Z_{i-1}=\Attr(M_i\cap \Attr^f)\to M_i\cap \Attr^f$ be the projection, $j_i\colon Z_i\setminus Z_{i-1}\hookrightarrow Z_i$ be the closed embedding, and define 
$$[\Stab]_{i}:=[\Stab]^\sim_{i-1}-j_{i*}p_i^*(\alpha). $$ 
Then $[\Stab]_{i}$ satisfies axiom (i) since $F\times F\subseteq   Z_{i-1}$.  
Note that we have  
$$\Attr(M_i\cap \Attr^f)=\Attr(M_i\cap Z_i)=\Attr(M_i)\cap Z_i, $$ 
and following Cartesian diagrams
$$
\xymatrix{
M_i\cap Z_i  \ar[r]^{ } \ar[d]^{ } \ar@{}[dr]|{\Box}  & \Attr(M_i)\cap Z_i \ar[d]^{ } \ar[r]^{\quad \quad j_i} \ar@{}[dr]|{\Box}  &  Z_i     \ar[d]^{ } \\
M_i \ar[r]^{ } & \Attr(M_i) \ar[r]^{ } &  X\times X^\sigma,   
} 
$$
therefore
\begin{align*}
    (M_i\hookrightarrow X\times X^\sigma)^*j_{i*}p_i^*(\alpha)&=(M_i\hookrightarrow \Attr(M_i))^*(\Attr(M_i)\hookrightarrow X\times X^\sigma)^*j_{i*}p_i^*(\alpha)\\
    &=e^\sT(N_{X^\sigma/X}^-)\cdot (M_i\hookrightarrow \Attr(M_i))^*p_i^*(\alpha)\\
    &=e^\sT(N_{X^\sigma/X}^-)\cdot \alpha,
\end{align*}
thus
\begin{align*}
    \deg_\sigma[\Stab]_{i}\big|_{M_i}=\deg_\sigma\left(\beta-e^\sT(N_{X^\sigma/X}^-)\cdot \alpha\right)<\deg_\sigma e^\sT(N_{X^\sigma/X}^-),
\end{align*}
i.e.~$[\Stab]_{i}$ satisfies axiom (ii).
This proves the cohomology version.

For the $K$-theory version, we claim that there exists $[\Stab^{\mathsf s}]_i$ in equivariant $K$-theory $K^\sT(Z_i)$ which satisfies the following two axioms:
\begin{enumerate}
\setlength{\parskip}{1ex}

\item[(i')] For any fixed component $F \in \Fix_\sA(X)$, $[\Stab^{\mathsf s}]_i|_{F\times F}  = e^\sT_K (N_{F / X}^-) \cdot [\Delta_F]$;

\item[(ii')] For any $F' \succ F$ such that $(F'\times F)\cap Z_{i}\neq \emptyset$, $\deg_\sigma [\Stab^{\mathsf s}]_i\big|_{F'\times F} \subseteq  \deg_\sigma e^\sT_K (N_{F' / X}^-)+\mathrm{shift}_{F'}-\mathrm{shift}_{F}$.

\end{enumerate}
Then $[\Stab^{\mathsf s}_{+ }]$ is $[\Stab^{\mathsf s}]_n$. Define 
$$[\Stab^{\mathsf s}]_1:=[\mathcal O_{Z_1}],$$ 
then $[\Stab]_1$ obviously satisfies axiom (i). For $i\geqslant 2$ and assume that we have constructed $[\Stab^{\mathsf s}]_{i-1}\in K^\sT(Z_{i-1})$, then we use the surjective map $K^\sT(Z_i)\twoheadrightarrow K^\sT(Z_{i-1})$ to find a preimage of $[\Stab^{\mathsf s}]_{i-1}$ in $K^\sT(Z_i)$, denoted by $[\Stab]^\sim_{i-1}$. Similar to the cohomological counterpart, we define
\begin{align*}
&\quad \quad \quad \quad \quad\quad\quad \alpha,\beta\in K^\sT\left(M_i\cap \Attr^f\right),\text{ such that }\beta:=[\Stab^{\mathsf s}]^\sim_{i-1}\big|_{M_i},\text{ and }\\
&\deg_\sigma\left((\beta-e^\sT_K(N_{X^\sigma/X}^-)\cdot \alpha)|_{F'\times F}\right)\subsetneq  \deg_\sigma e^\sT_K (N_{F' / X}^-)+\mathrm{shift}_{F'}-\mathrm{shift}_{F}.
\end{align*}
Define $$[\Stab]_{i}:=[\Stab]^\sim_{i-1}-j_{i*}p_i^*(\alpha). $$ Then $[\Stab]_{i}$ satisfies axiom (i) since $F\times F\subseteq   Z_{i-1}$; $[\Stab]_{i}$ satisfies axiom (ii) as
\begin{align*}
    (M_i\hookrightarrow X\times X^\sigma)^*j_{i*}p_i^*(\alpha)&=(M_i\hookrightarrow \Attr(M_i))^*(\Attr(M_i)\hookrightarrow X\times X^\sigma)^*j_{i*}p_i^*(\alpha)\\
    &=e^\sT_K(N_{X^\sigma/X}^-)\cdot (M_i\hookrightarrow \Attr(M_i))^*p_i^*(\alpha)\\
    &=e^\sT_K(N_{X^\sigma/X}^-)\cdot \alpha,
\end{align*}
thus 
\begin{align*}
    \deg_\sigma[\Stab]_{i}\big|_{F'\times F}=\deg_\sigma\left((\beta-e^\sT_K(N_{X^\sigma/X}^-)\cdot \alpha)|_{F'\times F}\right)\subsetneq  \deg_\sigma e^\sT_K (N_{F' / X}^-)+\mathrm{shift}_{F'}-\mathrm{shift}_{F}.
\end{align*}
This proves the $K$-theory version.
\end{proof}
When ambient spaces are symplectic, we have the following existence result.
\begin{Theorem}[{\cite[Prop.~3.5.1]{MO}}]
\label{thm on mo stab ex}
Suppose that $X$ is a smooth symplectic variety with symplectic form $\Omega$ such that $\sA$ fixes $\Omega$ and $\sT$ scales $\Omega$. Fix a chamber $\fC$,
and a $\sT$-invariant function $\sw\colon X\to \C$. 
Then cohomological stable envelope correspondence exists.
\end{Theorem}

\begin{proof}
By Lemma \ref{can map induce stab}, we are left construct cohomological stable envelope correspondence for $(X,0,\sT,\sA,\fC )$, which 
is done in \cite[Prop.~3.5.1]{MO}.
\end{proof}

When ambient spaces are \textit{symmetric GIT quotients}, we have the following existence result.
\begin{Theorem}\label{exi of coh stab on symm quiver var}
Suppose that $X$ is a symmetric GIT quotient with an action of tori $\sA\subseteq \sT$ as Definition \ref{def of sym var}. Fix a chamber $\fC$,
and a $\sT$-invariant function $\sw\colon X\to \C$. 
Then cohomological stable envelope correspondence exists. 
\end{Theorem}
\begin{proof}
By Lemma \ref{can map induce stab}, we are left to construct cohomological stable envelope correspondence for $(X,0,\sT,\sA,\fC )$, which 
is constructed in Theorem \ref{construct coh stab corr on symm quiver var}.
\end{proof}

\begin{Remark}\label{rmk on sd con}
Things might go wrong without the self-duality assumption on the $\sA$-action. For example, take $X=\mathrm{Tot}\left(\mathcal O_{\bP^1}(-1)^{\oplus 2}\right)$ with zero potential and let $\sT=\sA=\bC^*_a\times\bC^*_b$ act on $X$, where $\bC^*_a$ is the natural torus symmetry on $\bP^1$ and extends to act on $X$ by setting one of $\mathcal O_{\bP^1}(-1)$ to have zero weight at $0$ and another one to have zero weight at $\infty$, and $\bC^*_b$ acts on $\bP^1$ trivially and scales the fibers of $\mathcal O_{\bP^1}(-1)^{\oplus 2}$ with weight $1$. Then cohomological stable envelope does not exist in the chamber $\{0<b<a\}$.
We refer to Proposition \ref{prop on 3 ex} for a $K$-theoretic example. 
\end{Remark}


\begin{Remark}\label{rmk on corr is need}
Elliptic stable envelopes have been constructed for symmetric quiver varieties with zero potential in \cite{IMZ}, and it is shown in \textit{loc.}~\textit{cit.} that $K$-theoretic and cohomological limits of elliptic stable envelopes give rise to classes $[\mathsf{Stab}^{\mathsf s}_\fC]\in K^\sT(X\times X^\sA)$ and $[\mathsf{Stab}_\fC]\in H^\sT(X\times X^\sA)$ that satisfy the axioms (i) (ii) in Definition \ref{stab corr_k} and \ref{stab corr_coh} respectively. Moreover one can show that $[\mathsf{Stab}^{\mathsf s}_\fC]$ and $[\mathsf{Stab}_\fC]$ are supported on $\Attr^f_\fC$, in particular they can be lifted to classes in $K^\sT(\Attr^f_\fC)$ and $H^\sT(\Attr^f_\fC)$ respectively. 

However, the lifts are \textit{not necessarily unique} because neither $K^\sT(\Attr^f_\fC)\to K^\sT(X\times X^\sA)$ nor $H^\sT(\Attr^f_\fC)\to H^\sT(X\times X^\sA)$ is injective in general. It is not clear whether we can choose a lift to make it a stable envelope correspondence for the zero potential in the sense of Definition \ref{stab corr_k} or \ref{stab corr_coh} respectively. Therefore introducing stable envelope correspondences is necessary. 
\end{Remark}

\subsection{Transpose of stable envelope correspondences} 

\subsubsection{Convolutions}
Let $(X_i,\sw_i)$ ($i=1,2,3$) be as in Setting \ref{setting of stab} and $p_{ij}\colon X_1\times X_2\times X_3\to X_i\times X_j$ the projection.
Let $Z_{12}\subseteq  X_1\times X_2$, $Z_{23}\subseteq  X_2\times X_3$ be $\sT$-invariant closed subschemes such that 
$$p_{13}\colon p_{12}^{-1}(Z_{12})\cap p_{23}^{-1}(Z_{23})\to X_1\times X_3$$
is a proper map and denote the image by 
$$Z_{12}\circ Z_{23}:=p_{13}(p_{12}^{-1}(Z_{12})\cap p_{23}^{-1}(Z_{23})). $$ 
We define the \textit{convolution} on critical cohomology (e.g.~\cite[pp.~112]{CG},~\cite[\S 2.4]{VV2}): 
\begin{equation}\label{conv coh}
\circ\colon H^\sT(X_1\times X_2, \sw_1\boxminus\sw_2)_{Z_{12}}\otimes H^\sT(X_2\times X_3, \sw_2\boxminus\sw_3)_{Z_{23}}  
\to H^\sT(X_1\times X_3, \sw_1\boxminus\sw_3)_{Z_{12}\circ Z_{23}}, \end{equation}
$$(\beta,\alpha)\mapsto \beta\circ \alpha:=p_{13*}(p_{12}^*\beta\otimes p_{23}^*\alpha). $$
Here the tensor product $\otimes$ is defined generally as follows: 
for a smooth $Y$, closed subschemes $Z,Z'\subseteq  Y$ and  
$$a\in H^\sT(Y,\sw)_{Z}, \quad b\in H^\sT(Y,\sw')_{Z'}, $$
the Thom-Sebastiani isomorphism implies 
$$a\boxtimes b\in H^\sT(Y\times Y,\sw\boxplus \sw')_{Z\times Z'}, $$
then the diagonal pullback gives 
$$a\otimes b:=\Delta^*(a\boxtimes b)\in H^\sT(Y,\sw+ \sw')_{Z\cap Z'}.$$
Similarly we define the \textit{convolution} on critical $K$-theory (e.g.~\cite[pp.~251]{CG}, \cite[\S 2.3.2]{VV2}):
\begin{equation}\label{conv k}
\circ\colon K^\sT(X_1\times X_2, \sw_1\boxminus\sw_2)_{Z_{12}}\otimes K^\sT(X_2\times X_3, \sw_2\boxminus\sw_3)_{Z_{23}}  
\to K^\sT(X_1\times X_3, \sw_1\boxminus\sw_3)_{Z_{12}\circ Z_{23}}, \end{equation}
$$(\beta,\alpha)\mapsto \beta\circ \alpha:=p_{13*}(p_{12}^*\beta\otimes p_{23}^*\alpha). $$
Here the derived tensor product $\otimes$ is defined generally as follows: for a smooth $Y$, closed subschemes $Z,Z'\subseteq  Y$ and 
$$a\in K^\sT(Y,\sw)_{Z}, \quad b\in K^\sT(Y,\sw')_{Z'}, $$
derived (exterior) tensor product gives 
$$a\boxtimes b\in K^\sT(Y\times Y,\sw\boxplus \sw')_{Z\times Z'}, $$
then the diagonal pullback gives 
$$a\otimes b:=\Delta^*(a\boxtimes b)\in K^\sT(Y,\sw+ \sw')_{Z\cap Z'}.$$

\begin{Remark}\label{rmk: conv on coh supp}
For the convolution on critical cohomology \eqref{conv coh} to be well-defined, it is enough to assume that $p_{12}^{-1}(Z_{12})\cap p_{23}^{-1}(Z_{23})\cap \left(\Crit(\sw_1)\times \Crit(\sw_2)\times \Crit(\sw_3)\right)$ is proper over $X_1\times X_3$. This is because the vanishing cycle sheaf $\varphi_{\sw_i}\omega_{X_i}$ is supported on $\Crit(\sw_i)$, $i\in \{1,2,3\}$.
\end{Remark}

\subsubsection{Transpose correspondences}\label{sec: transpose correspondences}

Given a class $[\alpha]\in H^\sT(X\times Y, \sw_1\boxminus\sw_2)_Z$, we define its \textit{transpose}  
\begin{equation}\label{equ on transp}[\alpha]^{\mathrm{t}}\in H^\sT(Y\times X, \sw_2\boxminus\sw_1)_Z \end{equation} 
to be the image of $[\alpha]$ under the natural isomorphism $X\times Y\cong Y\times X$ followed by the automorphism $\bA^1\xrightarrow{-1}\bA^1$ on the target of the potential function. The $K$-theory version is defined similarly. 

\begin{Lemma}\label{adj of stab}
Under Setting \ref{setting: proper to affine}, let $[\Stab_{\fC }]$ be a cohomological stable envelope correspondence for $(X,\sw,\sT,\sA,\fC )$, then 
the convolution  
$$[\Stab_{-\fC}]^{\mathrm{t}}\circ [\Stab_{\fC }]\in H^\sT(X^\sA\times X^\sA,\sw^\sA\boxminus \sw^\sA)_{(\Attr^f_{-\fC})^{\mathrm{t}}\circ \Attr^f_{\fC}}$$
equals to 
the diagonal class $[\Delta]$. 

Let $\mathsf s$ be a generic slope, and $[\Stab^{\mathsf s}_{\fC }]$ be a $K$-theoretic stable envelope correspondence for $(X,\sw,\sT,\sA,\fC ,\mathsf s)$, then the convolution
$$[\Stab^{-\mathsf s}_{-\fC}]^{\mathrm{t}}\circ [\Stab^{\mathsf s+\frac{\mathcal K}{2}}_{\fC }]\in K^\sT(X^\sA\times X^\sA,\sw^\sA\boxminus \sw^\sA)_{(\Attr^f_{-\fC})^{\mathrm{t}}\circ \Attr^f_{\fC}}$$ equals to $[\mathcal O_\Delta]$. Here $\mathcal K=\det\Omega^1_X$ is the canonical bundle of $X$ and $+$ denotes the group operation of $\Pic_\sA(X)\otimes_{\bZ}\bR$.
\end{Lemma}

\begin{proof}
The proof is similar to \cite[Thm.\,4.4.1]{MO} for the cohomology version and \cite[Prop.\,1]{OS} for the $K$-theory version. We give details for the $K$-theory version. To simplify notation, let us write
\begin{align*}
    \mathsf s_1=\mathsf s+\frac{\mathcal K}{2}\:.
\end{align*}
The diagonal map $X\hookrightarrow X\times X$ induces an embedding
$$\Delta_{23}\colon X^\sA\times X\times X^\sA\hookrightarrow X^\sA\times X\times X\times X^\sA.$$
Denote $C=\Delta_{23}^{-1}\left((\Attr^f_{-\fC})^{\mathrm{t}}\times \Attr^f_{\fC}\right)$. By the proof of \cite[Thm.\,4.4.1]{MO}, $C$ is proper over $X^\sA\times X^\sA$; therefore 
\begin{align}\label{equ on stabtrans}
    [\Stab^{-\mathsf s}_{-\fC}]^{\mathrm{t}}\circ [\Stab^{\mathsf s_1}_{\fC }]:=p_{*}\Delta_{23}^*\left([\Stab^{-\mathsf s}_{-\fC}]^{\mathrm{t}}\boxtimes [\Stab^{\mathsf s_1}_{\fC }]\right)\in K^\sT(X^\sA\times X^\sA,\sw^\sA\boxminus \sw^\sA)_{p(C)}
\end{align}
is well-defined. Here $p\colon X^\sA\times X\times X^\sA\to X^\sA\times X^\sA$ is the projection, and we denote $(\Attr^f_{-\fC})^{\mathrm{t}}\circ \Attr^f_{\fC}:=p(C)$.

We compute \eqref{equ on stabtrans} using $\sA$-localization: its restriction to the component $F_1\times F_2$ equals to
\begin{align}\label{gamma loc form}
    \sum_{F\in \Fix_\sA(X)} \frac{\left([\Stab^{-\mathsf s}_{-\fC}]|_{F\times F_1}\right)^{\mathrm{t}}\circ \left([\Stab^{\mathsf s_1}_{\fC }]|_{F\times F_2}\right)}{e^\sT_K(N_{F/X})}.
\end{align}
According to axiom (ii) in Definition \ref{stab corr_k}, we have an inclusion 
\begin{multline*}
    \deg_\sA \left([\Stab^{-\mathsf s}_{-\fC}]\big|_{F\times F_1}\right)^{\mathrm{t}}\circ \left([\Stab^{\mathsf s_1}_{\fC }]\big|_{F\times F_2}\right)\subseteq  \deg_\sA e^\sT_K(N_{F/X})\\
    +\wt_\sA\left(\det(N^+_{F_1/X})^{-1/2}\otimes\mathsf s|_{F_1}\right)-\wt_\sA\left(\det(N^-_{F_2/X})^{1/2}\otimes\mathsf s_1|_{F_2}\right).
\end{multline*}
In the above, we have used $\wt_\sA\det(N^-_{F/X})\otimes\det(N^+_{F/X})\otimes\mathcal K|_F=0$. Therefore the function 
\begin{align}\label{shifted gamma}
    [\Stab^{-\mathsf s}_{-\fC}]^{\mathrm{t}}\circ [\Stab^{\mathsf s_1}_{\fC }]\cdot \left(\det(N^-_{F_1/X})^{-1/2}\otimes(\mathsf s|_{F_1})^{-1}\boxtimes \det(N^-_{F_2/X})^{1/2}\otimes\mathsf s_1|_{F_2}\right)
\end{align}
in the variable $a\in \sA$ is bounded in any direction of taking limit $a\to\infty$. 

Since $[\Stab^{-\mathsf s}_{-\fC}]^{\mathrm{t}}\circ [\Stab^{\mathsf s_1}_{\fC }]$ is defined in integral $K$-theory, then \eqref{shifted gamma} is a constant function in $a$.
 In particular,
\begin{align*}
    \deg_\sA ([\Stab^{-\mathsf s}_{-\fC}]^{\mathrm{t}}\circ [\Stab^{\mathsf s_1}_{\fC }])\big|_{F_1\times F_2}=\wt_\sA\left(\det(N^-_{F_1/X})^{1/2}\otimes\mathsf s|_{F_1}\right)-\wt_\sA\left(\det(N^-_{F_2/X})^{1/2}\otimes\mathsf s_1|_{F_2}\right).
\end{align*}
Since $[\Stab^{-\mathsf s}_{-\fC}]^{\mathrm{t}}\circ [\Stab^{\mathsf s_1}_{\fC }]$ is supported on $(\Attr^f_{-\fC})^{\mathrm{t}}\circ \Attr^f_{\fC}$, $([\Stab^{-\mathsf s}_{-\fC}]^{\mathrm{t}}\circ [\Stab^{\mathsf s_1}_{\fC }])\big|_{F_1\times F_2}$ vanishes if $F_2\npreceq F_1$. If $F_2\prec F_1$, then $\wt_\sA\left(\det(N^-_{F_1/X})^{1/2}\otimes\mathsf s|_{F_1}\right)-\wt_\sA\left(\det(N^-_{F_2/X})^{1/2}\otimes\mathsf s_1|_{F_2}\right)$ is nonintegral since $\mathsf s$ is generic, and this implies that $([\Stab^{-\mathsf s}_{-\fC}]^{\mathrm{t}}\circ [\Stab^{\mathsf s_1}_{\fC }])\big|_{F_1\times F_2}$ vanishes. When $F_1=F_2$, the only nonzero summand in \eqref{shifted gamma} is $F=F_1$, which gives $[\oO_\Delta]$. It follows that $[\Stab^{-\mathsf s}_{-\fC}]^{\mathrm{t}}\circ [\Stab^{\mathsf s_1}_{\fC }]=[\oO_\Delta]$.
\end{proof}

\subsection{Existence of stable envelopes via convolutions}\label{sect on exi of stb}
We work under Setting \ref{setting: proper to affine}. 
According to Remark \ref{proper-over-affine}, $\Attr^f_\fC$ is proper over $X$; therefore any class 
$$[\alpha]\in H^\sT(X\times X^\sA,\sw\boxminus \sw^\sA)_{\Attr^f_\fC}$$ induces a map (by convolution \eqref{conv coh}):
\begin{equation}\label{equ on convo on coh}\alpha\colon H^\sT(X^\sA,\sw^\sA)_Z\to H^\sT(X,\sw)_{\Attr^f_\fC\circ Z}, \quad \gamma\mapsto [\alpha]\circ \gamma. \end{equation}
Similarly, any $K$-theory class $$[\beta]\in K^\sT(X\times X^\sA,\sw\boxminus \sw^\sA)_{\Attr^f_\fC}$$ 
induces a map (by convolution \eqref{conv k}):
\begin{equation}\label{equ on convo on k}\beta\colon K^\sT(X^\sA,\sw^\sA)_Z\to K^\sT(X,\sw)_{\Attr^f_\fC\circ Z}, \quad \gamma\mapsto [\beta]\circ \gamma. \end{equation}
One can check that convolutions \eqref{conv coh}, \eqref{conv k} are associative, so compositions of convolutions induce
compositions of maps induced. 
Convolutions are also compatible with canonical maps as shown below.
\begin{Lemma}\label{lem on can cm w conv}
Let $[\beta]\in K^\sT(X\times X^\sA)_{\Attr^f_\fC}$ and $\can([\beta])\in K^\sT(X\times X^\sA,\sw\boxminus \sw^\sA)_{\Attr^f_\fC}$ be its image under canonical map \eqref{equ on can map k with sp}. Then we have a commutative diagram 
\begin{equation}\label{cd: can cm w conv_k}
\xymatrix{
K^{\sT}(X^\sA) \ar[d]_{\beta}  & \ar[l]_{i^\sA_*} K^{\sT}(X^\sA)_{Z(\sw^{\sA})}  \ar[r]^{\can} \ar[d]_{\beta} &  K^{\sT}(X^\sA,\sw^\sA) \ar[d]^{\can(\beta)}  \\
K^{\sT}(X) & \ar[l]_{i_*} K^{\sT}(X)_{Z(\sw)} \ar[r]^{\can} &  K^{\sT}( X,\sw),
} 
\end{equation}
where $i\colon Z(\sw)\hookrightarrow X$ and $i^\sA\colon Z(\sw^\sA)\hookrightarrow X^\sA$ are embeddings. 

Similarly, let $[\alpha]\in H^\sT(X\times X^\sA)_{\Attr^f_\fC}$ and $\can([\alpha])\in H^\sT(X\times X^\sA,\sw\boxminus \sw^\sA)_{\Attr^f_\fC}$ be its image under canonical map \eqref{can_coh}. Then we have a commutative diagram 
\begin{equation}\label{cd: can cm w conv_coh}
\xymatrix{
H^{\sT}(X^\sA) \ar[d]_{\alpha}  & \ar[l]_{i^\sA_*} H^{\sT}(X^\sA)_{Z(\sw^{\sA})}  \ar[r]^{\can} \ar[d]_{\alpha} &  H^{\sT}(X^\sA,\sw^\sA) \ar[d]^{\can(\alpha)}  \\
H^{\sT}(X) & \ar[l]_{i_*} H^{\sT}(X)_{Z(\sw)} \ar[r]^{\can} &  H^{\sT}( X,\sw).
} 
\end{equation}
\end{Lemma}
\begin{proof}
The canonical maps are well-defined as $\Attr^f_\fC\subseteq  Z\left(\sw\boxminus \sw^\sA\right)$. 
Note also that 
$$\Attr^f_\fC\cap\, (X\times Z(\sw^{\sA}))\subseteq  Z(\sw^{})\times Z(\sw^{\sA}), $$
therefore the convolutions $\beta$ in the middle vertical arrow of \eqref{cd: can cm w conv_k} and $\alpha$ in the middle vertical arrow of \eqref{cd: can cm w conv_coh} are well-defined. The left square of \eqref{cd: can cm w conv_k} commutes as proper pushforward commutes with flat pullbacks and Gysin pullbacks, similarly the left square of \eqref{cd: can cm w conv_coh} commutes for the same reason. The right square of \eqref{cd: can cm w conv_k} commutes as canonical map \eqref{equ on can map k with sp} commutes with flat pullbacks, proper pushforwards and Gysin pullbacks 
(e.g.~\cite[Lem.~2.4]{VV2}\footnote{When working with torus \textit{invariant} functions, the flatness condition on functions imposed in  \cite[Lem.~2.4]{VV2} can be dropped off by following the argument of \cite[Lem.~2.4.7]{Toda1}.}). Similarly the right square of \eqref{cd: can cm w conv_coh} commutes as canonical map \eqref{can_coh} commutes with smooth pullbacks, proper pushforwards and Gysin pullbacks.
\end{proof}


By applying convolutions to stable envelope correspondences (Definitions \ref{stab corr_coh}, \ref{stab corr_k}), we obtain:
\begin{Proposition}\label{corr induce stab}
Let $[\Stab_{\fC }]$ be a cohomological stable envelope correspondence for $(X,\sw,\sT,\sA,\fC )$, then 
$$\Stab_{\fC }\colon H^\sT(X^\sA,\sw^\sA)\to H^\sT(X,\sw)$$ is a cohomological stable envelope.

Let $[\Stab^{\mathsf s}_{\fC }]$ be a $K$-theoretic stable envelope correspondence for $(X,\sw,\sT,\sA,\fC ,\mathsf s)$, then 
$$\Stab^{\mathsf s}_{\fC }\colon K^\sT(X^\sA,\sw^\sA)\to K^\sT(X,\sw)$$ is a $K$-theoretic stable envelope.
\end{Proposition}

\begin{proof}
We check the axioms in Definition \ref{def of stab coho} for the cohomological case, the $K$-theoretic counterpart is similar. 
We work with a torus fixed component $F$ of $X^\sA$. 
Denote projections $\pi_{1}\colon X\times F\to X$, $\pi_2\colon X\times F \to F$. 
The axiom (i) is obvious as we have 
$$\pi_1\left(\Attr^f_\fC\,\bigcap \,(X\times F)\right)\subseteq  \Attr^f_\fC(F). $$ 
The axiom (ii) holds because
\begin{align*}
(F\hookrightarrow X)^*\pi_{1*}\left([\Stab_{\fC }]\otimes \pi_2^*\gamma\right)&=\pi_{1*}\left((F\times F\hookrightarrow X\times F)^*[\Stab_{\fC }]\otimes (F\times F\to F)^*\gamma\right)\\
&= e^\sT(N^-_{F/X})\cdot \gamma.
\end{align*}
The axiom (iii) holds because
\begin{align*}
(F'\hookrightarrow X)^*\pi_{1*}\left([\Stab_{\fC }]\otimes \pi_2^*\gamma\right)=\pi_{1*}\left((F'\times F\hookrightarrow X\times F)^*[\Stab_{\fC }]\otimes (F'\times F\to F)^*\gamma\right),
\end{align*}
and $\deg_\sA [\Stab_{\fC }]\big|_{F'\times F} < \deg_\sA e^\sT (N_{F' / X}^-)$.
\end{proof}
Combining Proposition \ref{corr induce stab} with Proposition \ref{prop on stable corr} and Theorems \ref{thm on mo stab ex}, \ref{exi of coh stab on symm quiver var}, we obtain:
\begin{Theorem}\label{thm on exist of stab}
Let $(X,\sw,\sT,\sA)$ be in Setting \ref{setting of stab} and Setting \ref{setting: proper to affine}, $\fC$ be a chamber and $\mathsf s$ be a slope. 

Its cohomological stable envelope exists  when either
$$ \,\, (1)\,\,\mathsf A\cong\bC^*, \, \mathrm{or} \,\,\, (2)\,\,  X \mathrm{\,\,is\,\, a\,\, smooth\,\,symplectic\,\, variety}, \mathrm{\,such\,\, that\,\, \sA\,\, fixes\,\,the\,\,symplectic\,\,form\,\, and\,\, \sT\,\, scales\,\, it},
\, \mathrm{or}$$ 
$$ (3)\,\, X \mathrm{\,\, is\,\, a\,\, symmetric\,\, GIT\,\,quotient},\, \mathrm{with\,\,an\,\,action\,\,by\,\,tori\,\, \sA\subseteq \sT\,\,as\,\,Definition\,\, \ref{def of sym var}}. $$
In the first case, the $K$-theoretic stable envelope also exists.
\end{Theorem} 

\begin{Definition}\label{def of stab1}
Let $(X,\sw,\sT)$ be in Settings \ref{setting of stab}, \ref{setting: proper to affine}. Fix a nontrivial $\sigma:\C^*\to \sT$, we take $\sA$ to be the image of $\sigma$. 
\begin{itemize} 

\item
The cohomological stable envelope in Theorem \ref{thm on exist of stab} (1) is denoted by $\Stab_{\sigma }$.  

\item
Fix $\mathsf s\in \Pic_\sigma(X)\otimes_\bZ \bR$, the $K$-theoretic stable envelope in Theorem \ref{thm on exist of stab} (1) is denoted 
by $\Stab^{\mathsf s}_{\sigma }$.
\end{itemize} 
\end{Definition}
The above cohomological stable envelope $\Stab_{\sigma }$  has the following nice property, which will be used only in the proof of
Proposition \ref{prop:vb and stab_repl_corr}.
\begin{Proposition}\label{prop:coh stab ex criterion}
Let $(X,\sw,\sT,\sA)$ be in the Setting \ref{setting of stab} and Setting \ref{setting: proper to affine} and $\fC$ a chamber. Then the cohomological stable envelope $\Stab_{\fC }$ exists if and only if $\Stab_{\sigma }=\Stab_{\sigma' }$ for all $(\sigma,\sigma')\in \fC\times\fC$.
\end{Proposition}

\begin{proof}
Suppose that $\Stab_{\fC }$ exists, then $\Stab_{\sigma }=\Stab_{\fC }$ for all $\sigma\in \fC$. Conversely, if $\Stab_{\sigma }=\Stab_{\sigma' }$ for all $(\sigma,\sigma')\in \fC\times\fC$, then for arbitrary $\sigma\in \fC$, $\Stab_{\sigma }$ satisfies all three axioms in Definition \ref{def of stab coho}. Axioms (i) and (ii) hold automatically, and the axiom (iii) holds since $\deg_\sA f<\deg_\sA g$ if and only if $\deg_\sigma f<\deg_\sigma g$ for a set of $\sigma$ that spans $\Lie\sA$. 
\end{proof}



\subsection{Existence of stable envelopes via Hall operations}\label{sect on ex kstab}

For any symmetric GIT quotient $X$ with an action of tori $\sA\subseteq \sT$ as Definition \ref{def of sym var}, and any $\sT$-invariant function $\sw\colon X\to \C$, 
its $K$-theoretic 
stable envelope $\Stab^{\mathsf s}_\fC$ (for a generic slope $\mathsf s$) can be obtained by commutative diagram 
\begin{equation}\label{diag define Kstab}
\xymatrix{
K^\sT(\fX^{\sA},\sw) \ar@{=}[r] & 
\bigoplus_{\phi/\sim} K^\sT(\fX^{\sA,\phi},\sw) \ar[rr]^{\quad \,\, \bigoplus_{\phi/\sim}\mathfrak{m}^\phi_{\fC}} & & K^\sT(\fX,\sw) \ar[d]^{\mathrm{res}} \\
K^\sT(X^{\sA},\sw) \ar@{=}[r] &
\bigoplus_{\phi/\sim} K^\sT(X^{\sA,\phi},\sw) \ar[u]^{\bigoplus_{\phi/\sim}\bPsi_K^{\phi,\mathsf s'}} \ar[rr]^{\quad \,\, \Stab^{\mathsf s}_\fC} & & K^\sT(X,\sw). 
}
\end{equation}
Here the direct sum is taken over all homomorphisms $\phi:\sA\to G$ modulo certain equivalence $\sim$ \eqref{equ on dis uni of X}, 
$\bPsi_K^{\phi,\mathsf s'}$ are nonabelian stable envelopes, defined using window subcategories, $\mathfrak{m}^\phi_{\fC}$ are Hall operations on the stack quotients, 
and $\mathrm{res}$ is the restriction map to the open substack.
This diagram will be proven in Theorem \ref{thm: hall k_sym quot}. A similar diagram also holds for the cohomological case (Theorem \ref{thm: hall coh_sym quot}). 
We leave the details to \S \ref{app on k symm var}. 

We remark that $\bPsi_K^{\phi,\mathsf s'}$ has a categorical lift to a functor. Combining with the categorical Hall product (e.g.~\cite{P}), 
we may define the composed functor $\Stab^{\mathsf s}_\fC$ as a ``categorical stable envelope''.



\subsection{Integrality of the inverse of stable envelopes}

The inverse of stable envelopes can be computed using the transpose of stable envelope correspondences by Lemma \ref{adj of stab}, this has the following application
on integrality.

\begin{Proposition}\label{prop: integral inverse stab}
Under Setting \ref{setting: proper to affine}, and assume that the $\sA$ weights in $\Gamma(X,\mathcal O_X)$ are nonpositive, and that stable envelope correspondences exist. Then $\Stab_\fC^{-1}$ is defined for the integral cohomology class, i.e.~it maps $H^\sT(X,\sw)$ to $H^\sT(X^\sA,\sw^\sA)$; similarly $(\Stab^{\mathsf s}_\fC)^{-1}$ is defined for the integral $K$-theory class,\,i.e.~it maps $K^\sT(X,\sw)$ to $K^\sT(X^\sA,\sw^\sA)$.
\end{Proposition}

\begin{proof}
Let $X_0=\Spec\Gamma(X,\mathcal O_X)$, then by the assumption the induced $\sA$-action on $X_0$ is attracting with respect to $\fC$, i.e.~$X_0=\Attr_{\fC}(X_0^\sA)$. Equivalently, $X_0^\sA=\Attr_{-\fC}(X_0^\sA)$. It follows that $\Attr^f_{-\fC}\subseteq  X\times_{X_0} X^\sA$. Then $\Attr^f_{-\fC}$ is proper over $X^\sA$ since $X$ is proper over $X_0$ by Setting \ref{setting: proper to affine}. Lemma \ref{adj of stab} implies that $\Stab_\fC^{-1}$ is induced by the correspondence $[\Stab_{-\fC}]^{\mathrm{t}}$ which is then defined over integral cohomology class. Similarly, $(\Stab^{\mathsf s}_\fC)^{-1}$ is induced by the correspondence $[\Stab^{\frac{\mathcal K}{2}-\mathsf s}_{-\fC}]^{\mathrm{t}}$ which is then defined over integral $K$-theory class.
\end{proof}

\subsection{Stable envelopes with coefficient in a subtorus}

Occasionally, we also consider the equivariant critical cohomology/$K$-theory with respect to a subtorus $\sT'\subseteq  \sT$ which does \textit{not} necessarily contain $\sA$. Then we can regard stable envelope correspondences (assuming existence) in the $\sT'$-equivariant critical cohomology/$K$-theory, and induce convolution maps:
\begin{align*}
    \Stab_\fC\colon H^{\sT'}(X^\sA,\sw^\sA)\to H^{\sT'}(X,\sw),\qquad \Stab^{\mathsf s}_\fC\colon K^{\sT'}(X^\sA,\sw^\sA)\to K^{\sT'}(X,\sw).
\end{align*}
The following injectivity result will be used in \cite{COZZ1}.

\begin{Proposition}\label{prop: stab_isom_subtorus}
Under Setting \ref{setting: proper to affine}, and assume that stable envelope correspondences exist. Let $\sT'\subseteq  \sT$ be the subtorus as above. 
\begin{itemize}
    \item If $\sA$-action on $\Crit(\sw)^{\sT'}$ is attracting with respect to $\fC$ \footnote{This means for arbitrary point $x\in \Crit(\sw)^{\sT'}$, $\lim_{t\to 0} \sigma(t)\cdot x$ exists for all cocharacter $\sigma$ in $\fC$.}, then $\Stab_\fC\colon H^{\sT'}(X^\sA,\sw^\sA)\to H^{\sT'}(X,\sw)$ is injective after $\sT'$-localization.
    \item If $\sA$-action on $X^{\sT'}$ is attracting with respect to $\fC$, then $\Stab^{\mathsf s}_\fC\colon K^{\sT'}(X^\sA,\sw^\sA)\to K^{\sT'}(X,\sw)$ is injective after $\sT'$-localization.
\end{itemize}
\end{Proposition}
\begin{proof}
Let us first prove the cohomology case. In view of Lemma \ref{adj of stab}, it is enough to show that the transpose correspondence $[\Stab_{-\fC}]^{\mathrm{t}}$ induces a well-defined convolution map $$H^{\sT'}(X,\sw)_{\loc}\to H^{\sT'}(X^\sA,\sw^\sA)_{\loc}. $$ By Remark \ref{rmk: conv on coh supp}, we only need to check that the $\sT'$-fixed locus in $\Attr^f_{-\fC}\cap \left(\Crit(\sw^\sA)\times \Crit(\sw)\right)$ is proper over $X^\sA$. Let $\pi\colon X\to X_0$ be a $\sT$-equivariant morphism with affine target $X_0$ as in the Setting \ref{setting: proper to affine}. Then by the assumption, we have $\pi(\Crit(\sw)^{\sT'})\subseteq  \Attr_{\fC}(X_0^\sA)$. Since $\Attr_{\fC}(X_0^\sA)\cap \Attr_{-\fC}(X_0^\sA)=X_0^\sA$, it follows that
\begin{align*}
    \left(\Attr^f_{-\fC}\cap \left(\Crit(\sw^\sA)\times \Crit(\sw)\right)\right)^{\sT'}\subseteq  X\times_{X_0} X^\sA
\end{align*}
and the latter is proper over $X^\sA$ because $X\to X_0$ is proper. This proves the cohomology case.

For the $K$-theory case, the argument is word-by-word the same as above except that we do not have an analog of Remark \ref{rmk: conv on coh supp} for $K$-theory so we need to replace the critical locus by the entire $X$.
\end{proof}

\subsection{Triangle lemma}\label{sec on tri lem}

Let $\fC$ be a chamber and let $\fC'$ be a face of some dimension. Consider $$\mathfrak{a}'=\Span\fC'\subseteq  \mathfrak{a}=\Lie(\sA)$$with associated subtorus $\sA'\subseteq \sA$. The cone $\fC$ projects to a cone in $\mathfrak{a}/\mathfrak{a}'$ that we denote by $\fC/\fC'$. 

We show that triangle lemma always holds for $K$-theoretic stable envelopes with generic slopes.  

\begin{Lemma}\label{triangle lemma for K}
Let $\mathsf s\in \Pic_\sA(X)\otimes_{\bZ}\bR$ be a generic slope. Suppose that $K$-theoretic stable envelopes $\Stab^{\mathsf s}_{\fC}$, $\Stab^{\mathsf s}_{\fC'}$, and $\Stab^{\mathsf s'}_{\fC/\fC'}$ exist, where $$\mathsf s'=\mathsf s|_{X^{\sA'}}\otimes \det\left(N^-_{X^{\sA'}/X}\right)^{1/2}\in \Pic_{\sA}(X^{\sA'})\otimes_\bZ \bR.$$ Then the following diagram 
\begin{equation*}
\xymatrix{
K^\sT(X^\sA,\sw^\sA) \ar[rr]^{\Stab^{\mathsf s}_{\fC}} \ar[dr]_{\Stab^{\mathsf s'}_{\fC/\fC'}} & & K^\sT(X,\sw) \\
 & K^\sT(X^{\sA'},\sw^{\sA'}) \ar[ur]_{\Stab^{\mathsf s}_{\fC'}} &
}
\end{equation*}
is commutative.
\end{Lemma}

\begin{proof}
The proof is essentially the same as \cite[Prop.~9.2.8]{Oko}. We choose generic $\xi'\in \fC'$, $\xi\in \fC$ such that $$|\!|\delta\xi |\!|\ll |\!|\xi' |\!|,$$ where $\delta\xi:=\xi-\xi'$. Then for arbitrary $F\in \Fix_\sA(X)$ (denote $F'\in \Fix_{\sA'}(X)$ such that $F\subseteq  F'$) and arbitrary $\gamma\in K^{\sT/\sA}(F,\sw^\sA)$, $\Stab^{\mathsf s}_{\xi' }\circ \Stab^{\mathsf s'}_{\delta\xi }(\gamma)$ is supported on $\Attr^f_\fC(F)$ and its Gysin pullback to $F$ equals to
\begin{align*}
\Stab^{\mathsf s}_{\xi' }\circ \Stab^{\mathsf s'}_{\delta\xi}(\gamma)|_F= e^\sT_K(N^-_{F'/X}) \Stab^{\mathsf s'}_{\delta\xi}(\gamma)|_F=  e^\sT_K(N^-_{F'/X}) e^\sT_K(N^-_{F/F'}) \gamma=  e^\sT_K(N^-_{F/X}) \gamma.
\end{align*}
Moreover, we claim that $\Stab^{\mathsf s}_{\xi' }\circ \Stab^{\mathsf s'}_{\delta\xi}(\gamma)$ satisfies the axiom (iii) in the Definition \ref{def of stab k}. Let $F_1\in \Fix_\sA(X)$ which is different from $F$, then there are two possibilities: 
\begin{itemize} 
\item
$F_1$ is contained in $F'$, or 
\item $F_1$ is contained in $F'_1\in \Fix_{\sA'}(X)$ such that $F_1'\neq F'$. 
\end{itemize} 
In the first case, 
\begin{align*}
    \Stab^{\mathsf s}_{\xi' }\circ \Stab^{\mathsf s'}_{\delta\xi}(\gamma)|_{F_1}= e^\sT_K(N^-_{F'/X}) \Stab^{\mathsf s'}_{\delta\xi}(\gamma)|_{F_1},
\end{align*}
then we have
\begin{align*}
    \deg_{\xi}\Stab^{\mathsf s}_{\xi' }&\circ \Stab^{\mathsf s'}_{\delta\xi}(\gamma)|_{F_1}=\deg_{\xi}\left( e^\sT_K(N^-_{F'/X}) \Stab^{\mathsf s'}_{\delta\xi}(\gamma)|_{F_1}\right)\\
    &\subsetneq  \deg_{\xi} \left(e^\sT_K(N^-_{F'/X}) e^\sT_K(N^-_{F_1/F'})\right) + \mathrm{weight}_{\delta\xi}\left(\det(N^-_{F_1/F'})^{1/2}\otimes\mathsf s'|_{F_1}\right)-\mathrm{weight}_{\delta\xi}\left(\det(N^-_{F/F'})^{1/2}\otimes\mathsf s'|_F\right)\\
    &=\deg_{\xi} e^\sT_K(N^-_{F_1/X}) + \mathrm{weight}_{\xi}\left(\det(N^-_{F_1/X})^{1/2}\otimes\mathsf s|_{F_1}\right)-\mathrm{weight}_{\xi}\left(\det(N^-_{F/X})^{1/2}\otimes\mathsf s|_F\right).
\end{align*}
In the second case, we have 
\begin{align*}
    \deg_{\xi'}\Stab^{\mathsf s}_{\xi' }\circ \Stab^{\mathsf s'}_{\delta\xi}(\gamma)|_{F_1}&\subseteq   \deg_{\xi'}\Stab^{\mathsf s}_{\xi' }\circ \Stab^{\mathsf s'}_{\delta\xi}(\gamma)|_{F'_1}\\
   \text{\tiny by axiom (iii) of $\Stab^{\mathsf s}_{\xi' }$}\quad &\subsetneq  \deg_{\xi'} e^\sT_K(N^-_{F'_1/X}) + \mathrm{weight}_{\xi'}\left(\det(N^-_{F'_1/X})^{1/2}\otimes\mathsf s|_{F'_1}\right)-\mathrm{weight}_{\xi'}\left(\det(N^-_{F'/X})^{1/2}\otimes\mathsf s|_{F'}\right)\\
   &= \deg_{\xi'} e^\sT_K(N^-_{F_1/X}) + \mathrm{weight}_{\xi'}\left(\det(N^-_{F_1/X})^{1/2}\otimes\mathsf s|_{F_1}\right)-\mathrm{weight}_{\xi'}\left(\det(N^-_{F/X})^{1/2}\otimes\mathsf s|_{F}\right).
\end{align*}
Note that for two Laurent polynomials $f$ and $g$, if $\deg_{\xi'}f\subsetneq  \deg_{\xi'}g$, then 
$\deg_{\xi}f\subsetneq  \deg_{\xi}g$ for small perturbation $\xi=\xi'+\delta\xi$. Applying the observation to 
$$f=\Stab^{\mathsf s}_{\xi' }\circ \Stab^{\mathsf s'}_{\delta\xi}(\gamma)|_{F_1}, \,\,\, g= e^\sT_K(N^-_{F_1/X})\otimes\det(N^-_{F_1/X})^{1/2}\otimes\mathsf s|_{F_1}\otimes\det(N^-_{F/X})^{-1/2}\otimes\left(\mathsf s|_{F}\right)^{-1},$$
we see that our claim holds in the second case as well. Then it follows from the uniqueness of stable envelopes (Proposition \ref{uniqueness of K stab}) that 
\begin{align*}
    \Stab^{\mathsf s}_{\xi' }\circ \Stab^{\mathsf s'}_{\delta\xi}(\gamma)= \Stab^{\mathsf s}_{\xi }(\gamma).
\end{align*}
The lemma then follows from $\Stab^{\mathsf s}_{\xi }=\Stab^{\mathsf s}_{\fC }$, $\Stab^{\mathsf s}_{\xi' }=\Stab^{\mathsf s}_{\fC' }$, and $\Stab^{\mathsf s'}_{\delta\xi}=\Stab^{\mathsf s'}_{\fC/\fC' }$.
\end{proof}

\begin{Remark}
A key fact used in the proof of Lemma \ref{triangle lemma for K} is  that small perturbation in $\xi\in \Lie(\sA)_\bR$ preserves the degree bound condition. 
The corresponding cohomological statement is \textit{not true}, as illustrated by the following example. Suppose that $\Lie(\sA)_\bR=\bR^2$ with coordinate $(x,y)$, and let $\fC=\{x>0,y>0\}$ and $\fC'=\{x>0,y=0\}$, and let $f=y^2,g=x$. For the coweight $\xi'$ such that $\langle x,\xi'\rangle=1,\langle y,\xi'\rangle=0$, we have $$0=\deg_{\xi'}(f)<\deg_{\xi'}(g)=1, $$ 
but for its perturbation $\xi$ such that $\langle x,\xi\rangle=1,\langle y,\xi\rangle=\delta$, we have $2=\deg_{\xi}(f)>\deg_{\xi}(g)=1$. So we do not have a straightforward analogue of Lemma \ref{triangle lemma for K} for cohomological case. 
\end{Remark}
Nevertheless, we show that the cohomological triangle lemma indeed holds for symmetric GIT quotients with any
torus invariant potential functions, see Theorem \ref{tri lem for coh stab}. 



\subsection{General flavour groups}\label{sec:general flavour}

Recall in Setting \ref{setting of stab}, we take $\sT$ to be a torus. This is a simplifying condition and it can be relaxed to the following.
\begin{Setting}\label{setting of stab_gen flav}
Let $X$ be a smooth quasi-projective variety over $\bC$ with a linear algebraic group $\sF$-action and $\sw\colon X \to \bC$ be an $\sF$-invariant regular function. Let $\sA \subseteq  Z(\sF)$ be a torus which lies in the center $Z(\sF)$ of $\sF$.
\end{Setting}
The definition of stable envelopes (Definitions \ref{def of stab coho}, \ref{def of stab k}) can be stated verbally with the torus $\sT$ replaced by linear algebraic group $\sF$. Results displayed in above remain unchanged with $\sT$ replaced by $\sF$. 

Moreover, when stable envelopes exist, they are compatible with changing of flavour groups. Namely, if $\sF'\subseteq  \sF$ is a subgroup which contains $\sA$, then the following diagram commutes:
\begin{equation}\label{change of group_coh}
\xymatrix{
H^\sF(X^\sA,\sw^\sA)\ar[r]^{\,\,\Stab_{\fC}} \ar[d]_{\mathrm{res}^F_{F'}} & H^\sF(X,\sw) \ar[d]^{\mathrm{res}^F_{F'}}\\
H^{\sF'}(X^\sA,\sw^\sA)\ar[r]^{\,\,\Stab_{\fC}} & H^{\sF'}(X,\sw) 
}
\end{equation}
where the vertical arrow is the restriction of coefficients map $\mathrm{res}^F_{F'}\colon H^\sF(\cdots)\to H^{\sF'}(\cdots)$. There is also a similar commutative diagram for critical $K$-theory:
\begin{equation}\label{change of group_k}
\xymatrix{
K^\sF(X^\sA,\sw^\sA)\ar[r]^{\,\, \Stab^{\mathsf s}_{\fC}} \ar[d]_{\mathrm{res}^F_{F'}} & K^\sF(X,\sw) \ar[d]^{\mathrm{res}^F_{F'}}\\
K^{\sF'}(X^\sA,\sw^\sA)\ar[r]^{\, \, \Stab^{\mathsf s}_{\fC}} & K^{\sF'}(X,\sw) 
}
\end{equation}
We will use stable envelopes under Setting \ref{setting of stab_gen flav} to construct 
interpolation maps in \S \ref{sec:na stab}.

\section{Cohomological stable envelopes on symmetric GIT quotients}\label{app on symm var}

In this section, we show that cohomological stable envelope correspondences exist for symmetric GIT quotients, 
given by the fundamental class of the closure $[\overline{\Attr}(\Delta)]$ of attracting set of the diagonal
(Theorem \ref{construct coh stab corr on symm quiver var}). 

We prove the triangle lemma for corresponding stable envelopes (Theorem \ref{tri lem for coh stab}). 
This is done by introducing the so-called nonabelian stable envelopes (see Definition \ref{def on nonab stab} and Theorem \ref{thm: deg bound na stab coh} for a characterization property), and a compatible diagram between stable envelopes and Hall operations (Theorem \ref{thm: hall coh_sym quot}). 
For symmetric quiver varieties, a different proof of the existence of stable envelopes as well as the triangle lemma is discussed in \cite{COZZ3}. 

As a side application, we derive an explicit formula of cohomological stable envelopes when potentials are zero (Corollary \ref{cor: explicit formula w=0_coh}).

\subsection{Symmetric GIT quotients}

\begin{Definition}\label{def of sym var}
Given the following data: 
\begin{itemize} 
\item A connected complex reductive group $G$, a Cartan and a Borel subgroup $H\subseteq B \subseteq G$,  
\item Two tori $\sA\subseteq \sT$,
\item A finite dimensional $(G\times \sT)$-representation $R$ such that $R\cong R^\vee$ as $(G\times \sA)$-representations, 
\item $\theta\colon G\to \C^*$ such that the semistable locus equals the stable locus $R^{{\theta}\emph{-}ss}=R^{{\theta}\emph{-}s}\neq \emptyset$ on which $G$-action is free, 
\end{itemize} 
we call $X:=R/\!\!/_{\theta} G=R^{{\theta}\emph{-}s}/G$ a \textit{symmetric GIT quotient}. 
It is endowed with the induced action by tori $\sT$ and $\sA$. 

We also have the associated variety $\widehat{X}:=R^{{\theta}\emph{-}s}\times_{G} (G/H)$ with a $G/H$-fibration $p\colon \widehat{X}\to X$. 
\end{Definition}
\begin{Remark}
For the induced character $\bar{\theta}\colon H\to \C^*$, there is an open immersion 
$\widehat{X}\cong R^{{\theta}\emph{-}s}/H\hookrightarrow R^{\bar{\theta}\emph{-}ss}/H$. 
\end{Remark}
The main theorem of this section is: 
\begin{Theorem}\label{construct coh stab corr on symm quiver var}
Given a symmetric GIT quotient $X$ with an action by tori $\sT$ and $\sA$ as Definition \ref{def of sym var}, 
fix a chamber $\fC$. Then 
$$\sum_{F\in \Fix_\sA(X)} [\overline{\Attr}_\fC(\Delta_F)]\in H^\sT\left(\Attr^f_\fC\right)$$ is a cohomological stable envelope correspondence for 
$(X,0,\sT,\sA,\fC )$.
\end{Theorem}

Our main examples of symmetric GIT quotients are given by symmetric quiver varieties introduced below.

\subsection{Quiver varieties}\label{sec: quiver var}
A \textit{quiver} $Q$ is a pair of finite sets $Q=(Q_0,Q_1)$ together with two maps $h,t\colon Q_1\to Q_0$. We will call $Q_0$ the set of nodes and $Q_1$ the set of arrows and $h$ (resp.~$t$) sends an arrow to its head (resp.~tail). If $a\in Q_1$, then we will write $t(a)\to h(a)$ to denote the arrow $a$. We define the 
\textit{adjacency matrix}
\begin{align*}
    (\mathsf Q_{ij})_{i,j\in Q_0}=\#(i\to j)
\end{align*}
and \textit{Cartan matrix}
\begin{align*}
    (\mathsf C_{ij})_{i,j\in Q_0}=2\delta_{ij}-\mathsf Q_{ij}-\mathsf Q_{ji}.
\end{align*}
Take dimension vectors $\mathbf{v}\in \bN^{Q_0}$,
$\underline{\mathbf d}=(\mathbf{d}_{\mathrm{in}},\mathbf d_{\mathrm{out}})\in \bN^{Q_0}\times \bN^{Q_0}$,  the \textit{space of framed representations} of $Q$ with \textit{gauge dimension} $\mathbf{v}$ and \textit{in-coming framing dimension} $\mathbf{d}_{\mathrm{in}}$ and \textit{out-going framing dimension} $\mathbf d_{\mathrm{out}}$ is
\begin{align}\label{equ rvab}
    R(\mathbf{v},\underline{\bd}):=R(Q,\mathbf{v},\underline{\bd})=\bigoplus_{a\in Q_1}\underbrace{\Hom(\bC^{\mathbf v_{t(a)}},\bC^{\mathbf v_{h(a)}})}_{X_a}\oplus \bigoplus_{i\in Q_0}\left(\underbrace{\Hom(\bC^{\bd_{\mathrm{in},i}},\bC^{\mathbf v_i})}_{A_i}\oplus \underbrace{\Hom(\bC^{\mathbf v_i},\bC^{\bd_{\Out,i}})}_{B_i}\right).
\end{align}
The \textit{gauge group} $G=\prod_{i\in Q_0}\GL(\mathbf v_i)$ naturally acts on $R(Q,\mathbf{v},\underline{\bd})$ by compositions with maps. 

Choose a stability condition $\theta\in \bQ^{Q_0}$ such that $\theta$-semistable representations are $\theta$-stable: 
$$R(Q,\mathbf{v},\underline{\bd})^{ss}=R(Q,\mathbf{v},\underline{\bd})^s\neq \emptyset, $$
and define the \textit{quiver variety} as the GIT quotient:
$$\mathcal M_{\theta}(\mathbf v,\underline{\bd}):=\mathcal M_{\theta}(Q,\mathbf v,\underline{\bd}):=R(Q,\mathbf{v},\underline{\bd})/\!\!/_{\theta} G=R(Q,\mathbf{v},\underline{\bd})^s/G. $$
As we do not impose any relation on the quiver, the above space is a smooth quasi-projective variety. 

We define an action of
\begin{align}\label{edge group}
    G_{\mathrm{edge}}:=\prod_{i,j\in Q_0}\GL(\mathsf Q_{ij})
\end{align}
on $R(Q,\mathbf{v},\underline{\bd})$: given a pair of nodes $i,j\in Q_0$, the contribution of the edges from $i$ to $j$ is $\Hom(\bC^{\mathbf v_i},\bC^{\mathbf{v}_j})\otimes \bC^{\mathsf Q_{ij}}$, then the factor $\GL(\mathsf Q_{ij})$ naturally acts on the second component.

We also consider the following group actions on $R(Q,\mathbf{v},\underline{\bd})$: 
\begin{align}\label{framing group}
    G_{\mathrm{fram}}^{\mathrm{in}}=\prod_{i\in Q_0}\GL(\bd_{\In,i})\curvearrowright \bigoplus_{i\in Q_0}\Hom(\bC^{\bd_{\In,i}},\bC^{\mathbf v_i}),\quad G_{\mathrm{fram}}^{\mathrm{out}}=\prod_{i\in Q_0}\GL(\bd_{\Out,i})\curvearrowright \bigoplus_{i\in Q_0}\Hom(\bC^{\mathbf v_i},\bC^{\bd_{\Out,i}}).
\end{align}
We define the \textit{flavour group}
\begin{align}\label{flavour group}
\sF:=G_{\mathrm{fram}}^{\mathrm{in}}\times G_{\mathrm{fram}}^{\mathrm{out}}\times G_{\mathrm{edge}}. 
\end{align}
It is easy to see that $\mathsf F\cong \Aut_{G}(R(Q,\mathbf{v},\underline{\bd}))$.

\begin{Definition} 
We say that a framing $\underline{\bd}$ is \textit{symmetric} if $\bd_{\In}=\bd_{\Out}=\bd$. In this case, we simplify the notations as 
\begin{equation}\label{equ on sym qu}R(\bv,\bd)=R(Q,\bv,\bd)= R(Q,\mathbf{v},\underline{\bd}), \quad 
\cM_\theta(\bv,\bd)=\cM_\theta(Q,\bv,\bd)=\mathcal M_{\theta}(Q,\mathbf v,\underline{\bd}), \end{equation} 
and we define $G_{\mathrm{fram}}^{\mathrm{\diag}}\subseteq  \mathsf F$ to be the diagonal subgroup of $G_{\mathrm{fram}}^{\mathrm{in}}\times G_{\mathrm{fram}}^{\mathrm{out}}$.
\end{Definition}

\begin{Definition}\label{def of sym quiver}
We say $Q$ is \textit{symmetric} if its adjacency matrix $ (\mathsf Q_{ij})_{i,j\in Q_0}$ is symmetric. 
For a symmetric quiver $Q$, we call the associated quiver variety \textit{symmetric quiver variety} (SQV) if the framing is symmetric.
\end{Definition}

\begin{Definition}\label{def:self-dual tori}
Let $(Q,\mathbf v,\mathbf{d})$ be a symmetric quiver with symmetric framing as above. Suppose that there exists a torus $\mathsf H$ together with a linear action of $\mathsf H$ on $R(Q,\mathbf v,\mathbf{d})$. We say that a linear action of a torus $\mathsf H$ on $R(Q,\mathbf v,\mathbf{d})$ is \textit{self-dual} if the $\mathsf H$-action commutes with the gauge group $G$-action, and $R(Q,\mathbf v,\mathbf{d})$ is self-dual as $(G\times \mathsf H)$-representation. 

We say that a torus action $\mathsf H$ on the symmetric quiver variety $\cM_\theta(Q,\mathbf v,\mathbf{d})$ is \textit{self-dual} if it is induced from a self-dual action of $\mathsf H$ on $R(Q,\mathbf v,\mathbf{d})$.
\end{Definition}
Symmetric quiver varieties with self-dual torus actions are examples of Definition \ref{def of sym var}, moreover their $\sA$-fixed loci are also 
symmetric quiver varieties as shown below. 

\begin{Lemma}\label{fix pts of sym quiv var}
Suppose that $X=\mathcal M_{\theta}(Q,\mathbf v,\mathbf{d})$ is a symmetric quiver variety, and $\sA$ is a torus with a self-dual action on $X$. Let $\sigma$ be a cocharacter of $\sA$, then the $\sigma$-fixed points locus $X^\sigma$ is a disjoint union of symmetric quiver varieties. Moreover, the induced action of $\sA$ on $X^\sigma$ is self-dual.
\end{Lemma}

\begin{proof}
The proof is similar to that of \cite[Prop.~2.3.1]{MO}. In fact, the same argument as \textit{loc.\,cit.} shows that
\begin{align}\label{eq: fix pts of sym quiv var}
    \cM_\theta(Q,\mathbf v,\mathbf{d})^\sigma=\bigsqcup_{\phi/\sim} \cM_{\theta_\phi}(Q_\phi,\mathbf v_\phi,\mathbf{d}_\phi).
\end{align}
Here $\phi:\bC^*\to G\times \sA$ is a lift of $\sigma:\bC^*\to \sA$ along the projection $G\times \sA\to  \sA$, and $(Q_\phi,\mathbf v_\phi,\mathbf{d}_\phi)$ is constructed in \cite[\S 2.3.3]{MO}, and two lifts $\phi_1$ and $\phi_2$ are equivalent (denoted as $\phi_1\sim \phi_2$) if they give the same action of $\sA$ on $R(Q,\mathbf v,\mathbf{d})$. The disjoint unions on the right-hand-side of \eqref{eq: fix pts of sym quiv var} are labelled by equivalent classes of lifts $\phi$. By construction, $R(Q_\phi,\mathbf v_\phi,\mathbf{d}_\phi)$ is the $\phi$-fixed points locus $R(Q,\mathbf v,\mathbf{d})^\phi$, and the gauge group for $Q_\phi$ is the $\phi$-fixed points subgroup $G^\phi$. Since $R(Q,\mathbf v,\mathbf{d})$ is a self-dual $G\times \sA$ representation and $\bC^*$ acts on it through $\phi\colon \bC^*\to G\times \sA$, it follows that $R(Q,\mathbf v,\mathbf{d})^\phi$ is a self-dual $G^\phi\times \sA$ representation. Hence $Q_\phi$ is a symmetric quiver, and the induced $\sA$-action on $\cM_{\theta_\phi}(Q_\phi,\mathbf v_\phi,\mathbf{d}_\phi)$ is self-dual. 
\end{proof}

\subsection{Generalities on attracting closures}
\begin{Definition}\label{def: cond *}
Let $(X,\sA)$ be in the Setting \ref{setting of stab} and $\sigma$ be a cocharacter of $\sA$ (not necessarily generic). Denote $\Delta$ to be the diagonal of $X^\sigma\times X^\sigma$, and $\overline{\Attr}_{\sigma}(\Delta)$ to be the closure of $\Attr_{\sigma}(\Delta)$ in $X\times X^\sigma$. 

Consider the following condition on $(X,\sA,\sigma)$: 
\begin{itemize}
    \item[($\star$)] for all $F'\neq F\in \Fix_\sigma(X)$, we have 
\begin{align}\label{eq on cond star}
    \dim \overline{\Attr}_\sigma(\Delta)\cap (F'\times F)<\dim F+\rk N^+_{F/X}-\rk N^+_{F'/X}.
\end{align}
\end{itemize}
Let $\fC$ be a chamber of $(X,\sA)$. We say that $(X,\sA,\fC)$ satisfies the condition ($\star$) if for a generic $\xi\in \fC$, $(X,\sA,\xi)$ satisfies the condition ($\star$).
\end{Definition}

\begin{Remark}
Assume moreover that $\dim N_{F/X}^+=\dim N_{F/X}^-$ holds for all $F\in \Fix_\sA(X)$, then \eqref{eq on cond star} is equivalent to 
\begin{equation}\label{equ on dim att on sym}
    \dim \overline{\Attr}_\sigma(\Delta)\cap (F'\times F)<\frac{1}{2}\dim (F'\times F).
\end{equation}
The condition $\dim N_{F/X}^+=\dim N_{F/X}^-$ is satisfied when the tangent bundle has the following self-dual property:
\begin{align*}
    T_X=T^\vee_X \in K^\sA(X).
\end{align*}
The $X$, $\widehat{X}$ in Definition \ref{def of sym var} satisfy this self-duality. 
\end{Remark}

\begin{Lemma}\label{cond star implies existence of stab coh}
Let $(X,\sT,\sA,\fC)$ be in the Setting \ref{setting of stab} and assume that $(X,\sA,\fC)$ satisfies the condition ($\star$) in Definition \ref{def: cond *}. Then 
$$\sum_{F\in \Fix_\sA(X)} [\overline{\Attr}_\fC(\Delta_F)]\in H^\sT\left(\Attr^f_\fC\right)$$ is a stable envelope correspondence for $(X,0,\sT,\sA,\fC)$.
\end{Lemma}
\begin{proof}
It suffices to show that for every $F\in \Fix_\sA(X)$, $[\overline{\Attr}_\fC(\Delta_F)]$ satisfies axioms (i) and (ii) in Definition \ref{stab corr_coh}. The axiom (i) is obvious. 
For axiom (ii), the class $$(F'\times F\hookrightarrow X\times F)^*[\overline{\Attr}_\fC(\Delta_F)]$$ is supported on a subvariety $\overline{\Attr}_\fC(\Delta_F)\cap (F'\times F)$ of dimension smaller than $\dim F+\rk N^+_{F/X}-\rk N^+_{F'/X}$. Note that $\dim F+\rk N^+_{F/X}=\dim \overline{\Attr}_\fC(\Delta_F)$. Therefore
\begin{align*}
    \deg_\sA(F'\times F\hookrightarrow X\times F)^*[\overline{\Attr}_\fC(\Delta_F)]&\leqslant \dim \overline{\Attr}_\fC(\Delta_F)\cap (F'\times F)-\left(\dim \overline{\Attr}_\fC(\Delta_F)+\dim (F'\times F)-\dim (X\times F)\right)\\
    &<\dim \overline{\Attr}_\fC(\Delta_F)-\rk N^+_{F'/X}-\left(\dim \overline{\Attr}_\fC(\Delta_F)+\dim (F'\times F)-\dim (X\times F)\right)\\
    &=\rk N^-_{F'/X}.
\end{align*}
This verifies the axiom (ii).
\end{proof}

\subsection{Generalities on family of smooth symplectic varieties}

Let $(X,\sA)$ be in the Setting \ref{setting of stab} and let $\sigma$ be a cocharacter of $\sA$. Suppose that there exists an $\sA$-invariant two-form $\omega\in \Omega^2(X)$, and there exists a smooth connected variety $B$ endowed with trivial $\sA$-action and a smooth and $\sA$-equivariant morphism $\phi\colon X\to B$, such that $\forall\, b\in B$, the restriction of $\omega$ to $X_b:=\phi^{-1}(b)$ is nondegenerate, i.e. a \textit{symplectic structure}.

We note that ${\Attr}_\sigma(\Delta)\subseteq  X\times_B X^\sigma$ because $\Delta\subseteq  X\times_B X^\sigma$ and $\phi\colon X\to B$ is $\sA$-equivariant. Then it follows that $\overline{\Attr}_\sigma(\Delta)\subseteq  X\times_B X^\sigma$. We also have $\dim N_{F/X}^+=\dim N_{F/X}^-$ holds for all $F\in \Fix_\sA(X)$ in this case.

\begin{Lemma}\label{attr is isotropic}
In the above situation, for all $F,F'\in \Fix_\sigma(X)$ and for all $b\in B$ we have
\begin{align*}
     \dim \overline{\Attr}_\sigma(\Delta)\cap (F'_b\times F_b)\leqslant\frac{1}{2}\dim (F'_b\times F_b),
\end{align*}
where $F_b=F\cap \phi^{-1}(b)$ and $F'_b=F'\cap \phi^{-1}(b)$.
\end{Lemma}

\begin{proof}
We consider the two-form $\omega':=\omega\boxminus \omega^\sigma$ on $X\times X^\sigma$, then the restriction of $\omega'$ to every fiber $X_b\times X^\sigma_b$ is a symplectic form. We claim that $\overline{\Attr}_\sigma(\Delta)\cap (X_b\times X^\sigma_b)$ is an isotropic subvariety in $X_b\times X^\sigma_b$. 
To prove this claim, first note that the restriction of $\omega'$ to $\Delta$ is trivial by the construction of $\omega'$. Since $\omega'$ is $\sA$-invariant, its restriction to ${\Attr}_\sigma(\Delta)$ must vanish as well. Thus the restriction of $\omega'$ to smooth locus of $\overline{\Attr}_\sigma(\Delta)$ vanishes by continuity. Let $W$ be an irreducible component of $\overline{\Attr}_\sigma(\Delta)\cap (X_b\times X^\sigma_b)$. For a general point $w\in W$, there exists a sequence of points $x_1,x_2,\cdots$ in the smooth locus of $\overline{\Attr}_\sigma(\Delta)$ approaching $w$ such that limit of $T_{x_i}\overline{\Attr}_\sigma(\Delta)$ exists as $i\to \infty$ and
contains the tangent space $T_w W$ \footnote{This can be seen by choosing a Whitney stratification of $\overline{\Attr}_\sigma(\Delta)$ for which $\overline{\Attr}_\sigma(\Delta)\cap (X_b\times X^\sigma_b)$ is a union of strata.}. Since the restriction of $\omega'$ to $T_{x_i}\overline{\Attr}_\sigma(\Delta)$ vanishes, our claim follows.

Then it follows that $\overline{\Attr}_\sigma(\Delta)\cap (F'_b\times F_b)$ is an isotropic subvariety in $F'_b\times F_b$ by \cite[Lem.~3.4.1]{MO}. Thus 
\begin{equation*}\dim \overline{\Attr}_\sigma(\Delta)\cap (F'_b\times F_b)\leqslant\frac{1}{2}\dim (F'_b\times F_b). \qedhere  \end{equation*}
\end{proof}

\begin{Lemma}\label{symp res has cond star}
In the same situation as Lemma \ref{attr is isotropic}, assume moreover that there exists an open dense subset $B^\circ\subseteq  B$ such that $\phi|_{X^\circ}\colon X^\circ\to B^\circ$ is a quasi-affine morphism where $X^\circ:=\phi^{-1}(B^\circ)$. Then $(X,\sA,\sigma)$ satisfies the condition ($\star$) \eqref{eq on cond star}.
\end{Lemma}

\begin{proof}
For every $x\in X^\sigma$, the tangent map $d\phi_x\colon T_xX\to T_{\phi(x)}B$ is surjective by the assumption that $\phi\colon X\to B$ is smooth. $d\phi_x$ is also $\sigma$-equivariant since $\phi$ is $\sigma$-equivariant. Then it follows that $d\phi_x$ maps $T_xX^\sigma=(T_xX)^\sigma$ surjectively onto $T_{\phi(x)}B$, so the restriction of $\phi$ to $\sigma$-fixed loci $\phi^\sigma\colon X^\sigma\to B$ is also smooth. In particular, $\phi^\sigma$ is flat. The flatness of $\phi^\sigma$ implies that $\forall\, F\in \Fix_\sigma(X)$ and any $b\in \phi(F)$,
we have 
$$\dim F=\dim F_b+\dim B.$$ 
For $F'\neq F\in \Fix_\sigma(X)$, let $F'^\circ:=F'\cap X^\circ$ and $F^\circ:=F\cap X^\circ$. We claim that 
$$\overline{\Attr}_\sigma(\Delta)\cap (F'^\circ\times F^\circ)=\emptyset.$$ 
Without loss of generality, we assume that $B^\circ$ is affine, then $X^\circ$ and $(X^{\circ})^\sigma$ are quasi-affine. Let $\Delta^\circ$ be the diagonal of $(X^{\circ})^\sigma\times (X^{\circ})^\sigma$, then ${\Attr}_\sigma(\Delta^\circ)$ is closed in $X^\circ\times(X^{\circ})^\sigma$ \footnote{Taking a finite set of $\sA$-homogeneous generators of $\bC[X^\circ\times(X^{\circ})^\sigma]$, we get a surjective $\sA$-equivariant algebra map $\bC[V]\twoheadrightarrow \bC[X^\circ\times(X^{\circ})^\sigma]$ where $V$ is a vector space with a linear $\sA$-action. Equivalently we get an $\sA$-equivariant closed embedding $X^\circ\times(X^{\circ})^\sigma\hookrightarrow V$. We note that $\Attr_\sigma(V^\sigma)=V^{\geqslant 0}$ where $V^{\geqslant 0}$ is the linear subspace of $V$ spanned by nonnegative $\sigma$-eigenvectors, in particular $\Attr_\sigma(V^\sigma)$ is closed in $V$. We also note that ${\Attr}_\sigma(\Delta^\circ)=p^{-1}(\Delta^\circ)\cap (X^\circ\times(X^{\circ})^\sigma)$, where $p\colon V^{\geqslant 0}\to V^\sigma$ is the attraction map and $\Delta^\circ$ is regarded as a closed subset of $V^\sigma$. Thus ${\Attr}_\sigma(\Delta^\circ)$ is closed in $X^\circ\times(X^{\circ})^\sigma$.}. Therefore $$\overline{\Attr}_\sigma(\Delta)\cap (X^\circ\times(X^{\circ})^\sigma)={\Attr}_\sigma(\Delta^\circ),$$ 
hence $\overline{\Attr}_\sigma(\Delta)\cap (F'^\circ\times F^\circ)=\emptyset$. In the general case we can cover $B^\circ$ with affine open subsets and apply the above argument to each of the affine open subset. This proves our claim.

Finally, we have the following dimension bound:
\begin{align*}
    \dim \overline{\Attr}_\sigma(\Delta)\cap (F'\times F)&\leqslant \dim (B\setminus B^\circ)+\sup_{b\in B}\dim \overline{\Attr}_\sigma(\Delta)\cap (F'_b\times F_b) \\
  \text{\tiny By Lem. \ref{attr is isotropic}}\quad &\leqslant \dim (B\setminus B^\circ)+\sup_{b\in B}\frac{1}{2}\dim (F'_b\times F_b) \\
  &\leqslant\frac{1}{2}\dim (F'\times F)-1.
\end{align*}
This shows that $(X,\sA,\sigma)$ satisfies the condition ($\star$) \eqref{eq on cond star}.
\end{proof}

\subsection{Proof of Theorem \ref{construct coh stab corr on symm quiver var}}

We first show that the attracting closure for the diagonal of $\widehat{X}$ in Definition \ref{def of sym var}
is a stable envelope correspondence. 

\begin{Lemma}\label{lem on hat X stab}
$$\sum_{F\in \Fix_\sA(\widehat{X})} [\overline{\Attr}_\fC(\Delta_F)]\in H^\sT\left(\Attr^f_\fC\right)$$ is a stable envelope correspondence for $(\widehat{X},0,\sT,\sA,\fC)$.
\end{Lemma} 
\begin{proof}
Since $R\cong R^\vee$ as $H$-representations, we have an isomorphism of $H$-representations: 
$$R\cong T^*N\oplus R_0, $$
where $N=\oplus_{\alpha\neq 0}\C_{\alpha}$ and $R_0$ is a trivial $H$-representation. In particular, there is a moment map 
$$\mu\colon T^*N\to \mathfrak{h}^*.$$
Define $\nu\colon R=T^*N\times R_0\xrightarrow{\mu\times \id} \mathfrak{h}^*\times R_0$, then the restriction of $\nu$ to $R^{{\theta}\emph{-}s}$ is smooth. By passing to the quotient, we obtain a smooth map 
$$\bar{\nu}\colon  \widehat{X}\to \mathfrak{h}^*\times R_0, $$
and for any $x\in \mathfrak{h}^*\times R_0$, $\bar{\nu}^{-1}(x)$ is a smooth symplectic variety. 

For a generic $x\in \mathfrak{h}^*\times R_0$, we have ${\nu}^{-1}(x)\subseteq R^{\bar{\theta}\emph{-}s}$; therefore $\bar{\nu}^{-1}(x)$ is an open subset of ${\nu}^{-1}(x)/\!\!/H$, in particular $\bar{\nu}^{-1}(x)$ is quasi-affine.
By Lemma \ref{symp res has cond star}, we know 
\eqref{equ on dim att on sym} holds for $\widehat{X}=R^{{\theta}\emph{-}s}\times_{G} (G/H)$ in Definition \ref{def of sym var}.  
By applying Lemma \ref{cond star implies existence of stab coh}, we are done. 
\end{proof}
Next we want to relate the geometry of $\widehat{X}$ and $X$. We define $\overline{X}:=R^{s}\times_{G} (G/B)$ and have fibrations: 
$$p\colon \widehat{X} \xrightarrow{p_1} \overline{X} \xrightarrow{p_2}  X. $$
For any $S\in \Fix_\sA(X)$, which is given by $S=\left(R^{\sA,\phi}\cap R^{{\theta}\emph{-}s}\right)/G^{\phi(\sA)}$ for a group homomorphism $\phi\colon \sA\to G$, we define 
$$\widehat{S}:= \left(R^{\sA,\phi}\cap R^{{\theta}\emph{-}s}\right)\times_{G^{\phi(\sA)}}\left(G^{\phi(\sA)}/H\right)\in \Fix_\sA(\widehat{X}), \quad  
\overline{S}:=\left(R^{\sA,\phi}\cap R^{{\theta}\emph{-}s}\right)\times_{G^{\phi(\sA)}} \left(G^{\phi(\sA)}/B^{\phi(\sA)}\right)\in \Fix_\sA(\overline{X}), $$
with natural morphisms 
$$q\colon \widehat{S} \xrightarrow{q_1} \overline{S} \xrightarrow{q_2}  S. $$
\begin{Proposition}
For any $S, S'\in \Fix_\sA(X)$ with $S\neq S'$, we have 
$$\deg_{\sA}\left(\left(S'\times S\hookrightarrow X\times S \right)^![\overline{\Attr}_\fC(\Delta_S)]\right)< \rk N_{S'/X}^{-}. $$
In particular, this finishes the proof of Theorem \ref{construct coh stab corr on symm quiver var}.  
\end{Proposition}
\begin{proof}
Let $Z:=\overline{\Attr}_\fC(\Delta_S)$, $\overline{Z}:=(p_2\times q_2)^{-1}(Z)$, $\widehat{Z}:=(p\times q)^{-1}(Z)$. Then 
$$\overline{\Attr}_\fC(\Delta_{\widehat{S}})\subseteq \widehat{Z}. $$ 
By Weyl conjugation, we choose $\phi:\sA\to G$ such that $N_{\overline{S}/p_2^{-1}(S)}$ are repelling in chamber $\fC$.
We have 
$$(p_1\times q_1)^!\colon H^{\sT}(\overline{Z})\xrightarrow{\cong} H^{\sT}(\widehat{Z}), \quad (p_2\times q_2)_*\colon H^{\sT}(\overline{Z})\to H^{\sT}(Z). $$
Here $p_1, q_1$ are $U$-fibrations where $B=H\cdot U$ and $U$ is a unipotent group, $p_2, q_2$ are proper maps as $G/B$ is proper. 

Consider the composition: 
\begin{equation}\label{equ on eta map}\eta:=\eta_{\widehat{X}\times \widehat{S}\to X\times S}:=(p_2\times q_2)_*\circ \left((p_1\times q_1)^! \right)^{-1}\colon H^{\sT}(\widehat{Z})\to H^{\sT}({Z}). \end{equation}
By $\sA$-localization, we have 
\begin{align*}
&\quad \,\, \left(S'\times S\hookrightarrow X\times S \right)^!\eta_{\widehat{X}\times \widehat{S}\to X\times S}\left([\overline{\Attr}_\fC(\Delta_{\widehat{S}})] \right) \\
&=\sum_{\widehat{S'_i}\to S'} \eta_{\widehat{S'_i}\times \widehat{S}\to S'\times S}\left(\frac{1}{e^{\sT}\left(N_{\overline{S'_i}/p_2^{-1}(S')}\right)} \left(\widehat{S'_i}\times \widehat{S}\hookrightarrow \widehat{X}\times \widehat{S} \right)^![\overline{\Attr}_\fC(\Delta_{\widehat{S}})] \right), 
\end{align*}
where the sum is taken over all $\sA$-fixed components $\{\widehat{S'_i}\}_{i}$ of $p^{-1}(S')$. Therefore by Lemma \ref{lem on hat X stab}, we have 
\begin{align*}
\deg_{\sA}\left(S'\times S\hookrightarrow X\times S \right)^!\eta_{\widehat{X}\times \widehat{S}\to X\times S}\left([\overline{\Attr}_\fC(\Delta_{\widehat{S}})] \right)
&<\frac{1}{2}\rk N_{\widehat{S'_i}/\widehat{X}}- \rk N_{\overline{S'_i}/p_2^{-1}(S')} \\
&=\frac{1}{2}\rk N_{\widehat{S'_i}/\widehat{X}}- \frac{1}{2}\rk N_{\widehat{S'_i}/p^{-1}(S')} \\
&= \frac{1}{2}\rk N_{S'/X}. 
\end{align*}
Finally, Lemma \ref{lem on eta map} below finishes the proof.
\end{proof}
\begin{Lemma}\label{lem on eta map}
In \eqref{equ on eta map}, we have 
$$\eta\left([\overline{\Attr}_\fC(\Delta_{\widehat{S}})] \right)=|W^{\phi}|\cdot (-1)^{\sharp}\cdot[\overline{\Attr}_\fC(\Delta_{{S}})]. $$
\end{Lemma}
\begin{proof}
By dimension counting, we have $\eta\left([\overline{\Attr}_\fC(\Delta_{\widehat{S}})] \right)\in H^{\sT}_{\mathrm{top}}(Z)$; therefore 
$$\eta\left([\overline{\Attr}_\fC(\Delta_{\widehat{S}})]\right)=\lambda\cdot [\overline{\Attr}_\fC(\Delta_{S})], $$
for some $\lambda\in \mathbb{Z}$, which we compute by restricting above to $S\times S$. By $\sA$-localization, we have 
\begin{align*}
\left(S\times S\hookrightarrow X\times S \right)^!\eta_{\widehat{X}\times \widehat{S}\to X\times S}\left([\overline{\Attr}_\fC(\Delta_{\widehat{S}})]\right)&=
\eta_{\widehat{S}\times \widehat{S}\to S\times S}\left(\frac{[\Delta(\widehat{S})]}{e^{\sT}\left(N_{\overline{S}/p_2^{-1}(S)}\right)} \right) \\
&=(p_2\times p_2)_*\left(\frac{e^{\sT}\left(N^\vee_{\overline{S}/p_2^{-1}(S)}\right)}{e^{\sT}\left(N_{\overline{S}/p_2^{-1}(S)}\right)} \cdot [\Delta(\widehat{S})]\right) \\
&=|W^{\phi}|\cdot (-1)^{\rk N_{\overline{S}/p_2^{-1}(S)}}\cdot[\Delta(S)].
\end{align*}
This implies that $\lambda=|W^{\phi}|\cdot (-1)^{\rk N_{\overline{S}/p_2^{-1}(S)}}$ and we are done.  
\end{proof}

\subsection{Triangle lemma for cohomological stable envelopes}
Let $X=R/\!\!/_\theta G$ be a symmetric GIT quotient and $\sA\subseteq \sT$ be the tori as in Definition \ref{def of sym var}. We fix a chamber $\fC$ and a face $\fC'$, and let $\sA'\subseteq  \sA$ be the subtorus associated with $\fC'$. Then $X^{\sA'}$ is disjoint union of symmetric GIT quotients, so there exists cohomological stable envelope correspondence for $(X^{\sA'},0,\sT,\sA,\fC/\fC' )$ by Theorem \ref{construct coh stab corr on symm quiver var}. Consider the induced stable envelopes for a $\sT$-invariant function $\sw$ on $X$, and we have the following triangle lemma.


\begin{Theorem}\label{tri lem for coh stab}
Using the above notation, the following diagram is commutative
\begin{equation*}
\xymatrix{
H^\sT(X^\sA,\sw^\sA) \ar[rr]^{\Stab_{\fC }} \ar[dr]_{\Stab_{\fC/\fC' }} & & H^\sT(X,\sw) \\
 & H^\sT(X^{\sA'},\sw^{\sA'}). \ar[ur]_{\Stab_{\fC' }} &
}
\end{equation*}
\end{Theorem}

Theorem \ref{tri lem for coh stab} will be proven in steps. 

\textbf{Step 1.} We formulate a nonabelian stable envelope for symmetric quotient stack $\fX=[R/G]$, to this end, we will need the following lemma.

\begin{Lemma}
Consider the closure of the diagonal $\Delta_X$ in $X\times \fX$, denoted $\overline{\Delta_X}$. Then the projection from $\overline{\Delta_X}$ to $\fX$ is representable and proper.
\end{Lemma}

\begin{proof}
We have the following explicit construction of $\overline{\Delta_X}$ as a quotient stack. Consider the graph 
$$\mathrm{Graph}(q)\subset X\times R^{\theta\emph{-}s}$$ of quotient map $q\colon R^{\theta\emph{-}s}\to X$, and take the closure of $\mathrm{Graph}(q)$ in $X\times R$, 
denoted by $\overline{\mathrm{Graph}(q)}$. Then 
$$\overline{\Delta_X}=[\overline{\mathrm{Graph}(q)}/G]. $$ 
This implies that $\overline{\Delta_X}\to \fX$ is representable. Take $X_0=R/\!\!/G=\Spec\bC[R]^G$ to be the affine quotient, then the natural projection $\pi\colon X\to X_0$ is projective. Note that the composition $\pi\circ q\colon R^{\theta\emph{-}s}\to X_0$ extends to the quotient map $Q\colon R\to X_0$, so $\pi\times \id\colon X\times R\to X_0\times R$ maps $\overline{\mathrm{Graph}(q)}$ to $\mathrm{Graph}(Q)$. In particular, the projection $\overline{\mathrm{Graph}(q)}\to R$ factors as a proper morphism $\overline{\mathrm{Graph}(q)}\to \mathrm{Graph}(Q)$ followed by an isomorphism $\mathrm{Graph}(Q)\cong R$; thus $\overline{\mathrm{Graph}(q)}\to R$ is proper. This implies that $\overline{\Delta_X}\to \fX$ is representable and proper.
\end{proof}
Since the function $\sw\boxminus\sw$ vanishes on $\overline{\Delta_X}$, the fundamental class $[\overline{\Delta_X}]$ induces a correspondence map
between critical cohomologies.
\begin{Definition}\label{def on nonab stab}
We call the induced map 
\begin{align}\label{na stab coh_sym quot}
\bPsi_H\colon H^\sT(X,\sw)\to H^\sT(\fX,\sw) 
\end{align}
by $[\overline{\Delta_X}]$, a (cohomological) \textit{nonabelian stable envelope}. 
\end{Definition}
It turns out that $\bPsi_H$ admits a \textit{characterization} in terms of a degree bound condition similar to that of stable envelopes. For any nonzero cocharacter $\sigma:\bC^*\to G$, let $R^\sigma\subset R$ be the fixed subspace with respect to the induced $\bC^*$-action and $G^\sigma\subset G$ be the $\sigma$-fixed subgroup of $G$. Define $\fX^\sigma:=[R^\sigma/G^\sigma]$, and denote $$\mathbf j_\sigma:\fX^\sigma\to \fX$$ the natural map between quotient stacks.

\begin{Theorem}\label{thm: deg bound na stab coh}
$\bPsi_H$ is the unique $H_\sT(\pt)$-linear map from $H^\sT(X,\sw)$ to $H^\sT(\fX,\sw)$ that satisfies the following:
\begin{itemize}
    \item $\mathrm{res}\circ \bPsi_H=\id$, where $\mathrm{res}\colon H^\sT(\fX,\sw)\to  H^\sT(X,\sw)$ is the map induced by restriction from $\fX$ to the open substack $X$,
    \item for all nonzero cocharacter $\sigma:\bC^*\to G$, the inequality
    \begin{align}\label{deg bound na stab coh_to show}
        \deg_\sigma \mathbf j_\sigma^* \bPsi_H(\gamma)<\frac{1}{2} (\dim \fX-\dim \fX^\sigma)
    \end{align}
    holds for all $\gamma\in H^\sT(X,\sw)$. Here $\deg_\sigma$ is the polynomial degree in $H_{\bC^*}(\pt)$ in the decomposition $H^{\sT}(\fX^{\sigma},\sw)\cong H^{\sT\times G^\sigma/\sigma(\bC^*)}(R^\sigma,\sw)\otimes H_{\bC^*}(\pt)$.
\end{itemize}
\end{Theorem}

Theorem \ref{thm: deg bound na stab coh} 
will be used in the proof of Theorem \ref{thm: hall coh_sym quot} below, and will be proven in Section \ref{Proof of thm: deg bound na stab coh}. \bigskip

\textbf{Step 2.} We introduce a Hall operation relating critical cohomology of ``$\sA$-fixed'' stack of $\fX$ to that of $\fX$, and compare it with the stable envelope. Take a group homomorphism $$\phi:\sA\to G, $$ 
define $R^{\sA,\phi}$ to be the $\sA$-fixed subspace of $R$ where $\sA$ acts via $\sA\xrightarrow{(\id,\phi)}\sA\times G\curvearrowright R$, and define 
$$\fX^{\sA,\phi}:=[R^{\sA,\phi}/G^{\phi(\sA)}], $$ 
where $G^{\phi(\sA)}$ is the subgroup of $G$ fixed by adjoint action of $\phi(\sA)$. We assume that the stable locus $R^{\sA,\phi}\cap R^{\theta\emph{-}s}$ is nonempty, so the GIT quotient $X^{\sA,\phi}:=R^{\sA,\phi}/\!\!/_\theta G^{\phi(\sA)}$ is a connected component of $\sA$-fixed loci $X^\sA$. Note that all connected components of $X^\sA$ arises in this way, and we have
\begin{equation}\label{equ on dis uni of X}
    \fX^\sA=\bigsqcup_{\phi/\sim} \fX^{\sA,\phi},
\end{equation}
where the sum is taken for all homomorphisms $\phi:\sA\to G$ modulo the equivalence relation: $\phi_1\sim\phi_2$ if and only if they give the isomorphic $(G\times \sA)$-module structures on $R$.

We define a map $\mathfrak{m}^\phi_\fC\colon H^{\sT}(\fX^{\sA,\phi},\sw)\to H^{\sT}(\fX,\sw)$, according to a choice of chamber $\fC$, as follows. Let 
$$L^\phi_\fC:=\Attr_\fC(R^{\sA,\phi})\subset R, \quad P^\phi_\fC:=\Attr_\fC(G^{\phi(\sA)})\subset G, $$ 
where $L^\phi_\fC$ is a linear subspace of $R$, and $P^\phi_\fC$ is a parabolic subgroup of $G$. 
There are maps 
$$
\xymatrix{
R^{\sA,\phi} &\ar[l]_-{q} L^\phi_\fC \ar[r]^-{p} & R,
}
$$
where $q$ is the attraction map and $p$ is the inclusion. 
Then $P^\phi_\fC$ naturally acts on $L^\phi_\fC$ and maps $p$, $q$ are $P^\phi_\fC$-equivariant,
where $P^\phi_\fC$-action on $R$ is via $P^\phi_\fC\subseteq G$ and $P^\phi_\fC$-action on $R^{\sA,\phi}$ 
is via the contraction map $P^\phi_\fC\to G^{\phi(\sA)}$.
Therefore we get a diagram
\begin{equation}\label{Hall diagram}
\xymatrix{
\fX^{\sA,\phi} &\ar[l]_-{\fq} \mathfrak L^\phi_{\fC}:=[L^\phi_{\fC}/P^\phi_\fC] \ar[r]^-{\fp} & \fX,
}
\end{equation}
where $\fq$ is smooth and $\fp$ is proper.

\begin{Definition}\label{def of hall op_sym quot}
We define the \textit{Hall operation}
\begin{align*}
\mathfrak{m}^\phi_{\fC}:=\fp_*\circ\fq^*\colon H^\sT(\fX^{\sA,\phi},\sw)\longrightarrow H^\sT(\fX,\sw).
\end{align*}
\end{Definition}

\begin{Remark}\label{rmk:rewrite hall op_sym quot}
Let $\widetilde{L}^\phi_\fC:=G\times^{P^\phi_\fC}{L}^\phi_\fC$ and $\widetilde{p}:\widetilde{L}^\phi_\fC\to R$ be the natural morphism induced by the $G$-action. Then there are natural isomorphisms:
\begin{align*}
    \mathrm{ind}^G_{P^\phi_\fC}\colon H^{\sT\times P^\phi_\fC}({L}^\phi_\fC,\sw)\cong H^{\sT\times G}(\widetilde{L}^\phi_\fC,\sw).
\end{align*}
Note that the inverse of $\mathrm{ind}^G_{P^\phi_\fC}$ is given by $i^*\circ \mathrm{res}^G_{P^\phi_\fC}$, where $i:{L}^\phi_\fC\hookrightarrow \widetilde{L}^\phi_\fC$ is the closed immersion induced by $P^\phi_\fC\hookrightarrow G$ and $\mathrm{res}^G_{P^\phi_\fC}$ is the base change functor that takes $G$-equivariant cohomology to the $P^\phi_\fC$-equivariant counterpart. Using the Cartesian diagram
\begin{equation*}
\xymatrix{
\widetilde{L}^\phi_\fC \ar[r]^{\widetilde{p}} \ar[d]_{\cdot/G}  \ar@{}[dr]|{\Box}  & R \ar[d]^{\cdot/G}\\
\fL^\phi_\fC \ar[r]^{\mathfrak{p}} & \fX, 
}
\end{equation*}
the Hall operation can be rewritten as
\begin{align*}
    \mathfrak{m}^\phi_{\fC}=\widetilde{p}_* \circ\mathrm{ind}^G_{P^\phi_\fC}\circ \fq^*\colon H^{\sT\times G^{\phi(\sA)}}(R^{\sA,\phi},\sw)\longrightarrow H^{\sT\times G}(R,\sw).
\end{align*}
\end{Remark}


\begin{Theorem}\label{thm: hall coh_sym quot}
The following diagram is commutative
\begin{equation}\label{diag on hall coh_sym quot}
\xymatrix{
H^\sT(\fX^{\sA,\phi},\sw) \ar[rr]^{\mathfrak{m}^\phi_{\fC}} & & H^\sT(\fX,\sw)\\
H^\sT(X^{\sA,\phi},\sw) \ar[u]^{\bPsi_H^\phi} \ar[rr]^{\Stab_\fC} & & H^\sT(X,\sw),
\ar[u]_{\bPsi_H}
}
\end{equation}
where the left vertical arrow is the map $\bPsi_H$ for the stack $\fX^{\sA,\phi}$.
\end{Theorem}

We will prove Theorem \ref{thm: hall coh_sym quot} in Section \ref{Proof of thm: hall coh_sym quot}. \bigskip

\textbf{Step 3.} To finish the proof of Theorem \ref{tri lem for coh stab}, consider the following diagram on a triangular prism:
\begin{equation*}
\xymatrix{
H^\sT(\fX^{\sA,\phi},\sw) \ar[rr]^{\mathfrak{m}^\phi_{\fC}}\ar[dr]_{\mathfrak{m}^\phi_{\fC/\fC'}} & & H^\sT(\fX,\sw) \ar[dd]^{\mathrm{res}} \\
 & H^\sT(\fX^{\sA',\phi},\sw) \ar[ur]_{\mathfrak{m}^\phi_{\fC'}} &\\
H^\sT(X^{\sA,\phi},\sw) \ar[rr]^{\Stab_{\fC}\qquad\qquad}|-{\phantom{aaa}} \ar[dr]_{\Stab_{\fC/\fC'}} \ar[uu]^{\bPsi^{\phi}_H} & & H^\sT(X,\sw)  \\
 & H^\sT(X^{\sA',\phi},\sw) \ar[ur]_{\,\,\, \Stab_{\fC'}} \ar[uu]_>>>>>>{\bPsi^{\phi}_H} &
}
\end{equation*}
By Theorem \ref{thm: hall coh_sym quot}, the right, the backward, and the left squares are commutative. The upper triangle is commutative because Hall operations are associative. These imply the commutativity of the lower triangle.

\subsection{Proof of Theorem \ref{thm: deg bound na stab coh}}\label{Proof of thm: deg bound na stab coh}

The equality $\mathrm{res}\circ \bPsi_H=\id$ is obvious from the definition of \eqref{na stab coh_sym quot}.

\subsubsection{Degree bound}
To show that $\bPsi_H$ satisfies the degree bound condition \eqref{deg bound na stab coh_to show}, it is enough to show:
\begin{align}\label{deg bound na stab coh}
    \deg_\sigma (\mathbf j_\sigma\times \id)^![\overline{\Delta_X}]<\frac{1}{2}(\dim R^{\sigma\emph{-}\text{moving}}-\dim \Lie(G)^{\sigma\emph{-}\text{moving}}),
\end{align}
where $(\mathbf j_\sigma\times \id)^![\overline{\Delta_X}]$ is the pullback of the fundamental class of $\overline{\Delta_X}$ along the map $\mathbf j_\sigma\times \id:\fX^\sigma\times X\to \fX\times X$, which is an element in $H^{\sT}(\fX^\sigma\times_{\fX}\overline{\Delta_X})$.

We first prove a similar statement for the abelian quotient. Define $\widehat{\fX}:=[R/H]$ and consider the diagonal
$$\Delta_{\widehat{X}}\subseteq \widehat{X}\times \widehat{X}. $$ 
Take its closure $\overline{\Delta_{\widehat{X}}}$ in $\widehat{\fX}\times \widehat{X}$. Without loss of generality, we assume that image of $\sigma$ is contained in $H$, and let 
$$\widehat{\mathbf j}_\sigma:\widehat{\fX}^{\sigma}:=[R^\sigma/H]\to \widehat{\fX}$$ be the natural map.
\begin{Lemma}
We have 
\begin{align}\label{deg bound na stab coh_abelian}
    \deg_\sigma (\widehat{\mathbf j}_\sigma\times \id)^![\overline{\Delta_{\widehat{X}}}]<\frac{1}{2}\dim R^{\sigma\emph{-}\text{moving}},
\end{align}
where $(\widehat{\mathbf j}_\sigma\times \id)^![\overline{\Delta_{\widehat{X}}}]$ is the pullback of the fundamental class of $\overline{\Delta_{\widehat{X}}}$ along the map $\widehat{\mathbf j}_\sigma\times \id:\widehat{\fX}^\sigma\times \widehat{X}\to \widehat{\fX}\times \widehat{X}$, which is an element in 
$H^{\sT}\left(\widehat{\fX}^\sigma\times_{\widehat{\fX}}\overline{\Delta_{\widehat{X}}}\right)$.
\end{Lemma}

\begin{proof}
Consider the graph of the quotient map $\widehat{q}\colon R^{\theta\emph{-}s}\to \widehat{X}$:
\begin{align*}
    \mathrm{Graph}(\widehat{q})\subset \widehat{X}\times R^{\theta\emph{-}s},
\end{align*}
and take the closure of $\mathrm{Graph}(\widehat{q})$ in $\widehat{X}\times R$, denoted by $\overline{\mathrm{Graph}(\widehat{q})}$. Then $\overline{\Delta_{\widehat{X}}}=[\overline{\mathrm{Graph}(\widehat{q})}/H]$, and 
\begin{align*}
H^{\sT}(\widehat{\fX}^\sigma\times_{\widehat{\fX}}\overline{\Delta_{\widehat{X}}})=H^{\sT\times H}\left((\widehat{X}\times R^{\sigma}) \:\bigcap\: \overline{\mathrm{Graph}(\widehat{q})}\right).
\end{align*}
To prove \eqref{deg bound na stab coh_abelian}, it suffices to show $$\dim (\widehat{X}\times R^{\sigma}) \:\bigcap\: \overline{\mathrm{Graph}(\widehat{q})}-\frac{1}{2}\dim R^{\sigma\emph{-}\text{moving}}<\dim \overline{\mathrm{Graph}(\widehat{q})}+\dim \widehat{X}\times R^{\sigma}-\dim \widehat{X}\times R,$$
where the right-hand-side is the homological degree of $(\widehat{X}\times R^\sigma\hookrightarrow \widehat{X}\times R)^![\overline{\mathrm{Graph}(\widehat{q})}]$. It is elementary to see that the above inequality is equivalent to
\begin{align}\label{dim bound to show}
\dim (\widehat{X}\times R^{\sigma}) \:\bigcap\: \overline{\mathrm{Graph}(\widehat{q})}<\frac{1}{2}(\dim R+\dim R^{\sigma}).
\end{align}
Since $R\cong R^\vee$ as $H$-representations, we have an isomorphism of $H$-representations:
$$R\cong T^*N\oplus R_0, $$ 
where $N$ is a direct sum of nontrivial weight spaces and $R_0$ is a trivial $H$-representation. Define $$\nu\colon R=T^*N\times R_0\xrightarrow{\mu\times \id} \mathfrak{h}^*\times R_0, $$ 
where $\mu\colon T^*N\to \mathfrak{h}^*$ is the moment map, then the restriction of $\nu$ to $R^{{\theta}\emph{-}s}$ is smooth. By passing to the quotient, we obtain a smooth map 
$$\bar{\nu}\colon  \widehat{X}\to \mathfrak{h}^*\times R_0, $$
and for any $(x,y)\in \mathfrak{h}^*\times R_0$, $\bar{\nu}^{-1}(x,y)$ is a smooth symplectic variety. Let $\mathrm{pr}\colon R\to R_0$ be the projection, then $\mathrm{Graph}(\widehat{q})$ is contained in $(\bar{\nu}\times \mathrm{pr})^{-1}(\mathfrak{h}^*\times \Delta_{R_0})\subset \widehat{X}\times R$. 

For any $(x,y,y)\in \mathfrak{h}^*\times \Delta_{R_0}$, we have a Lagrangian: 
$$(\bar{\nu}\times \mathrm{pr})^{-1}(x,y,y)\:\bigcap \: \mathrm{Graph}(\widehat{q})\subseteq \bar{\nu}^{-1}(x,y)\times T^*N. $$ 
After taking the closure, we get a Lagrangian
$$(\bar{\nu}\times \mathrm{pr})^{-1}(x,y,y)\:\bigcap \: \overline{\mathrm{Graph}(\widehat{q})}\subseteq \bar{\nu}^{-1}(x,y)\times T^*N. $$ 
Let $\mathrm{pr}^\sigma\colon R^\sigma=T^*N^\sigma\times R_0\to R_0$ be the projection. Using the same argument as Lemma \ref{attr is isotropic}, 
$$(\bar{\nu}\times \mathrm{pr}^\sigma)^{-1}(x,y,y)\:\bigcap \: \overline{\mathrm{Graph}(\widehat{q})}\subseteq \bar{\nu}^{-1}(x,y)\times T^*N^\sigma $$ 
is isotropic, and we have
\begin{align*}
    \dim \:(\bar{\nu}\times \mathrm{pr}^\sigma)^{-1}(x,y,y)\:\bigcap \: \overline{\mathrm{Graph}(\widehat{q})}\leqslant \frac{1}{2}\dim\: \bar{\nu}^{-1}(x,y)\times T^*N^\sigma=\frac{1}{2}(\dim R+\dim R^{\sigma})-\dim \mathfrak{h}^*\times R_0.
\end{align*}
There exists a dense open subset $U\subset \mathfrak{h}^*$ such that $\mu^{-1}(U)\subseteq (T^*N)^{\bar{\theta}\emph{-}s}$. It follows that ${\nu}^{-1}(U\times R_0)\subseteq R^{\bar{\theta}\emph{-}s}$. We note that $\mathrm{Graph}(\widehat{q})$ is contained in the closed subset $\widehat{X}\times_{\mathfrak{h}^*\times R_0}R$, so $\overline{\mathrm{Graph}(\widehat{q})}\subseteq\widehat{X}\times_{\mathfrak{h}^*\times R_0}R$. Therefore,
\begin{align*}
\left(\bar{\nu}^{-1}(U\times R_0)\times R\right)\:\bigcap\: \overline{\mathrm{Graph}(\widehat{q})}\subseteq \left(\bar{\nu}^{-1}(U\times R_0)\times R\right)\:\bigcap\: \left(\widehat{X}\times_{\mathfrak{h}^*\times R_0}R\right)\subseteq \bar{\nu}^{-1}(U\times R_0)\times R^{\bar{\theta}\emph{-}s}.
\end{align*}
Since $R^\sigma \cap R^{\bar{\theta}\emph{-}s}=\emptyset$, we see that 
\begin{align*}
    (\bar{\nu}\times \mathrm{pr}^\sigma)^{-1}(x,y,y)\:\bigcap \: \overline{\mathrm{Graph}(\widehat{q})}=\emptyset
\end{align*}
for all $(x,y,y)\in U\times \Delta_{R_0}$. As a result, we have
\begin{align*}
    \dim (\widehat{X}\times R^{\sigma}) \:\bigcap\: \overline{\mathrm{Graph}(\widehat{q})}\leqslant \dim \mathfrak{h}^*\times \Delta_{R_0}-1+ \frac{1}{2}\dim\: \bar{\nu}^{-1}(x,y)\times T^*N^\sigma=\frac{1}{2}(\dim R+\dim R^{\sigma})-1.
\end{align*}
This proves \eqref{dim bound to show}.
\end{proof}

Next, we deduce \eqref{deg bound na stab coh} from \eqref{deg bound na stab coh_abelian}. An argument similar to Lemma \ref{lem on eta map} shows that
\begin{align}\label{eta identity_na stab}
    \eta_{\widehat{\fX}\times \widehat{X}\to \fX\times X}([\overline{\Delta_{\widehat{X}}}])=|W|\cdot (-1)^\sharp\cdot [\overline{\Delta_{X}}].
\end{align}
Here $\eta_{\widehat{\fX}\times \widehat{X}\to \fX\times X}$ is defined for the $G/H\times G/H$-bundle $\widehat{\fX}\times \widehat{X}\to \fX\times X$ by $\eta_{\widehat{\fX}\times \widehat{X}\to \fX\times X}=(p_2\times p_2)_*\circ \left((p_1\times p_1)^!\right)^{-1}$, where 
\begin{align*}
    p\colon \widehat{X} \xrightarrow{p_1} \overline{X}:=R^{\theta\emph{-}s}\times_G(G/B) \xrightarrow{p_2}  X,\quad p\colon \widehat{\fX} \xrightarrow{p_1} \overline{\fX}:=[R/B] \xrightarrow{p_2}  \fX
\end{align*}
are the natural projections.

Define $\widetilde{\mathbf j}_\sigma:\widetilde{\fX}^\sigma:=\fX^\sigma\times_{\fX}\widehat{\fX}\to \widehat{\fX}$ the pullback of $\mathbf j_\sigma$ to $\widehat{\fX}$, then by base change we have
\begin{align}\label{base change}
({\mathbf j}_\sigma\times \id)^!\eta_{\widehat{\fX}\times \widehat{X}\to \fX\times X}([\overline{\Delta_{\widehat{X}}}])=\eta_{\widetilde{\fX}^\sigma\times \widehat{X}\to {\fX}^\sigma\times X}\left((\widetilde{\mathbf j}_\sigma\times\id)^![\overline{\Delta_{\widehat{X}}}]\right).
\end{align}
We compute $\eta_{\widetilde{\fX}^\sigma\times \widehat{X}\to {\fX}^\sigma\times X}$ as follows. Note that
\begin{align*}
    \widetilde{\fX}^\sigma\cong [R^\sigma\times_{G^\sigma}G/H].
\end{align*}
Here $R^\sigma\times_{G^\sigma}G$ is the quotient of $R^\sigma\times G$ by the diagonal action of $G^\sigma$, where $G^\sigma$ acts on $R^\sigma$ naturally and acts on $G$ by left multiplication. $R^\sigma\times_{G^\sigma}G$ has a right $H$-action on $G$, and $[R^\sigma\times_{G^\sigma}G/H]$ is the quotient stack. We notice that for every $w\in W^\sigma\backslash W$ ($W$ and $W^\sigma$ are Weyl groups of $G$ and $G^\sigma$ respectively), $\widehat{\fX}^{w^{-1}(\sigma)}$ is naturally a closed substack of $\widetilde{\fX}^\sigma$ via the embedding
\begin{align*}
    \widehat{\fX}^{w^{-1}(\sigma)}=[R^{w^{-1}(\sigma)}/H]\cong [R^{\sigma}\times_{G^\sigma} G^{\sigma}\cdot w/H]\hookrightarrow [R^{\sigma}\times_{G^\sigma} G/H]\cong \widetilde{\fX}^\sigma.
\end{align*}
Therefore, after $H_{\bC^*}(\pt)$-localization, we have
\begin{align}\label{sigma loc for eta}
\eta_{\widetilde{\fX}^\sigma\times \widehat{X}\to {\fX}^\sigma\times X}(\cdots)=\sum_{w\in W^\sigma\backslash W}\eta_{\widehat{\fX}^{w^{-1}(\sigma)}\times \widehat{X}\to {\fX}^\sigma\times X}\left(\frac{(\widehat{\fX}^{w^{-1}(\sigma)}\times \widehat{X}\hookrightarrow \widetilde{\fX}^\sigma\times \widehat{X})^!(\cdots)}{e^{\sT\times G^{\sigma}}(N_{w})}\right).
\end{align}
Here $N_w$ is the normal bundle of $[R^{\sigma}\times_{G^\sigma} P_\sigma\cdot w/B]$ in $[R^{\sigma}\times_{G^\sigma} G/B]$, where $P_\sigma=\Attr_+(G^\sigma)$ is a parabolic subgroup of $G$. Note that $\sigma$ fixes $R^{\sigma}\times G^{\sigma}\cdot w/H$, and $N_w$ is $\sigma$-moving with 
$$\rk N_w=\dim G/P_\sigma=\frac{1}{2}\dim \Lie(G)^{\sigma\emph{-}\text{moving}}. $$ 
Thus $\eta_{\widehat{\fX}^{w^{-1}(\sigma)}\times \widehat{X}\to {\fX}^\sigma\times X}$ does not change the $\sigma$-degree and $e^{\sT\times G^{\sigma}}(N_{w})^{-1}$ decreases the $\sigma$-degree by $\frac{1}{2}\dim \Lie(G)^{\sigma\emph{-}\text{moving}}$. Combining \eqref{eta identity_na stab}, \eqref{base change} and \eqref{sigma loc for eta}, we get
\begin{align*}
\deg_\sigma (\mathbf j_\sigma\times \id)^![\overline{\Delta_X}]&=\deg_\sigma \eta_{\widetilde{\fX}^\sigma\times \widehat{X}\to {\fX}^\sigma\times X}\left((\widetilde{\mathbf j}_\sigma\times\id)^![\overline{\Delta_{\widehat{X}}}]\right)\\
&\leqslant \max_{w\in W^\sigma\backslash W}\left\{\deg_{w^{-1}(\sigma)} (\widehat{\mathbf j}_{w^{-1}(\sigma)}\times \id)^![\overline{\Delta_{\widehat{X}}}]- \frac{1}{2}\dim \Lie(G)^{\sigma\emph{-}\text{moving}}\right\}\\
\text{\scriptsize{by \eqref{deg bound na stab coh_abelian}}}&<\frac{1}{2}(\dim R^{\sigma\emph{-}\text{moving}}-\dim \Lie(G)^{\sigma\emph{-}\text{moving}}).
\end{align*}
This proves \eqref{deg bound na stab coh}. 

\subsubsection{Uniqueness}\label{subsubsec: uniquess of Psi_H} Next we show that $\bPsi_H$ is the unique $H_\sT(\pt)$-linear map from $H^\sT(X,\sw)$ to $H^\sT(\fX,\sw)$ that satisfies the conditions in Theorem \ref{thm: deg bound na stab coh}. The idea is similar to that of uniqueness of stable envelopes. 

It is known that the unstable locus 
$$R^{\theta\emph{-}u}:=R\setminus R^{\theta\emph{-}ss}$$ admits equivariant Kempf-Ness (KN) stratification by connected locally-closed subvarieties \cite{Kir,DH}. 
The KN stratification is constructed iteratively by selecting a pair $(\sigma_i\in \cochar(G),Z_i:=R^{\sigma_i})$ which maximizes the numerical invariant
\begin{align}
\mu(\sigma):=\frac{(\sigma,\theta)}{|\sigma|}
\end{align}
among those $(\sigma,Z)$ for which $Z$ is not contained in the union of the previously defined strata. Here $|\cdot|$ is a fixed conjugation-invariant norm on the cocharacters of $G$. 

One defines the open subvariety $Z_i^*\subseteq  Z_i$ to consist of those points not lying on previously defined strata, and $$Y_i:=\Attr_{\sigma_i}(Z_i^*).$$ Then define the new strata $$S_i=G\cdot Y_i.$$ Note that $Y_i$ is invariant under the parabolic subgroup $P_i:=\Attr_{\sigma_i}(G^{\sigma_i})$, and it is known that the natural map 
\begin{equation}\label{equ on GYS}
G\times^{P_i}Y_i\to S_i 
\end{equation} 
is an isomorphism.

\begin{Remark}\label{rmk: KN strata coh}
For a KN stratum $S_{i}=G\cdot \Attr_\sigma(Z^*_i)$ associated with a cocharacter $\sigma$, we have natural isomorphisms
\begin{align*}
    H^G(S_i,\sw)\cong H^{P_i}(\Attr_{\sigma_i}(Z^*_i),\sw)\cong H^{G^{\sigma_i}}(\Attr_{\sigma_i}(Z^*_i),\sw)\cong H^{G^{\sigma_i}}(Z^*_i,\sw).
\end{align*}
The first isomorphism is induced by \eqref{equ on GYS}, which equals $(\Attr_{\sigma_i}(Z^*_i)\hookrightarrow S_i)^*\circ \mathrm{res}^G_{P_i}$; the second isomorphism follows from general fact that $$H^{P_i}(-)\cong H^{G^{\sigma_i}}(-),$$ since $G^{\sigma_i}$ is a Levi subgroup in $P_i$; the third isomorphism is induced by Gysin pullback $(Z^*_i\hookrightarrow \Attr_{\sigma_i}(Z^*_i))^*$ \cite[Eqn.~(37)]{Dav}. The composition $H^G(S_i,\sw)\cong H^{G^{\sigma_i}}(Z^*_i,\sw)$ is given by $(Z^*_i\hookrightarrow S_i)^*\circ \mathrm{res}^G_{G^{\sigma_i}}$, or equivalently $([Z^*_i/G^{\sigma_i}]\to [S_i/G])^*$.
\end{Remark}

We denote the following natural map between quotient stacks: 
\begin{equation}\label{equ on na map zig}
\mathbf i_{\sigma_i}:[Z^*_i/G^{\sigma_i}]\to \fX, \end{equation} 
which factors as an open immersion $[Z^*_i/G^{\sigma_i}]\hookrightarrow [Z_i/G^{\sigma_i}]$ followed by $\mathbf j_{\sigma_i}:[R^{\sigma_i}/G^{\sigma_i}]\to [R/G]$. In particular, we have well-defined Gysin pullbacks:
$$\mathbf i_{\sigma_i}^*\colon H^\sT(\fX,\sw)\to H^\sT([Z^*_i/G^{\sigma_i}],\sw). $$
And according to what we have established, 
\begin{align*}
        \deg_{\sigma_i} \mathbf i_{\sigma_i}^* \bPsi_H(\gamma)<\frac{1}{2} (\dim \fX-\dim \fX^{\sigma_i})=\rk N_{S_i/R}
\end{align*}
holds for all $\gamma\in H^\sT(X,\sw)$. By Remark \ref{rmk: KN strata coh}, $\mathbf i_{\sigma_i}^*$ can be identified with the Gysin pullback $(S_i\hookrightarrow R)^*$ followed by the isomorphism $H^G(S_i,\sw)\cong H^{G^{\sigma_i}}(Z^*_i,\sw)$. 

Suppose that there exists a nonzero class $\alpha\in H^{\sT\times G}(R,\sw)$ supported on $R^{\theta\emph{-}u}$ with the property that 
$$\deg_{\sigma_j}(S_j\hookrightarrow R)^*\alpha<\rk N_{S_j/R}$$ 
for all KN strata $S_j$. Let $S$ be a union of KN strata with a maximal strata $S_i$ (open in $S$), such that $\alpha$ is supported on $S$ and $(S_i\hookrightarrow R)^*\alpha\neq 0$. Then 
\begin{align*}
    \alpha=(S\hookrightarrow R)_*\widetilde{\alpha} 
\end{align*}
for some $\widetilde{\alpha}\in H^{\sT\times G}(S,\sw)$, and we have
\begin{align*}
    \deg_{\sigma_i}(S_i\hookrightarrow R)^*\alpha=\deg_{\sigma_i}e^{G}(N_{S_i/R})\cdot(\widetilde{\alpha}|_{S_i})\geqslant \rk N_{S_i/R}.
\end{align*}
We get a contradiction, and therefore we must have $\alpha=0$. This proves the uniqueness.

\subsection{Proof of Theorem \ref{thm: hall coh_sym quot}}\label{Proof of thm: hall coh_sym quot}

We claim that $\mathfrak{m}^\phi_{\fC}\circ \bPsi_H^\phi$ is induced by the correspondence $$[Z]\in H^\sT(\fX\times X^{\sA,\phi},\sw\boxminus\sw)_{Z}, $$
where $Z=\overline{\Attr}_\fC\left(\Delta_{X^{\sA,\phi}}\right)$ is the closure of $\Attr_\fC\left(\Delta_{X^{\sA,\phi}}\right)$ inside the stack $\fX\times X^{\sA,\phi}$ (here attracting set is taken inside the stable locus $X\times X^{\sA,\phi}$), and $[Z]$ is the image of the fundamental class of $Z$ under the canonical map \eqref{can with supp_coh}.

By construction, $\bPsi_H^\phi$ is induced by the correspondence $[\overline{\Delta_{X^{\sA,\phi}}}]$. Then the map $\fq^*\circ \bPsi_H^\phi$ is induced by the pullback correspondence $[(\fq\times\id)^{-1}(\overline{\Delta_{X^{\sA,\phi}}})]$ along the morphism 
$$(\fq\times\id):\fL^\phi_\fC\times X^{\sA,\phi}\to \fX^{\sA,\phi}\times X^{\sA,\phi}.$$ 
Since $\fq$ is smooth with connected fibers and $\overline{\Delta_{X^{\sA,\phi}}}$ is an irreducible stack, $(\fq\times\id)^{-1}(\overline{\Delta_{X^{\sA,\phi}}})$ is an irreducible closed substack of $\fL^\phi_\fC\times X^{\sA,\phi}$. Note that $(\fq\times\id)^{-1}(\overline{\Delta_{X^{\sA,\phi}}})$ contains an open substack $(\fq\times\id)^{-1}(\Delta_{X^{\sA,\phi}})$ and the latter is isomorphic to $\Attr_\fC\left(\Delta_{X^{\sA,\phi}}\right)$. So the map $\fp_*\circ\fq^*\circ \bPsi_H^\phi$ is induced by the correspondence $\left[(\fp\times\id)\left((\fq\times\id)^{-1}(\overline{\Delta_{X^{\sA,\phi}}})\right)\right]$. Note that $(\fp\times\id)$ maps $(\fq\times\id)^{-1}(\Delta_{X^{\sA,\phi}})$ isomorphically onto $\Attr_\fC\left(\Delta_{X^{\sA,\phi}}\right)$. Then by the irreducibility of $(\fq\times\id)^{-1}(\overline{\Delta_{X^{\sA,\phi}}})$, we have
$$(\fp\times\id)\left((\fq\times\id)^{-1}(\overline{\Delta_{X^{\sA,\phi}}})\right)=Z.$$ 
This proves the claim that $\mathfrak{m}^\phi_{\fC}\circ \bPsi_H^\phi$ is induced by the correspondence $[Z]$.

Since the restriction of $Z$ to the open substack $X\times X^{\sA,\phi}$ is the closure of $\Attr_\fC\left(\Delta_{X^{\sA,\phi}}\right)$ inside $X\times X^{\sA,\phi}$, we have 
\begin{align*}
    \mathrm{res}\circ\mathfrak{m}^\phi_{\fC}\circ \bPsi_H^\phi=\Stab_\fC.
\end{align*}
Since $\mathrm{res}\circ\bPsi_H=\id$, it remains to show that the image of $\mathfrak{m}^\phi_{\fC}\circ \bPsi_H^\phi$ lies in the image of $\bPsi_H$. By Theorem \ref{thm: deg bound na stab coh}, this amounts to showing that for all nonzero cocharacter $\sigma:\bC^*\to G$ (by conjugation we can assume that the image of $\sigma$ is in $G^{\phi(\sA)}$), the inequality
\begin{align*}
\deg_\sigma \mathbf j_\sigma^* \circ\mathfrak{m}^\phi_{\fC}\circ \bPsi_H^\phi(\gamma)<\frac{1}{2} (\dim \fX-\dim \fX^\sigma)
\end{align*}
holds for all $\gamma\in H^\sT(X^{\sA,\phi},\sw)$. Equivalently, the limit
\begin{align}\label{deg bound 1_coh}
    \lim_{t\to \infty} \frac{{\mathbf j}_{\sigma}^*\circ\mathfrak{m}^\phi_{\fC}\circ\bPsi_H^{\phi}(\gamma)}{\sqrt{e^{\sT\times G^{\sigma}}\left((R-\Lie(G))^{\sigma\text{-moving}}\right)}}=0\,\, \text{ for all cocharacters $\sigma:\bC^*\to G^{\phi(\sA)}$},
\end{align}
where $t$ is the equivariant parameter in $H_{\bC^*} (\pt)=\bQ[t]$ and 
\begin{align*}
    \frac{{\mathbf j}_{\sigma}^*\circ\mathfrak{m}^\phi_{\fC}\circ\bPsi_H^{\phi}(\gamma)}{\sqrt{e^{\sT\times G^{\sigma}}\left((R-\Lie(G))^{\sigma\text{-moving}}\right)}}\in H^{\sT\times G^{\sigma}/\sigma(\bC^*)}(R^{\sigma},\sw)(t^{1/2}).
\end{align*}
Consider the morphism $\widetilde{p}:\widetilde{L}(\bv,\bd)^\phi_\fC\to R(\bv,\bd)$ defined in Remark \ref{rmk:rewrite hall op_sym quot}. The $\sigma$-fixed locus of $\widetilde{L}^\phi_\fC$ decomposes into connected components 
\begin{align*}
   \left(\widetilde{L}^\phi_\fC\right)^\sigma =\bigsqcup_{w\in W^\sigma\backslash W/W^\phi}\left(\widetilde{L}^\phi_\fC\right)^\sigma_w,\qquad \left(\widetilde{L}^\phi_\fC\right)^\sigma_w=G^{\sigma}\times^{P^{\phi,\sigma}_{\fC,w}}\left(L^\phi_\fC\right)^{w^{-1}(\sigma)},
\end{align*}
where $W^\sigma,W,W^\phi$ are Weyl groups of $G^\sigma,G,G^{\phi(\sA)}$ respectively, $w^{-1}(\sigma):=w^{-1}\cdot\sigma\cdot w$ is a cocharacter of $G^{\phi(\sA)}$, and 
$$P^{\phi,\sigma}_{\fC,w}=G^{\sigma}\bigcap wP^\phi_{\fC}w^{-1},$$ which acts on $\left(L^\phi_\fC\right)^{w^{-1}(\sigma)}$ via $g\mapsto w^{-1}gw\curvearrowright \left(L^\phi_\fC\right)^{w^{-1}(\sigma)}$. 

Denote the induced maps between $\sigma$-fixed loci:
\begin{align*}
    \widetilde{p}^\sigma_{w}:\left(\widetilde{L}^\phi_\fC\right)^\sigma_w\to R^\sigma,\quad q^\sigma_w\colon \left(L^\phi_\fC\right)^{w^{-1}(\sigma)}\to \left(R^{\sA,\phi}\right)^{w^{-1}(\sigma)},
\end{align*}
and the induced maps between quotient stacks:
\begin{align*}
    \mathfrak q^\sigma_w\colon \left[\frac{\left(L^\phi_\fC\right)^{w^{-1}(\sigma)}}{P^{\phi,\sigma}_{\fC,w}}\right]\longrightarrow \left[\frac{\left(R^{\sA,\phi}\right)^{w^{-1}(\sigma)}}{(G^{\phi(\sA)})^{w^{-1}(\sigma)}}\right],\quad \mathbf j^\phi_{w^{-1}(\sigma)}\colon \left[\frac{\left(R^{\sA,\phi}\right)^{w^{-1}(\sigma)}}{(G^{\phi(\sA)})^{w^{-1}(\sigma)}}\right]\longrightarrow \left[\frac{R^{\sA,\phi}}{G^{\phi(\sA)}}\right].
\end{align*}
Applying equivariant localization to the morphism $\widetilde{p}:\widetilde{L}^\phi_\fC\to R$ with respect to equivariant parameter $t$:
\begin{align*}
&\frac{{\mathbf j}_{\sigma}^*\circ\mathfrak{m}^\phi_{\fC}\circ\bPsi_H^{\phi}(\gamma)}{e^{\sT\times G^{\sigma}}\left(N_{R^\sigma/R}\right)}=\sum_{w\in W^\sigma\backslash W/W^\phi}\widetilde{p}^\sigma_{w*}\left(\frac{\mathrm{ind}^{G^\sigma}_{P^{\phi,\sigma}_{\fC,w}}\circ \mathfrak q^{\sigma*}_w\circ \mathbf j^{\phi*}_{w^{-1}(\sigma)}\circ\bPsi_H^{\phi}(\gamma)}{e^{\sT\times G^{\sigma}}\left(N_{\left(\widetilde{L}^\phi_\fC\right)^\sigma_w/\widetilde{L}^\phi_\fC}\right)}\right)\\
&=\sum_{w\in W^\sigma\backslash W/W^\phi}\widetilde{p}^\sigma_{w*}\circ\mathrm{ind}^{G^\sigma}_{P^{\phi,\sigma}_{\fC,w}} \left(\frac{\mathfrak q^{\sigma*}_w\circ \mathbf j^{\phi*}_{w^{-1}(\sigma)}\circ\bPsi_H^{\phi}(\gamma)}{e^{\sT\times P^{\phi,\sigma}_{\fC,w}}\left(N_{\left(\widetilde{L}^\phi_\fC\right)^\sigma_w/\widetilde{L}^\phi_\fC}\bigg|_{\left(L^\phi_\fC\right)^{w^{-1}(\sigma)}}\right)}\right)\\
&=\sum_{w\in W^\sigma\backslash W/W^\phi}\widetilde{p}^\sigma_{w*}\circ\mathrm{ind}^{G^\sigma}_{P^{\phi,\sigma}_{\fC,w}} \circ\mathfrak q^{\sigma*}_w\left(\frac{ \mathbf j^{\phi*}_{w^{-1}(\sigma)}\circ\bPsi_H^{\phi}(\gamma)}{e^{\sT\times (G^{\phi(\sA)})^{w^{-1}(\sigma)}}\left((L^\phi_\fC-\Lie(P^{\phi,\sigma}_{\fC,w})+\Lie(G))^{w^{-1}(\sigma)\text{-moving}}\right)}\right).
\end{align*}
Then to prove \eqref{deg bound 1_coh} for a given cocharacter $\sigma$, it is enough to show that
\begin{align*}
    \lim_{t\to \infty} \frac{\sqrt{ e^{(G^{\phi(\sA)})^{w^{-1}(\sigma)}}\left((N_{R^{w^{-1}(\sigma)}/R}+\Lie(G))^{w^{-1}(\sigma)\text{-moving}}\right)}\cdot\mathbf j^{\phi*}_{w^{-1}(\sigma)}\circ\bPsi_H^{\phi}(\gamma)}{e^{(G^{\phi(\sA)})^{w^{-1}(\sigma)}}\left((L^\phi_\fC-\Lie(P^{\phi,\sigma}_{\fC,w})+\Lie(G))^{w^{-1}(\sigma)\text{-moving}}\right)}=0\,\, \text{ for all $w\in W^\sigma\backslash W/W^\phi$}.
\end{align*}
Here we have used the following facts: 
\begin{align*}
    \deg_{w^{-1}(\sigma)} e^{(G^{\phi(\sA)})^{w^{-1}(\sigma)}}\left((N_{R^{w^{-1}(\sigma)}/R}+\Lie(G))^{w^{-1}(\sigma)\text{-moving}}\right)=\deg_\sigma e^{(G^{\phi(\sA)})^{\sigma}}\left((N_{R^{\sigma}/R}+\Lie(G))^{\sigma\text{-moving}}\right).
\end{align*}
Since 
\begin{align*}
    \deg_{w^{-1}(\sigma)} e^{(G^{\phi(\sA)})^{w^{-1}(\sigma)}}\left((N_{R^{w^{-1}(\sigma)}/R}+\Lie(G))^{w^{-1}(\sigma)\text{-moving}}\right)= \deg_{w^{-1}(\sigma)}e^{(G^{\phi(\sA)})^{w^{-1}(\sigma)}}\left((R+\Lie(G))^{w^{-1}(\sigma)\text{-moving}}\right),
\end{align*}
and
\begin{align*}
    & \quad \, \deg_{w^{-1}(\sigma)}  e^{(G^{\phi(\sA)})^{w^{-1}(\sigma)}}\left((L^\phi_\fC-\Lie(P^{\phi,\sigma}_{\fC,w})+\Lie(G))^{w^{-1}(\sigma)\text{-moving}}\right)=\\
    &\frac{1}{2}\deg_{w^{-1}(\sigma)}e^{(G^{\phi(\sA)})^{w^{-1}(\sigma)}}\left((R+\Lie(G))^{w^{-1}(\sigma)\text{-moving}}\right)+\frac{1}{2}\deg_{w^{-1}(\sigma)}e^{(G^{\phi(\sA)})^{w^{-1}(\sigma)}}\left((R^{\sA,\phi}-\Lie(G^{\phi(\sA)}))^{w^{-1}(\sigma)\text{-moving}}\right),
\end{align*}
it remains to show that 
\begin{align}\label{deg bound 2_coh}
    \lim_{t\to \infty} \frac{\mathbf j^{\phi*}_{w^{-1}(\sigma)}\circ\bPsi_H^{\phi}(\gamma)}{\sqrt{e^{(G^{\phi(\sA)})^{w^{-1}(\sigma)}}\left((R^{\sA,\phi}-\Lie(G^{\phi(\sA)}))^{w^{-1}(\sigma)\text{-moving}}\right)}}=0\,\, \text{ for all $w\in W^\sigma\backslash W/W^\phi$}.
\end{align}
By Theorem \ref{thm: deg bound na stab coh}, we have
\begin{align*}
    \deg_{w^{-1}(\sigma)} \mathbf j^{\phi*}_{w^{-1}(\sigma)}\bPsi^{\phi}(\gamma)< \frac{1}{2}\deg_{w^{-1}(\sigma)} e^{(G^{\phi(\sA)})^{w^{-1}(\sigma)}}\left((R^{\sA,\phi}-\Lie(G^{\phi(\sA)}))^{w^{-1}(\sigma)\text{-moving}}\right),
\end{align*}
then \eqref{deg bound 2_coh} follows from the above degree bound.

\subsection{Explicit formulas of cohomological stable envelopes}\label{sect on explicit form_coh}

Theorem \ref{thm: hall coh_sym quot} provides a tool to produce explicit formulas of cohomological stable envelopes for symmetric GIT quotient. 
As critical cohomology lacks a fundamental class in general, in the following corollary, we restrict to the case when $\sw=0$, which reduces to the ordinary BM homology. Results in this section will be used to calculate examples in \cite{COZZ1}.

\begin{Corollary}\label{cor: explicit formula w=0_coh}
Assume $\sw=0$, then we have 
\begin{align}\label{explicit formula w=0_coh}
    \bPsi_H\circ\Stab_\fC([X^{\sA,\phi}])=\sum_{w\in W/W^\phi}w\left(e^{\sT\times G^{\phi(\sA)}}\left(T\fX^{\sA,\phi\text{-repl}}\right)\right) \cdot [\fX]\:,
\end{align}
in particular,
\begin{align*}
    \Stab_\fC([X^{\sA,\phi}])=\sum_{w\in W/W^\phi}w\left(e^{\sT\times G^{\phi(\sA)}}\left(T\fX^{\sA,\phi\text{-repl}}\right)\right) \cdot [X]\:.
\end{align*}
Here $W$, $W^\phi$ are Weyl groups of $G$, $G^{\phi(\sA)}$ respectively,
$T\fX=R-\Lie(G)$ and $(-)^{\sA, \phi\text{-repl}}$ is the repelling part with respect to homomorphism $\phi\colon \sA\to G$, and $w$ acts on a weight $\mu$ of $G$ by $w(\mu)(g):=\mu(w^{-1}\cdot g\cdot w)$.
\end{Corollary}

\begin{proof}
By definition of nonabelian stable envelope, we have
$$\bPsi^\phi_H([X^{\sA,\phi}])=[\fX^{\sA,\phi}]. $$
Then \eqref{explicit formula w=0_coh} follow from Theorem \ref{thm: hall coh_sym quot} and explicit formulas of cohomological Hall operations in \cite[Thm.~2]{CoHA}.
\end{proof}

\section{\texorpdfstring{$K$}{K}-theoretic stable envelopes on symmetric GIT quotients}\label{app on k symm var}
In this section, we show that $K$-theoretic stable envelopes exist for symmetric GIT quotients with any $\sT$-invariant potential functions (Theorem \ref{thm: hall k_sym quot}). 
The proof makes use of nonabelian stable envelopes (see Definition \ref{def of na stab k_window}, and Theorem \ref{thm: deg bound na stab k} for a characterization property) 
and Hall operations, similarly as in the previous section. 
As a side application, we derive an explicit formula of $K$-theoretic stable envelopes when potentials are zero (Corollary \ref{cor: explicit formula w=0_k}).

\subsection{Window category and nonabelian stable envelope}
We use the same notations as in Definition \ref{def of sym var}.

\begin{Definition}
Let $\fX=[R/G]$, $\mathsf s\in \mathrm{Char}(G)\otimes_\bZ\bR$ be a \textit{generic slope}, that is, the projection of $\mathsf s$ to $\mathrm{Char}(H)\otimes_\bZ\bR$ is not integral for any subgroup $H\subseteq G$. The \textit{window subcategory} of $\D^b_{\coh}(\fX)$ with slope $\mathsf s$ is the full triangulated dg-subcategory $\mathbb M^{\mathsf s}$ in $\D^b_{\coh}(\fX)$ generated by $\mathcal O_{\fX}\otimes U$, where $U$ is an irreducible $G$-module whose characters lie in the polytope $$\frac{1}{2}\deg_G e^G_K(R(\bv,\bd)-\Lie(G))+\wt_G\mathsf s.$$
Let $\sw\colon \fX\to \bA^1$ be a function, the \textit{window subcategory} of $\mathrm{MF}_\coh(\fX,\sw)$ with slope $\mathsf s$ is the full subcategory $\mathrm{MF}_\coh(\mathbb M^{\mathsf s},\sw)$ in $\mathrm{MF}_\coh(\fX,\sw)$ generated by matrix factorizations $(\mathcal E_\bullet,d_\bullet)$ with $\cE_0,\cE_1\in \mathbb M^{\mathsf s}$. 
\end{Definition}
It is shown in \cite[Thm.~1.2]{HS} that the restriction functor $$\mathbb M^{\mathsf s}\stackrel{\simeq}{\to} \D^b_\coh(X)$$ is an equivalence of triangulated dg-categories. Then the restriction functor 
\begin{equation}\mathrm{MF}_\coh(\mathbb M^{\mathsf s},\sw)\stackrel{\simeq}{\to} \mathrm{MF}_\coh(X,\sw)\nonumber \end{equation} is an equivalence of triangulated dg-categories. Passing to Grothendieck's group, we get an isomorphism 
$$K_0(\mathrm{MF}_\coh(\mathbb M^{\mathsf s},\sw))\cong K(X,\sw). $$ 
Note that \cite[Thm.~1.2]{HS} holds in the $\sT$-equivariant case (mentioned in the proof \cite[Cor.~5.2]{HS}), so the above statements still hold in the $\sT$-equivariant version.

\begin{Definition}\label{def of na stab k_window}
The inverse of the above isomorphism followed by the natural map $K_0(\mathrm{MF}_\coh(\mathbb M^{\mathsf s},\sw))\to K(\fX,\sw)$ induces a map 
\begin{equation}\label{na stab_window}\bPsi^{\mathsf s}_K\colon K^\sT(X,\sw)\to K^\sT(\fX,\sw), \end{equation}
which we call the ($K$-theoretic) \textit{nonabelian stable envelope}.
\end{Definition}

By construction we have $\mathrm{res}\circ \bPsi^{\mathsf s}_K=\id$, where 
$$\mathrm{res}\colon K^\sT(\fX,\sw)\to K^\sT(X,\sw)$$ is the restriction to the stable locus. Moreover, we have the following analog of Theorem \ref{thm: deg bound na stab coh}.

\begin{Theorem}\label{thm: deg bound na stab k}
$\bPsi^{\mathsf s}_K$ is the unique $K_\sT(\pt)$-linear map from $K^\sT(X,\sw)$ to $K^\sT(\fX,\sw)$ that satisfies the following:
\begin{itemize}
    \item $\mathrm{res}\circ \bPsi^{\mathsf s}_K=\id$,
    \item for all nonzero cocharacter $\sigma:\bC^*\to G$ and all $\gamma\in K^\sT(X,\sw)$, there is a strict inclusion of polytopes
    \begin{align*}
    \deg_\sigma \mathbf j^*_{\sigma}\bPsi^{\mathsf s}_K(\gamma)\subsetneq  \frac{1}{2}\deg_{\sigma} e^G_K(R(\bv,\bd)-\Lie(G))+\wt_{\sigma}\mathsf s,
\end{align*}
where $\mathbf j_\sigma:[R^\sigma/G^\sigma]\to \fX$ is the natural morphism between quotient stacks.
\end{itemize}
\end{Theorem}

\begin{proof}
By the construction of window categories, $\bPsi^{\mathsf s}_K$ satisfies the two conditions. The uniqueness is proven similarly as Theorem \ref{thm: deg bound na stab coh} and we omit the details.
\end{proof}

\subsection{Hall operations}

\begin{Definition}\label{def of hall op k_sym quot}
Consider the diagram \eqref{Hall diagram}, and we define the $K$-theoretic \textit{Hall operations}
\begin{align*}
\mathfrak{m}^\phi_{\fC}:=\fp_*\circ\fq^*\colon K^\sT(\fX^{\sA,\phi},\sw)\longrightarrow K^\sT(\fX,\sw).
\end{align*}
\end{Definition}

\begin{Theorem}\label{thm: hall k_sym quot}
Let $\mathsf s\in \mathrm{Char}(G)\otimes_\bZ \bR$ be a generic slope. Then the composition
\begin{align*}
\mathrm{res}\circ \mathfrak{m}^\phi_{\fC}\circ \bPsi_K^{\phi,\mathsf s'}\colon K^\sT(X^{\sA,\phi},\sw)\to K^\sT(X,\sw)
\end{align*}
is a $K$-theoretic stable envelope with slope $\mathsf s$, where $\bPsi_K^{\phi,\mathsf s'}$ is the $K$-theoretic nonabelian stable envelope for the stack $\fX^{\sA,\phi}$ with the slope
\begin{align*}
    \mathsf s'=\mathsf s\otimes \det\left(T\fX^{\sA,\phi\text{-repl}}\right)^{1/2}\in \mathrm{Char}(G^{\phi(\sA)})\otimes_\bZ \bR.
\end{align*}
In the above, $T\fX^{\sA,\phi\text{-repl}}$ is the $\fC$-repelling part of $T\fX=R-\Lie(G)$ with respect to $\sA$-action via $\sA\xrightarrow{(\phi, \id)} G\times \sA$.
Moreover, the following diagram is commutative
\begin{equation}\label{diag on hall k_sym quot}
\xymatrix{
K^\sT(\fX^{\sA,\phi},\sw) \ar[rr]^{\mathfrak{m}^\phi_{\fC}} & & K^\sT(\fX,\sw)\\
K^\sT(X^{\sA,\phi},\sw) \ar[u]^{\bPsi_K^{\phi,\mathsf s'}} \ar[rr]^{\Stab^{\mathsf s}_\fC} & & K^\sT(X,\sw).
\ar[u]_{\bPsi^{\mathsf s}_K}
}
\end{equation}
\end{Theorem}

\begin{proof}
We claim that the image of $\mathfrak{m}^\phi_{\fC}\circ \bPsi_K^{\phi,\mathsf s'}$ is contained in the image of $\bPsi_K^{\mathsf s}$. By Theorem \ref{thm: deg bound na stab k}, it is enough to show: 
\begin{itemize}
    \item for all nonzero cocharacter $\sigma:\bC^*\to G^{\phi(\sA)}$ and all $\gamma\in K^\sT(X,\sw)$, there is a strict inclusion of polytopes
    \begin{align*}
    \deg_\sigma \mathbf j^*_{\sigma}\circ \mathfrak{m}^\phi_{\fC}\circ \bPsi_K^{\phi,\mathsf s'}(\gamma)\subsetneq  \frac{1}{2}\deg_{\sigma} e^G_K(R-\Lie(G))+\wt_{\sigma}\mathsf s.
    \end{align*}
\end{itemize}

For a vector bundle $V$, denote 
\begin{align}\label{hat e}
    \hat e^G_K(V):=\det(V)^{1/2}\cdot e^G_K(V), 
\end{align} 
and extend the definition to $K$-theory class by setting $$\hat e^G_K(V_1-V_2)=\hat e^G_K(V_1)/\hat e^G_K(V_2)$$ for vector bundles $V_1,V_2$. Then the degree condition can be rewritten as 
\begin{align*}
    \deg_\sigma \mathbf j^*_{\sigma}\circ \mathfrak{m}^\phi_{\fC}\circ \bPsi_K^{\phi,\mathsf s'}(\gamma)\subsetneq  \frac{1}{2}\deg_{\sigma} \hat e^G_K(N_{R^{\sigma}/R}-\Lie(G))+\wt_{\sigma}\mathsf s,
\end{align*}
equivalently the limits
\begin{align}\label{deg bound 1}
    \lim_{t\to 0,\infty} \frac{{\mathbf j}_{\sigma}^*\circ\mathfrak{m}^\phi_{\fC}\circ\bPsi_K^{\phi,\mathsf s'}(\gamma)}{\sqrt{\hat e^{G^{\sigma}}_K\left(N_{R^\sigma/R}-\Lie(G)\right)}\cdot\sigma^*(\mathsf s)}\,\, \text{ exist for all cocharacters $\sigma:\bC^*\to G^{\phi(\sA)}$},
\end{align}
where $t$ is the equivariant parameter in $K_{\bC^*}(\pt)=\bQ[t^\pm]$ and 
$$\frac{{\mathbf j}_{\sigma}^*\circ\mathfrak{m}^\phi_{\fC}\circ\bPsi_K^{\phi,\mathsf s'}(\gamma)}{\sqrt{\hat e^{G^{\sigma}}_K\left(N_{R^\sigma/R}-\Lie(G)\right)}\cdot\sigma^*(\mathsf s)}\in K^{\sT\times G^{\sigma}/\sigma(\bC^*)}(R^{\sigma},\sw)(t^{1/2}).$$ 
Consider the morphism $\widetilde{p}:\widetilde{L}^\phi_\fC\to R$ in Remark \ref{rmk:rewrite hall op_sym quot}. The $\sigma$-fixed locus of $\widetilde{L}^\phi_\fC$ decomposes into connected components 
\begin{align*}
   \left(\widetilde{L}^\phi_\fC\right)^\sigma =\bigsqcup_{w\in W^\sigma\backslash W/W^\phi}\left(\widetilde{L}^\phi_\fC\right)^\sigma_w,\qquad \left(\widetilde{L}^\phi_\fC\right)^\sigma_w=G^{\sigma}\times^{P^{\phi,\sigma}_{\fC,w}}\left(L^\phi_\fC\right)^{w^{-1}(\sigma)},
\end{align*}
where $W^\sigma,W,W^\phi$ are Weyl groups of $G^\sigma,G,G^{\phi(\sA)}$ respectively, $w^{-1}(\sigma):=w^{-1}\cdot\sigma\cdot w$ is a cocharacter of $G^{\phi(\sA)}$, and 
$$P^{\phi,\sigma}_{\fC,w}=G^{\sigma}\bigcap wP^\phi_{\fC}w^{-1},$$ which acts on $\left(L^\phi_\fC\right)^{w^{-1}(\sigma)}$ via $g\mapsto w^{-1}gw\curvearrowright \left(L^\phi_\fC\right)^{w^{-1}(\sigma)}$. 

Denote the induced maps between $\sigma$-fixed loci:
\begin{align*}
    \widetilde{p}^\sigma_{w}\colon\left(\widetilde{L}^\phi_\fC\right)^\sigma_w\to R^\sigma,\quad q^\sigma_w\colon \left(L^\phi_\fC\right)^{w^{-1}(\sigma)}\to \left(R^{\sA,\phi}\right)^{w^{-1}(\sigma)},
\end{align*}
and the induced maps between quotient stacks:
\begin{align*}
    \mathfrak q^\sigma_w\colon \left[\frac{\left(L^\phi_\fC\right)^{w^{-1}(\sigma)}}{P^{\phi,\sigma}_{\fC,w}}\right]\longrightarrow \left[\frac{\left(R^{\sA,\phi}\right)^{w^{-1}(\sigma)}}{(G^{\phi(\sA)})^{w^{-1}(\sigma)}}\right],\quad \mathbf j^\phi_{w^{-1}(\sigma)}\colon \left[\frac{\left(R^{\sA,\phi}\right)^{w^{-1}(\sigma)}}{(G^{\phi(\sA)})^{w^{-1}(\sigma)}}\right]\longrightarrow \left[\frac{R^{\sA,\phi}}{G^{\phi(\sA)}}\right].
\end{align*}
Applying equivariant localization to the morphism $\widetilde{p}:\widetilde{L}^\phi_\fC\to R$ with respect to equivariant parameter $t$:
\begin{align*}
&\frac{{\mathbf j}_{\sigma}^*\circ\mathfrak{m}^\phi_{\fC}\circ\bPsi_K^{\phi,\mathsf s'}(\gamma)}{e^{\sT\times G^{\sigma}}_K\left(N_{R^\sigma/R}\right)}=\sum_{w\in W^\sigma\backslash W/W^\phi}\widetilde{p}^\sigma_{w*}\left(\frac{\mathrm{ind}^{G^\sigma}_{P^{\phi,\sigma}_{\fC,w}}\circ \mathfrak q^{\sigma*}_w\circ \mathbf j^{\phi*}_{w^{-1}(\sigma)}\circ\bPsi_K^{\phi,\mathsf s'}(\gamma)}{e^{\sT\times G^{\sigma}}_K\left(N_{\left(\widetilde{L}^\phi_\fC\right)^\sigma_w/\widetilde{L}^\phi_\fC}\right)}\right)\\
&=\sum_{w\in W^\sigma\backslash W/W^\phi}\widetilde{p}^\sigma_{w*}\circ\mathrm{ind}^{G^\sigma}_{P^{\phi,\sigma}_{\fC,w}} \left(\frac{\mathfrak q^{\sigma*}_w\circ \mathbf j^{\phi*}_{w^{-1}(\sigma)}\circ\bPsi_K^{\phi,\mathsf s'}(\gamma)}{e^{\sT\times P^{\phi,\sigma}_{\fC,w}}_K\left(N_{\left(\widetilde{L}^\phi_\fC\right)^\sigma_w/\widetilde{L}^\phi_\fC}\bigg|_{\left(L^\phi_\fC\right)^{w^{-1}(\sigma)}}\right)}\right)\\
&=\sum_{w\in W^\sigma\backslash W/W^\phi}\widetilde{p}^\sigma_{w*}\circ\mathrm{ind}^{G^\sigma}_{P^{\phi,\sigma}_{\fC,w}} \circ\mathfrak q^{\sigma*}_w\left(\frac{ \mathbf j^{\phi*}_{w^{-1}(\sigma)}\circ\bPsi_K^{\phi,\mathsf s'}(\gamma)}{e^{\sT\times (G^{\phi(\sA)})^{w^{-1}(\sigma)}}_K\left((L^\phi_\fC-\Lie(P^{\phi,\sigma}_{\fC,w})+\Lie(G))^{w^{-1}(\sigma)\text{-moving}}\right)}\right).
\end{align*}
Note that by self-duality, we have 
$$e^{\sT\times G^{\sigma}}_K\left(N_{R^\sigma/R}\right)=\hat e^{\sT\times G^{\sigma}}_K\left(N_{R^\sigma/R}\right)\cdot(\text{some $\sT/\sA$-character}). $$ Then to prove \eqref{deg bound 1} for a given cocharacter $\sigma$, it is enough to show that
\begin{align*}
    \lim_{t\to 0,\infty} \frac{\sqrt{\hat e^{(G^{\phi(\sA)})^{w^{-1}(\sigma)}}_K\left(N_{R^{w^{-1}(\sigma)}/R}+\Lie(G)\right)}\cdot\mathbf j^{\phi*}_{w^{-1}(\sigma)}\circ\bPsi_K^{\phi,\mathsf s'}(\gamma)}{e^{(G^{\phi(\sA)})^{w^{-1}(\sigma)}}_K\left((L^\phi_\fC-\Lie(P^{\phi,\sigma}_{\fC,w})+\Lie(G))^{w^{-1}(\sigma)\text{-moving}}\right)\cdot w^{-1}(\sigma)^*(\mathsf s)}\,\, \text{ exist for all $w\in W^\sigma\backslash W/W^\phi$}.
\end{align*}
Here we have used the following facts: $w^{-1}(\sigma)^*(\mathsf s)=\sigma^*(\mathsf s)$ since $\mathsf s$ is Weyl-invariant, and 
\begin{align*}
    \deg_{w^{-1}(\sigma)} \hat e^{(G^{\phi(\sA)})^{w^{-1}(\sigma)}}_K\left(N_{R^{w^{-1}(\sigma)}/R}+\Lie(G)\right)=\deg_\sigma \hat e^{(G^{\phi(\sA)})^{\sigma}}_K\left(N_{R^{\sigma}/R}+\Lie(G)\right).
\end{align*}
According to the definition of $\mathsf s'$, we have
\begin{align*}
    \mathsf s'=\mathsf s\otimes\det\left(L^\phi_{\fC}-\Lie(P^\phi_\fC)\right)^{-1/2},
\end{align*}
and
\begin{align*}
    & \quad \,\, e^{(G^{\phi(\sA)})^{w^{-1}(\sigma)}}_K\left((L^\phi_\fC-\Lie(P^{\phi,\sigma}_{\fC,w})+\Lie(G))^{w^{-1}(\sigma)\text{-moving}}\right)\cdot w^{-1}(\sigma)^*(\mathsf s)\\
    &=\hat e^{(G^{\phi(\sA)})^{w^{-1}(\sigma)}}_K\left((L^\phi_\fC-\Lie(P^{\phi,\sigma}_{\fC,w})+\Lie(G))^{w^{-1}(\sigma)\text{-moving}}\right)\cdot w^{-1}(\sigma)^*(\mathsf s').
\end{align*}
Since 
\begin{align*}
    \deg_{w^{-1}(\sigma)}\hat e^{(G^{\phi(\sA)})^{w^{-1}(\sigma)}}_K\left(N_{R^{w^{-1}(\sigma)}/R}+\Lie(G)\right)= \deg_{w^{-1}(\sigma)}\hat e^{(G^{\phi(\sA)})^{w^{-1}(\sigma)}}_K\left(R+\Lie(G)\right),
\end{align*}
and
\begin{align*}
    & \quad \, \deg_{w^{-1}(\sigma)} \hat e^{(G^{\phi(\sA)})^{w^{-1}(\sigma)}}_K\left((L^\phi_\fC-\Lie(P^{\phi,\sigma}_{\fC,w})+\Lie(G))^{w^{-1}(\sigma)\text{-moving}}\right)\\
    &=\deg_{w^{-1}(\sigma)} \hat e^{(G^{\phi(\sA)})^{w^{-1}(\sigma)}}_K\left(L^\phi_\fC-\Lie(P^{\phi,\sigma}_{\fC,w})+\Lie(G)\right),
\end{align*}
it remains to show that 
\begin{align}\label{deg bound 2}
    \lim_{t\to 0,\infty} \frac{\mathbf j^{\phi*}_{w^{-1}(\sigma)}\circ\bPsi_K^{\phi,\mathsf s'}(\gamma)}{\sqrt{\hat e^{(G^{\phi(\sA)})^{w^{-1}(\sigma)}}_K\left(R^{\sA,\phi}-\Lie(G^{\phi(\sA)})\right)}\cdot w^{-1}(\sigma)^*(\mathsf s')}\,\, \text{ exist for all $w\in W^\sigma\backslash W/W^\phi$}.
\end{align}
By Theorem \ref{thm: deg bound na stab k}, we have
\begin{align*}
    \deg_{w^{-1}(\sigma)} \mathbf j^{\phi*}_{w^{-1}(\sigma)}\bPsi^{\phi,\mathsf s'}_K(\gamma)\subseteq  \frac{1}{2}\deg_{w^{-1}(\sigma)} e^{G^{\phi(\sA)}}_K(R^{\sA,\phi}-\Lie(G^{\phi(\sA)}))+\wt_{w^{-1}(\sigma)}\mathsf s',
\end{align*}
then \eqref{deg bound 2} follows from the above degree bound. This proves our claim.

In the remainder of the proof, we need to show that $\mathrm{res}\circ\mathfrak{m}^\phi_{\fC}\circ \bPsi_K^{\phi,\mathsf s'}$ satisfies the axioms of $K$-theoretic stable envelopes. 

The image of $\fp:\fL^\phi_\fC \to \fX$ is a closed substack of $\fX$ and it contains 
$$\fp(\fq^{-1}(X^{\sA,\phi}))=\Attr_\fC(X^{\sA,\phi})$$ as dense substack. Then $\mathrm{res}\circ\mathfrak{m}^\phi_{\fC}\circ \bPsi_K^{\phi,\mathsf s'}$ is supported on $\overline{\Attr}_\fC(X^{\sA,\phi})\subseteq  \Attr^f_\fC(X^{\sA,\phi})$. This verifies the support axiom (Definition \ref{def of stab k}(i)). 
Moreover, we have $$\fp^{-1}(\Attr_\fC(X^{\sA,\phi}))=\fq^{-1}(X^{\sA,\phi}), $$ where $\fp$ induces an isomorphism $\fq^{-1}(X^{\sA,\phi})\cong \Attr_\fC(X^{\sA,\phi})$, and $\fq:\fq^{-1}(X^{\sA,\phi})\to X^{\sA,\phi}$ is given by the attraction map. In particular, we have the following Cartesian diagram
\begin{equation*}
\xymatrix{
X^{\sA,\phi} \ar@{^{(}->}[d]_{\mathbf k^{\phi}}  \ar@{}[dr]|{\Box} & \ar[l]_-{a} \Attr_\fC(X^{\sA,\phi}) \ar@{^{(}->}[r]^-{i} \ar@{^{(}->}[d]  \ar@{}[dr]|{\Box} & U \ar@{^{(}->}[d]^{\mathbf k\circ u} \\
\fX^{\sA,\phi} & \ar[l]_{\fq} \fL^{\phi}_\fC \ar[r]^{\fp} & \fX,
}
\end{equation*}
where $\mathbf k\colon X\hookrightarrow \fX$ and $\mathbf k^\phi\colon X^{\sA,\phi}\hookrightarrow \fX^{\sA,\phi}$ are the open immersions of the stable loci, $a$ is the attraction map,
\begin{align*}
u\colon U=X\setminus\left(\Attr^f_\fC(X^{\sA,\phi})\setminus \Attr_\fC(X^{\sA,\phi})\right)\hookrightarrow X
\end{align*}
is the natural open immersion, and $i$ is the natural closed immersion. It follows that 
\begin{align*}
    u^*\circ \mathrm{res}\circ\mathfrak{m}^\phi_{\fC}\circ \bPsi_K^{\phi,\mathsf s'}=i_*\circ a^*\circ \mathbf k^{\phi*}\circ \bPsi_K^{\phi,\mathsf s'}=i_*\circ a^*.
\end{align*}
Thus 
$$\mathrm{res}\circ\mathfrak{m}^\phi_{\fC}\circ \bPsi_K^{\phi,\mathsf s'}(\gamma)\big|_{X^{\sA,\phi}}=i_*\circ a^*(\gamma)\big|_{X^{\sA,\phi}}=e(N^-_{X^{\sA,\phi}/X})\cdot\gamma$$ 
for all $\gamma\in K^{\sT}(X^{\sA,\phi},\sw)$. This verifies the normalization axiom (Definition \ref{def of stab k}(ii)).

Finally, we verify the degree axiom (Definition \ref{def of stab k}(iii)) as follows. Let $\phi':\sA\to G$ be a homomorphism that is not equivalent to $\phi$. We assume that $X^{\sA,\phi'}$ is nonempty. Since $X^{\sA,\phi'}=X^{\sA,g\phi'g^{-1}}$ for any $g\in G$, we can assume that $\phi'(\sA)$ lies in $G^{\phi(\sA)}$ without loss of generality. The degree axiom can be restated as a strict inclusion of polytopes 
\begin{multline*}
\deg_\sA\mathrm{res}\circ\mathfrak{m}^\phi_{\fC}\circ \bPsi_K^{\phi,\mathsf s'}(\gamma)\big|_{X^{\sA,\phi'}}\subsetneq  \deg_\sA \hat e^\sT_K \left(N_{X^{\sA,\phi'} / X}^-\right)+\wt_\sA\phi'^{*}(\mathsf s)-\wt_\sA\phi^{*}(\mathsf s)
-\wt_\sA \det\left(N_{X^{\sA,\phi} / X}^-\right)^{1/2},
\end{multline*}
for all $\gamma\in K^{\sT/\sA}(X^{\sA,\phi},\sw)$.

Consider the natural morphism $\mathbf j^{\phi'}:[R^{\sA,\phi'}/G^{\phi'(\sA)}]\to \fX$, which induces an isomorphism between the stable locus of $[R^{\sA,\phi'}/G^{\phi'(\sA)}]$ and $X^{\sA,\phi'}$. Note that 
\begin{align*}
    \deg_\sA \hat e^\sT_K \left(N_{X^{\sA,\phi'} / X}^-\right)=\frac{1}{2}\deg_\sA \hat e^{\sT\times G^{\phi'(\sA)}}_K \left(N_{R^{\sA,\phi'} / R}-\Lie(G)\right),
\end{align*}
where $\sA$ acts on $N_{R^{\sA,\phi'} / R}-\Lie(G)$ via the homomorphism $\sA\xrightarrow{(\phi',\id)}G^{\phi'(\sA)}\times \sA$ and the latter $\sA$-action factors through the flavour group $\sT$. We also note that
\begin{align*}
\wt_\sA\phi^{*}(\mathsf s)+\wt_\sA \det(N_{X^{\sA,\phi} / X}^-)^{1/2}=\wt_\sA\phi^*(\mathsf s')+\wt_\sA \det\left(L^\phi_{\fC}-\Lie(P^\phi_\fC)\right)^{-1/2},
\end{align*}
where $\sA$ action on $\det\left(L^\phi_{\fC}-\Lie(P^\phi_\fC)\right)^{-1/2}$ factors through the flavour group $\sT$. In the following discussions we denote this character by $\det_{\mathrm{flav}}\left(L^\phi_{\fC}-\Lie(P^\phi_\fC)\right)^{-1/2}$. Then the degree axiom follows from the following condition:
\begin{equation}\label{deg bound 3}
\begin{split}
\lim \frac{\phi^*(\mathsf s')\cdot\det_{\mathrm{flav}}\left(L^\phi_{\fC}-\Lie(P^\phi_\fC)\right)^{-1/2}\cdot\mathbf j^{\phi'*}\circ\mathfrak{m}^\phi_{\fC}\circ\bPsi_K^{\phi,\mathsf s'}(\gamma)}{\phi'^*(\mathsf s)\cdot\sqrt{\hat e^{\sT\times G^{\phi'(\sA)}}_K\left(N_{R^{\sA,\phi'}/R}-\Lie(G)\right)}} \\
\text{exist in all directions of taking equivariant parameter $\mathsf a\to \infty$.}
\end{split}
\end{equation}
Consider the morphism $\widetilde{p}:\widetilde{L}^\phi_\fC\to R$ defined in Remark \ref{rmk:rewrite hall op_sym quot}. Let $\sA$ act on $\widetilde{L}^\phi_\fC$ and $R$ by the homomorphism $\sA\xrightarrow{(\phi',\id)}G^{\phi(\sA)}\times\sA$. The $\sA$-fixed locus, denoted by $\left(\widetilde{L}^\phi_\fC\right)^{\phi'}$, decomposes into connected components 
\begin{align*}
   \left(\widetilde{L}^\phi_\fC\right)^{\phi'} =\bigsqcup_{w\in W^{\phi'}\backslash W/W^\phi}\left(\widetilde{L}^\phi_\fC\right)^{\phi'}_w,\qquad \left(\widetilde{L}^\phi_\fC\right)^{\phi'}_w=G^{\phi'(\sA)}\times^{P^{\phi,\phi'}_{\fC,w}}\left(L^\phi_\fC\right)^{\sA, w^{-1}(\phi')},
\end{align*}
where $W^{\phi'},W,W^\phi$ are Weyl groups of $G^{\phi'(\sA)},G,G^{\phi(\sA)}$ respectively, $w^{-1}(\phi'):=w^{-1}\cdot \phi'\cdot w$ is a homomorphism from $\sA$ to $G^{\phi(\sA)}$, and 
$$P^{\phi,\phi'}_{\fC,w}=G^{\phi'(\sA)}\bigcap wP^\phi_{\fC}w^{-1}, $$ which acts on $\left(L^\phi_\fC\right)^{\sA, w^{-1}(\phi')}$ via $g\mapsto w^{-1}gw\curvearrowright \left(L^\phi_\fC\right)^{\sA, w^{-1}(\phi')}$. 

Denote the induced maps between $\sA$-fixed loci:
\begin{align*}
    \widetilde{p}^{\phi'}_{w}:\left(\widetilde{L}^\phi_\fC\right)^{\phi'}_w\to R^{\phi'},\quad q^{\phi'}_w\colon \left(L^\phi_\fC\right)^{\sA, w^{-1}(\phi')}\to \left(R^{\sA,\phi}\right)^{\sA, w^{-1}(\phi')},
\end{align*}
and the induced map between quotient stacks:
\begin{align*}
    \mathfrak q^{\phi'}_w\colon \left[\frac{\left(L^\phi_\fC\right)^{\sA, w^{-1}(\phi')}}{P^{\phi,{\phi'}}_{\fC,w}}\right]\longrightarrow \left[\frac{\left(R^{\sA,\phi}\right)^{\sA, w^{-1}(\phi')}}{(G^{\phi(\sA)})^{\sA, w^{-1}(\phi')}}\right],\quad \mathbf j^{\phi}_{\sA, w^{-1}(\phi')}\colon \left[\frac{\left(R^{\sA,\phi}\right)^{\sA, w^{-1}(\phi')}}{(G^{\phi(\sA)})^{w^{-1}(\phi'(\sA))}}\right]\longrightarrow \left[\frac{R^{\sA,\phi}}{G^{\phi(\sA)}}\right].
\end{align*}
Applying equivariant localization to the morphism $\widetilde{p}:\widetilde{L}^\phi_\fC\to R$ with respect to the $\sA$ action via $\sA\xrightarrow{(\phi',\id)}G^{\phi'(\sA)}\times \sA$, we have
\begin{align*}
&\quad \,\, \frac{{\mathbf j}^{\phi'*}\circ\mathfrak{m}^\phi_{\fC}\circ\bPsi_K^{\phi,\mathsf s'}(\gamma)}{e^{\sT\times G^{\phi'(\sA)}}_K\left(N_{R^{\sA,\phi'}/R}\right)}\\
&=\sum_{w\in W^{\phi'}\backslash W/W^\phi}\widetilde{p}^{\phi'}_{w*}\circ\mathrm{ind}^{G^{\phi'(\sA)}}_{P^{\phi,{\phi'}}_{\fC,w}} \circ\mathfrak q^{{\phi'}*}_w\left(\frac{ \mathbf j^{\phi*}_{\sA,w^{-1}({\phi'})}\circ\bPsi_K^{\phi,\mathsf s'}(\gamma)}{e^{\sT\times (G^{\phi(\sA)})^{w^{-1}(\phi'(\sA))}}_K\left((L^\phi_\fC-\Lie(P^{\phi,{\phi'}}_{\fC,w})+\Lie(G))^{\sA,w^{-1}(\phi')\text{-moving}}\right)}\right).
\end{align*}
Then to prove \eqref{deg bound 3}, it is enough to show that the limit of
\begin{equation*}
{\scriptstyle{
\xymatrix{
\frac{\phi^*(\mathsf s')\cdot\det_{\mathrm{flav}}\left(L^\phi_{\fC}-\Lie(P^\phi_\fC)\right)^{-1/2}\cdot\sqrt{\hat e^{\sT\times(G^{\phi(\sA)})^{w^{-1}(\phi'(\sA))}}_K\left(N_{R^{\sA,w^{-1}({\phi'})}/R}+\Lie(G)\right)}}{w^{-1}({\phi'})^*(\mathsf s)\cdot e^{\sT\times (G^{\phi(\sA)})^{w^{-1}(\phi'(\sA))}}_K\left((L^\phi_\fC-\Lie(P^{\phi,{\phi'}}_{\fC,w})+\Lie(G))^{\sA,w^{-1}({\phi'})\text{-moving}}\right)}
\times\mathbf j^{\phi*}_{\sA,w^{-1}({\phi'})}\circ\bPsi_K^{\phi,\mathsf s'}(\gamma)
}}}
\end{equation*}
exist for all $w\in W^{\phi'}\backslash W/W^\phi$ and all directions of $\mathsf a\to \infty$. We note that
\begin{multline*}
\deg_\sA e^{\sT\times (G^{\phi(\sA)})^{w^{-1}(\phi'(\sA))}}_K\left((L^\phi_\fC-\Lie(P^{\phi,{\phi'}}_{\fC,w}))^{\sA,w^{-1}({\phi'})\text{-moving}}\right)=\wt_\sA{\det}_{(w^{-1}(\phi'),\id)}\left(L^\phi_{\fC}-\Lie(P^\phi_\fC)\right)^{-1/2}\\
+\deg_\sA \hat e^{\sT\times (G^{\phi(\sA)})^{w^{-1}(\phi'(\sA))}}_K\left((L^\phi_\fC-\Lie(P^{\phi,{\phi'}}_{\fC,w}))^{\sA,w^{-1}({\phi'})\text{-moving}}\right),
\end{multline*}
where ${\det}_{(w^{-1}(\phi'),\id)}\left(L^\phi_{\fC}-\Lie(P^\phi_\fC)\right)^{-1/2}$ is the $\sA$-character induced by 
\begin{align*}
    \sA\xrightarrow{(\phi',\id)}G^{\phi(\sA)}\times \sA\curvearrowright {\det}\left(L^\phi_{\fC}-\Lie(P^\phi_\fC)\right)^{-1/2}.
\end{align*}
We have equation between $\sA$-characters $${\det}_{(w^{-1}(\phi'),\id)}\left(L^\phi_{\fC}-\Lie(P^\phi_\fC)\right)^{-1/2}={\det}_{w^{-1}(\phi')}\left(L^\phi_{\fC}-\Lie(P^\phi_\fC)\right)^{-1/2}\cdot {\det}_{\mathrm{flav}}\left(L^\phi_{\fC}-\Lie(P^\phi_\fC)\right)^{-1/2}.$$
Therefore it remains to show that 
\begin{equation}\label{deg bound 4}
\begin{split}
\lim \frac{\phi^*(\mathsf s')\cdot \mathbf j^{\phi*}_{\sA,w^{-1}({\phi'})}\circ\bPsi_K^{\phi,\mathsf s'}(\gamma)}{w^{-1}({\phi'})^*(\mathsf s')\cdot \sqrt{\hat e^{\sT\times (G^{\phi(\sA)})^{w^{-1}(\phi'(\sA))}}_K\left(R^{\sA,\phi}-\Lie(G^{\phi(\sA)})\right)}}\\
\text{exist for all $w\in W^{\phi'}\backslash W/W^\phi$ and all directions of $\mathsf a\to \infty$.}
\end{split}
\end{equation}
Note that the action 
$$\sA\xrightarrow{(w^{-1}(\phi'),\id)}G^{\phi(\sA)}\times \sA\curvearrowright R^{\sA,\phi}-\Lie(G^{\phi(\sA)})$$ factors through $\sA\xrightarrow{w^{-1}(\phi')/\phi}G^{\phi(\sA)}$, where $w^{-1}(\phi')/\phi$ is the multiplicative way of writing the difference between two homomorphisms $w^{-1}(\phi')$ and $\phi$. 
By Theorem \ref{thm: deg bound na stab k}, we have
\begin{align*}
    \deg_{\sA} \mathbf j^{\phi*}_{\sA,w^{-1}(\phi')}\bPsi^{\phi,\mathsf s'}_K(\gamma)\subsetneq  \frac{1}{2}\deg_{\sA} e^{\sT\times G^{\phi(\sA)}}_K(R^{\sA,\phi}-\Lie(G^{\phi(\sA)}))+\wt_{\sA}(w^{-1}(\phi')/\phi)^*(\mathsf s').
\end{align*}
Then \eqref{deg bound 4} follows from the above degree bound. This finishes the proof.
\end{proof}

\subsection{Explicit formulas of \texorpdfstring{$K$}{K}-theoretic stable envelopes}\label{sect on explicit form_k}

\begin{Corollary}\label{cor: explicit formula w=0_k}
Let $\mathsf s\in \mathrm{char}(G)\otimes_\bZ \bR$ be a generic slope and $\chi\in \mathrm{char}(G^{\phi(\sA)})$ such that $\mathsf s\otimes \det\left(T\fX^{\sA,w(\phi)\text{-repl}}\right)^{1/2}$ is in a sufficiently small neighbourhood of $\chi$. Assume $\sw=0$, then
\begin{align}\label{explicit formula w=0_k}
    \bPsi^{\mathsf s}_K\circ\Stab^{\mathsf s}_\fC([\mathcal L_\chi])=\sum_{w\in W/W^\phi}w\left(\chi\cdot e^{\sT\times G^{\phi(\sA)}}_K\left(T\fX^{\sA,\phi\text{-repl}}\right)\right) \cdot [\mathcal O_{\fX}]\:,
\end{align}
in particular,
\begin{align*}
    \Stab^{\mathsf s}_\fC([\mathcal L_\chi])=\sum_{w\in W/W^\phi}w\left(\chi\cdot e^{\sT\times G^{\phi(\sA)}}_K\left(T\fX^{\sA,\phi\text{-repl}}\right)\right)\cdot [\mathcal O_{X}]\:.
\end{align*}
Here $\mathcal L_\chi\in \Pic(X^{\sA,\phi})$ is the descent of the character $\chi$, $W$, $W^\phi$ are Weyl groups of $G$, $G^{\phi(\sA)}$ respectively,
$T\fX=R-\Lie(G)$ and $(-)^{\sA, \phi\text{-repl}}$ is the repelling part with respect to homomorphism $\phi\colon \sA\to G$, and $w$ acts on a weight $\mu$ of $G$ by $w(\mu)(g):=\mu(w^{-1}\cdot g\cdot w)$.
\end{Corollary}

\begin{proof}
Using Theorem \ref{thm: hall k_sym quot}, we see that 
$$\bPsi^{\phi,\mathsf s'}_K([\mathcal L_\chi])=\chi\otimes\mathcal O_{\fX^{\sA,\phi}}, $$ for the above choice of $\mathsf s'$. Then \eqref{explicit formula w=0_k} follows from Theorem \ref{thm: hall k_sym quot} and explicit formula of $K$-theoretic Hall operations in \cite[Prop.~3.4]{P}.
\end{proof}


\begin{Remark}
(1) When $X^{\sA,\phi}$ is an affine space, $[X^{\sA,\phi}]$ generates $H^\sT([X^{\sA,\phi}])$ over $H_\sT(\pt)$ and $[\mathcal L_\chi]$ generates $K^\sT([X^{\sA,\phi}])$ over $K_\sT(\pt)$. In this case, Corollaries \ref{cor: explicit formula w=0_coh}, \ref{cor: explicit formula w=0_k} completely determine the formulas of $\bPsi_H\circ\Stab_\fC$ and $\bPsi^{\mathsf s}_K\circ\Stab^{\mathsf s}_\fC$ respectively. (2) We also expect an explicit formula in the setting of elliptic cohomology. 
\end{Remark}

\section{Deformed dimensional reductions and stable envelopes}

In this section, we introduce dimensional reduction data and their compatibility (Definition \ref{compatible dim red data}, Proposition \ref{condition (ii')}). 
Compatible dimensional reduction data have natural interpolation maps between their critical cohomology and $K$-theory, which are always isomorphisms 
in the cohomological case (Proposition \ref{prop: coh compatible map is iso}), and are isomorphisms in the $K$-theoretic case when deformed dimensional reduction holds
(Proposition \ref{prop on def dim red}).

We show that stable envelopes are compatible with the interpolation maps (Theorems \ref{dim red and stab_coh}, \ref{dim red and stab_K}).
This provides flexible tools to compare critical theories and corresponding stable envelopes for different quiver models of the same critical loci. 
 
We provide examples of compatible dimensional reduction data in \S \ref{sect on cpt deform dim red}. In particular, this implies that 
(the critical) stable envelopes on tripled quivers with canonical cubic potentials reproduce stable envelopes of Nakajima quiver varieties.
More examples will be discussed in \cite{COZZ1}.

\subsection{Dimensional reduction data}\label{sect on dim red}

Consider the following situation. Let $Y$ be a smooth quasi-projective $\sT$-variety with a $\sT$-equivariant vector bundle $E$, and let $X:=\mathrm{Tot}_Y(E)$ with natural projection $\pi\colon X\to Y$. Assume that $\sw\colon X\to \bC$ is a function of the form $$\sw=\langle e,s\rangle+\pi^*(\phi),$$ where $e$ is the fiber coordinate of $E$, $s\in \Gamma(Y,E^\vee)$ is an 
$\sT$-invariant section, and $\phi\colon Y\to \bC$ is a $\sT$-invariant map. Note that the pairing $\langle e,s\rangle$ is $\sT$-invariant. We call the tuple $(X,Y,s,\phi)$ a \textit{dimensional reduction datum}. Let $Z(s)$ be the classical vanishing locus of $s$ in $Y$, and $Z^\der(s)$ be the derived vanishing locus of $s$ in $Y$. They fit into the following diagram 
$$
\xymatrix{
 \pi^{-1}(Z^{\der}(s))  \ar[r]^{\quad \quad \,\, \mathbf i^\der} \ar[d]^{\pi} \ar@{}[dr]|{\Box}  & X \ar[d]^{\pi} \ar@{}[dr]|{\Box}  &   \pi^{-1}(Z^{}(s))  \ar[l]_{ \mathbf i \quad \,\, }   \ar[d]^{\pi} \\
Z^{\der}(s) \ar[r]^{ } & Y &  Z^{}(s),  \ar[l]^{ } 
} 
$$
where $\mathbf i$ (resp.~$\mathbf i^\der$) is the closed embedding of $\pi^{-1}(Z(s))$ (resp.~$\pi^{-1}(Z^\der(s))$) into $X$.
\begin{Definition}\label{dim red data}
Let $(X,Y,s,\phi)$ be a dimensional reduction datum. We say that \textit{cohomological deformed dimensional reduction} holds for the tuple $(X,Y,s,\phi)$ if the composition
\begin{align}\label{deformed dim red_coh}
    \mathbf i_*\circ \pi^*\colon H^{\sT}(Z(s),\phi)\cong H^{\sT}(X,\sw)
\end{align}
is an isomorphism. We say that the $K$-\textit{theoretic deformed dimensional reduction} holds for the tuple $(X,Y,s,\phi)$ if the composition
\begin{align}\label{deformed dim red_K}
    \mathbf i^\der_*\circ \pi^*\colon K^{\sT}(Z^\der(s),\phi)\cong K^{\sT}(X,\sw)
\end{align}
is an isomorphism.
\end{Definition}

\begin{Definition}\label{compatible dim red data}
We say that dimensional reduction data $(X,Y,s,\phi)$ and $(X',Y',s',\phi')$ are \textit{compatible} if there exists
\begin{enumerate}
    \item a $\sT$-equivariant closed embedding $j\colon Y'\hookrightarrow Y$,
    \item a $\sT$-equivariant surjective vector bundle map $\pr\colon E|_{Y'}\twoheadrightarrow E'$,
\end{enumerate}
as in the following commutative diagram with $X'':=Y'\times_Y X=\mathrm{Tot}_{Y'}(E|_{Y'})$:
\begin{equation}\label{diag on xx' yy'}
\xymatrix{
 X'  \ar[dr]_{\pi'} &  & X'' \ar@{>>}[ll]_{\pr} \ar[dl] \ar@{^{(}->}[rr]^{\tilde j}  \ar@{}[dr]|{\Box} & & X \ar[dl]^{\pi}  \\
& Y' \ar@{^{(}->}[rr]^{j} & & Y &
} 
\end{equation}
such that
\begin{enumerate}[(i)]
     \item $\phi|_{Y'}=\phi'$, $s|_{Y'}=\pr^\vee\circ s' $. \footnote{This is equivalent to $\pr^*(\sw')=\sw|_{X''}$.}
    \item the natural map $Z^\der(s')\to Z^\der(s)$ induced by $\pr:j^* E\twoheadrightarrow E'$ is an isomorphism.
\end{enumerate}
Here we explain how the map in (ii) is induced. It is known that $$Z^\der(s)\cong \textbf{Spec}_Y \mathrm{Kos}(E,s), $$ 
for Koszul complex $$\mathrm{Kos}(E,s)=\left(\bigwedge^{\rk E}E\to\cdots\to \bigwedge^2E\to E\overset{s}{\to} \mathcal O_Y\right), $$ which is a dg-algebra over $\mathcal O_Y$. The map $Z^\der(s')\to Z^\der(s)$ is given by the dg-algebra map $$\mathrm{Kos}(E,s)\to j_*\mathrm{Kos}(E',s'), $$ 
which is induced by $\pr:j^* E\twoheadrightarrow E'$.

\end{Definition}
We explain how to re-state condition (ii) in Definition \ref{compatible dim red data} \textit{without} involving derived algebraic geometry. 
\begin{Lemma}
Suppose that $(X,Y,s,\phi)$ and $(X',Y',s',\phi')$ are two sets of dimensional reduction data, and assume that there exists maps (1) and (2) as in the Definition \ref{compatible dim red data} which satisfy condition (i), and assume moreover that $Z(s')$ is not empty, then there exists a natural vector bundle map on $Z(s')$: 
\begin{equation}\label{equ rho on vbm}\rho:\ker(\pr)|_{Z(s')}\to N^\vee_{Y'/Y}|_{Z(s')}. \end{equation}
\end{Lemma}
\begin{proof}
In fact, denote $\mathcal K:=\ker(\pr)$. For arbitrary $y\in Z(s')$, there exists an open neighbourhood $y\in U\subseteq  Y$ such that $\mathcal K|_{U\cap Y'}$ extends to a sub-bundle $\widetilde{\mathcal K}_U\subseteq  E|_U$. Note that $s|_U$ maps $\widetilde{\mathcal K}_U$ to $I_{Y'}|_{U}$, then we get
\begin{align*}
    \rho_{U\cap Z(s')}:\mathcal K|_{U\cap Z(s')}=\widetilde{\mathcal K}_U/I_{Z(s')}|_{U}\widetilde{\mathcal K}_U\longrightarrow (I_{Y'}/ I_{Z(s')}\cdot I_{Y'})|_{U\cap Z(s')}= N^\vee_{Y'/Y}|_{U\cap Z(s')}.
\end{align*}
The map $\rho_{U\cap Z(s')}$ does not depend on the choice of extension $\widetilde{\mathcal K}_U$ because any other choice is different from a given one by a map $\widetilde{\mathcal K}_U\to I_{Y'}|_U\cdot E|_U$ and after restricting to $Z(s')$, the image of this map is inside $I_{Z(s')}\cdot I_{Y'}|_{U\cap Z(s')}$. Consequently, for a pair of open subset $U$ and $V$, $\rho_{U\cap Z(s')}|_{U\cap V\cap Z(s')}=\rho_{V\cap Z(s')}|_{U\cap V\cap Z(s')}$. Therefore the family  of maps $\{\rho_{U\cap Z(s')}\}_U$ glues to a map 
\eqref{equ rho on vbm}. 

An equivalent construction of \eqref{equ rho on vbm} is given as follows. The map $\pr:j^* E\twoheadrightarrow E'$ induces two dg-algebra maps: 
$$f:\mathrm{Kos}(E,s)\to j_*\mathrm{Kos}(E',s'), \quad g:\mathbf Lj^*\mathrm{Kos}(E,s)\to \mathrm{Kos}(E',s'). $$ 
Applying $\mathbf Lj^*$ to the first map, and we get the following commutative diagram of dg-algebras on $Y'$:
\begin{equation}\label{cd Koszul complexes between compatible dim red pair}
\xymatrix{
 \mathbf Lj^*\mathrm{Kos}(E,s)   \ar[d]_{\mathbf Lj^*(f)} \ar[r]^{\,\,\, g}  & \mathrm{Kos}(E',s')  \ar@{=}[d]\\
\mathbf Lj^*j_*\mathrm{Kos}(E',s') \ar[r]^{\quad a} & \mathrm{Kos}(E',s') ,
} 
\end{equation}
where $a$ is the adjunction map $\mathbf Lj^*j_*\to \mathrm{id}$. Then we have induced map on the mapping cones 
$$\mathrm{Cone}(g)\to \mathrm{Cone}(a).$$ 
Let $i'\colon Z(s')\hookrightarrow Y'$ be the natural embedding, then 
$$H^{-1}(\mathbf Li'^*(\mathrm{Cone}(g)))\cong i'^*\mathcal K, \quad H^{-1}(\mathbf Li'^*(\mathrm{Cone}(a)))\cong i'^* N^\vee_{Y'/Y}, $$ 
and the induced map $i'^*\mathcal K\to i'^* N^\vee_{Y'/Y}$ equals to $\rho$ in \eqref{equ rho on vbm}.
\end{proof}

\begin{Proposition}\label{condition (ii')}
Suppose that $(X,Y,s,\phi)$ and $(X',Y',s',\phi')$ are two sets of dimensional reduction data, and assume that there exists maps (1) and (2) as in the Definition \ref{compatible dim red data} which satisfy condition (i). Then the condition (ii) is equivalent to
\begin{itemize}
    \item[(ii')] $Z(s)$ is set-theoretically contained in $Y'$, and $\rho:\ker(\pr)|_{Z(s')}\to N^\vee_{Y'/Y}|_{Z(s')}$ \eqref{equ rho on vbm} is an isomorphism.
\end{itemize}
\end{Proposition}

\begin{proof}
(ii)$\Longrightarrow$(ii'). Since $Z^\der(s')\cong Z^\der(s)$, $Z(s)$ is set-theoretically contained in $Y'$. Moreover 
$$\mathrm{Kos}(E,s)\to j_*\mathrm{Kos}(E',s')$$ being a quasi-isomorphism implies that the map $$\mathrm{Cone}(g)\to \mathrm{Cone}(a)$$ induced from the commutative diagram \eqref{cd Koszul complexes between compatible dim red pair} is a quasi-isomorphism. In particular, 
$$\rho:i'^*\ker(\pr)\cong H^{-1}(\mathbf Li'^*(\mathrm{Cone}(g)))\to H^{-1}(\mathbf Li'^*(\mathrm{Cone}(a)))\cong i'^* N^\vee_{Y'/Y}$$ is an isomorphism.

(ii')$\Longrightarrow$(ii). We only need to show $\mathrm{Kos}(E,s)\to j_*\mathrm{Kos}(E',s')$ is a quasi-isomorphism. Since $Z(s)\subseteq  Y'$ and 
$s|_{Y'}=\pr^\vee\circ s' $, we have $Z(s)=Z(s')$ as sets. Assume $Z(s')\neq\emptyset$ (otherwise (ii) trivially holds). To show that $\mathrm{Kos}(E,s)\to j_*\mathrm{Kos}(E',s')$ is a quasi-isomorphism, it is enough to check it locally in a Zariski neighbourhood of any point in $Y'$. To this end, we can assume that there is a decomposition of vector bundle 
$$E\cong \widetilde{\mathcal K}\oplus \widetilde{E}', \,\, \,\,    \mathrm{such}\,\, \mathrm{that}\,\, \, \widetilde{\mathcal K}|_{ Y'}\cong \ker(\pr), \,\,\, \widetilde{E}'|_{Y'}\cong E',$$ 
and $\pr$ equals to the natural projection $(\widetilde{\mathcal K}\oplus \widetilde{E}')|_{Y'}\to \widetilde{E}'|_{Y'}$. Let us write $$s=(s_1,s_2):\widetilde{\mathcal K}\oplus \widetilde{E}'\to \mathcal O_Y, $$ then the image of $s_1$ is contained in $I_{Y'}$ (the ideal that defines $Y'$). Then we consider the vector bundle map on $Y'$:
\begin{align*}
    \tau\colon \ker(\pr)=\widetilde{\mathcal K}/I_{Y'}\widetilde{\mathcal K}\longrightarrow I_{Y'}/I^2_{Y'}= N^\vee_{Y'/Y},
\end{align*}
which is defined as the image of $s_1$. By construction, $\tau|_{Z(s')}=\rho$ which is an isomorphism by assumption, so there exists an open neighbourhood $V$ of $Z(s')$ in $Y$ such that $\tau|_{V\cap Y'}$ is an isomorphism. Then it suffices to show that $\mathrm{Kos}(E,s)|_V\to j_*\mathrm{Kos}(E',s')|_V$ is a quasi-isomorphism. 
We have decomposition $$\mathrm{Kos}(E,s)|_V\cong \mathrm{Kos}(\widetilde{\mathcal K}_{V}, s_1|_V)\otimes \mathrm{Kos}(\widetilde E'_{V}, s_2|_V). $$ 
Since $\tau|_{V\cap Y'}$ is an isomorphism, (locally) $s_1$ maps a basis of $\widetilde{\mathcal K}_{V}$ to a regular sequence that generates $I_{Y'}|_V$; thus $\mathrm{Kos}(\widetilde{\mathcal K}_{V}, s_1|_V)$ is quasi-isomorphic to $j_*\mathcal O_{V\cap Y'}$ as a dg $\mathcal O_Y$-algebra. It follows that 
\begin{equation*}\mathrm{Kos}(E,s)|_V\cong j_*\mathbf Lj^*\mathrm{Kos}(\widetilde E'_{V}, s_2|_V)\cong j_*\mathrm{Kos}(E',s')|_V. \qedhere \end{equation*}
\end{proof}
Note that compatible dimensional reduction data gives a further consequence. 
\begin{Proposition}\label{compatible dim red crit loci}
Suppose that $(X,Y,s,\phi)$ and $(X',Y',s',\phi')$ are compatible dimensional reduction data, then $\Crit(\sw)$ is a subscheme of $X''$, and $\pr\colon X''\to X'$ maps $\Crit(\sw)$ isomorphically onto $\Crit(\sw')$.\,\footnote{In fact, the derived critical loci are also isomorphic. This can be proven using the $(-1)$-shifted version of \cite[Cor.~5.3]{KP}, and the fact that the derived zero locus of $s$ and $s'$ are isomorphic and the restriction of $\phi$ and $\phi'$ on them coincide. We thank Tasuki Kinjo to pointing out this. }
\end{Proposition}

\begin{proof}
We first notice that $$\Crit(\sw)\subseteq  \pi^{-1}(Z(s)).$$ Then by Definition \ref{compatible dim red data} (i), we have embedding $$\Crit(\sw)\subseteq  \pi^{-1}(Z(s'))\subseteq  X''.$$ 
Note that for any smooth closed subvariety $Z\subseteq  X$, there is an inclusion $$\Crit(\sw)\cap Z\subseteq  \Crit(\sw|_{Z}),$$ 
and taking $Z=X''$ leads to $$\Crit(\sw)\subseteq  \Crit(\sw|_{X''}).$$ 
By the condition (1) in Definition \ref{compatible dim red data}, we have $$\Crit(\sw|_{X''})=\pr^{-1}(\Crit(\sw')),$$ thus $\pr\colon X''\to X'$ restricts to give $\pr:\Crit(\sw)\to \Crit(\sw')$. It remains to show the map is surjective and any fiber contains exactly one point. Since we can check this locally, we replace $Y$ by an open subset and without loss of generality we assume $E\cong \widetilde{\mathcal K}\oplus \widetilde{E}'$ such that $\widetilde{\mathcal K}|_{Y'}\cong \ker(\pr)$ and $\widetilde{E}'|_{Y'}\cong E'$ and $\pr$ equals to the natural projection $(\widetilde{\mathcal K}\oplus \widetilde{E}')|_{Y'}\to \widetilde{E}'|_{Y'}$. 
Shrink $Y$ if necessary, we assume $$\widetilde{\mathcal K}^\vee\cong \bigoplus_{i=1}^n \mathcal O_Y\cdot k_i, \quad \widetilde{E}'^\vee\cong \bigoplus_{j=1}^m \mathcal O_Y\cdot e_j, $$ 
such that the morphism given by $\{t_i=s(k_i^{\vee})\}_{i=1}^n\colon Y\to \bA^n$ is smooth,
where $k_i^\vee$ is dual basis of $k_i$. 
We have $$\sw=\sum_{i=1}^n k_i \cdot t_i+\sum_{j=1}^m e_j\cdot s(e_j^\vee)+\phi, \quad \sw'=\sum_{j=1}^m e_j\cdot s(e_j^\vee)|_{Y'}+\phi|_{Y'}, $$ 
then it follows that
\begin{align*}
    \Crit(\sw)= \pr^{-1}(\Crit(\sw'))\bigcap\left\{k_i+\sum_{j=1}^me_j\cdot\frac{\partial s(e_j^\vee)}{\partial t_i}+\frac{\partial \phi}{\partial t_i}=0\:\bigg|\: 1\leqslant i \leqslant n\right\}.
\end{align*}
Since $k_i=-\sum_{j=1}^me_j\cdot\frac{\partial s(e_j^\vee)}{\partial t_i}-\frac{\partial \phi}{\partial t_i}$ $(1\leqslant i\leqslant n)$ gives a section of $\widetilde{\mathcal K}$ on $\mathrm{Tot}(\widetilde{E}')$, $\Crit(\sw)$ is section of the projection $\pr^{-1}(\Crit(\sw'))\to \Crit(\sw')$; therefore $\pr$ maps $\Crit(\sw)$ isomorphically onto $\Crit(\sw')$.
\end{proof}

We note that for a dimensional reduction datum $(X,Y,s,\phi)$, its $\sA$-fixed part $(X^\sA,Y^\sA,s^\sA,\phi^\sA)$ is also a dimensional reduction datum.

\begin{Lemma}\label{lem on cpt on afix}
Suppose that $(X,Y,s,\phi)$ and $(X',Y',s',\phi')$ are compatible dimensional reduction data, then the fixed locus data $(X^\sA,Y^\sA,s^\sA,\phi^\sA)$ and $(X'^{\sA},Y'^{\sA},s'^{\sA},\phi'^{\sA})$ are also compatible.
\end{Lemma}

\begin{proof}
Taking $\sA$-invariance, we get $\sT$-equivariant closed embedding $j^\sA\colon Y'^\sA\hookrightarrow Y^\sA$ and $\sT$-equivariant surjective vector bundle map 
$$\pr^\sA\colon E|_{Y'^\sA}^\mathrm{fix}\twoheadrightarrow E'|_{Y'^\sA}^{\mathrm{fix}}. $$ 
Then the condition (i) in Definition \ref{compatible dim red data} is obviously satisfied. By Proposition \ref{condition (ii')}, it is enough to verify condition (ii'). Since $s\in \Gamma(Y,E^\vee)^\sA$, we have $Z(s^\sA)=Z(s)^\sA$, so $Z(s^\sA)$ is set-theoretically contained in $Y'^\sA$. The natural map $\rho^\sA:\ker(\pr^\sA)|_{Z(s'^\sA)}\to N^\vee_{Y'^\sA/Y^\sA}|_{Z(s'^\sA)}$ is the $\sA$-fixed part of the pullback of $\rho:\ker(\pr)|_{Z(s')}\cong  N^\vee_{Y'/Y}|_{Z(s')}$ to $Z(s'^\sA)$, so $\rho^\sA$ is an isomorphism.
\end{proof}

\begin{Proposition}\label{prop:compatible map_fixed pt}
Suppose that $(X,Y,s,\phi)$ and $(X',Y',s',\phi')$ are compatible dimensional reduction data. Let $\delta_H$ and $\delta^\sA_H$ be the natural maps
\begin{align}\label{def of deltah}
    \delta_H:=\tilde j_*\circ \pr^*\colon  H^\sT(X',\sw')\to H^\sT(X,\sw),\quad \delta^\sA_H:=\tilde j^\sA_*\circ \pr^{\sA*}\colon  H^\sT(X'^\sA,\sw'^\sA)\to H^\sT(X^\sA,\sw^\sA),
\end{align}
where we refer to diagram \eqref{diag on xx' yy'} for those maps.
Then we have commutative diagram
$$
\xymatrix{
 H^{\sT}(X^\sA,\sw^\sA)  & & & H^{\sT}(X,\sw) \ar[lll]_{\quad i^*}, \\
H^{\sT}(X'^\sA,\sw'^\sA)   \ar[u]^{\delta^\sA_H } & & & H^{\sT}(X',\sw') \ar[u]_{\delta_H} \ar[lll]_{\quad e^\sT\left(N_{Y'/Y}|_{X'^\sA}^{\mathrm{mov}}\right)\cdot\, i'^*},
} 
$$
where $i'\colon X'^\sA\hookrightarrow X'$ and $i\colon X^\sA\hookrightarrow X$ are closed embeddings of fixed points loci. 

Similarly let $\delta_K$ and $\delta^\sA_K$ be the natural maps
\begin{align}\label{def of deltak}
    \delta_K:=\tilde j_*\circ \pr^*\colon  K^\sT(X',\sw')\to K^\sT(X,\sw),\quad \delta^\sA_K:=\tilde j^\sA_*\circ \pr^{\sA*}\colon  K^\sT(X'^\sA,\sw'^\sA)\to K^\sT(X^\sA,\sw^\sA),
\end{align}
then we have commutative diagram
$$
\xymatrix{
K^{\sT}(X^\sA,\sw^\sA)  & & & K^{\sT}(X,\sw) \ar[lll]_{\quad i^*}  \\
K^{\sT}(X'^\sA,\sw'^\sA)   \ar[u]^{\delta^\sA_K } & & & K^{\sT}(X',\sw') \ar[u]_{\delta_K} \ar[lll]_{\quad e^\sT_K \left(N_{Y'/Y}|_{X'^\sA}^{\mathrm{mov}}\right)\cdot \, i'^*}.
} 
$$
\end{Proposition}

\begin{proof}
Consider the following Cartesian diagram of smooth varieties:
\begin{equation*}
\xymatrix{
X''^\sA \ar[r]^{\tilde j^\sA} \ar[d]_{i''} \ar@{}[dr]|{\Box}  & X^\sA \ar[d]^{i} \\
X'' \ar[r]^{\tilde j} & X
}
\end{equation*}
According to Remark \ref{equiv excess_coh}, we have 
\begin{align*}
    i^*\circ \tilde j_*=\tilde j^\sA_*\circ e^\sT(N_{X''/X}|_{X''^\sA}^{\mathrm{mov}})\cdot i''^*.
\end{align*}
We note that $\pr\circ i''=i'\circ \pr^\sA$. Then we have
\begin{multline*}
i^*\circ \delta_H=i^*\circ \tilde j_*\circ \pr^*=\tilde j^\sA_*\circ e^\sT(N_{X''/X}|_{X''^\sA}^{\mathrm{mov}})\cdot i''^*\circ \pr^*= \tilde j^\sA_*\circ e^\sT(N_{X''/X}|_{X''^\sA}^{\mathrm{mov}})\cdot \pr^{\sA *}\circ\, i'^*\\
=\tilde j^\sA_*\circ \pr^{\sA *}\circ\, e^\sT(N_{Y'/Y}|_{X'^\sA}^{\mathrm{mov}})\cdot i'^*=\delta^\sA_H\circ e^\sT(N_{Y'/Y}|_{X'^\sA}^{\mathrm{mov}})\cdot i'^*.
\end{multline*}
This proves the statement for cohomology. For the $K$-theory counterpart, we use Remark \ref{equiv excess_K} to get $$i^*\circ \tilde j_*=\tilde j^\sA_*\circ e^\sT_K(N_{X''/X}|_{X''^\sA}^{\mathrm{mov}})\cdot i''^*,$$ and the rest of computation is the same as above.
\end{proof}

\begin{Proposition}\label{prop on def dim red}
Suppose that $(X,Y,s,\phi)$ and $(X',Y',s',\phi')$ are compatible dimensional reduction data.
If $K$-theoretic deformed dimensional reduction holds for both $(X,Y,s,\phi)$ and $(X',Y',s',\phi')$, then the natural map $\delta_K\colon  K^\sT(X',\sw')\to K^\sT(X,\sw)$ defined in Proposition \ref{prop:compatible map_fixed pt} is an isomorphism.
\end{Proposition}

\begin{proof}
Using $Z^{\der}(s)\cong Z^{\der}(s')$, we have Cartesian diagram:
\begin{equation*}
\xymatrix{
\pi^{-1}(Z^{\der}(s))\ar[r]^{\pr} \ar[d]_{\mathbf i''^{\der}} \ar@{}[dr]|{\Box}  & \pi'^{-1}(Z^{\der}(s')) \ar[d]^{\mathbf i'^{\der}}\\
X'' \ar[r]^{\pr} & X'. \\
}
\end{equation*}
Since $\pr$ is smooth, we have $$\mathbf i''^{\der}_*\circ \pr^*=\pr^*\circ \mathbf i'^{\der}_*.$$ 
It follows that
\begin{align*}
    \delta_K\circ(\mathbf i'^{\der}_*\circ \pi'^*)=\tilde j_*\circ \pr^*\circ \mathbf i'^{\der}_*\circ \pi'^*=\tilde j_*\circ \mathbf i''^{\der}_*\circ\pr^* \circ \pi'^*= \mathbf i^{\der}_*\circ \pi^*,
\end{align*}
where in the last equation we use $\tilde j\circ \mathbf i''^{\der}=\mathbf i^{\der}$ and $\pi'\circ \pr=\pi$. By assumption, both $\mathbf i'^{\der}_*\circ \pi'^*$ and $\mathbf i^{\der}_*\circ \pi^*$ are isomorphisms; thus $\delta_K$ is an isomorphism. 
\end{proof}

Proposition \ref{prop on def dim red} holds for cohomology in more general setting, 
where deformed dimensional reductions for $(X,Y,s,\phi)$ and $(X',Y',s',\phi')$ are not needed as assumptions. 


\begin{Proposition}\label{prop: coh compatible map is iso}
Suppose that $(X,Y,s,\phi)$ and $(X',Y',s',\phi')$ are compatible dimensional reduction data. Then the map $ \delta_H\colon  H^\sT(X',\sw')\to H^\sT(X,\sw)$ defined in Proposition \ref{prop:compatible map_fixed pt} is an isomorphism.\footnote{In view of the previous footnote, this isomorphism can also be proven from the gluing of DT perverse sheaves, as the condition in Definition \ref{compatible dim red data} implies the orientation data of the derived critical locus of $\sw$ and $\sw'$ are equivalent.}
\end{Proposition}

\begin{proof}
The map $ \delta_H\colon  H^\sT(X',\sw')\to H^\sT(X,\sw)$ is induced by the following chain of maps
\begin{align*}
    H^\sT_{*-2\rk\ker(\pr)}(X', \varphi_{\sw'}\omega_{X'})\xrightarrow{\cong} H^\sT_*(X'',\varphi_{\sw|_{X''}}\tilde j^!\omega_X)\xrightarrow{\cong} H^\sT_*(X,\tilde j_*\varphi_{\sw|_{X''}}\tilde j^!\omega_X)\longrightarrow H^\sT_*(X, \varphi_{\sw}\omega_X),
\end{align*}
where the first two isomorphisms are obvious, the last map is induced by the natural morphism in $\D^b_c(X)$:
\begin{align}\label{van cyc comp dim red}
    \tilde j_*\varphi_{\sw|_{X''}}\tilde j^!\omega_X\longrightarrow \varphi_{\sw}\omega_X.
\end{align}
To show that $\delta_H$ is an isomorphism, it remains to show the pushforward of \eqref{van cyc comp dim red} to $Y$:
\begin{equation}\label{equ on pitj}
    \pi_*\tilde j_*\varphi_{\sw|_{X''}}\tilde j^!\omega_X\longrightarrow \pi_*\varphi_{\sw}\omega_X
\end{equation}
is an isomorphism in $\D^b_c(Y)$. In the proof of Proposition \ref{compatible dim red crit loci}, we know that \'etale locally on $Y$, we have a decomposition 
and a basis:
$$E\cong \widetilde{\cK}\oplus\widetilde{E}', \quad \widetilde{\mathcal K}^\vee\cong \bigoplus_{i=1}^n \mathcal O_Y\cdot k_i, \quad \widetilde{E}'^\vee\cong \bigoplus_{j=1}^m \mathcal O_Y\cdot e_j,$$ 
such that the morphism given by $\{t_i=s(k_i^{\vee})\}_{i=1}^n\colon Y\to \bA^n$ is smooth, where $k_i^\vee$ is dual basis of $k_i$. In this coordinate, 
$$\sw=\sum_{i=1}^n k_i \cdot t_i+\sum_{j=1}^m e_j\cdot s(e_j^\vee)+\phi.$$ 
Let $\pi_1\colon \mathrm{Tot}(E)\to \mathrm{Tot}(\widetilde{E}')$ be the projection whose fibers contribute to the function $\sw$ by 
$\sum_{i=1}^n k_i \cdot t_i$, where $\bigcap_{i=1}^n\{t_i=0\}$ cuts out $Y'\subseteq  Y$ in an \'etale neighbourhood.
Then the deformed dimensional reduction for regular sections (Theorem \ref{thm: def dim red reg sec}) implies that the natural map
\begin{align*}
    \pi_{1*}\tilde j_*\varphi_{\sw|_{X''}}\tilde j^!\omega_X\longrightarrow \pi_{1*}\varphi_{\sw}\omega_X
\end{align*}
is an isomorphism. Pushing forward to $Y$, we obtain \eqref{equ on pitj}.
\end{proof}

\subsection{Compatibility with stable envelopes}
Suppose that $(X,Y,s,\phi)$ and $(X',Y',s',\phi')$ are compatible dimensional reduction data as in the Definition \ref{compatible dim red data}. 
As $X\to Y$ and $X'\to Y'$ are vector bundles projections equivariant under $\sA$, 
there are one-to-one correspondences $$\mathrm{Fix}_\sA(X)\cong \mathrm{Fix}_\sA(Y), \quad \mathrm{Fix}_\sA(X')\cong \mathrm{Fix}_\sA(Y'), $$ 
with isomorphisms given by projections and $\sA$-fixed parts of inverse maps of projections.

The embedding $Y'\hookrightarrow Y$ induces a map 
$$\mathfrak{r}:\mathrm{Fix}_\sA(X')\cong \mathrm{Fix}_\sA(Y')\longrightarrow \mathrm{Fix}_\sA(Y) \cong \mathrm{Fix}_\sA(X),$$ 
which is neither injective nor surjective in general. The map $\mathfrak{r}$ respects the ample partial order $\le$ defined in the Remark \ref{partial order by line bundle}. Namely, $\forall \, F,F'\in \mathrm{Fix}_\sA(X')$,
\begin{align}\label{preserve ample partial order}
    F<F' \Longleftrightarrow \mathfrak{r}(F)<\mathfrak{r}(F').
\end{align}

\begin{Theorem}\label{dim red and stab_coh}
Suppose that $(X,Y,s,\phi)$ and $(X',Y',s',\phi')$ are compatible dimensional reduction data as in the Definition \ref{compatible dim red data}. Fix a normalizer 
$\epsilon\in \{\pm 1\}^{\Fix_\sA(X)}$, and define $\epsilon'\in \{\pm 1\}^{\Fix_\sA(X')}$ by $$\epsilon'_{F}:=\epsilon_{\mathfrak{r}(F)}\cdot (-1)^{\rk N_{Y'/Y}|_F^+}.$$
Fix a chamber $\fC\subseteq  \Lie(\sA)_\bR$, and we assume that cohomological stable envelopes exist for $(X,\sw,\sT,\sA,\fC )$ and $(X',\sw',\sT,\sA,\fC)$. Then we have the following commutative diagram:
\begin{equation}\label{equ on dim red and stab_coh}
\xymatrix{
 H^{\sT}(X^\sA,\sw^\sA) \ar[rr]^{\Stab_{\fC,\epsilon}} & & H^{\sT}( X,\sw) \\
H^{\sT}(X'^\sA,\sw'^\sA)  \ar[rr]^{\Stab_{\fC, \epsilon'}} \ar[u]^{\delta^\sA_H } & & H^{\sT}(X',\sw') \ar[u]_{\delta_H} .
} 
\end{equation}
\end{Theorem}

\begin{proof}
Define $$\mathcal S=\delta_H\circ \Stab_{\fC , \epsilon'}\colon H^{\sT}(X'^\sA,\sw'^\sA) \to H^{\sT}(X,\sw).$$ 
For an arbitrary $F\in \Fix_\sA(X')$, and take an arbitrary $\gamma\in H^{\sT/\sA}(F,\sw'^\sA)$, we need to show that $\mathcal S(\gamma)$ satisfies the axioms (ii) and (iii) in Definition \ref{def of stab coho} and axiom (i') in Remark \ref{weak axiom coh}, namely
\begin{itemize}
\setlength{\parskip}{1ex}
    \item[(1)] $\mathcal S(\gamma)$ is supported on $\Attr^{\le}_\fC(\mathfrak{r}(F))$;
    \item[(2)] $\mathcal S (\gamma)\big|_{\mathfrak{r}(F)} = \epsilon_{\mathfrak{r}(F)}\cdot e^\sT (N_{\mathfrak{r}(F) / X}^-) \cdot \delta^\sA_H(\gamma)$;
    \item[(3)] For any $F'\neq\mathfrak{r}(F)$, the inequality $\deg_\sA \mathcal S(\gamma)\big|_{F'} < \deg_\sA e^\sT (N_{F' / X}^-)$ holds.
\end{itemize}
Here we choose the ample partial order $\le$ as in Remark \ref{partial order by line bundle}.

Denote $\beta=\Stab_{\fC, \epsilon'}(\gamma)$, then $\beta$ is supported on $\Attr^{\le}_\fC(F)$. It follows that $\delta_H(\beta)=\tilde j_*\circ\pr^*(\beta)$ is supported on $\pr^{-1}(\Attr^{\le}_\fC(F))$. Since the complex $\varphi_\sw\omega_X$ is supported on $\Crit(\sw)$, $\delta_H(\beta)$ is supported on $\pr^{-1}(\Attr^{\le}_\fC(F))\cap \Crit(\sw)$. According to Proposition \ref{compatible dim red crit loci}, we have 
\begin{align*}
    \pr^{-1}(\Attr^{\le}_\fC(F))\cap \Crit(\sw)=(\pr|_{\Crit(\sw)})^{-1}(\Attr^{\le}_\fC(F)\cap \Crit(\sw')).
\end{align*}
By Lemma \ref{refined partial order and closed subset}, we have 
$$\Attr^{\le}_\fC(F)\cap \Crit(\sw')=\bigcup_{F\le F'}\Attr_\fC(F'\cap \Crit(\sw'))_{\Crit(\sw')}, $$ 
where the subscript means taking attracting set inside $\Crit(\sw')$. Therefore $\delta_H(\beta)$ is supported on
\begin{align*}
\bigcup_{F\le F'}\left(\pr|_{\Crit(\sw)})^{-1}(\Attr_\fC(F'\cap \Crit(\sw'))_{\Crit(\sw')}\right)&=\bigcup_{F\le F'}\Attr_\fC(\mathfrak{r}(F')\cap \Crit(\sw))_{\Crit(\sw)}\\
\text{\tiny by \eqref{preserve ample partial order}}\quad &=\bigcup_{\mathfrak{r}(F)\le \mathfrak{r}(F')}\Attr_\fC(\mathfrak{r}(F')\cap \Crit(\sw))_{\Crit(\sw)}\\
\text{\tiny since $F''\cap \Crit(\sw)\neq \emptyset\Longleftrightarrow F''=\mathfrak{r}(F')$ for some $F'$}\quad &=\bigcup_{\mathfrak{r}(F)\le F''}\Attr_\fC(F''\cap \Crit(\sw))_{\Crit(\sw)}\\
\text{\tiny by Lemma \ref{refined partial order and closed subset}}\quad &=\Attr^{\le}_\fC(\mathfrak{r}(F))\cap \Crit(\sw).
\end{align*}
In particular, $\delta_H(\beta)$ is supported on $\Attr^{\le}_\fC(\mathfrak{r}(F))$. This proves (1).

To show (2) and (3), we notice that
\begin{align*}
    \mathcal S (\gamma)\big|_{X^\sA} =i^*\circ \delta_H (\beta)=\delta^\sA_H\circ e^\sT\left(N_{Y'/Y}|_{X'^\sA}^{\mathrm{mov}}\right)\cdot i'^*(\beta)=\delta^\sA_H\left( e^\sT\left(N_{Y'/Y}|_{X'^\sA}^{\mathrm{mov}}\right)\cdot \Stab_{\fC, \epsilon'}(\gamma)|_{X'^\sA}\right),
\end{align*}
where the second equality follows from Proposition \ref{prop:compatible map_fixed pt}. Then
\begin{align*}
    \mathcal S (\gamma)\big|_{\mathfrak{r}(F)}=\delta^\sA_H\left( e^\sT\left(N_{Y'/Y}|_{F}^{\mathrm{mov}}\right)\cdot \Stab_{\fC, \epsilon'}(\gamma)|_{F}\right)=\delta^\sA_H\left( \epsilon'_F\cdot e^\sT\left(N_{Y'/Y}|_{F}^{\mathrm{mov}}\right)\cdot e^\sT(N_{F/X'}^-)\cdot\gamma\right).
\end{align*}
According to Proposition \ref{condition (ii')}, we have 
$$N_{Y'/Y}|_{Z(s')^\sA}^{+}=\left(\ker(\pr)|_{Z(s')^\sA}^{-}\right)^\vee \in K^\sT\left(Z(s')^\sA\right). $$ 
Via pullback along $\pi'^{\sA}\colon X'^{\sA}\to Y'^{\sA}$, we have 
\begin{equation}\label{equ on ny'y}e^\sT\left(\pi'^{\sA*}N_{Y'/Y}|_{Z(s')^\sA}^{+}\right)=(-1)^{\rk N_{Y'/Y}|_{X'^\sA}^{+}}e^\sT\left(\pi'^{\sA*}\ker(\pr)|_{Z(s')^\sA}^{-}\right)\in K^\sT\left((\pi'^{\sA})^{-1}(Z(s')^\sA)\right). \end{equation}
As we have inclusion
$$ \Crit(\sw'^{\sA})\subseteq  (\pi'^{\sA})^{-1}(Z(s')^\sA), $$
we can restrict \eqref{equ on ny'y} to $\Crit(\sw'^{\sA})$ and the restriction 
\begin{equation}\label{equ on ny'y2}e^\sT\left(\pi'^{\sA*}N_{Y'/Y}|_{Z(s')^\sA}^{+}\right)\Big|_{\Crit(\sw'^{\sA})}=(-1)^{\rk N_{Y'/Y}|_{X'^\sA}^{+}}e^\sT\left(\pi'^{\sA*}\ker(\pr)|_{Z(s')^\sA}^{-}\right)\Big|_{\Crit(\sw'^{\sA})} \end{equation}
acts on $H^\sT(X'^\sA,\sw'^\sA)$. 
Note that 
\begin{equation}\label{equ on ny'y3} \left(\pi'^{\sA*}N_{Y'/Y}|_{Z(s')^\sA}^{+}\right)\Big|_{\Crit(\sw'^{\sA})}\cong \left(\Crit(\sw'^{\sA}) \hookrightarrow X'^{\sA} \right)^*\left(\pi'^{\sA*}N_{Y'/Y}|_{Y'^{\sA}}^{+} \right), \end{equation}
\begin{equation}\label{equ on ny'y4} \left(\pi'^{\sA*}\ker(\pr)|_{Z(s')^\sA}^{-}\right)\Big|_{\Crit(\sw'^{\sA})}\cong \left(\Crit(\sw'^{\sA}) \hookrightarrow X'^{\sA} \right)^*\left(\pi'^{\sA*}\ker(\pr)|_{Z(s')^\sA}^{-}\right). \end{equation}
For simplicity, we write   
$$N_{Y'/Y}|_{X'^\sA}^{+}=\left(\pi'^{\sA*}N_{Y'/Y}|_{Y'^{\sA}}^{+} \right), \quad \ker(\pr)|_{X'^\sA}^{-}=\left(\pi'^{\sA*}\ker(\pr)|_{Z(s')^\sA}^{-}\right) $$
and omit the pullback along $\Crit(\sw'^{\sA}) \hookrightarrow X'^{\sA}$ in \eqref{equ on ny'y3}, \eqref{equ on ny'y4}. Then \eqref{equ on ny'y2} becomes
\begin{equation}\label{equ on enyy}e^\sT\left(N_{Y'/Y}|_{X'^\sA}^{+}\right)=(-1)^{\rk N_{Y'/Y}|_{X'^\sA}^{+}}e^\sT\left(\ker(\pr)|_{X'^\sA}^{-}\right), \end{equation}
which acts on $H^\sT(X'^\sA,\sw'^\sA)$. Consequently, we have
\begin{align*}
    \mathcal S (\gamma)\big|_{\mathfrak{r}(F)}&=\delta^\sA_H\left( (-1)^{\rk N_{Y'/Y}|_{F}^+}\cdot\epsilon'_F\cdot e^\sT\left(\ker(\pr)|_{F}^{-}+N_{Y'/Y}|_{F}^{-}\right)\cdot e^\sT(N_{F/X'}^-)\cdot\gamma\right)\\
    &=\epsilon_{\mathfrak{r}(F)}\cdot e^\sT(N_{\mathfrak{r}(F)/X}^-)\cdot \delta^\sA_H(\gamma).
\end{align*}
This proves (2). 

Finally, for any $F'\neq\mathfrak{r}(F)$, 
\begin{align*}
\mathcal S (\gamma)\big|_{F'} =\delta^\sA_H\left( e^\sT\left(N_{Y'/Y}|_{\mathfrak{r}^{-1}(F')}^{\mathrm{mov}}\right)\cdot \Stab_{\fC, \epsilon'}(\gamma)|_{\mathfrak{r}^{-1}(F')}\right).
\end{align*}
Suppose that $\mathfrak{r}^{-1}(F')=\{F_1,\cdots,F_n\}$, then $F_i\neq F$ for all $i$. The same argument as above shows that
\begin{align*}
e^\sT\left(N_{Y'/Y}|_{F_i}^{\mathrm{mov}}\right)\cdot \Stab_{\fC, \epsilon'}(\gamma)|_{F_i}=\pm e^\sT\left(\ker(\pr)|_{F_i}^{-}+N_{Y'/Y}|_{F_i}^{-}\right)\cdot\Stab_{\fC, \epsilon'}(\gamma)|_{F_i}
\end{align*}
Since $\delta^\sA_H$ preserves $\deg_\sA$, we have
\begin{align*}
    \deg_\sA \delta^\sA_H\left(e^\sT\left(N_{Y'/Y}|_{F_i}^{\mathrm{mov}}\right)\cdot \Stab_{\fC, \epsilon'}(\gamma)|_{F_i}\right)&=\deg_\sA e^\sT\left(\ker(\pr)|_{F_i}^{-}+N_{Y'/Y}|_{F_i}^{-}\right)+\deg_\sA \Stab_{\fC, \epsilon'}(\gamma)|_{F_i}\\
    & <\deg_\sA e^\sT\left(\ker(\pr)|_{F_i}^{-}+N_{Y'/Y}|_{F_i}^{-}\right)+\deg_\sA e^\sT(N_{F_i/X'}^-)\\
    &=\deg_\sA e^\sT(N_{F'/X}^-),
\end{align*}
where the second inequality follows from the axiom (iii). Therefore
\begin{align*}
    \deg_\sA\mathcal S (\gamma)\big|_{F'} \leqslant \max_{1\leqslant i\leqslant n}\deg_\sA\delta^\sA_H\left(e^\sT\left(N_{Y'/Y}|_{F_i}^{\mathrm{mov}}\right)\cdot \Stab_{\fC, \epsilon'}(\gamma)|_{F_i}\right)<\deg_\sA e^\sT(N_{F'/X}^-).  
\end{align*}
This proves (3).
\end{proof}

\begin{Theorem}\label{dim red and stab_K}
Suppose that $(X,Y,s,\phi)$ and $(X',Y',s',\phi')$ are compatible dimensional reduction data as in the Definition \ref{compatible dim red data}. Fix a normalizer $\mathcal E\in \pm \Pic_\sT(X^\sA)$, and define $\mathcal E'\in \pm \Pic_\sT(X'^\sA)$ by
$$\mathcal E'|_F:=\mathcal E|_F\otimes (-1)^{\rk N_{Y'/Y}|_F^+}\det(N_{Y'/Y}|_F^+).$$
Fix a generic slope $\mathsf s\in \Pic_\sA(Y)\otimes_\bZ \bR$, and define $\mathsf s'\in \Pic_\sA(Y')\otimes_\bZ \bR$ by
$$\mathsf s':=\mathsf s|_{Y'}\otimes \det(N_{Y'/Y})^{1/2}.$$
Fix a chamber $\fC\subseteq  \Lie(\sA)_\bR$, and assume that $K$-theoretic stable envelopes exist for $(X,\sw,\sT,\sA,\fC ,\mathsf s)$ and $(X',\sw',\sT,\sA,\fC,\mathsf s')$. 

Assume moreover that $\ker(\pr)=N^\vee_{Y'/Y}$ in $K^{\sT}(Y')$.
Then we have the following commutative diagram:
\begin{equation}\label{equ on dim red and stab_k}
\xymatrix{
  K^{\sT}(X^\sA,\sw^\sA) \ar[rr]^{\Stab^{\mathsf s}_{\fC,\mathcal E}} & & K^{\sT}( X,\sw) \\
K^{\sT}(X'^\sA,\sw'^\sA)  \ar[rr]^{\Stab^{\mathsf s'}_{\fC, \mathcal E'}} \ar[u]^{\delta^\sA_K } & & K^{\sT}(X',\sw') \ar[u]_{\delta_K}.
} 
\end{equation}
\end{Theorem}

\begin{proof}
Define $$\mathcal S=\delta_K\circ \Stab^{\mathsf s'}_{\fC, \mathcal E'}\colon K^{\sT}(X'^\sA,\sw'^\sA) \to K^{\sT}(X,\sw).$$ 
For an arbitrary $F\in \Fix_\sA(X')$, and take an arbitrary $\gamma\in K^{\sT/\sA}(F,\sw'^\sA)$, we need to show that $\mathcal S(\gamma)$ satisfies the axioms (ii) and (iii) in Definition \ref{def of stab k} and axiom (i') in Remark \ref{weak axiom K}, namely
\begin{itemize}
\setlength{\parskip}{1ex}
    \item[(1)] $\mathcal S(\gamma)$ is supported on $\Attr^{\le}_\fC(\mathfrak{r}(F))$;
    \item[(2)] $\mathcal S (\gamma)\big|_{\mathfrak{r}(F)} = \mathcal E|_{\mathfrak{r}(F)}\cdot\ e^\sT_K (N_{\mathfrak{r}(F) / X}^-) \cdot \delta^\sA_K(\gamma)$;
    \item[(3)] For any $F'\neq \mathfrak{r}(F)$, there is a strict inclusion of polytopes
    \begin{align*}
        \deg_\sA \mathcal S(\gamma)\big|_{F'} \subsetneq  \deg_\sA e^\sT_K (N_{F' / X}^-)&+\mathrm{weight}_\sA\left(\det(N_{F'/X}^-)^{1/2}\otimes\mathsf s|_{F'}\right)-\mathrm{weight}_\sA\left(\det(N_{\mathfrak{r}(F)/X}^-)^{1/2}\otimes\mathsf s|_{\mathfrak{r}(F)}\right)\\
        &+\mathrm{weight}_\sA \mathcal E|_{\mathfrak{r}(F)}.
    \end{align*}
\end{itemize}
Here we choose the ample partial order $\le$ as in Remark \ref{partial order by line bundle}.

Denote $\beta= \Stab^{\mathsf s'}_{\fC, \mathcal E'}(\gamma)$, then $\beta$ is supported on $\Attr^{\le}_\fC(F)$. It follows that $\delta_K(\beta)=\tilde j_*\circ\pr^*(\beta)$ is supported on $\pr^{-1}(\Attr^{\le}_\fC(F))$. Since the critical $K$-theory is supported on $\Crit(\sw)$, $\delta_K(\beta)$ is supported on $\pr^{-1}(\Attr^{\le}_\fC(F))\cap \Crit(\sw)$. The same argument as in the proof of Theorem \ref{dim red and stab_coh} shows that
\begin{align*}
    \pr^{-1}(\Attr^{\le}_\fC(F))\cap \Crit(\sw)=\Attr^{\le}_\fC(\mathfrak{r}(F))\cap \Crit(\sw).
\end{align*}
Thus $\delta_K(\beta)$ is supported on $\Attr^{\le}_\fC(\mathfrak{r}(F))$. This proves (1).

To show (2) and (3), we notice that
\begin{align*}
    \mathcal S (\gamma)\big|_{X^\sA} =i^*\circ \delta_K (\beta)=\delta^\sA_K\circ e^\sT_K\left(N_{Y'/Y}|_{X'^\sA}^{\mathrm{mov}}\right)\cdot i'^*(\beta)=\delta^\sA_K\left( e^\sT_K\left(N_{Y'/Y}|_{X'^\sA}^{\mathrm{mov}}\right)\cdot  \Stab^{\mathsf s'}_{\fC, \mathcal E'}(\gamma)|_{X'^\sA}\right),
\end{align*}
where the second equality follows from Proposition \ref{prop:compatible map_fixed pt}. Then
\begin{align*}
    \mathcal S (\gamma)\big|_{\mathfrak{r}(F)}&=\delta^\sA_K\left( e^\sT_K\left(N_{Y'/Y}|_{F}^{\mathrm{mov}}\right)\cdot \Stab^{\mathsf s'}_{\fC, \mathcal E'}(\gamma)|_{F}\right)\\
    &=\delta^\sA_K\left( \mathcal E'|_{F}\cdot  e^\sT_K\left(N_{Y'/Y}|_{F}^{\mathrm{mov}}\right)\cdot e^\sT_K(N_{F/X'}^-)\cdot\gamma\right).
\end{align*}
Similarly as \eqref{equ on enyy}, we have 
$$e^\sT_K(N_{Y'/Y}|_{X'^\sA}^{+})=(-1)^{\rk N_{Y'/Y}|_{X'^\sA}^{+}}\det(N_{Y'/Y}|_{X'^\sA}^{+})^{-1} e^\sT_K(\ker(\pr)|_{X'^\sA}^{-}).$$
Therefore
\begin{align*}
    \mathcal S (\gamma)\big|_{\mathfrak{r}(F)}&=\delta^\sA_K\left((-1)^{\rk N_{Y'/Y}|_{F}^+} \det(N_{Y'/Y}|_{F}^+)^{-1}\mathcal E'|_{F}\cdot e^\sT_K\left(\ker(\pr)|_{F}^{-}+N_{Y'/Y}|_{F}^{-}\right)\cdot e^\sT_K(N_{F/X'}^-)\cdot\gamma\right)\\
    &=\mathcal E|_{\mathfrak{r}(F)}\cdot e^\sT_K(N_{\mathfrak{r}(F)/X}^-)\cdot \delta^\sA_K(\gamma).
\end{align*}
This proves (2). 

Finally, for any $F'\neq\mathfrak{r}(F)$, 
\begin{align*}
\mathcal S (\gamma)\big|_{F'} =\delta^\sA_K\left( e^\sT_K\left(N_{Y'/Y}|_{\mathfrak{r}^{-1}(F')}^{\mathrm{mov}}\right)\cdot \Stab^{\mathsf s'}_{\fC, \mathcal E'}(\gamma)|_{\mathfrak{r}^{-1}(F')}\right).
\end{align*}
Suppose that $\mathfrak{r}^{-1}(F')=\{F_1,\cdots,F_n\}$, then $F_i\neq F$ for all $i$. The same argument as above shows that
\begin{align*}
e^\sT_K\left(N_{Y'/Y}|_{F_i}^{\mathrm{mov}}\right)\cdot \Stab^{\mathsf s'}_{\fC, \mathcal E'}(\gamma)|_{F_i}=\pm \det(N_{Y'/Y}|_{F_i}^{+})^{-1}e^\sT_K\left(\ker(\pr)|_{F_i}^{-}+N_{Y'/Y}|_{F_i}^{-}\right)\cdot\Stab^{\mathsf s'}_{\fC, \mathcal E'}(\gamma)|_{F_i}.
\end{align*}
Since $\delta^\sA_K$ preserves $\deg_\sA$, we have
\begin{align*}
\deg_\sA \delta^\sA_K&\left(e^\sT_K\left(N_{Y'/Y}|_{F_i}^{\mathrm{mov}}\right)\cdot \Stab^{\mathsf s'}_{\fC, \mathcal E'}(\gamma)|_{F_i}\right)\\
&\subseteq  \text{Convex hull}\left[\deg_\sA\left(\det(N_{Y'/Y}|_{F_i}^{+})^{-1}e^\sT_K\left(\ker(\pr)|_{F_i}^{-}+N_{Y'/Y}|_{F_i}^{-}\right)\right) +\deg_\sA e^\sT_K(N^-_{F_i/X'}) \right]\\
&\qquad +\wt_\sA(\det(N_{F_i/X'}^-)^{1/2}\otimes\mathsf s'|_{F_i})-\wt_\sA(\det(N_{F/X'}^-)^{1/2}\otimes\mathsf s'|_{F})+\wt_\sA\mathcal E'|_{F}\\
&=\deg_\sA e^\sT_K(N_{F'/X}^-)+\mathrm{weight}_\sA\left(\det(N_{F'/X}^-)^{1/2}\otimes\mathsf s|_{F'}\right)-\mathrm{weight}_\sA\left(\det(N_{\mathfrak{r}(F)/X}^-)^{1/2}\otimes\mathsf s|_{\mathfrak{r}(F)}\right)\\
&\qquad +\mathrm{weight}_\sA \mathcal E|_{\mathfrak{r}(F)}.
\end{align*}
Therefore
\begin{align*}
    &\deg_\sA\mathcal S (\gamma)\big|_{F'} \subseteq  \text{Convex hull}\left[\sum_{i=1}^n\deg_\sA \delta^\sA_K\left(e^\sT_K\left(N_{Y'/Y}|_{F_i}^{\mathrm{mov}}\right)\cdot \Stab^{\mathsf s'}_{\fC, \mathcal E'}(\gamma)|_{F_i}\right)\right]\\
    &\subseteq  \deg_\sA e^\sT_K(N_{F'/X}^-)+\mathrm{weight}_\sA\left(\det(N_{F'/X}^-)^{1/2}\otimes\mathsf s|_{F'}\right)-\mathrm{weight}_\sA\left(\det(N_{\mathfrak{r}(F)/X}^-)^{1/2}\otimes\mathsf s|_{\mathfrak{r}(F)}\right)\\
    &\qquad+\mathrm{weight}_\sA \mathcal E|_{\mathfrak{r}(F)}.  
\end{align*}
Since $\mathsf s$ is generic, the above inclusion must be strict by Remark \ref{rmk: generic slope}. This proves (3).
\end{proof}

\subsection{Examples of compatible dimensional reduction data}\label{sect on cpt deform dim red}

\begin{Example}\label{ex: dim red}
Suppose that $(X,Y,s,\phi)$ is a dimensional reduction datum such that $s$ is regular.
Then $Z(s)$ is a smooth subvariety in $Y$ with codimension equals to $\rk E$, and $Z(s)\cong Z^\der(s)$. It is easy to see that $(X,Y,s,\phi)$ and 
$\left(Z(s),Z(s),0,\phi|_{Z(s)}\right)$ are compatible dimensional reduction data (Definition \ref{compatible dim red data}).
\end{Example}
In particular, starting with a quiver, the associated tripled quivers with canonical cubic potentials and Nakajima varieties of the double quiver 
provide such examples. 
\begin{Example}\label{ex: doubled vs tripled}
Given a quiver $Q=(Q_0,Q_1)$, we define the \textit{doubled quiver} to be $\overline{Q}:=(Q_0,Q_1\sqcup Q_1^*)$, where $Q_1^*$ is the same as $Q_1$ but with reversed arrow direction. We define the \textit{tripled quiver} $\widetilde{Q}:=(Q_0,Q_1\sqcup Q_1^*\sqcup Q_0)$, which adds one edge loop to each node on top of $\overline{Q}$. 

Fix $\mathbf{v},\mathbf{d}\in \bN^{Q_0}$, we consider two \textit{symmetric quiver varieties} (see Definition \ref{def of sym quiver} for details of construction):
\begin{enumerate}
    \item $\cM_\theta(\overline{Q},\mathbf{v},\mathbf{d})$ associated with the doubled quiver $\overline{Q}$;
    \item $\cM_\theta(\widetilde{Q},\mathbf{v},\mathbf{d})$ associated with the tripled quiver $\widetilde{Q}$.
\end{enumerate}
Assume $\theta$ is generic, so that $\theta$-semistable loci coincide with $\theta$-stable loci: 
$$R(\overline{Q},\mathbf{v},\mathbf{d})^{ss}=R(\overline{Q},\mathbf{v},\mathbf{d})^{s}, \quad R(\widetilde{Q},\mathbf{v},\mathbf{d})^{ss}=R(\widetilde{Q},\mathbf{v},\mathbf{d})^{s}. $$  Let $G=\prod_{i\in Q_0}\GL(\mathbf v_i)$ be the \textit{gauge group}, then the quotient map 
$$R(\overline{Q},\mathbf{v},\mathbf{d})^{s}\to \cM_\theta(\overline{Q},\mathbf{v},\mathbf{d})=R(\overline{Q},\mathbf{v},\mathbf{d})^{s}/G$$ is a principal $G$-bundle.
Denote the associated adjoint bundle to be $\mathcal G\to \cM_\theta(\overline{Q},\mathbf{v},\mathbf{d})$. Then the \textit{moment map} for $G\curvearrowright R(\overline{Q},\mathbf{v},\mathbf{d})$ descends to a section $\mu\in \Gamma\left(\cM_\theta(\overline{Q},\mathbf{v},\mathbf{d}),\mathcal G^\vee\right)$. The total space $\mathrm{Tot}(\mathcal G)$ is naturally identified with an open subscheme $\overset{\bullet}{\cM}_\theta(\widetilde{Q},\mathbf{v},\mathbf{d})\subseteq  {\cM}_\theta(\widetilde{Q},\mathbf{v},\mathbf{d})$ consisting of $G$-equivalence classes in $R(\widetilde{Q},\mathbf{v},\mathbf{d})$, which are $\theta$-stable when restricted to $\overline{Q}$. 
The zero locus of $\mu$:
$$\cN_{\theta}(\overline{Q},\mathbf{v},\mathbf{d}):=Z(\mu)$$
is the \textit{Nakajima quiver variety} \cite{Nak1,Nak2} (which we assume to be nonempty).  
Note that $\cN_{\theta}(\overline{Q},\mathbf{v},\mathbf{d})$ is a smooth subvariety in $\cM_{\theta}(\overline{Q},\mathbf{v},\mathbf d)$ with codimension equals to $\mathcal G$, and $Z(\mu)=Z^\der(\mu)$. 

Denote $\bC^*_\hbar$ the torus that acts on $R(\overline{Q},\bv,\bd)$ by assigning weights $R(\overline{Q},\bv,\bd)=M\oplus\hbar^{-1}M^\vee$, where
\begin{align*}
    M=\bigoplus_{a\in Q_1}\Hom(\bC^{\bv_{t(a)}},\bC^{\bv_{t(a)}})\oplus\bigoplus_{i\in Q_0}\Hom(\bC^{\bd_{i}},\bC^{\bv_{i}}).
\end{align*}
We choose torus $\sT$ which contains $\bC^*_\hbar$ such that $\sT=(\sT/\bC^*_\hbar)\times \bC^*_\hbar$, and let $\sA\subseteq  \sT/\bC^*_\hbar$ be a subtorus. Choose a function $\phi:\cM_\theta(\overline{Q},\mathbf{v},\mathbf{d})\to \bC$ and a section $\mu\in \Gamma(\cM_\theta(\overline{Q},\bv,\bd),\mathcal \hbar^{-1}\mathcal G^\vee)$ and appropriate $\sT$ action such that $\phi$ and $\mu$ are $\sT$-invariant. Let us define the $\sT$ action on edge loops by scaling with weight $\hbar$, then $\sT$ fixes the pairing $\langle e,\mu\rangle$. Then according to Example \ref{ex: dim red}, the dimensional reduction data 
$$\left(\overset{\bullet}{\cM}_\theta(\widetilde{Q},\mathbf{v},\mathbf{d}),\cM_\theta(\overline{Q},\mathbf{v},\mathbf{d}),\mu, \phi \right), \quad \left(\cN_{\theta}(\overline{Q},\mathbf{v},\mathbf{d}),\cN_{\theta}(\overline{Q},\mathbf{v},\mathbf{d}),0, \phi|_{\cN_{\theta}(\overline{Q},\mathbf{v},\mathbf{d})}\right)$$ are compatible.

\begin{Remark}\label{rmk on dim red is}
In the above example, if we take $\phi=0$, then dimensional reduction for critical cohomology \cite[Thm.~A.1]{Dav} and respectively for critical $K$-theory \cite[Thm.~3.6]{Isi} implies that
\begin{align*}
\delta_H\colon H^\sT(\cN_{\theta}(\overline{Q},\mathbf{v},\mathbf{d}))&\to H^\sT(\overset{\bullet}{\cM}_\theta(\widetilde{Q},\mathbf{v},\mathbf{d}),\widetilde{\sw})\\
\delta_K\colon K^\sT(\cN_{\theta}(\overline{Q},\mathbf{v},\mathbf{d}))&\to K^\sT(\overset{\bullet}{\cM}_\theta(\widetilde{Q},\mathbf{v},\mathbf{d}),\widetilde{\sw})
\end{align*}
are isomorphisms. Here $\delta_H$ and $\delta_K$ are the maps defined in Proposition \ref{prop:compatible map_fixed pt} and 
\begin{equation}
\label{trip pot}\widetilde{\sw}=\sum_{i\in Q_0}\tr(\varepsilon_i\mu_i), \end{equation} 
is the \textit{cubic potential}, where $\varepsilon_i$ is the edge loop on the node $i$, and $\mu_i$ is the $i$-th component of moment map. 

It is known that \cite[Lem.~6.3]{Dav2}: 
$$\Crit(\widetilde{\sw})\subseteq  \overset{\bullet}{\cM}_\theta(\widetilde{Q},\mathbf{v},\mathbf{d}), $$ therefore the restriction from $\cM_\theta(\widetilde{Q},\mathbf{v},\mathbf{d})$ to $\overset{\bullet}{\cM}_\theta(\widetilde{Q},\mathbf{v},\mathbf{d})$ gives isomorphisms:
\begin{align*}
H^\sT(\cM_\theta(\widetilde{Q},\mathbf{v},\mathbf{d}),\widetilde{\sw})\cong H^\sT(\overset{\bullet}{\cM}_\theta(\widetilde{Q},\mathbf{v},\mathbf{d}),\widetilde{\sw}),\quad K^\sT(\cM_\theta(\widetilde{Q},\mathbf{v},\mathbf{d}),\widetilde{\sw})\cong K^\sT(\overset{\bullet}{\cM}_\theta(\widetilde{Q},\mathbf{v},\mathbf{d}),\widetilde{\sw}).
\end{align*}
\end{Remark}
More examples on compatible dimensional reduction data are discussed in \cite{COZZ1}, which are used to study shifted Yangians of $\mathfrak{gl}_2$,  
shifted affine Yangians of $\mathfrak{gl}_1$ and their representations. 

\end{Example}

\section{Deformations of potentials and stable envelopes }\label{sect on sp and stab}

In this section, we discuss deformations of potentials and how they interplay with stable envelopes. This provides powerful tools  
in the computations of stable envelopes, relating different modules of quantum groups, see \cite{COZZ1}.

\begin{Setting}\label{setting of def of potentials}
Fix $(X,\sT,\sA)$ in Setting \ref{setting of stab}. Suppose there is a torus action $\bG_m\curvearrowright X$ which commutes with $\sT$-action. 
\begin{itemize}
\item Assume that $\mathsf{f}\colon X\to \bC$ is a $(\sT\times \bG_m)$-invariant function, and $\mathsf{g}\colon X\to \bC$ is a $\sT$-invariant function together with a decomposition 
$$\sg=\sum_{i=1}^m\sg_i, $$ such that $\bG_m$ scales $\mathsf{g}_i$ with weight $-\mathsf n_i<0$,~i.e.~$(u\cdot \mathsf{g}_i)(x):=\mathsf{g}_i(u^{-1}\cdot x)=u^{-n_i}\cdot \mathsf{g}_i(x)$.

\item Let $\widetilde{X}:=X\times \bA^1$ with $(\sT\times \bG_m)$-action given by the aforementioned action on $X$ and the action
\begin{align*}
    \sT\times \bG_m  \xrightarrow{\pr} \bG_m \curvearrowright \bA^1
\end{align*}
on $\bA^1$, where $\bG_m \curvearrowright \bA^1$ is of weight $-1$. Let $t$ be the coordinate on $\bA^1$, with $(\sT\times \bG_m)$-invariant function:
\begin{align*}
    \sw:=\mathsf f+\sum_{i=1}^mt^{\mathsf n_i}\mathsf g_i\colon \widetilde{X}\to \bC, \quad (x,t)\mapsto f(x)+\sum_{i=1}^mt^{\mathsf n_i}\mathsf g_i(x).
\end{align*}
\end{itemize}
\end{Setting}
When we discuss the relation with stable envelopes, we work in the following setting. 
\begin{Setting}\label{setting of def of potentials2}
Let $(X,\sT,\sA,\mathsf f,\mathsf g)$ be as in Setting \ref{setting of def of potentials}. Fix
a chamber $\fC\subseteq \Lie(\sA)_\bR$ and a slope $\mathsf s\in \Pic_\sA(X)\otimes_\bZ \bR$. Assume 
cohomological and $K$-theoretic stable envelope correspondences exist for $(X,0,\sT,\sA,\fC)$ and $(X,0,\sT,\sA,\fC,\mathsf s)$ respectively.

\end{Setting}
\begin{Remark}
Since canonical maps take stable envelope correspondences for zero potential to stable envelope correspondences for any potential, then according to Proposition \ref{corr induce stab}, cohomological and $K$-theoretic stable envelopes exist for $(X,\mathsf f,\sT,\sA,\fC),(X,\mathsf f+\sg,\sT,\sA,\fC)$ and $(X,\mathsf f,\sT,\sA,\fC,\mathsf s),(X,\mathsf f+\sg,\sT,\sA,\fC,\mathsf s)$ respectively in Setting \ref{setting of def of potentials2}.
\end{Remark}

\subsection{Specialization for Borel-Moore homology and \texorpdfstring{$K$}{K}-theory of zero loci}

\begin{Lemma}
Under Setting \ref{setting of def of potentials}, let $Z(\sw)\subseteq  \widetilde{X}$ denote the zero locus of $\sw$ and $Z(\sw)^*$ its open subset fitting into commutative diagrams 
$$
\xymatrix{
 Z(\sw)^*  \ar[d]^{} \ar[r]_{ }  \ar@{}[dr]|{\Box}  & Z(\sw)  \ar[d]^{} \ar@{^{(}->}[r]   \ar@{}[dr]|{\curvearrowright }   &  \widetilde{X}     \ar[d]^{\pr} \\
\bC^* \ar@{^{(}->}[r] & \bA^1  \ar@{=}[r] &  \bA^1,   
} 
$$
where $\pr$ is the projection and the left square is Cartesian. Then the left vertical map is a trivial fibration such that 
$$Z(\sw)^*\cong Z(\mathsf f+\mathsf g)\times \bC^*.$$
\end{Lemma}
\begin{proof}
There is an automorphism 
\begin{equation}\label{equ on iso a}\mathbf a\colon \widetilde{X}^*:=X\times \bC^*\to \widetilde{X}^*, \quad (x,t) \mapsto (t^{-1}\cdot x,t), \end{equation}
where $t\cdot x$ means identifying $\bC^*$ with $\bG_m$ and $t$ acts on $x$ via the $\bG_m$ action. Then 
$$\mathbf a^*(\sw)=\mathsf f+\mathsf g, $$ 
therefore $\mathbf a^{-1}(Z(\sw)^*)=Z(\mathsf f+\mathsf g)\times \bC^*$.
\end{proof}

It follows that the projection $Z(\sw)\to \bC$ fits in the setting of \cite[\S2.6.30]{CG}. 
\begin{Definition}\label{def of sp for zero loci_coh}
Under Setting \ref{setting of def of potentials},
we define a \textit{specialization map}: 
\begin{align}\label{sp zero locus_coh}
    \mathsf{sp}\colon H^\sT(Z(\mathsf f+\mathsf g))\to H^\sT(Z(\mathsf f)),
\end{align}
as the composition of the following maps
\begin{align*}
\xymatrix{
H^\sT_i(Z(\mathsf f+\mathsf g))\ar[r]^-{\pr^*} & H^\sT_{i+2}(Z(\mathsf f+\mathsf g)\times \bC^*) \ar[r]^-{\mathbf a_*}_-{\cong}
&  H^\sT_{i+2}(Z(\sw)^*) \ar[r]^-{\underset{t\to 0}{\lim}} & H^\sT_i(Z(\sw)\cap \{t=0\})=H_i^\sT(Z(\mathsf f)), 
}
\end{align*}
where $\underset{t\to 0}{\lim}$ is the map defined in \cite[\S2.6.30]{CG}.
\end{Definition}
\begin{Remark}\label{rmk: equiv def of sp zero locus_coh}
According to the construction in \cite[\S2.6.30]{CG}, we have an equality
\begin{align*}
    \underset{t\to 0}{\lim}(-)=\delta\circ (-\cap \alpha).
\end{align*}
Here $\alpha\in H_1(\bC^*)\cong \bQ$ is the Borel-Moore homology class given by pushforward of the fundamental class $[\bR]\in H_1(\bR)$ under the exponential map 
$$\bR\to \bC^*, \quad s\mapsto e^s, $$ 
and for any $\theta\in [0,2\pi)$, the maps $s\mapsto e^{s+i\theta}$ and $s\mapsto e^{s}$ induce the same map $H_1(\bR)\to H_1(\bC^*)$ on homologies.
$(-)\cap \alpha\colon H^\sT_{j}(Z(\sw)^*)\to H^\sT_{j-1}(Z(\sw)^*)$ is the intersection pairing with $\alpha$. $\delta$ is the boundary map in the excision long exact sequence
\begin{align*}
\xymatrix{
H^\sT_{i+1}(Z(\sw)^*) \ar[r]^-{\delta} & H^\sT_i(Z(\sw)\cap \{t=0\}) \ar[r] & H^\sT_i(Z(\sw)) \ar[r] & H^\sT_i(Z(\sw)^*) \ar[r]^-{\delta} & . \\
}
\end{align*}
Then it follows that $\mathsf{sp}$ \eqref{sp zero locus_coh} equals to the composition of the following maps
\begin{align*}
\xymatrix{
H^\sT_i(Z(\mathsf f+\mathsf g))\ar[r]^-{\id\otimes\alpha} & H^\sT_{i+1}(Z(\mathsf f+\mathsf g)\times \bC^*) \ar[r]^-{\mathbf a_*}_-{\cong}
&  H^\sT_{i+1}(Z(\sw)^*) \ar[r]^-{\delta} & H^\sT_i(Z(\sw)\cap \{t=0\})=H_i^\sT(Z(\mathsf f)).
}
\end{align*}
\end{Remark}
Similarly, there is a $K$-theoretic 
\textit{specialization map} (ref.~\cite[\S5.3]{CG}). 
\begin{Definition}\label{def of sp for zero loci_k}
Under Setting \ref{setting of def of potentials}, we define 
\begin{align}\label{sp zero locus_k}
    \mathsf{sp}\colon K^\sT(Z(\mathsf f+\mathsf g))\to K^\sT(Z(\mathsf f)),
\end{align}
as the composition of the following maps
\begin{align*}
\xymatrix{
K^\sT(Z(\mathsf f+\mathsf g))\ar[r]^-{\pr^*} & K^\sT(Z(\mathsf f+\mathsf g)\times \bC^*) \ar[r]^-{\mathbf a_*}_-{\cong}
&  K^\sT(Z(\sw)^*) \ar[r]^-{\underset{t\to 0}{\lim}} & K^\sT(Z(\sw)\cap \{t=0\})=K^\sT(Z(\mathsf f)). 
}
\end{align*}
Here $\underset{t\to 0}{\lim}$ is defined by an extension to a class on $Z(\sw)$ followed by 
refined Gysin pullback  $i^!_0\colon K^\sT(Z(\sw))\to K^\sT(Z(\sw)\cap \{t=0\})$. This is well-defined as we have the excision exact sequence:
\begin{align*}
\xymatrix{
K^\sT(Z(\sw)\cap \{t=0\})\ar[r]^-{i_{0 *}} & K^\sT(Z(\sw)) \ar[r] & K^\sT(Z(\sw)^*) \ar[r] &  0,\\
}
\end{align*}
and $i^!_0\circ i_{0 *}(-)=0 \in K^\sT(Z(\sw)\cap \{t=0\})$. 
\end{Definition}

\begin{Remark}\label{rmk: g=0 zero locus}
If $\sg=0$, then the specialization maps \eqref{sp zero locus_coh} and \eqref{sp zero locus_k} are identity maps. This can be seen as follows. In this case, $Z(\sw)=Z(\mathsf f)\times \bC$ and $\mathbf a$ maps $Z(\sw)^*$ to itself. Then $\mathbf a_*$ is the identity map because the $\bG_m$ action on $H^\sT(Z(\mathsf f))$ is trivial. In particular $\mathbf a_*\circ\pr^*(\gamma)=\gamma\otimes[\bC^*]$ for any $\gamma\in H^\sT(Z(\mathsf f))$. It follows that 
\begin{align*}
    \mathsf{sp}(\gamma)=\underset{t\to 0}{\lim}(\gamma\otimes[\bC^*])=\gamma\otimes\underset{t\to 0}{\lim}([\bC^*])=\gamma,
\end{align*}
where in the second equality we use the fact that $\underset{t\to 0}{\lim}$ is compatible with external tensor product. The $K$-theory version can be argued similarly and we omit the details.
\end{Remark}

\begin{Remark}\label{rmk: sp comp w/ pushforward}
Specialization maps \eqref{sp zero locus_coh} and \eqref{sp zero locus_k} are compatible with pushforward maps induced by the embedding of the zero locus into $X$. Namely the following diagrams are commutative:
\begin{align*}
\xymatrix{
H^\sT(Z(\mathsf f+\sg)) \ar[r]^-{\mathsf{sp}} \ar[dr] & H^\sT(Z(\mathsf f))  \ar[d] \\
& H^\sT(X),
}
\qquad
\xymatrix{
K^\sT(Z(\mathsf f+\sg)) \ar[r]^-{\mathsf{sp}} \ar[dr] & K^\sT(Z(\mathsf f))  \ar[d] \\
& K^\sT(X).
}
\end{align*}
In particular, if $\mathsf f=0$, then the specialization maps $\mathsf{sp}\colon H^\sT(Z(\mathsf g))\to H^\sT(X)$ and $\mathsf{sp}\colon K^\sT(Z(\mathsf g))\to K^\sT(X)$ equal to the pushforward maps induced by the closed immersion $Z(\mathsf g)\hookrightarrow X$.
\end{Remark}

\begin{Remark}\label{rmk: sp zero loci functorial}
Specialization maps \eqref{sp zero locus_coh} and \eqref{sp zero locus_k} are \textit{functorial} with respect to \textit{proper pushforward} and \textit{lci pullback}. Namely, let $(Y,\sT,\sA,\bG_m)$ be in Setting \ref{setting of def of potentials}, and assume that there exists a $(\sT\times \bG_m)$-equivariant map $\pi\colon Y\to X$. Then we have commutative diagrams
\begin{align*}
\xymatrix{
H^\sT(Z((\mathsf f+\sg)\circ\pi)) \ar[r]^-{\mathsf{sp}}  & H^\sT(Z(\mathsf f\circ \pi))   \\
H^\sT(Z(\mathsf f+\sg)) \ar[u]^{\pi^!} \ar[r]^-{\mathsf{sp}} & H^\sT(Z(\mathsf f)) \ar[u]_{\pi^!},
}
\qquad
\xymatrix{
K^\sT(Z((\mathsf f+\sg)\circ\pi)) \ar[r]^-{\mathsf{sp}}  & K^\sT(Z(\mathsf f\circ \pi))   \\
K^\sT(Z(\mathsf f+\sg)) \ar[u]^{\pi^!} \ar[r]^-{\mathsf{sp}} & K^\sT(Z(\mathsf f)) \ar[u]_{\pi^!},
}
\end{align*}
where we use the fact that $X$ and $Y$ are smooth, so the above lci pullbacks $\pi^!$ are well-defined. 

If $\pi$ is moreover proper, then we have commutative diagrams
\begin{align*}
\xymatrix{
H^\sT(Z((\mathsf f+\sg)\circ\pi)) \ar[r]^-{\mathsf{sp}} \ar[d]_{\pi_*} & H^\sT(Z(\mathsf f\circ \pi))  \ar[d]^{\pi_*} \\
H^\sT(Z(\mathsf f+\sg)) \ar[r]^-{\mathsf{sp}} & H^\sT(Z(\mathsf f)),
}
\qquad
\xymatrix{
K^\sT(Z((\mathsf f+\sg)\circ\pi)) \ar[r]^-{\mathsf{sp}} \ar[d]_{\pi_*} & K^\sT(Z(\mathsf f\circ \pi))  \ar[d]^{\pi_*} \\
K^\sT(Z(\mathsf f+\sg)) \ar[r]^-{\mathsf{sp}} & K^\sT(Z(\mathsf f)).
}
\end{align*}
\end{Remark}
Specialization maps are compatible with stable envelopes for zero loci. 
\begin{Proposition}\label{prop: stab comp with sp zero loci}
Fix $(X,\sT,\sA,\mathsf f,\sg,\fC,\mathsf s)$ as in Setting \ref{setting of def of potentials2}, then the following diagrams are commutative:
\begin{align*}
\xymatrix{
H^\sT(Z(\mathsf f+\sg)^\sA) \ar[r]^-{\mathsf{sp}^\sA} \ar[d]_{\Stab_\fC} & H^\sT(Z(\mathsf f)^\sA)  \ar[d]^{\Stab_\fC} \\
H^\sT(Z(\mathsf f+\sg)) \ar[r]^-{\mathsf{sp}} & H^\sT(Z(\mathsf f)),
}
\qquad
\xymatrix{
K^\sT(Z(\mathsf f+\sg)^\sA) \ar[r]^-{\mathsf{sp}^\sA} \ar[d]_{\Stab^{\mathsf s}_\fC} & K^\sT(Z(\mathsf f)^\sA)  \ar[d]^{\Stab^{\mathsf s}_\fC} \\
K^\sT(Z(\mathsf f+\sg)) \ar[r]^-{\mathsf{sp}} & K^\sT(Z(\mathsf f)),
}
\end{align*}
where $\Stab_\fC$ and $\Stab^{\mathsf s}_\fC$ are maps induced by the stable envelope correspondences as in Lemma \ref{lem on can cm w conv}, and $\mathsf{sp}^\sA$ is the specialization maps \eqref{sp zero locus_coh} and \eqref{sp zero locus_k} for the $\sA$-fixed loci.
\end{Proposition}

\begin{proof}
We first prove the cohomology version. By definition, it is enough to show that every square in the following diagram is commutative:
\begin{align*}
\xymatrix{
H^\sT_i(Z(\mathsf f+\mathsf g)^\sA) \ar[d] \ar[r]^-{\pr^*} & H^\sT_{i+2}(Z(\mathsf f+\mathsf g)^\sA\times \bC^*) \ar[d]  \ar[r]^-{\mathbf a_*}
&  H^\sT_{i+2}((Z(\sw)^*)^\sA) \ar[d]  \ar[r]^-{\underset{t\to 0}{\lim}} & H^\sT_i(Z(\sw)^\sA\cap \{t=0\}) \ar[d] \\
H^\sT_i(Z(\mathsf f+\mathsf g))\ar[r]^-{\pr^*} & H^\sT_{i+2}(Z(\mathsf f+\mathsf g)\times \bC^*) \ar[r]^-{\mathbf a_*}
&  H^\sT_{i+2}(Z(\sw)^*) \ar[r]^-{\underset{t\to 0}{\lim}} & H^\sT_i(Z(\sw)\cap \{t=0\}),
}
\end{align*}
where the first and the fourth vertical arrows are induced by $[\Stab_\fC]\in H^\sT(X\times X^\sA)_{\Attr^f_\fC}$, the second and the third ones are induced by $[\Stab_\fC]\otimes[\Delta_{\bC^*}]\in H^\sT(X\times X^\sA\times (\bC^*)^2)_{\Attr^f_\fC\times (\bC^*)^2}$, for the diagonal $\Delta_{\bC^*}$ in $(\bC^*)^2$. It is obvious that the first two squares commute. The commutativity of the third square follows from the fact that specialization commutes with the convolution \cite[Prop.~2.7.23]{CG}.

Similarly, in the $K$-theory version, we have 
\begin{align*}
\xymatrix{
K^\sT(Z(\mathsf f+\mathsf g)^\sA) \ar[d] \ar[r]^-{\pr^*} & K^\sT(Z(\mathsf f+\mathsf g)^\sA\times \bC^*) \ar[d]  \ar[r]^-{\mathbf a_*}
&  K^\sT((Z(\sw)^*)^\sA) \ar[d]  \ar[r]^-{\underset{t\to 0}{\lim}} & K^\sT(Z(\sw)^\sA\cap \{t=0\}) \ar[d] \\
K^\sT(Z(\mathsf f+\mathsf g))\ar[r]^-{\pr^*} & K^\sT(Z(\mathsf f+\mathsf g)\times \bC^*) \ar[r]^-{\mathbf a_*}
&  K^\sT(Z(\sw)^*) \ar[r]^-{\underset{t\to 0}{\lim}} & K^\sT(Z(\sw)\cap \{t=0\}),
}
\end{align*}
where the first and the fourth vertical arrows are induced by $[\Stab^{\mathsf s}_\fC]\in K^\sT(X\times X^\sA)_{\Attr^f_\fC}$, the second and the third ones are induced by $[\Stab^{\mathsf s}_\fC]\otimes[\mathcal O_{\Delta_{\bC^*}}]\in K^\sT(X\times X^\sA\times (\bC^*)^2)_{\Attr^f_\fC\times (\bC^*)^2}$. It is obvious that the first two squares 
commute. The commutativity of the third square follows from \cite[Thm.~5.3.9]{CG}.
\end{proof}

The goal of the rest of this section is to adapt the above construction to critical cohomology and $K$-theory (in several settings) and to prove various compatibility properties.

\subsection{Specialization for critical cohomology}

\begin{Assumption}\label{assumption for def of potential_crit}
In Setting \ref{setting of def of potentials}, we further assume the affinization map 
$$\pi\colon X\to X_0=\Spec \Gamma(X,\mathcal O_X)$$
is proper and the induced $\bG_m$ action on $\Gamma(X,\mathcal O_X)$ is nonpositive (equivalently, the induced $\bG_m$ action attracts $X_0$ to the fixed locus $X_0^{\bG_m}$).
\end{Assumption}

\begin{Lemma}\label{zero fiber = attr}
Under Assumption \ref{assumption for def of potential_crit}, we have $$X\times \{0\}=\bigcup_{F\in \Fix_{\bG_m}(\widetilde{X})}\Attr_+(F).$$
\end{Lemma}

\begin{proof}
Let $\widetilde{X}_0:=X_0\times \bC$ be the affinization of $\widetilde{X}$. By Setting \ref{setting of def of potentials}, the $\bG_m$ action on $\bC$ is repelling, so the $\bG_m$-fixed point locus $\widetilde{X}_0^{\bG_m}$ equals to $X_0^{\bG_m}\times \{0\}$, and the attracting set $\Attr_+(\widetilde{X}_0^{\bG_m})$ equals to $X_0$ by Assumption \ref{assumption for def of potential_crit}. Denote the affinization map by $p:\widetilde{X}\to \widetilde{X}_0$, then 
for any $F\in \Fix_{\bG_m}(\widetilde{X})$, 
$$\Attr_+(F)\subseteq  p^{-1}(\Attr_+(p(F)))\subseteq  p^{-1}(\Attr_+(\widetilde{X}_0^{\bG_m}))=X\times \{0\}.$$ 
This proves one direction of inclusion. For the other direction, take any point $x\in X\times \{0\}$, the limit  
$$\lim_{u\to 0}p(u\cdot x)=\lim_{u\to 0}u\cdot p(x)$$ exists by Assumption \ref{assumption for def of potential_crit}, 
where $u\cdot$ is the action of $u\in \bG_m$. Since $p$ is proper by Assumption \ref{assumption for def of potential_crit}, $\lim_{u\to 0}u\cdot x$ exists by valuative criterion of properness; thus $x\in \Attr_+(F)$ for some $F\in \Fix_{\bG_m}(\widetilde{X})$. This concludes the proof.
\end{proof}

\begin{Definition}\label{def of sp crit_coh}
Under Assumption \ref{assumption for def of potential_crit}, we define a \textit{specialization map} 
\begin{align}\label{sp crit_coh}
    \mathsf{sp}\colon H^\sT(X,\mathsf f+\mathsf g)\to H^\sT(X,\mathsf f),
\end{align}
as the composition of the following maps
\begin{align*}
\xymatrix{
H^\sT_i(X,\mathsf f+\mathsf g)\ar[r]^-{\id\otimes\alpha\,\,} & H^\sT_{i+1}(X\times \bC^*,\mathsf f+\mathsf g) \ar[r]^-{\mathbf a_*}_-{\cong}
&  H^\sT_{i+1}(\widetilde{X}^*, \sw|_{\widetilde{X}^*}) \ar[r]^-{\delta} & H^\sT_i(X\times \{0\},\sw|_{\{t=0\}})=H_i^\sT(X,\mathsf f).
}
\end{align*}
Here $\alpha\in H_{1}(\bC^*)$ is the class defined in Remark \ref{rmk: equiv def of sp zero locus_coh}, $\mathbf a$ is given in \eqref{equ on iso a}, and $\delta$ is the boundary map in the following excision long exact sequence (exactness follows from Proposition \ref{excision for crit coh} and Lemma \ref{zero fiber = attr}):
\begin{align*}
\xymatrix{
H^\sT_{i+1}(\widetilde{X}^*,\sw|_{\widetilde{X}^*}) \ar[r]^-{\delta} & H^\sT_i(X\times\{0\},\sw|_{\{t=0\}}) \ar[r] & H^\sT_i(\widetilde{X},\sw) \ar[r] & H^\sT_i(\widetilde{X}^*,\sw|_{\widetilde{X}^*}) \ar[r]^-{\delta} & . \\
}
\end{align*}
\end{Definition}


\begin{Remark}\label{rmk: g=0 crit}
If $\sg=0$, then specialization map \eqref{sp crit_coh} is the identity map. This can be seen as follows. In this case, $\mathbf a_*$ is the identity map because the $\bG_m$ action on $H^\sT(X,\mathsf f)$ is trivial. Then for any $\gamma\in H^\sT(X,\mathsf f)$,
\begin{align*}
    \mathsf{sp}(\gamma)=\delta(\gamma\otimes\alpha)=\gamma\otimes\delta(\alpha)=\gamma,
\end{align*}
where by using Thom-Sebastiani isomorphism, the above equalities reduce to the calculation in Remark \ref{rmk: equiv def of sp zero locus_coh}.
\end{Remark}

\begin{Remark}\label{rmk: sp zero loci vs crit}
Comparing Definition \ref{def of sp crit_coh} with Remark \ref{rmk: equiv def of sp zero locus_coh}, and noticing that canonical maps commute with $(-)\otimes\alpha$, $\mathbf a_*$ and boundary map $\delta$, we see that canonical maps commute with specialization maps \eqref{sp zero locus_coh} and \eqref{sp crit_coh}, i.e. we have commutative diagram:
\begin{align*}
\xymatrix{
H^\sT(Z(\mathsf f+\sg)) \ar[r]^-{\mathsf{sp}} \ar[d]_{\can} & H^\sT(Z(\mathsf f))  \ar[d]^{\can} \\
H^\sT(X,\mathsf f+\sg) \ar[r]^-{\mathsf{sp}} & H^\sT(X,\mathsf f).
}
\end{align*}
In particular, if $\mathsf f=0$, then $\can\circ\,\mathsf{sp}$ equals to the pushforward induced by the inclusion $Z(\sg)\hookrightarrow X$. 
\end{Remark}

Specialization map \eqref{sp crit_coh} is \textit{functorial} with respect to \textit{proper pushforward} and \textit{lci pullback},~i.e. 
\begin{Proposition}\label{prop on sp comm with stab coho}

Let $(X,\sT,\sA,\bG_m,\mathsf f,\sg)$ be in Setting \ref{setting of def of potentials} and 
suppose there is a $(\sT\times \bG_m)$-equivariant map $\pi\colon Y\to X$, so that $(Y,\sT,\sA,\bG_m,\mathsf f\circ \pi,\sg\circ \pi)$ is also in Setting \ref{setting of def of potentials}.
Assume $X$, $Y$ both satisfy Assumption \ref{assumption for def of potential_crit}.

Then we have commutative diagram
\begin{align*}
\xymatrix{
H^\sT(Y,(\mathsf f+\sg)\circ\pi) \ar[r]^-{\mathsf{sp}}  & H^\sT(Y, \mathsf f\circ \pi)   \\
H^\sT(X,\mathsf f+\sg) \ar[u]^{\pi^!} \ar[r]^-{\mathsf{sp}} & H^\sT(X,\mathsf f) \ar[u]_{\pi^!},
}
\end{align*}
where we use the fact that $X$ and $Y$ are smooth, so the above lci pullbacks $\pi^!$ are well-defined. 

If $\pi$ is moreover proper, then we have commutative diagram
\begin{align*}
\xymatrix{
H^\sT(Y, (\mathsf f+\sg)\circ\pi) \ar[r]^-{\mathsf{sp}} \ar[d]_{\pi_*} & H^\sT(Y,\mathsf f\circ \pi)  \ar[d]^{\pi_*} \\
H^\sT(X,\mathsf f+\sg) \ar[r]^-{\mathsf{sp}} & H^\sT(X,\mathsf f).
}
\end{align*}
\end{Proposition}
\begin{proof}
Lci pullback commutes with $(-)\otimes\alpha$, $\mathbf a_*$ and boundary map $\delta$, so it commutes with specialization maps for critical cohomology, similarly for the proper pushforward case. 
\end{proof}
Specialization map \eqref{sp crit_coh} is compatible with stable envelopes.  
\begin{Proposition}\label{prop: stab comp with sp crit}
Fix $(X,\sT,\sA,\bG_m,\mathsf f,\sg,\fC)$ as in Setting \ref{setting of def of potentials2} which satisfies Assumption \ref{assumption for def of potential_crit}. Then the following diagram is commutative
\begin{align*}
\xymatrix{
H^\sT(X^\sA,\mathsf f^\sA+\sg^\sA) \ar[d]_{\Stab_\fC} \ar[r]^-{\mathsf{sp}^\sA} & H^\sT(X^\sA,\mathsf f^\sA) \ar[d]^{\Stab_\fC} \\
H^\sT(X,\mathsf f+\sg) \ar[r]^{\,\,\,\,\mathsf{sp}} &  H^\sT(X,\mathsf f).
}
\end{align*}
\end{Proposition}

\begin{proof}
It is easy to see that $\Stab_\fC$ commutes with $(-)\otimes\alpha$ and $\mathbf a_*$. It remains to show that $\Stab_\fC$ commutes with the boundary map $\delta$. Note that $\Stab_\fC$ is induced from the stable envelope correspondence $[\Stab_\fC]$ via the critical convolution (see \S\ref{sect on exi of stb}), which consists of a Gysin pullback
and a proper pushforward. Since they both commute with $\delta$, $\Stab_\fC$ commutes with $\delta$ as well.
\end{proof}

We provide examples where Assumption \ref{assumption for def of potential_crit} holds, so that there are specialization maps for critical cohomology. 
\begin{Example}\label{ex of sp for crit_coh}
Let $\mathcal M_{\theta}(Q,\bv,\underline{\bd})$ be a quiver variety together with a flavor torus $\sT$-action (see \S \ref{sec: quiver var}). Every regular function on $\mathcal M_{\theta}(Q,\bv,\underline{\bd})$ is a gauge group $G$-invariant function on the representation space $R(Q,\bv,\underline{\bd})$, which can be written as
$$\sw=\tr(\sW)$$
where $\sW$ is a linear combination of cycles in the Crawley-Boevey quiver associated with $(Q,\underline{\bd})$ \footnote{Crawley-Boevey quiver is obtained from $Q$ by adding a new node $\infty$ to $Q$ and $\bd_{\In,i}$ arrows from $\infty$ to $i\in Q_0$ and $\bd_{\Out,i}$ arrows from $i$ to $\infty$ \cite{CB}.}. Suppose that $\sW$ does not contain constant term,\,\,i.e.\,\,there is no trivial cycle, then $\sw$ has negative weights with respect to the $\bC^*$-action which scale every arrow with weight $-1$. The data $X=\mathcal M_{\theta}(Q,\bv,\underline{\bd})$, $\mathsf f=0$, $\sg=\sw$, and the above $\C^*$ action fit into 
Setting \ref{setting of def of potentials} and 
satisfy Assumption \ref{assumption for def of potential_crit}. 
In particular, there exists a specialization map for critical cohomology (Definition \ref{def of sp crit_coh}):
$$\mathsf{sp}\colon H^\sT(\mathcal M_{\theta}(Q,\bv,\underline{\bd}),\sw)\to H^\sT(\mathcal M_{\theta}(Q,\bv,\underline{\bd})), $$ 
which makes the following diagram commutative (Remark \ref{rmk: sp zero loci vs crit})
\begin{align*}
\xymatrix{
H^\sT(Z(\sw)) \ar[d]_{\can} \ar[dr] \\
H^\sT(\mathcal M_{\theta}(Q,\bv,\underline{\bd}),\sw) \ar[r]^-{\mathsf{sp}} & H^\sT(\mathcal M_{\theta}(Q,\bv,\underline{\bd})).
}
\end{align*}
Moreover, in the case when $Q$ is symmetric and $\bd_{\In}=\bd_{\Out}=\bd$ and $\sA$-action is self-dual, the specialization maps are compatible with stable envelopes by Proposition \ref{prop: stab comp with sp crit},\,i.e.\,the following diagram is commutative:
\begin{align*}
\xymatrix{
H^\sT(\mathcal M_{\theta}(Q,\mathbf v,\bd)^\sA,\sw^\sA) \ar[d]_{\Stab_\fC} \ar[r]^-{\mathsf{sp}^\sA} & H^\sT(\mathcal M_{\theta}(Q,\mathbf v,\bd)^\sA) \ar[d]^{\Stab_\fC} \\
H^\sT(\mathcal M_{\theta}(Q,\mathbf v,\bd),\sw) \ar[r]^{\,\,\,\,\mathsf{sp}} &  H^\sT(\mathcal M_{\theta}(Q,\mathbf v,\bd)).
}
\end{align*}
The above construction also applies to Nakajima variety $\cN_\theta(\bv,\bd)$, and for any potential $\sw=\tr(\sW)$ with negative weight under the $\bC^*$ action that scale every arrow with weight $-1$, we get a specialization
$$\mathsf{sp}\colon H^\sT(\cN_{\theta}(Q,\bv,\underline{\bd}),\sw)\to H^\sT(\cN_{\theta}(Q,\bv,\underline{\bd})), $$ 
which is compatible with canonical map and stable envelope.

\end{Example}

\subsection{Specialization for critical \texorpdfstring{$K$}{K}-theory}

\subsubsection{Definition and properties}
Contrary to the case of critical cohomology, the canonical map in critical $K$-theory is always surjective. Therefore it is convenient 
to formulate its specalization map to be the unique one (if exists) which is compatible with the specialization map of the zero locus. 

\begin{Definition}\label{def of sp crit_k}
Under Setting \ref{setting of def of potentials}, a \textit{specialization} for critical $K$-theory is a map $\mathsf{sp}\colon K^\sT(X,\mathsf f+\sg)\to K^\sT(X,\mathsf f)$ which makes the following diagram commutative:
\begin{align}\label{pre sp for crit k}
\xymatrix{
K^\sT(Z(\mathsf f+\sg)) \ar[r]^-{\mathsf{sp}} \ar@{>>}[d]_{\can} &  K^\sT(Z(\mathsf f))  \ar@{>>}[d]^{\can} \\
K^\sT(X,\mathsf f+\sg) \ar[r]^-{\mathsf{sp}} &  K^\sT(X,\mathsf f),
}
\end{align}
where the upper horizontal line is given in \eqref{sp zero locus_k}.
\end{Definition}

Specialization map in Definition \ref{def of sp crit_k} is \textit{functorial} with respect to \textit{proper pushforward} and \textit{lci pullback},~i.e. 

\begin{Proposition}\label{prop on sp comm with stab}
Let $(X,\sT,\sA,\bG_m,\mathsf f,\sg)$ be in Setting \ref{setting of def of potentials} and 
suppose there is a $(\sT\times \bG_m)$-equivariant map $\pi\colon Y\to X$, so that 
$(Y,\sT,\sA,\bG_m,\mathsf f\circ \pi,\sg\circ \pi)$ also fits into Setting \ref{setting of def of potentials}.
Assume that specialization maps $\mathsf{sp}\colon K^\sT(X,\mathsf f+\sg)\to K^\sT(X,\mathsf f)$ and $\mathsf{sp}\colon K^\sT(Y,(\mathsf f+\sg)\circ\pi)\to K^\sT(Y,\mathsf f\circ \pi)$ exist. Then we have a commutative diagram: 
\begin{align*}
\xymatrix{
K^\sT(Y,(\mathsf f+\sg)\circ\pi) \ar[r]^-{\mathsf{sp}}  & K^\sT(Y, \mathsf f\circ \pi)   \\
K^\sT(X,\mathsf f+\sg) \ar[u]^{\pi^!} \ar[r]^-{\mathsf{sp}} & K^\sT(X,\mathsf f) \ar[u]_{\pi^!},
}
\end{align*}
where we use the fact that $X$ and $Y$ are smooth, so the above lci pullbacks $\pi^!$ are well-defined. 

If $\pi$ is moreover proper, then we have a commutative diagram:
\begin{align*}
\xymatrix{
K^\sT(Y, (\mathsf f+\sg)\circ\pi) \ar[r]^-{\mathsf{sp}} \ar[d]_{\pi_*} & K^\sT(Y,\mathsf f\circ \pi)  \ar[d]^{\pi_*} \\
K^\sT(X,\mathsf f+\sg) \ar[r]^-{\mathsf{sp}} & K^\sT(X,\mathsf f).
}
\end{align*}
\end{Proposition}

\begin{proof}
Consider the following diagram:
\begin{align*}
\xymatrix{
K^\sT(Z((\mathsf f+\sg)\circ\pi))\ar[rrr]^{\mathsf{sp}} \ar[dr]^{\can} & & & K^\sT(Z(\mathsf f\circ\pi)) \ar[dl]_{\can} \\
& K^\sT(Y,(\mathsf f+\sg)\circ\pi) \ar[r]^-{\mathsf{sp}}  & K^\sT(Y, \mathsf f\circ \pi)   \\
& K^\sT(X,\mathsf f+\sg) \ar[u]^{\pi^!} \ar[r]^-{\mathsf{sp}} & K^\sT(X,\mathsf f) \ar[u]_{\pi^!}\\
K^\sT(Z(\mathsf f+\sg))\ar[rrr]^{\mathsf{sp}} \ar[ur]_{\can}  \ar[uuu]^{\pi^!}  & & & K^\sT(Z(\mathsf f)) \ar[uuu]_{\pi^!}\ar[ul]^{\can}
}
\end{align*}
The outer square is commutative by Remark \ref{rmk: sp zero loci functorial}. The left and right trapezoids are commutative by \cite[Lem.~2.4(a)]{VV2}. The top and bottom trapezoids are commutative by Definition \ref{def of sp crit_k}. Then it follows that
\begin{align*}
    \pi^!\circ\, \mathsf{sp}\circ \can=\mathsf{sp}\circ\pi^!\circ \can.
\end{align*}
Since $K^\sT(Z(\mathsf f+\sg))\xrightarrow{\can} K^\sT(X,\mathsf f+\sg)$ is surjective, we get $\pi^!\circ\, \mathsf{sp}=\mathsf{sp}\circ\pi^!$. 
The commutativity in the proper pushforward case can be proven similarly, with \cite[Lem.~2.4(a)]{VV2} replaced by \cite[Lem.~2.4(b)]{VV2}.
\end{proof}


Specialization map in Definition \ref{def of sp crit_k} is compatible with stable envelopes,~i.e. 

\begin{Proposition}\label{prop: stab comp with sp crit_k}
Fix $(X,\sT,\sA,\bG_m,\mathsf f,\sg,\fC,\mathsf s)$ as in Setting \ref{setting of def of potentials2}. Suppose that there exist specialization maps $\mathsf{sp}\colon K^\sT(X,\mathsf f+\sg)\to K^\sT(X,\mathsf f)$ and $\mathsf{sp}^\sA\colon K^\sT(X^\sA,\mathsf f^\sA+\sg^\sA)\to K^\sT(X^\sA,\mathsf f^\sA)$. Then they are compatible with $K$-theoretic stable envelopes, namely the following diagram is commutative
\begin{align*}
\xymatrix{
K^\sT(X^\sA,\mathsf f^\sA+\sg^\sA) \ar[d]_{\Stab^{\mathsf s}_\fC} \ar[r]^-{\mathsf{sp}^\sA} & K^\sT(X^\sA,\mathsf f^\sA) \ar[d]^{\Stab^{\mathsf s}_\fC} \\
K^\sT(X,\mathsf f+\sg) \ar[r]^-{\mathsf{sp}} &  K^\sT(X,\mathsf f).
}
\end{align*}
\end{Proposition}

\begin{proof}
The proof is similar as in above. 
Consider the following diagram:
\begin{align*}
\xymatrix{
K^\sT(Z(\mathsf f^\sA+\sg^\sA))\ar[rrr]^{\mathsf{sp}^\sA} \ar[dr]^{\can} \ar[ddd]_{\Stab^{\mathsf s}_\fC} & & & K^\sT(Z(\mathsf f^\sA)) \ar[dl]_{\can} \ar[ddd]^{\Stab^{\mathsf s}_\fC}\\
& K^\sT(X^\sA,\mathsf f^\sA+\sg^\sA) \ar[d]_{\Stab^{\mathsf s}_\fC} \ar[r]^-{\mathsf{sp}^\sA} & K^\sT(X^\sA,\mathsf f^\sA) \ar[d]^{\Stab^{\mathsf s}_\fC} \\
& K^\sT(X,\mathsf f+\sg) \ar[r]^-{\mathsf{sp}} &  K^\sT(X,\mathsf f) &\\
K^\sT(Z(\mathsf f+\sg))\ar[rrr]^{\mathsf{sp}} \ar[ur]_{\can}  & & & K^\sT(Z(\mathsf f)) \ar[ul]^{\can}
}
\end{align*}
The outer square is commutative by Proposition \ref{prop: stab comp with sp zero loci}. The left and right trapezoids are commutative by Lemma \ref{lem on can cm w conv}. The top and bottom trapezoids are commutative by assumption (i.e.~\eqref{pre sp for crit k} commutes). Then it follows that
\begin{align*}
    \Stab^{\mathsf s}_\fC\circ\, \mathsf{sp}^\sA\circ \can=\mathsf{sp}\circ\Stab^{\mathsf s}_\fC\circ \can.
\end{align*}
Since $K^\sT(Z(\mathsf f^\sA+\sg^\sA))\xrightarrow{\can} K^\sT(X^\sA,\mathsf f^\sA+\sg^\sA)$ is surjective, we get $\Stab^{\mathsf s}_\fC\circ\, \mathsf{sp}^\sA=\mathsf{sp}\circ\Stab^{\mathsf s}_\fC$. 
\end{proof}

In below, we provide two existence results for specialization maps in critical $K$-theory. 

\subsubsection{Existence result I}
The first existence is due to P\u{a}durariu in his study of $K$-theoretic Hall algebras for quivers with potentials and relation to shuffle algebras.

\begin{Proposition}[{\cite[Prop.~3.6]{P}}]\label{prop: Tudor's sp}
Suppose that $\mathsf f=0$. Assume that
\begin{itemize}
    \item the $\sT$-fixed locus of $X$ lies in the zero locus of $\sg$,~i.e.~$X^\sT\subseteq  Z(\sg)$,
    \item $K^\sT(X)$ is torsion free over $K_\sT(\pt)$.
\end{itemize}
Then the specialization map $\mathsf{sp}\colon K^\sT(X,\sg)\to K^\sT(X)$ exists.
\end{Proposition}

\begin{Remark}
The argument in the proof of \cite[Prop.~3.6]{P} works for a reductive group $\sT$, so is Proposition \ref{prop: Tudor's sp}.
\end{Remark}

\subsubsection{Existence result II}

The second result is about how to induce specialization maps for GIT semistable loci from specialization maps of their ambient spaces.
This is motivated by \cite[Eqn.~(104)]{N} where the induced specialization map is used.

Let $G$ be a complex reductive group, $\sT$ be a torus such that 
$\sT\times G$ act on $\bP^n\times \bA^m$. Let $X\subset \bP^n\times \bA^m$ be a smooth projective-over-affine variety invariant under the action of $\sT \times G$, $\sg\colon X\to \bC$ be a $(\sT \times G)$-invariant regular function, which satisfies the weight conditions in Setting \ref{setting of def of potentials} with $\sT$ replaced with $\sT \times G$. Let $\cL=\cO(1)$ be a chosen linearization, and $X^{ss}$ be the semistable locus of $X$ with respect to the $G$ action and polarization $\cL$.

\begin{Proposition}\label{prop: induced sp}
Assume that $\sg\colon X\to \bC$ is not a constant function. Suppose that there exists a specialization map $\mathsf{sp}_X\colon K^{\sT\times G}(X,\sg)\to K^{\sT\times G}(X)$. Then there exists a specialization map $\mathsf{sp}_{X^{ss}}\colon K^{\sT\times G}(X^{ss},\sg)\to K^{\sT\times G}(X^{ss})$, and the diagram
\begin{equation}\label{compatible sp for stable locus}
\xymatrix{
K^{\sT\times G}(X,\sg)\ar[r]^-{\mathsf{sp}_X} \ar[d] & K^{\sT\times G}(X) \ar[d]\\
K^{\sT\times G}(X^{ss},\sg)\ar[r]^-{\mathsf{sp}_{X^{ss}}} & K^{\sT\times G}(X^{ss})
}
\end{equation}
is commutative, where the vertical arrows are the restriction to the semistable locus.

Moreover, if $\mathsf{sp}_X$ is injective, then so is $\mathsf{sp}_{X^{ss}}$.
\end{Proposition}

\begin{proof}
We know that the unstable locus $X\setminus X^{ss}$ admits a Kempf-Ness (KN) stratification \cite{Kir,DH}, see \S \ref{subsubsec: uniquess of Psi_H} for a description of KN stratification when $X$ is a linear representation of $G$. By induction on the open strata, it suffices to prove the proposition replacing $X^{ss}$ with $V:=X\setminus S$ for a closed KN strata $S$. We note that there exists a cocharacter $\lambda:\bC^*\to G$ such that $Z:=X^\lambda$ is a closed subvariety of $S$. Let $$Y=\Attr_\lambda(Z):=\{x\in X,\lim_{t\to 0}\lambda(t)\cdot x\in Z\}, $$ then $Y$ is a closed subvariety in $S$ and the action map $G\times_{P_\lambda}Y\to S$ induces an isomorphism. Here $P_\lambda=\{g\in G:\lim_{t\to 0} \lambda(t)g\lambda(t)^{-1}\text{ exists}\}$ is the parabolic subgroup associated to $\lambda$. We introduce notations
\begin{equation*}
\xymatrix{
Z \ar[r]^{\sigma} & Y \ar@/^/[l]^{\pi} \ar[r]^{j} & S
},
\end{equation*}
where $\sigma$, $j$ are natural closed immersions, and $\pi$ is the attraction map $\pi(y)=\lim_{t\to 0}\lambda(t)\cdot y$. We also let $$\fX:=[X/(\sT\times G)], \quad \mathfrak{V}=[V/(\sT\times G)], \quad \mathfrak{S}:=[S/(\sT\times G)], $$ and we put a subscript $0$ to denote the zero locus of $\sg$, for example $\fX_0:=[Z(\sg)/(\sT\times G)]$, $\mathfrak{V}_0=[(V\cap Z(\sg))/(\sT\times G)]$, and so on.

Fix $w\in \bZ$, and recall that Halpern-Leistner constructed a semiorthogonal decomposition (SOD) of the bounded derived category $\D^b(\fX)$ of coherent sheaves on $\fX$ in \cite[Thm.~3.35]{Hal}:
\begin{align*}
\D^b(\fX)=\left\langle \D^b_{\mathfrak{S}}(\fX)_{<w},\mathbf G_w,\D^b_{\mathfrak{S}}(\fX)_{\geq w}\right\rangle.
\end{align*}
Moreover, the restriction from $\fX$ to $\mathfrak{V}$ induces an equivalence:
\begin{align*}
\mathbf G_w\cong \D^b(\mathfrak{V}).
\end{align*}
On the other hand, since $\sg$ is $\sT\times G$ invariant, we have $Z_0=X_0^{\lambda}$, and $Y_0=\Attr_\lambda(Z_0)$, and the action map $G\times_{P_\lambda}Y_0\to S_0$ induces an isomorphism. We have the following induced maps
\begin{equation*}
\xymatrix{
Z_0 \ar[r]^{\sigma_0} & Y_0 \ar@/^/[l]^{\pi_0} \ar[r]^{j} & S_0
}
\end{equation*}
Notice that $Y_0\cong Z_0\times_Z Y$ where $Z_0\to Z$ is the inclusion and $Y\to Z$ is $\pi$. In particular, $\pi_0\colon Y_0\to Z_0$ is a locally trivial bundle of affine spaces, which is the condition (A) in \cite[pp.~7]{Hal}. We claim that the condition (L+) in \cite[pp.~7]{Hal} is also satisfied for the variety $X_0$ with KN strata $S_0$, that is, $L\sigma_0^*\mathbf L_{S_0/X_0}$ has nonnegative weights w.r.t. $\lambda$.
\begin{itemize}
    \item If $\mathsf g$ is not constantly zero on $S$, then $\mathsf g|_S\colon S\to \bA^1$ is a flat morphism (since $\mathsf g|_S$ can not be a nonzero constant function due to weight conditions in Setting \ref{setting of def of potentials}). In this case, $S_0$ is isomorphic to the derived zero locus $S\times^{\mathbf L}_{\bA^1}\{0\}$, and we have
\begin{align*}
L\sigma_0^*\mathbf L_{S_0/X_0}\cong L(Z_0\hookrightarrow Z)^* L\sigma^*\mathbf L_{S/X}.
\end{align*}
Since $L\sigma^*\mathbf L_{S/X}$ has nonnegative weights w.r.t. $\lambda$ (see \cite[Lem.~2.7]{Hal}), so is $L\sigma_0^*\mathbf L_{S_0/X_0}$.
\item If $\mathsf g|_S=0$, then there is an exact triangle
\begin{align*}
L\sigma_0^*L(S_0\hookrightarrow X_0)^*\mathbf L_{X_0/X}\to L\sigma_0^*\mathbf L_{S_0/X}\to L\sigma_0^*\mathbf L_{S_0/X_0}\to +1.
\end{align*}
Since $S_0=S$ and $\sigma_0=\sigma$, $L\sigma_0^*\mathbf L_{S_0/X}=L\sigma^*\mathbf L_{S/X}$ has nonnegative weights w.r.t. $\lambda$ (see \cite[Lem.~2.7]{Hal}).
Since $\mathsf g\colon X\to \bA^1$ is flat, we have $\mathbf L_{X_0/X}\cong \mathcal O_{X_0}[1]$. In particular, $L\sigma_0^*L(S_0\hookrightarrow X_0)^*\mathbf L_{X_0/X}\cong \mathcal O_{Z_0}[1]$ has zero $\lambda$ weight. It follows that $L\sigma_0^*\mathbf L_{S_0/X_0}$ has nonnegative weights w.r.t. $\lambda$.
\end{itemize}
In both cases, (L+) condition holds. Then by \cite[Thm.~3.35]{Hal}, we have an SOD
\begin{align*}
\D^b(\fX_0)=\left\langle \D^b_{\mathfrak{S}_0}(\fX_0)_{<w},\mathbf G_{0,w},\D^b_{\mathfrak{S}_0}(\fX_0)_{\geqslant w}\right\rangle,
\end{align*}
and the restriction from $\fX_0$ to $\mathfrak{V}_0$ induces an equivalence:
\begin{align*}
\mathbf G_{0,w}\cong \D^b(\mathfrak{V}_0).
\end{align*}
The canonical functor $$\D^b(\fX_0)\to \D^b(\fX_0)/\Perf(\fX_0)\cong \mathrm{MF}(\fX,\sg)$$ induces an SOD for the matrix factorization category (e.g.~\cite[Prop.~2.1]{P}):
\begin{align*}
    \mathrm{MF}(\fX,\sg)=\left\langle \mathrm{MF}_{\mathfrak{S}}(\fX,\sg)_{<w}, \mathrm{MF}(\mathbf G_{w},\sg), 
    \mathrm{MF}_{\mathfrak{S}}(\fX,\sg)_{\geqslant w}\right\rangle,
\end{align*}
where $\mathrm{MF}_{\mathfrak{S}}(\fX,\sg)_{<w}$, $\mathrm{MF}(\mathbf G_{w},\sg)$, and $\mathrm{MF}_{\mathfrak{S}}(\fX,\sg)_{\geq w}$ are the essential images of $\D^b_{\mathfrak{S}_0}(\fX_0)_{<w}$, $\mathbf G_{0,w}$, and $\D^b_{\mathfrak{S}_0}(\fX_0)_{\geqslant w}$ respectively. Moreover, the restriction from $\fX$ to $\mathfrak{V}$ induces an equivalence:
\begin{align*}
\mathrm{MF}(\mathbf G_{w},\sg)\cong \mathrm{MF}(\mathfrak{V},\sg).
\end{align*}
Taking Grothendieck groups, and we obtain a commutative diagram:
\begin{equation}\label{can comm with window}
\xymatrix{
K(\fX_0) \ar[r]^-{\can} & K(\fX,\sg) \\
K(\mathfrak{V}_0) \ar[u]^{s_w} \ar[r]^-{\can} & K(\mathfrak{V},\sg) \ar[u]_{\psi_w},
}
\end{equation}
where horizontal arrows are canonical maps, and vertical arrows are induced by the equivalences $\mathbf G_{0,w}\cong \D^b(\mathfrak{V}_0)$ and $\mathrm{MF}(\mathbf G_{w},\sg)\cong \mathrm{MF}(\mathfrak{V},\sg)$. Note that $(\mathfrak{V}_0\hookrightarrow \fX_0)^*s_w=\id$, and $(\mathfrak{V}\hookrightarrow \fX)^*\psi_w=\id$. Let $\mathsf{sp}_{\fX}=\mathsf{sp}_X\colon K(\fX,\sg)\to K(\fX)$ denote the specialization map. Then the pushforward map $(\mathfrak{V}_0\hookrightarrow \mathfrak{V})_*: K(\mathfrak{V}_0)\to K(\mathfrak{V})$ factors into 
\begin{align*}
(\mathfrak{V}_0\hookrightarrow \mathfrak{V})_*= (\mathfrak{V}_0\hookrightarrow \mathfrak{V})_*(\mathfrak{V}_0\hookrightarrow \fX_0)^*s_w&=(\mathfrak{V}\hookrightarrow\fX)^*(\fX_0\hookrightarrow \fX)_*s_w\\
\text{\scriptsize by def. of specialization}\;&=(\mathfrak{V}\hookrightarrow\fX)^*\circ\mathsf{sp}_{\fX}\circ\can \circ s_w\\
\text{\scriptsize by \eqref{can comm with window}}\;&=(\mathfrak{V}\hookrightarrow\fX)^*\circ\mathsf{sp}_{\fX}\circ \psi_w\circ\can 
\end{align*}
Therefore, there exists a specialization map $\mathsf{sp}_{\mathfrak{V}}\colon K(\mathfrak{V},\sg)$ to $K(\mathfrak{V})$ and $\mathsf{sp}_{\mathfrak{V}}=(\mathfrak{V}\hookrightarrow\fX)^*\circ\mathsf{sp}_{\fX}\circ \psi_w$.

Finally, Let $\bar\psi_w\colon K(\mathfrak{V})\to K(\fX)$ be the map obtained by taking Grothendieck groups to the equivalence $\mathbf G_w\cong \D^b(\mathfrak{V})$. Then $\bar\psi_w$ induces an isomorphism $K(\mathfrak{V})\cong K_0(\mathbf G_w)$. The image of $\mathsf{sp}_{\mathfrak{X}}\circ \psi_w$, which equals to the image of $\mathsf{sp}_{\mathfrak{X}}\circ \psi_w\circ\can=\mathsf{sp}_{\mathfrak{X}}\circ\can\circ s_w=(\fX_0\hookrightarrow \fX)_*s_w$, is contained in the Grothendieck group of $(\fX_0\hookrightarrow \fX)_*(\mathbf G_{0,w})$. Notice that $(\fX_0\hookrightarrow \fX)_*(\mathbf G_{0,w})$ is contained in $\mathbf G_w$; thus image of $\mathsf{sp}_{\mathfrak{X}}\circ \psi_w$ is contained in the image of $\bar\psi_w$. In particular, $\bar\psi_w\circ (\mathfrak{V}\hookrightarrow\fX)^*$ is identity on the image of $\mathsf{sp}_{\mathfrak{X}}\circ \psi_w$, and it follows that $$\bar\psi_w\circ \mathsf{sp}_{\mathfrak{V}}=\bar\psi_w\circ (\mathfrak{V}\hookrightarrow\fX)^*\circ\mathsf{sp}_{\mathfrak{X}}\circ \psi_w=\mathsf{sp}_{\mathfrak{X}}\circ \psi_w.$$ Now assume that $\mathsf{sp}_{\mathfrak{X}}\colon K(\fX,\sg)\to K(\fX)$ is injective, then $\mathsf{sp}_{\mathfrak{X}}\circ \psi_w$ is also injective since $(\mathfrak{V}\hookrightarrow \fX)^*\psi_w=\id$, and the equation $\bar\psi_w\circ \mathsf{sp}_{\mathfrak{V}}=\mathsf{sp}_{\mathfrak{X}}\circ \psi_w$ implies that $\mathsf{sp}_{\mathfrak{V}}$ is injective.
\end{proof}

Combine Propositions \ref{prop: Tudor's sp} and \ref{prop: induced sp}, we obtain the following.

\begin{Corollary}
Let $X$ be a linear representation of $\sT\times G$, $\theta\colon G\to \bC^*$ be a character of $G$, $\sg\colon X\to \bC$ be a $(\sT \times G)$-invariant regular function, which satisfies the weight conditions in Setting \ref{setting of def of potentials} with $\sT$ replaced with $\sT \times G$. Let $X^{ss}$ be the $\theta$-semistable locus of $X$ with respect to the $G$ action. Assume $X^{\sT\times G}\subset Z(\sg)$, then there exist specialization maps $\mathsf{sp}_X$ and $\mathsf{sp}_{X^{ss}}$ for $X$ and $X^{ss}$ respectively, and the diagram \eqref{compatible sp for stable locus} is commutative. Moreover, if $\mathsf{sp}_X$ is injective, then so is $\mathsf{sp}_{X^{ss}}$.
\end{Corollary}

\begin{Example}
Let $X=R(Q,\bv,\underline{\bd})$ for a quiver $Q$ with gauge dimension $\bv$ and framing dimension $\underline{\bd}$, $G=\prod_{i\in Q_0}\GL(\bv_i)$ be the gauge group, $\sT$ be a flavour torus (see \S \ref{sec: quiver var}). Suppose that every edge loop has onzero $\sT$ weight, then $X^{\sT\times G}=\{0\}\subset Z(\sg)$ for any $\sg$ that satisfies the weight conditions in Setting \ref{setting of def of potentials} with $\sT$ replaced with $\sT \times G$. Then we have specializations $$\mathsf{sp}_{\fM}\colon K^\sT(\fM(\bv,\underline{\bd}),\sg)\to K^\sT(\fM(\bv,\underline{\bd})), \quad \mathsf{sp}_{\cM}\colon K^\sT(\cM(\bv,\underline{\bd}),\sg)\to K^\sT(\cM(\bv,\underline{\bd})),$$ where $\fM(\bv,\underline{\bd})=[X/G]$ and $\cM(\bv,\underline{\bd})=[X^{ss}/G]$. Moreover, the diagram
\begin{equation*}
\xymatrix{
K^\sT(\fM(\bv,\underline{\bd}),\sg) \ar[r]^-{\mathsf{sp}_{\fM}} \ar[d] & K^\sT(\fM(\bv,\underline{\bd})) \ar[d]\\
K^\sT(\cM(\bv,\underline{\bd}),\sg)\ar[r]^-{\mathsf{sp}_{\cM}} & K^\sT(\cM(\bv,\underline{\bd}))
}
\end{equation*}
is commutative, where the vertical arrows are the restriction to the semistable locus.

As an application, we give a proof of \cite[Conj.~1.9]{N}, which we thank Andrei Negu\c{t} for helpful communications. 

Let $(Q,\sW)$ be the quiver with potential associated to a symmetrizable Kac-Moody algebra $\mathfrak{g}_{A,D}$ in \cite[Def.~A.1]{COZZ1}, $\sT$ be the torus $\bC^*_\hbar$ that scales the arrows in \textit{loc.\,cit.}, and set $\sg=\tr\sW\colon R(Q,\bv,\underline{\bd})\to \bC$. Then the argument of \cite[Lem.~A.4]{COZZ3} can be applied to $K$-theory and the localized specialization map
\begin{align*}
\mathsf{sp}_{\fM}\colon K^\sT(\fM(\bv,\underline{\bd}),\sg)\otimes_{\bC[\hbar^\pm]}\bC(\hbar) \to K^\sT(\fM(\bv,\underline{\bd}))\otimes_{\bC[\hbar^\pm]}\bC(\hbar)
\end{align*}
is injective. According to the above discussion, the localized specialization map 
\begin{align*}
\mathsf{sp}_{\cM}\colon K^\sT(\cM(\bv,\underline{\bd}),\sg)\otimes_{\bC[\hbar^\pm]}\bC(\hbar) \to K^\sT(\cM(\bv,\underline{\bd}))\otimes_{\bC[\hbar^\pm]}\bC(\hbar)
\end{align*}
is also injective. This proves the injectivity statement in \cite[Conj.~1.9]{N}. The surjectivity statement in \cite[Conj.~1.9]{N} follows from the surjectivity of restriction-to-open map for critical $K$-theory (see \S \ref{subsubsec: excision}) and torus localization \eqref{prop on loc of crit k}.
\end{Example}

\subsubsection{Existence result III}\label{sec: sp for crit k existence by dim red}

The third existence result is relevant to examples considered in \cite{COZZ2}.   
The idea of construction uses dimensional reduction to the case when there is no potential.  

\begin{Construction}\label{const of dim r}
Let $G$ be a linear reductive group, $\sT$ be a torus, with a direct sum of finite dimensional linear $(G\times \sT\times \bG_m)$-representations: 
$$W=V\oplus U.$$ 
Let $\pi\colon W\to V$ be the projection with a pair of $(G\times \sT)$-equivariant sections:
$$s_0,s_\epsilon\in \Gamma(V, \underline{U}^\vee)$$ 
of the dual bundle $\underline{U}^\vee:=V\times U^\vee\to V$ of $\pi$. Assume that $s_0$ is $\bG_m$-equivariant, and there exists a decomposition $$s_\epsilon=\sum_{i=1}^m s_{\epsilon,i}$$ such that $\bG_m$ scales $s_{\epsilon,i}$ with weight $-\mathsf n_i<0$, i.e. $u^{\mathsf n_i}s_{\epsilon,i}$ is $\bG_m$-equivariant ($i=1,\cdots, m$).
Define the section 
$$s_t:=s_0+\sum_{i=1}^mt^{\mathsf n_i}s_{\epsilon,i}, \quad t\in \bA^1.$$ 
Let $Z^{\der}(s_t)$ be the derived zero locus of $s_t$, 
which fits into the following diagram 
\begin{equation*}
\xymatrix{
 \pi^{-1}(Z^{\der}(s_t))  \ar[r]^{\quad \quad \iota_t} \ar[d]^{ } \ar@{}[dr]|{\Box}  & W \ar[d]^{\pi}  \\
Z^{\der}(s_t) \ar[r]^{ } & V.
}
\end{equation*}
Let $e$ be the coordinate of the fiber of $\pi$ and define $(G\times \sT)$-invariant regular function:
\begin{equation*}
\sw_t:=\langle e,s_t \rangle\colon W\to \C.
\end{equation*}
Choose a stability condition such that the semistable locus equals the stable locus:
$$W^{ss}=W^{s}\neq \emptyset.$$ 
By abuse of notations, we write the descent of $\sw_{t}$ to be 
$$\sw_{t}\colon X:=W^s/G\to \C. $$ 
Then $(X,\sT,\sA,\mathsf f=\sw_0,\sg=\sw_1-\sw_0)$ fits into Setting \ref{setting of def of potentials}.
\end{Construction}

By applying dimensional reduction (Proposition \ref{prop on dim red for kgp}, Remark \ref{rmk on teqiuv}), we obtain 
\begin{align*}
\iota_{t*}\colon K^\sT((\pi^{-1}(Z^{\der}(s_t)))^s/G) \xrightarrow{\cong} K^\sT(X,\sw_t), \,\,\, \forall\,\, t\in  \bA^1.
\end{align*}
There is a standard specialization map for the left hand side: 
$$\mathsf{sp}\colon K^\sT((\pi^{-1}(Z^{\der}(s_1)))^s/G)\to K^\sT((\pi^{-1}(Z^{\der}(s_0)))^s/G), $$
defined by an extension to $K$-theory class of the family $\pi^{-1}(Z^{\der}(s_t)))^s/G$ ($t\in \A^1$), followed by the Gysin pullback to the special fiber $\pi^{-1}(Z^{\der}(s_0)))^s/G$ (e.g.~\cite[\S5.3]{CG}, Definition \ref{def of sp for zero loci_k}).  
\begin{Definition}\label{def on sp of kgp}
Under the setting of Construction \ref{const of dim r}, 
we define $\nu$ as the unique map which makes 
the following diagram commute: 
\begin{equation}\label{equ on sp of kgp}
\xymatrix{
K^\sT(X,\sw_1)  \ar[r]^{\nu}  & K^\sT(X,\sw_0)   \\
K^\sT((\pi^{-1}(Z^{\der}(s_1)))^s/G) \ar[r]^{\mathsf{sp}} \ar[u]_{\cong}^{\iota_{1*}} & K^\sT((\pi^{-1}(Z^{\der}(s_0)))^s/G). \ar[u]^{\cong}_{\iota_{0*}}
} 
\end{equation}

\end{Definition}

\begin{Proposition}\label{prop: nu is sp}
The map $\nu$ \eqref{equ on sp of kgp} is the specialization map for critical $K$-theory (Definition \ref{def of sp crit_k}).
\end{Proposition}

\begin{proof}
We have the following diagram
$$
\xymatrix{
 & K^\sT(Z(\sw_1)) \ar@{->>}[dr]^{\can} \ar@/^2pc/[ddd]^{\mathsf{sp}} & \\
K^\sT((\pi^{-1}(Z^{\der}(s_1)))^s/G) \ar[ru]^{i_{1*}}  \ar[d]_{\mathsf{sp}}^{ }  \ar[rr]^{\quad \,\, \iota_{1*}}_{\quad \,\, \cong}  &  & K^\sT(W^s/G,\sw_1) \ar[d]_{}^{\nu}  \\
K^\sT((\pi^{-1}(Z^{\der}(s_0)))^s/G)  \ar[rd]_{i_{0*}}  \ar[rr]^{\quad \,\, \iota_{0*}}_{\quad \,\,  \cong} & &  K^\sT(W^s/G,\sw_0) \\
 & K^\sT(Z(\sw_0)). \ar@{->>}[ur]_{\can} & 
} 
$$
Here the upper and lower triangles commute as $\pi^{-1}(Z^{\der}(s_t))\stackrel{i_{t}}{\hookrightarrow}   Z(\sw_{t})$, and the canonical map is simply the pushforward map.  The middle square commutes by Definition \ref{def on sp of kgp}. 
The commutativity of  
\begin{equation}\label{comm diag on can sp}
\xymatrix{
K^\sT(Z(\sw_{1}))  \ar[r]^{\can \quad  \,\,\, } \ar[d]_{\mathsf{sp}}^{ } & K^\sT(W^s/G,\sw_{1}) \ar[d]_{}^{\nu}  \\
K^\sT(Z(\sw_{0})) \ar[r]^{\can \quad \,\,\,}  & K^\sT(W^s/G,\sw_{0}),
} 
\end{equation}
is equivalent to the equality of maps
$$\mathsf{sp}\circ (\iota_{1*}^{-1}\circ \can)=(\iota_{0*}^{-1}\circ \can)\circ \mathsf{sp}\colon K^\sT(Z(\sw_{1}))\to K^\sT((\pi^{-1}(Z^{\der}(s_0)))^s/G).$$
Let $\mathrm{Kos}(\tau,s_t)$ denote the Koszul factorization \eqref{Kos fact}. Then 
$$(\iota_{t*}^{-1}\circ \can)(-)=\left(i_{Z(\sw_{t})}^*\mathrm{Kos}(\tau,s_t)^\vee\right)\otimes_{\mathcal{O}_{Z(\sw_{t})}}(-)\colon 
K^\sT(Z(\sw_{t})) \to K^\sT((\pi^{-1}(Z^{\der}(s_t)))^s/G), $$
where $i_{Z(\sw_{t})}\colon Z(\sw_{t})\to W^s/G$ is the inclusion. 
Therefore, for any $[\eE]\in K^\sT(Z(\sw_{1}))$, we have 
\begin{align*}\mathsf{sp}\circ (\iota_{1*}^{-1}\circ \can)[\eE]&=\mathsf{sp}\left(i_{Z(\sw_{1})}^*\mathrm{Kos}(\tau,s_1)^\vee\otimes_{\mathcal{O}_{Z(\sw_{1})}}[\eE] \right)  \\
&=\mathsf{sp}\left(i_{Z(\sw_{1})}^*\mathrm{Kos}(\tau,s_1)^\vee)\right)\otimes _{\mathcal{O}_{Z(\sw_{0})}} \mathsf{sp}([\eE]) \\
&=\left(i_{Z(\sw_{0})}^*\mathrm{Kos}(\tau,s_0)^\vee)\right)\otimes _{\mathcal{O}_{Z(\sw_{0})}} \mathsf{sp}([\eE]) \\
&=(\iota_{0*}^{-1}\circ \can)\circ \mathsf{sp} ([\eE]),\end{align*}
where we use the fact that the specialization of tensor products 
is the tensor product of specializations. 
\end{proof}

\begin{Example}\label{ex of sp for crit_k}
Let $G=\GL_n$ be the gauge group and consider representation spaces $W, V, U$ of the following quivers
\begin{equation*}
W:\quad\xymatrix{
\text{\tiny{\boxed{N}}}  \ar@/^0.3pc/[r]^{A}  & \ar@/^0.3pc/[l]^{B} \text{\textcircled{n}}  \ar@(dr,ur)_{\xi} \ar@(ru,lu)_{x} \ar@(ld,rd)_{y} }
\qquad
V:\quad\xymatrix{
\text{\tiny{\boxed{N}}}  \ar@/^0.3pc/[r]^{A}  &  
\text{\textcircled{n}}   \ar@(ld,rd)_{y}  \ar@(dr,ur)_{\xi} }
\qquad
U:\quad
\xymatrix{
\text{\tiny{\boxed{N}}}    & \ar@/^0.3pc/[l]^{B}  \text{\textcircled{n}} \ar@(ru,lu)_{x}  }\,\,.
\end{equation*}
We define \textit{cyclic stability} on $W$, so the stable locus $W^s$ are those representations which satisfy 
$\C\langle x,\xi,y\rangle A(\C^N)=\C^n$,
and define the GIT quotient $$X:=W^s/\GL_n.$$ 
We choose a $\bG_m$-action on $W$ by
\begin{align*}
    u\cdot(\xi,x,y,A,B)=(u^{-2}\xi,ux,uy,uA,uB),
\end{align*}
and choose a $\sT$-action on $W$ which commutes with $(\GL_n\times \bG_m)$-action, so that 
$$W=V\oplus U$$ is a direct sum of $(G\times \sT\times \bG_m)$-representations. For any $t\in \bA^1$, we define a map 
$$s_t\colon V\to U^\vee, \quad (y,\xi,A)\mapsto ([y,\xi],\,\xi A-t^{2}A\Xi),$$
where $\Xi$ is a fixed $N$ by $N$ matrix. The potential function is determined by $s_t$ via the trace pairing:
\begin{equation*}
\sw_t:=\tr(\xi[x,y]+\xi A B )-t^{2}\cdot\tr(A\Xi B)\colon W\to \C.
\end{equation*} 
Then $\sw_0$ is $\bG_m$-invariant and $\bG_m$ scales $\sw_1-\sw_0$ with weight $-2$.
These data fit into the setting of Construction \ref{const of dim r}. By Proposition \ref{prop: nu is sp}, the specialization map
\begin{align*}
\mathsf{sp}\colon K^{\sT}(X,\sw_{1}) \to K^{\sT}(X,\sw_{0})
\end{align*} 
exists in critical $K$-theory. This map is useful to relate different modules of quantum toroidal algebra $U_{q_1,q_2,q_3}(\widehat{\mathfrak{gl}}_1)$.
\end{Example}

\section{Vector bundles and stable envelopes}\label{sect on stab and vb}

Let $(X,\sw,\sT,\sA)$ be as in Setting \ref{setting of stab}, and take a $\sT$-equivariant vector bundle $E$ on $X$ with projection 
\begin{equation}\label{equ on vb yx}\pi\colon Y:=\mathrm{Tot}(E)\to X. \end{equation}
Define $i\colon X\hookrightarrow Y$ as the zero section map. We note that $F\mapsto i^{-1}(F)$ gives a bijection  
$$\mathrm{Fix}_\sA(Y)\cong \mathrm{Fix}_\sA(X), $$
where the inverse map is given by $G\mapsto (\pi^{-1}(G))^\sA$.

In this section, we prove several compatibility results for stable envelopes of $(X,\sw,\sT,\sA)$ and $(Y,\widetilde\sw,\sT,\sA)$, 
when $\widetilde\sw$ is a $\sT$-invariant function on $Y$ such that $\widetilde{\sw}|_X=\sw$, and $E$ is attracting or repelling. 
We show that the closure of the attracting set of the diagonal gives stable envelope correspondence on $(X,\sw,\sT,\sA)$ when $Y$ is a symmetric quiver variety 
and $E$ has a decomposition into attracting and repelling subbundles (Theorem \ref{thm:attr plus repl}). This explicit description is particularly useful in the manipulation of 
stable envelopes. 
Finally, we prove a triangle lemma for attracting closure correspondence without attracting or repelling condition on $E$ (Proposition \ref{tri lem for unsym fram}),
which will be used to prove the triangle lemma of Hall envelopes in \S \ref{sect on hall comp with hall}.

\subsection{Definitions and properties}

\begin{Definition}
An $\sA$-equivariant vector bundle $E$ on $X$ is called \textit{attracting} (resp.~\textit{repelling}) with respect to a chamber $\fC$ if $\forall\, F\in \Fix_\sA(X)$, $E|_F$ has nonnegative (resp.~nonpositive) $\sA$-weights with respect to $\fC$.
\end{Definition}

\begin{Proposition}\label{prop:vb and stab_attr}
Assume that $E$ is attracting with respect to $\fC$. Then cohomological stable envelope exists for $(X,\sw,\sT,\sA,\fC)$ if and only if it exists for $(Y,\sw\circ\pi,\sT,\sA,\fC)$. When they exist, the following diagram commutes
\begin{equation}\label{vb and stab_attr_coh}
\xymatrix{
H^\sT(Y^\sA,(\sw\circ \pi)^\sA) \ar[r]^-{\Stab_\fC} & H^\sT(Y,\sw\circ \pi)\\
H^\sT(X^\sA,\sw^\sA) \ar[r]^{\,\,\Stab_\fC} \ar[u]^{\pi^{\sA*}} & H^\sT(X,\sw) \ar[u]_{\pi^{*}}.
}
\end{equation}
Similarly, let $\mathsf s\in \Pic_\sA(X)\otimes_\bZ \bR$ be a generic slope, then $K$-theoretic stable envelope exists for $(X,\sw,\sT,\sA,\fC,\mathsf s)$ if and only if it exists for $(Y,\sw\circ\pi,\sT,\sA,\fC,\mathsf s)$. When they exist, the following diagram commutes
\begin{equation}\label{vb and stab_attr_k}
\xymatrix{
K^\sT(Y^\sA,(\sw\circ \pi)^\sA) \ar[r]^-{\Stab^{\mathsf s}_\fC} & K^\sT(Y,\sw\circ \pi)\\
K^\sT(X^\sA,\sw^\sA) \ar[r]^{\,\,\Stab^{\mathsf s}_\fC} \ar[u]^{\pi^{\sA *}} & K^\sT(X,\sw) \ar[u]_{\pi^{*}}.
}
\end{equation}
\end{Proposition}

\begin{proof}
We prove the cohomology version. Assume that cohomological stable envelope exists for $(X,\sw,\sT,\sA,\fC)$. Define $$\mathcal S=\pi^*\circ \Stab_{\fC}\colon H^{\sT}(X^\sA,\sw^\sA) \to H^{\sT}(Y,\sw\circ \pi).$$ 
For an arbitrary $F\in \Fix_\sA(Y)$, and take an arbitrary $\gamma\in H^{\sT/\sA}(i^{-1}(F),\sw^\sA)$, we claim that $\mathcal S(\gamma)$ satisfies the axioms (ii) and (iii) in Definition \ref{def of stab coho} and axiom (i') in Remark \ref{weak axiom coh}, namely
\begin{itemize}
\setlength{\parskip}{1ex}
    \item[(1)] $\mathcal S(\gamma)$ is supported on $\Attr^{\le}_\fC(F)$;
    \item[(2)] $\mathcal S (\gamma)\big|_{F} = e^\sT (N_{F / Y}^-) \cdot \pi^{\sA*}(\gamma)$;
    \item[(3)] For any $F'\neq F$, the inequality $\deg_\sA \mathcal S(\gamma)\big|_{F'} < \deg_\sA e^\sT (N_{F'/Y}^-)$ holds.
\end{itemize}
Here we choose the ample partial order $\le$ as in Remark \ref{partial order by line bundle}. Since $\pi^{\sA*}\colon H^\sT(X^\sA,\sw^\sA)\to H^\sT(Y^\sA,(\sw\circ \pi)^\sA)$ is an isomorphism with the inverse map given by $i^{\sA *}$, the above three properties imply that $\mathcal S \circ i^{\sA *}$ is a cohomological stable envelope for $(Y,\sw\circ\pi,\sT,\sA,\fC)$, moreover the diagram \eqref{vb and stab_attr_coh} commutes.

Since $E$ is attracting, for all $G\in \Fix_\sA(Y)$, we have 
$$\Attr_\fC(G)=\pi^{-1}(\Attr_\fC(i^{-1}(G))).$$ 
Therefore $\mathcal S(\gamma)$ is supported on $$\pi^{-1}(\Attr^{\le}_\fC(i^{-1}(F)))=\pi^{-1}\left(\bigcup_{i^{-1}(F)\le i^{-1}(G)}\Attr_\fC(i^{-1}(F))\right)=\bigcup_{F\le G}\Attr_\fC(G)=\Attr^{\le}_\fC(F).$$
This proves (1). For (2), we notice that 
$$\mathcal S (\gamma)\big|_{F}=\pi^{\sA*}\left(\Stab_\fC(\gamma)|_{i^{-1}(F)}\right)=\pi^{\sA*}\left(e^\sT\left(N^-_{i^{-1}(F)/X}\right)\cdot \gamma\right)= e^\sT (N_{F / Y}^-) \cdot \pi^{\sA*}(\gamma),$$
where in the last equation we use the fact that $E$ is attracting so that $N^-_{F/Y}\cong \pi^{\sA*}\left(N^-_{i^{-1}(F)/X}\right)$. 

For (3), we notice that $\mathcal S (\gamma)\big|_{F'}=\pi^{\sA*}\left(\Stab_\fC(\gamma)|_{i^{-1}(F')}\right)$. Since $\pi^{\sA*}$ preserves $\deg_\sA$, we have $$\deg_\sA \mathcal S (\gamma)\big|_{F'}=\deg_\sA\left(\Stab_\fC(\gamma)|_{i^{-1}(F')}\right)<\deg_\sA e^\sT \left(N_{i^{-1}(F')/X}^-\right)=\deg_\sA e^\sT \left(N_{F'/Y}^-\right).$$
This proves one direction. For the other direction, assume that cohomological stable envelope exists for $(Y,\sw\circ\pi,\sT,\sA,\fC)$, then a similar argument as above shows that $i^*\circ \Stab_\fC\circ \pi^{\sA*}$ is a cohomological stable envelope for $(X,\sw,\sT,\sA,\fC)$. This proves the cohomology version.

The $K$-theory version is proven similarly, and we omit the details.
\end{proof}
\begin{Proposition}\label{prop:vb and stab_repl}
Assume that $E$ is repelling with respect to $\fC$. Let $\widetilde\sw$ be a $\sT$-invariant function on $Y$ such that $\widetilde\sw|_X=\sw$. Moreover, assume that cohomological stable envelopes exist for both $(X,\sw,\sT,\sA,\fC)$ and $(Y,\widetilde\sw,\sT,\sA,\fC)$. Then the following diagram is commutative
\begin{equation}\label{vb and stab_repl_coh}
\xymatrix{
H^\sT(Y^\sA,\widetilde\sw^\sA) \ar[r]^-{\Stab_\fC} & H^\sT(Y,\widetilde\sw)\\
H^\sT(X^\sA,\sw^\sA) \ar[r]^{\,\,\Stab_\fC} \ar[u]^{i^{\sA}_*} & H^\sT(X,\sw) \ar[u]_{i_*}.
}
\end{equation}
Similarly, let $\mathsf s\in \Pic_\sA(X)\otimes_\bZ \bR$ be a generic slope, and assume that $K$-theoretic stable envelopes exist for both $(X,\sw,\sT,\sA,\fC,\mathsf s')$ and $(Y,\widetilde\sw,\sT,\sA,\fC,\mathsf s)$ where $\mathsf s'=\mathsf s\otimes (\det E)^{1/2}$. Then the following diagram is commutative
\begin{equation}\label{vb and stab_repl_k}
\xymatrix{
K^\sT(Y^\sA,\widetilde\sw^\sA) \ar[r]^-{\Stab^{\mathsf s}_\fC} & K^\sT(Y,\widetilde\sw)\\
K^\sT(X^\sA,\sw^\sA) \ar[r]^{\,\,\Stab^{\mathsf s'}_\fC} \ar[u]^{i^{\sA}_*} & K^\sT(X,\sw) \ar[u]_{i_*}.
}
\end{equation}
\end{Proposition}

\begin{proof}
We prove the cohomology version ($K$-theory version can be proven similarly). 
Define $$\mathcal S=i_*\circ \Stab_{\fC}\colon H^{\sT}(X^\sA,\sw^\sA) \to H^{\sT}(Y,\widetilde\sw).$$ 
For an arbitrary $F\in \Fix_\sA(Y)$, and take an arbitrary $\gamma\in H^{\sT/\sA}(i^{-1}(F),\sw^\sA)$, we need to show that $\mathcal S(\gamma)$ satisfies the axioms (ii) and (iii) in Definition \ref{def of stab coho} and axiom (i') in Remark \ref{weak axiom coh}, namely
\begin{itemize}
\setlength{\parskip}{1ex}
    \item[(1)] $\mathcal S(\gamma)$ is supported on $\Attr^{\le}_\fC(F)$;
    \item[(2)] $\mathcal S (\gamma)\big|_{F} = e^\sT (N_{F / Y}^-) \cdot i^{\sA}_*(\gamma)$;
    \item[(3)] For any $F'\neq F$, the inequality $\deg_\sA \mathcal S(\gamma)\big|_{F'} < \deg_\sA e^\sT (N_{F'/Y}^-)$ holds.
\end{itemize}
Here we choose the ample partial order $\le$ as in Remark \ref{partial order by line bundle}. 

Since $E$ is repelling, we have $i^{-1}(\Attr_\fC(G))=\Attr_\fC(i^{-1}(G))$ for all $G\in \Fix_\sA(Y)$; therefore $\mathcal S(\gamma)$ is supported on $$i(\Attr^{\le}_\fC(i^{-1}(F)))=i\left(\bigcup_{i^{-1}(F)\le i^{-1}(G)}\Attr_\fC(i^{-1}(F))\right)=\bigcup_{F\le G}X\cap\Attr_\fC(G)=X\cap\Attr^{\le}_\fC(F)\subseteq  \Attr^{\le}_\fC(F).$$
This prove (1). For (2), we notice that 
$$\mathcal S (\gamma)\big|_{F}=e^\sT(E|_{F}^{\mathrm{mov}})\cdot i^{\sA}_*\left(\Stab_\fC(\gamma)|_{i^{-1}(F)}\right)=e^\sT(E|_{F}^{\mathrm{mov}})\cdot i^{\sA}_*\left(e^\sT\left(N^-_{i^{-1}(F)/X}\right)\cdot \gamma\right)= e^\sT (N_{F / Y}^-) \cdot i^{\sA}_*(\gamma),$$
where in the last equation we have used the fact that $E$ is repelling so that 
$$N^-_{F/Y}= \pi^{\sA^*}\left(N^-_{i^{-1}(F)/X}\right)+E|_{F}^{\mathrm{mov}}\in K^\sT(F). $$ 
For (3), we notice that $\mathcal S (\gamma)\big|_{F'}=e^\sT\left(E|_{F'}^{\mathrm{mov}}\right)\cdot i^{\sA}_*\left(\Stab_\fC(\gamma)|_{i^{-1}(F')}\right)$. Since $i^{\sA}_*$ preserves $\deg_\sA$, 
we have 
\begin{align*}\deg_\sA \mathcal S (\gamma)\big|_{F'}&=\deg_\sA e^\sT(E|_{F'}^{\mathrm{mov}})+\deg_\sA\left(\Stab_\fC(\gamma)|_{i^{-1}(F')}\right)\\
&<\deg_\sA e^\sT(E|_{F'}^{\mathrm{mov}})+\deg_\sA e^\sT \left(N_{i^{-1}(F')/X}^-\right)=\deg_\sA e^\sT (N_{F'/Y}^-).  \qedhere \end{align*}

\end{proof}

The following lemma will be used in Appendix \ref{sec:proof of prop psi}.
\begin{Lemma}\label{lem:repl vb induce stab_K}
Assume that $E$ is repelling with respect to $\fC$. Suppose that $F\in \Fix_\sA(X)$ is a fixed component in $X$ such that $E|_F$ is moving. Let $\bC^*_u$ act on $Y$ by fixing $X$ and scaling the fibers of $E$ with weight $1$. Suppose that $K$-theoretic stable envelope exists for $(Y,\sw\circ\pi,\sT\times \bC^*_u,\sA,\fC,\mathsf s)$, then there exists a $K$-theoretic stable envelope map\footnote{This means a map satisfying axioms (i)-(iii) in Definition \ref{def of stab k} for 
component $F$.} $\Stab^{\mathsf s'}_\fC\colon K^\sT(F,\sw^\sA)\to K^\sT(X,\sw)$, where $\mathsf s'=\mathsf s\otimes(\det E)^{ 1/2}$.
\end{Lemma}

\begin{proof}
By assumption, the projection $(\pi^{-1}(F))^\sA\to F$ is an isomorphism. Let $\sigma:\bC^*\to \sA$ be a cocharacter in the chamber $\fC$, then we claim that 
$\Stab^{\mathsf s'}_\sigma\colon K^\sT(F,\sw^\sA)\to K^\sT(X,\sw)$ (in Definition \ref{def of stab1}) satisfies axioms (i)-(iii) in Definition \ref{def of stab k} with respect to the whole chamber $\fC$. 

It is enough to check axiom (iii). We note that $\bC^*_u$ acts on $X$ trivially, so $\Stab^{\mathsf s'}_\sigma$ extends to a $K$-theoretic stable envelope for $(X,\sw,\sT\times \bC^*_u,\bC^*,+)$. Proposition \ref{prop:vb and stab_repl} implies that
\begin{align*}
    i_*\circ \Stab^{\mathsf s'}_\sigma=\Stab^{\mathsf s}_\fC\circ \,i^\sA_*.
\end{align*}
As $E|_F$ is moving, $i^\sA_*$ is an isomorphism when restricted to $F$, we see that for all $\gamma\in K^{\sT\times \bC^*_u}(F,\sw^\sA)$ and for all $F'\in\Fix_\sA(X), F'\neq F$, there is a strict inclusion of polytopes $$\deg_\sA e^{\sT\times \bC^*_u}_K (E|_{F'})\cdot \Stab^{\mathsf s'}_{\sigma}(\gamma)\big|_{F'} \subsetneq  \deg_\sA e^{\sT\times \bC^*_u}_K \left(N_{\pi^{-1}(F')^\sA / Y}^-\right)+\mathrm{shift}_{F'}-\mathrm{shift}_{F},$$ where $\mathrm{shift}_{F}:=\mathrm{weight}_\sA\left(\det(N_{F/X}^-)^{1/2}\otimes\mathsf s'|_F\right)$. It follows that 
$$\deg_\sA\Stab^{\mathsf s'}_{\sigma}(\gamma)\big|_{F'} \subseteq  \deg_\sA e^{\sT\times \bC^*_u}_K (N_{F'/X}^-)+\mathrm{shift}_{F'}-\mathrm{shift}_{F}.$$
Since $\mathsf s$ is generic, the above inclusion must be strict by Remark \eqref{rmk: generic slope}. Evaluating at $u=1$, $\Stab^{\mathsf s'}_\sigma$ induces a $K$-theoretic stable envelope map $\Stab^{\mathsf s'}_\fC\colon K^\sT(F,\sw^\sA)\to K^\sT(X,\sw)$.
\end{proof}

\begin{Remark}\label{rmk:repl vb induce stab_K triangle}
Under the assumption of Lemma \ref{lem:repl vb induce stab_K}, suppose that $\fC'$ is a face of $\fC$ and let $\sA'$ be the subtorus of $\sA$ associated with $\fC'$, and we also assume the existence of $K$-theoretic stable envelope for $(Y,\sw\circ\pi,\sT\times \bC^*_u,\sA',\fC',\mathsf s)$. Assume moreover that $F$ is an $\sA'$-fixed component and $E|_F$ is moving with respect to $\sA'$-action. Then by Lemma \ref{lem:repl vb induce stab_K} we have $K$-theoretic stable envelope maps $\Stab^{\mathsf s'}_\fC$ and $\Stab^{\mathsf s'}_{\fC'}$ 
from $K^\sT(F,\sw^\sA)$ to $K^\sT(X,\sw)$, where $\mathsf s'=\mathsf s\otimes(\det E)^{1/2}$. Moreover, there exists an obvious $K$-theoretic stable envelope map $\Stab^{\mathsf s''}_{\fC/\fC'}\colon K^\sT(F,\sw^\sA)\to K^\sT(F,\sw^\sA)$ for the $\mathsf s''$ determined by Lemma \ref{triangle lemma for K}, which is the identity map. Therefore by Lemma \ref{triangle lemma for K}, we have $\Stab^{\mathsf s'}_\fC=\Stab^{\mathsf s'}_{\fC'}$.
\end{Remark}

Combining Propositions \ref{prop:vb and stab_attr}, \ref{prop:vb and stab_repl}, we obtain the following.

\begin{Proposition}
Assume that there exists a decomposition $E=E_+\oplus E_-$ such that $E_+$ is attracting and $E_-$ is repelling with respect to $\fC$. Suppose that $\sw'$ is a $\sT$-invariant function on $\mathrm{Tot}(E_-)$ such that $\sw'|_{X}=\sw$, and define $\widetilde\sw:=(Y\to \mathrm{Tot}(E_-))^*(\sw')$. Moreover, assume that cohomological stable envelopes exist for both $(X,\sw,\sT,\sA,\fC)$ and $(Y,\widetilde\sw,\sT,\sA,\fC)$. Then the following diagram is commutative
\begin{equation}\label{vb and stab_coh}
\xymatrix{
H^\sT(Y^\sA,\widetilde\sw^\sA) \ar[r]^-{\Stab_\fC} & H^\sT(Y,\widetilde\sw)\\
H^\sT(X^\sA,\sw^\sA) \ar[r]^{\Stab_\fC} \ar[u]^{\psi^\sA} & H^\sT(X,\sw) \ar[u]_{\psi},
}
\end{equation}
where
\begin{equation}\label{eq: psi}
\begin{split}
\psi&=(Y\to \mathrm{Tot}(E_-))^*\circ (X\hookrightarrow \mathrm{Tot}(E_-))_*=(\mathrm{Tot}(E_+)\hookrightarrow Y)_*\circ (\mathrm{Tot}(E_+)\to X)^*,\\
\psi^\sA&=\left(Y^\sA\to \mathrm{Tot}(E_-)^\sA\right)^*\circ \left(X^\sA\hookrightarrow \mathrm{Tot}(E_-)^\sA\right)_*=\left(\mathrm{Tot}(E_+)^\sA\hookrightarrow Y^\sA\right)_*\circ \left(\mathrm{Tot}(E_+)^\sA\to X^\sA\right)^*.
\end{split}
\end{equation}
Similarly, let $\mathsf s\in \Pic_\sA(X)\otimes_\bZ \bR$ be a generic slope, and assume that $K$-theoretic stable envelopes exist for both $(X,\sw,\sT,\sA,\fC,\mathsf s')$ and $(Y,\widetilde\sw,\sT,\sA,\fC,\mathsf s)$, where $\mathsf s'=\mathsf s\otimes (\det E_-)^{1/2}$. Then the following diagram is commutative
\begin{equation}\label{vb and stab_k}
\xymatrix{
K^\sT(Y^\sA,\widetilde\sw^\sA) \ar[r]^-{\Stab^{\mathsf s}_\fC} & K^\sT(Y,\widetilde\sw)\\
K^\sT(X^\sA,\sw^\sA) \ar[r]^{\Stab^{\mathsf s'}_\fC} \ar[u]^{\psi^{\sA}} & K^\sT(X,\sw) \ar[u]_{\psi},
}
\end{equation}
where $\psi$ and $\psi^\sA$ are given by $K$-theoretic version of \eqref{eq: psi}.
\end{Proposition}

\subsection{Attracting closure as cohomological stable envelopes}

The following explicit descriptions of cohomological stable envelopes will be useful. 
In below, we omit the canonical map in Lemma \ref{can map induce stab} and simply write 
$$[\overline{\Attr}_\fC\left(\Delta_{(-)}\right)]:=\can\circ\,[\overline{\Attr}_\fC\left(\Delta_{(-)}\right)].$$

\begin{Proposition}\label{prop:vb and stab_attr_corr}
Assume that $E$ is attracting with respect to $\fC$. If $[\overline{\Attr}_\fC(\Delta_{Y^\sA})]$ induces cohomological stable envelope for $(Y,\sw\circ\pi,\sT,\sA,\fC)$,
then $[\overline{\Attr}_\fC(\Delta_{X^\sA})]$ induces cohomological stable envelope for $(X,\sw,\sT,\sA,\fC)$.
\end{Proposition}
\begin{proof}
Let $S\colon H^\sT(X^\sA,\sw^\sA)\to H^\sT(X,\sw)$ be the map induced by $[\overline{\Attr}_\fC(\Delta_{X^\sA})]$, 
$\Stab_\fC$ the cohomological stable envelope for $(Y,\sw\circ\pi,\sT,\sA,\fC)$.
By Proposition \ref{prop:vb and stab_attr}, it suffices to show that $S$ satisfies $$\pi^*\circ S=\Stab_\fC\circ\,\pi^{\sA *}. $$ 
Let $W\subseteq  Y\times X^\sA$ be the total space of vector bundle $\pr^*E$ on $\overline{\Attr}_\fC(\Delta_{X^\sA})$, 
where $\pr:\overline{\Attr}_\fC(\Delta_{X^\sA})\to X$ is the projection, then 
$\pi^*\circ S$ is induced by the correspondence $[W]$. 

Note that $\Attr_\fC(\Delta_{Y^\sA})\times (\id\times \pi^\sA)(Y^\sA) $ intersects transversely with $Y\times \Delta_{Y^\sA} \times X^\sA$
inside $(Y\times Y^\sA)\times (Y^\sA\times X^\sA)$, and their intersection equals to $\Attr_\fC((\Delta\times \pi^\sA)(Y^\sA))\subseteq  Y\times Y^\sA \times X^\sA$, where $\Delta\colon Y^\sA\hookrightarrow Y^\sA\times Y^\sA$ is the diagonal map. The intersection of $\overline{\Attr}_\fC(\Delta_{Y^\sA})\times (\id\times \pi^\sA)(Y^\sA)$ with $Y\times \Delta_{Y^\sA} \times X^\sA$ is contained in $\overline{\Attr}_\fC((\Delta\times \pi^\sA)(Y^\sA))$, where the latter closure is taken inside $Y\times Y^\sA\times X^\sA$; thus in the equivariant Chow group, we have
\begin{align*}
    \left[\overline{\Attr}_\fC(\Delta_{Y^\sA})\times (\id\times \pi^\sA)(Y^\sA)\right]\,\bigcap\, \left[Y\times \Delta_{Y^\sA} \times X^\sA]=[\overline{\Attr}_\fC((\Delta\times \pi^\sA)(Y^\sA))\right]\in A^\sT(Y\times Y^\sA\times X^\sA).
\end{align*}
Then $\Stab_\fC\circ\,\pi^{\sA *}$ is induced by the correspondence $p_{13*}[\overline{\Attr}_\fC((\Delta\times \pi^\sA)(Y^\sA))]$, where $p_{13}\colon Y\times Y^\sA\times X^\sA\to Y\times X^\sA$ is the projection to the first and third components. Since $p_{13}$ induces isomorphism $(\Delta\times \pi^\sA)(Y^\sA)\cong (\id\times \pi^\sA)(Y^\sA)$, we have $$p_{13*}[\overline{\Attr}_\fC((\Delta\times \pi^\sA)(Y^\sA))]=[\overline{\Attr}_\fC((\id\times \pi^\sA)(Y^\sA))],$$ 
where the latter closure is taken in $Y\times X^\sA$. Furthermore, ${\Attr}_\fC((\id\times \pi^\sA)(Y^\sA))$ is the total space of vector bundle $\pr^*E$ on $\Attr_\fC(\Delta_{X^\sA})$ where $\pr\colon \Attr_\fC(\Delta_{X^\sA})\to X$ is the natural projection. Therefore, we have $\overline{\Attr}_\fC((\id\times \pi^\sA)(Y^\sA))=W$, and it follows that $\pi^*\circ S=\Stab_\fC\circ\,\pi^{\sA *}$.
\end{proof}

\begin{Proposition}\label{prop:vb and stab_repl_corr}
Assume that $E$ is repelling with respect to $\fC$. Let $\bC^*_u$ act on $Y$ by fixing $X$ and scale the fibers of $E$ with weight $1$, and let $\widetilde\sw$ be a $\sT\times\bC^*_u$-invariant function on $Y$ such that $\widetilde\sw|_X=\sw$. 

If $[\overline{\Attr}_\fC(\Delta_{Y^\sA})]$ induces cohomological stable envelope for $(Y,\widetilde{\sw},\sT\times\bC^*_u,\sA,\fC)$, then $[\overline{\Attr}_\fC(\Delta_{X^\sA})]$ induces cohomological stable envelope for $(X,\sw,\sT,\sA,\fC)$.
\end{Proposition}

\begin{proof}
\textbf{Step\,1.} We show that cohomological stable envelope exists for $(X,\sw,\sT,\sA,\fC)$. Let $\sigma:\bC^*\to \sA$ be a cocharacter that lies in $\fC$. Then cohomological stable envelope $\Stab_\sigma$ exists for $(X,\sw,\sT\times \bC^*_u,\bC^*,+)$ by Proposition \ref{prop on stable corr}. By Proposition \ref{prop:vb and stab_repl}, we have 
$$i_*\circ \Stab_\sigma=\Stab_{\fC,Y}\circ\, i^\sA_*. $$ 
Here $\Stab_{\fC,Y}$ is the stable envelope for $(Y,\widetilde{\sw},\sT\times \bC^*_u,\sA,\fC)$ and we use the fact that it is the same 
as the stable envelope when replacing $\sA$ by above $\sigma$ in chamber $\fC$ (ref.~Proposition \ref{prop:coh stab ex criterion}). 

Let $\sigma':\bC^*\to \sA$ be another cocharacter that lies in $\fC$, then we have $$i_*\circ \Stab_\sigma=i_*\circ \Stab_{\sigma'}.$$ Applying $i^*$ to two sides of the above equation, we get $$e^{\sT\times \bC^*_u}(E)\cdot (\Stab_\sigma-\Stab_{\sigma'})=0. $$ We note that $H^{\sT\times \bC^*_u}(X,\sw)\cong H^{\sT}(X,\sw)\otimes\bQ[u]$, and multiplication by $e^{\sT\times \bC^*_u}(E)$ is injective on $H^{\sT\times \bC^*_u}(X,\sw)$ because $e^{\sT\times \bC^*_u}(E)=u^{\rk E}+\,$lower degree in $u$. Thus $\Stab_\sigma=\Stab_{\sigma'}$, and it follows from Proposition \ref{prop:coh stab ex criterion} that cohomological stable envelope exists for $(X,\sw,\sT\times \bC^*_u,\sA,\fC)$. Evaluating at $u=0$, we are done.

\textbf{Step\,2.} We show that $\Stab_\fC$ for $(X,\sw,\sT,\sA,\fC)$ is induced by the correspondence $[\overline{\Attr}_\fC(\Delta_{X^\sA})]$.

By Lemmata \ref{lem: attr vb}, \ref{lem: attr closure vb}, ${\Attr}_{\fC}(\Delta_{Y^{\sA}})$ is a vector bundle of rank $(\rk E|_{X^\sA}^{\text{fixed}})$ over ${\Attr}_{\fC}(\Delta_{X^{\sA}})$, and $$\overline{\Attr}_\fC(\Delta_{Y^\sA})\bigcap\, (X\times X^\sA)=\overline{\Attr}_\fC(\Delta_{X^\sA}), $$ 
this implies that  
\begin{align*}
    (i\times i^\sA)^![\overline{\Attr}_{\fC}(\Delta_{Y^{\sA}})]=u^{\rk E}\cdot[\overline{\Attr}_{\fC}(\Delta_{X^{\sA}})]+\text{lower order terms in }u.
\end{align*}
Let $\overline{p}_{1},\overline{p}_{2}$ be the projections from $Y\times Y^\sA$ to $Y,Y^\sA$ respectively, and let $p_{1},p_{2}$ be the projections from $X\times X^\sA$ to $X,X^\sA$ respectively. Then we have
\begin{align}\label{res of Stab_Y 1}
i^*\circ \Stab_{\fC,Y}\circ \:i^{\sA}_*(-) &= i^*\circ \overline{p}_{2*}\left([\overline{\Attr}_{\fC}(\Delta_{Y^{\sA}})]\otimes \overline{p}_{1}^*(i^{\sA}_*(-))\right)\\ \nonumber 
\text{\tiny by base change }&= p_{2*}\circ(i\times\id)^*\left([\overline{\Attr}_{\fC}(\Delta_{Y^{\sA}})]\otimes \overline{p}_{1}^*(i^{\sA}_*(-))\right)\\\nonumber
\text{\tiny by base change }&= p_{2*}\left(((i\times i^\sA)^![\overline{\Attr}_{\fC}(\Delta_{Y^{\sA}})])\otimes p_{1}^*(-)\right)\\\nonumber
&=u^{\rk E}\cdot p_{2*}\left([\overline{\Attr}_{\fC}(\Delta_{X^{\sA}})]\otimes p_{1}^*(-)\right)+\text{lower order terms in }u.
\end{align}
By Proposition \ref{prop:vb and stab_repl}, we also have
\begin{align}\label{res of Stab_Y 2}
i^*\circ \Stab_{\fC,Y}\circ \:i^{\sA}_*(-) = i^*\circ i_*\circ \Stab_{\fC,X}(-) = u^{\rk E}\cdot \Stab_{\fC,X}(-)+\text{lower order terms in }u.
\end{align}
Comparing the leading terms in \eqref{res of Stab_Y 1} and in \eqref{res of Stab_Y 2}, we arrive at 
\begin{align*}
    \Stab_{\fC,X}(-)=p_{2*}\left([\overline{\Attr}_{\fC}(\Delta_{X^{\sA}})]\otimes p_{1}^*(-)\right),
\end{align*}
i.e. $[\overline{\Attr}_{\fC}(\Delta_{X^{\sA}})]$ induces $\Stab_{\fC,X}$.
\end{proof}

\subsection{More existence results on cohomological stable envelopes}\label{sect on more exist on costab}

Cohomological stable envelopes are known to exist on (1) symmetric GIT quotients and symplectic varieties with general $\sA$ and on (2) smooth varieties with $\sA=\C^*$ 
(ref.~Theorem \ref{thm on exist of stab}). As an application of results obtained in the previous section, we deduce another existence result for some 
\textit{not necessarily symmetric GIT quotients} when $\sA$ is general.

\begin{Theorem}\label{thm:attr plus repl}
Suppose that in \eqref{equ on vb yx}, $\mathrm{Tot}(E)=Y$ is a symmetric GIT quotient with an action by tori $\sA\subseteq \sT$ as Definition \ref{def of sym var}.  
Assume that there exists a decomposition of vector bundle $$E=E_+\oplus E_-$$ into attracting ($E_+$) and repelling ($E_-$) subbundles with respect to $\fC$. 
Then $[\overline{\Attr}_\fC(\Delta_{X^\sA})]$ induces cohomological stable envelope for $(X,\sw,\sT,\sA,\fC)$.

Moreover, let $\fC'$ be a face of $\fC$ and let $\sA'$ be the subtorus of $\sA$ associated with $\fC'$. Then the cohomological stable envelopes $\Stab_{\fC/\fC'}$ and $\Stab_{\fC'}$ exist, and they fit into the following commutative diagram 
\begin{equation*}
\xymatrix{
H^\sT(X^\sA,\sw^\sA) \ar[rr]^{\Stab_{\fC}} \ar[dr]_{\Stab_{\fC/\fC'}} & & H^\sT(X,\sw) \\
 & H^\sT(X^{\sA'},\sw^{\sA'}) \ar[ur]_{\Stab_{\fC'}}. &
}
\end{equation*}
\end{Theorem}
\begin{proof}
Let $\bC^*_u$ act on $Y=\mathrm{Tot}(E)$ by fixing $X$ and $E_{+}$, and scaling the fibers of $E_{-}$ with weight $1$.
By Theorem \ref{construct coh stab corr on symm quiver var}, $[\overline{\Attr}_\fC(\Delta_{Y^\sA})]$ is a cohomological stable envelope correspondence for $(Y,\sw\circ\pi,\sT\times \bC^*_u,\sA,\fC)$.
By Propositions \ref{prop:vb and stab_attr_corr}, \ref{prop:vb and stab_repl_corr}, $[\overline{\Attr}_\fC(\Delta_{X^\sA})]$ is a cohomological stable envelope correspondence for $(X,\sw,\sT,\sA,\fC)$. The triangle lemma then follows from Proposition \ref{tri lem for unsym fram} below. 
\end{proof}

\subsection{The triangle lemma for attracting closure correspondence}

In this subsection, we do not impose any attracting or repelling condition on $E$. Let $\bC^*_u$ act on $Y=\mathrm{Tot}(E)$ by fixing $X$ and scaling the fibers of $E$ with weight $1$. Let $\fC'$ be a face of $\fC$ and let $\sA'$ be the subtorus of $\sA$ associated with $\fC'$. 

Let $\overline{\Attr}_\fC(\Delta_{Y^\sA})$ be the closure of ${\Attr}_\fC(\Delta_{Y^\sA})$ in $Y\times Y^\sA$, similarly for $\overline{\Attr}_\fC(\Delta_{X^\sA})\subseteq  X\times X^\sA$.
We show that if the triangle lemma holds for attracting closure correspondences for $Y$, then it also holds for attracting closure correspondences for $X$.
\begin{Proposition}\label{tri lem for unsym fram}
The equation 
\begin{align*}
    [\overline{\Attr}_{\fC'}(\Delta_{Y^{\sA'}})]\circ [\overline{\Attr}_{\fC/\fC'}(\Delta_{Y^{\sA}})]=[\overline{\Attr}_{\fC}(\Delta_{Y^{\sA}})] \in \Hom(H^{\sT\times \bC^*_u}(Y^\sA,\sw\circ \pi),H^{\sT\times \bC^*_u}(Y,\sw\circ \pi))
\end{align*}
implies that 
\begin{align*}
    [\overline{\Attr}_{\fC'}(\Delta_{X^{\sA'}})]\circ [\overline{\Attr}_{\fC/\fC'}(\Delta_{X^{\sA}})]=[\overline{\Attr}_{\fC}(\Delta_{X^{\sA}})]\in \Hom(H^{\sT}(X^\sA,\sw),H^{\sT}(X,\sw)).
\end{align*}
Here we use the correspondence to denote the induced map between critical cohomologies.
\end{Proposition}

\begin{proof}
Let $i\colon X\hookrightarrow Y$, $i^{\sA'}\colon X^{\sA'}\hookrightarrow Y^{\sA'}$, $i^{\sA}\colon X^{\sA}\hookrightarrow Y^{\sA}$ be zero sections. 
By Lemmata \ref{lem: attr vb}, \ref{lem: attr closure vb}, ${\Attr}_{\fC}(\Delta_{Y^{\sA}})$ is a vector bundle of rank $(\rk E|_{X^\sA}^++\rk E|_{X^\sA}^{\text{fixed}})$ over ${\Attr}_{\fC}(\Delta_{X^{\sA}})$, and $\overline{\Attr}_\fC(\Delta_{Y^\sA})\bigcap\, (X\times X^\sA)=\overline{\Attr}_\fC(\Delta_{X^\sA})$, then 
\begin{align*}
    (i\times i^\sA)^![\overline{\Attr}_{\fC}(\Delta_{Y^{\sA}})]=u^{\rk E-\rk E|_{X^\sA}^+}\cdot[\overline{\Attr}_{\fC}(\Delta_{X^{\sA}})]+\text{lower order terms in }u.
\end{align*}
Therefore we have
\begin{equation}\label{convolution 1}
\begin{split}
&\quad \left((i\times i^{\sA'})^![\overline{\Attr}_{\fC'}(\Delta_{Y^{\sA'}})]\right)\circ \left((i^{\sA'}\times i^\sA)^![\overline{\Attr}_{\fC/\fC'}(\Delta_{Y^{\sA}})]\right)\\
&= u^{\rk E-\rk E|_{X^\sA}^++\rk E|_{X^\sA}^{\sA'\text{-fixed}}}\cdot[\overline{\Attr}_{\fC'}(\Delta_{X^{\sA'}})]\circ [\overline{\Attr}_{\fC/\fC'}(\Delta_{X^{\sA}})]+\text{lower order terms in }u.
\end{split}
\end{equation}
Let $\overline{p}_{ij}$ be the projection from $Y\times Y^{\sA'}\times Y^\sA$ to the $ij$-th component, for example $\overline{p}_{13}\colon Y\times Y^{\sA'}\times Y^\sA\to Y\times Y^\sA$. Similarly, let ${p}_{ij}$ be the projection from $X\times X^{\sA'}\times X^\sA$ to the $ij$-th component. Let $\widetilde{p}_{13}\colon X\times Y^{\sA'}\times X^\sA\to X\times X^\sA$ be the natural projection. The LHS of \eqref{convolution 1} can be evaluated as:
\begin{align*}
    &\quad\,\, p_{13*}\left((i\times i^{\sA'}\times i^{\sA})^!\left(\overline{p}_{12}^![\overline{\Attr}_{\fC'}(\Delta_{Y^{\sA'}})]\otimes \overline{p}_{23}^![\overline{\Attr}_{\fC/\fC'}(\Delta_{Y^{\sA}})]\right)\right)\\
    &=\widetilde{p}_{13*}\left(e^{\sT\times \bC^*_u}(E|_{X^{\sA'}}^{\sA'\text{-fixed}})\cdot (i\times \id\times i^{\sA})^!\left(\overline{p}_{12}^![\overline{\Attr}_{\fC'}(\Delta_{Y^{\sA'}})]\otimes \overline{p}_{23}^![\overline{\Attr}_{\fC/\fC'}(\Delta_{Y^{\sA}})]\right)\right)\\
    &=u^{\rk E|_{X^{\sA}}^{\sA'\text{-fixed}}}\cdot\widetilde{p}_{13*}\left((i\times \id\times i^{\sA})^!\left(\overline{p}_{12}^![\overline{\Attr}_{\fC'}(\Delta_{Y^{\sA'}})]\otimes \overline{p}_{23}^![\overline{\Attr}_{\fC/\fC'}(\Delta_{Y^{\sA}})]\right)\right)+\text{lower order terms in }u\\
    \text{\tiny by base change }&=u^{\rk E|_{X^{\sA}}^{\sA'\text{-fixed}}}\cdot (i\times  i^{\sA})^!\overline{p}_{13*}\left(\overline{p}_{12}^![\overline{\Attr}_{\fC'}(\Delta_{Y^{\sA'}})]\otimes \overline{p}_{23}^![\overline{\Attr}_{\fC/\fC'}(\Delta_{Y^{\sA}})]\right)+\text{lower order terms in }u\\
    &=u^{\rk E|_{X^{\sA}}^{\sA'\text{-fixed}}}\cdot (i\times  i^{\sA})^!\left([\overline{\Attr}_{\fC'}(\Delta_{Y^{\sA'}})]\circ [\overline{\Attr}_{\fC/\fC'}(\Delta_{Y^{\sA}})]\right)+\text{lower order terms in }u\\
    \text{\tiny by assumption }&=u^{\rk E|_{X^{\sA}}^{\sA'\text{-fixed}}}\cdot (i\times  i^{\sA})^![\overline{\Attr}_{\fC}(\Delta_{Y^{\sA}})]+\text{lower order terms in }u\\
    &=u^{\rk E-\rk E|_{X^\sA}^++\rk E|_{X^{\sA}}^{\sA'\text{-fixed}}}\cdot[\overline{\Attr}_{\fC}(\Delta_{X^{\sA}})]+\text{lower order terms in }u.
\end{align*}
Compared with the leading term on the RHS of \eqref{convolution 1}, and we obtain the desired equation
\begin{align*}
    [\overline{\Attr}_{\fC'}(\Delta_{X^{\sA'}})]\circ [\overline{\Attr}_{\fC/\fC'}(\Delta_{X^{\sA}})]=[\overline{\Attr}_{\fC}(\Delta_{X^{\sA}})].
\end{align*}
Here we used the fact that $H^{\sT\times \bC^*_u}(X,\sw)\cong H^{\sT}(X,\sw)\otimes\bQ[u]$ is free over $\bQ[u]$.
\end{proof}

In the above proof, we have used the following: 
\begin{Lemma}\label{lem: attr vb}
${\Attr}_{\fC}(\Delta_{Y^{\sA}})$ is a vector bundle of rank $(\rk E|_{X^\sA}^++\rk E|_{X^\sA}^{\text{fixed}})$ over ${\Attr}_{\fC}(\Delta_{X^{\sA}})$.
\end{Lemma}
\begin{proof}
The attraction map 
$$a:{\Attr}_{\fC}(\Delta_{Y^{\sA}})\to \Delta_{Y^{\sA}}$$ is an affine fibration with $(\rk N_{X^\sA/X}^+ +\rk E|_{X^\sA}^+)$-dimensional fibers, and $\Delta_{Y^{\sA}}$ is the vector bundle $E|_{X^\sA}^{\text{fixed}}$ over $\Delta_{X^\sA}$, so ${\Attr}_{\fC}(\Delta_{Y^{\sA}})$ is smooth of dimension $(\dim {\Attr}_{\fC}(\Delta_{X^{\sA}})+\rk E|_{X^\sA}^++\rk E|_{X^\sA}^{\text{fixed}})$. 

It is elementary to see that ${\Attr}_{\fC}(\Delta_{Y^{\sA}})$ is contained in $(\pi\times \pi^\sA)^{-1}({\Attr}_{\fC}(\Delta_{X^{\sA}}))$, and 
${\Attr}_{\fC}(\Delta_{Y^{\sA}})\bigcap\,(\pi\times \pi^\sA)^{-1}(\Delta_{X^{\sA}})$ is isomorphic to the vector bundle $E|_{X^\sA}^+\oplus E|_{X^\sA}^{\text{fixed}}$, which implies that $$(\pi\times \pi^\sA)|_{{\Attr}_{\fC}(\Delta_{Y^{\sA}})}:{\Attr}_{\fC}(\Delta_{Y^{\sA}})\to {\Attr}_{\fC}(\Delta_{X^{\sA}})$$ is a smooth morphism in a Zariski open neighbourhood of ${\Attr}_{\fC}(\Delta_{Y^{\sA}})\bigcap\,(\pi\times \pi^\sA)^{-1}(\Delta_{X^{\sA}})$. Since $(\pi\times \pi^\sA)|_{{\Attr}_{\fC}(\Delta_{Y^{\sA}})}$ is $\sA$-equivariant and the $\sA$-action (in the chamber $\fC$) is attracting, $(\pi\times \pi^\sA)|_{{\Attr}_{\fC}(\Delta_{Y^{\sA}})}$ is a smooth morphism. 

Finally, suppose that $s_1$ and $s_2$ are two local sections of the map $(\pi\times \pi^\sA)|_{{\Attr}_{\fC}(\Delta_{Y^{\sA}})}$ on an open subset $U\subseteq  {\Attr}_{\fC}(\Delta_{X^{\sA}})$, then for all $a,b\in \mathcal O_{{\Attr}_{\fC}(\Delta_{X^{\sA}})}(U)$, $as_1+bs_2$ is also a section of $(\pi\times \pi^\sA)|_{{\Attr}_{\fC}(\Delta_{Y^{\sA}})}$ because $\lim_{\fC}(as_1+bs_2)$ exists and is equal to $a\lim_\fC(s_1)+b\lim_{\fC}s_2$. It follows that ${\Attr}_{\fC}(\Delta_{Y^{\sA}})$ is a sub-bundle of $(\pi\times \pi^\sA)^{-1}({\Attr}_{\fC}(\Delta_{X^{\sA}}))$ of rank $(\rk E|_{X^\sA}^++\rk E|_{X^\sA}^{\text{fixed}})$.
\end{proof}
\begin{Lemma}\label{lem: attr closure vb}
\begin{align*}
    \overline{\Attr}_\fC(\Delta_{Y^\sA})\bigcap\, (X\times X^\sA)=\overline{\Attr}_\fC(\Delta_{X^\sA}).
\end{align*}
\end{Lemma}

\begin{proof}
It is easy to see that $${\Attr}_\fC(\Delta_{Y^\sA})\bigcap\,(X\times X^\sA)={\Attr}_\fC(\Delta_{X^\sA}),$$ thus the inclusion ``$\supseteq$'' is obvious. For the opposite direction of inclusion, we notice that $(\pi\times\pi^\sA)^{-1}({\Attr}_\fC(\Delta_{X^\sA}))$ contains $\Delta_{Y^\sA}$ and is closed under taking the attracting set, so we have 
$${\Attr}_\fC(\Delta_{Y^\sA})\subseteq (\pi\times\pi^\sA)^{-1}({\Attr}_\fC(\Delta_{X^\sA})).$$ 
It follows that $$\overline{\Attr}_\fC(\Delta_{Y^\sA})\subseteq (\pi\times\pi^\sA)^{-1}(\overline{\Attr}_\fC(\Delta_{X^\sA})),$$ 
therefore, 
\begin{equation*}\overline{\Attr}_\fC(\Delta_{Y^\sA})\bigcap\, (X\times X^\sA)\subseteq  (i\times i^\sA)^{-1}(\pi\times\pi^\sA)^{-1}(\overline{\Attr}_\fC(\Delta_{X^\sA}))=\overline{\Attr}_\fC(\Delta_{X^\sA}). \qedhere \end{equation*}
\end{proof}

\section{Hall envelopes and stable envelopes }\label{sec:na stab}
 
The theory of stable envelopes on symplectic varieties \cite{MO, Oko, AO2} provides an effective way of producing $R$-\textit{matrices of quantum groups}, 
where the triangle lemma plays the key role. In the above constructions of stable envelopes 
for critical loci, we prove the triangle lemma for symmetric quiver varieties (or more generally symmetric GIT quotients) with any potentials (Theorem \ref{tri lem for coh stab}, Lemma \ref{triangle lemma for K})
and also for certain asymmetric case (Theorem \ref{thm:attr plus repl}). 
In the symmetric cases, stable envelopes are compatible with Hall operations via nonabelian stable envelopes 
(Theorems \ref{thm: hall coh_sym quot}, \ref{thm: hall k_sym quot}). This motivates the definition 
of \textit{Hall envelopes} (Definition \ref{defi of prestab}) on any (not necessarily symmetric) quiver varieties with potentials as the following compositions: 
\begin{equation}
\xymatrix{
H^\sT(\fM(\bv,\underline{\bd})^{\sA,\phi},\sw) \ar[rr]^{\mathfrak{m}^\phi_{\fC}} & & H^\sT(\fM(\bv,\underline{\bd}),\sw) \ar[d]_{} \\
H^\sT(\cM_\theta(\bv,\underline{\bd})^{\sA,\phi},\sw) \ar[u]^{\widetilde{\bPsi}_H^\phi} \ar[rr]^{\HallEnv_\fC} & & H^\sT(\cM_\theta(\bv,\underline{\bd}),\sw),
}
\quad 
\xymatrix{
K^\sT(\fM(\bv,\underline{\bd})^{\sA,\phi},\sw) \ar[rr]^{\mathfrak{m}^\phi_{\fC}} & & K^\sT(\fM(\bv,\underline{\bd}),\sw)  \ar[d]_{}\\
K^\sT(\cM_\theta(\bv,\underline{\bd})^{\sA,\phi},\sw) \ar[u]^{\widetilde{\bPsi}_K^{\phi,\mathsf s'}} \ar[rr]^{\HallEnv^{\mathsf s}_\fC} & & K^\sT(\cM_\theta(\bv,\underline{\bd}),\sw),
}\nonumber \end{equation}
where the downward arrows are restrictions to open substacks, and the upward arrows are called \textit{interpolation maps} (Definition \ref{def of nona stab}), 
defined by adding extra framings on quivers and using (abelian) stable envelopes 
for one dimensional tori.
In the symmetric cases, we show that interpolation maps agree with nonabelian stable envelopes (Proposition \ref{prop: compare na stab_sym quiver}); therefore 
Hall envelopes coincide with stable envelopes.  
We also give a description of nonabelian stable envelopes using BPS cohomology in \S \ref{sect on bps coh}. For tripled quivers with canonical cubic potentials, we show that  nonabelian stable envelopes constructed in this paper reproduce the nonabelian stable envelopes for Nakajima quiver varieties in \cite{AO1} along the dimensional reductions (Proposition \ref{prop:compare na stab nak}).


We note that triangle lemma could fail for either Hall envelopes or stable envelopes in general (Example \ref{ex on fail of tri lem}, Proposition \ref{prop on 3 ex}).   
For a class of asymmetric quiver varieties, where the asymmetric part comes only from the framing (Definition \ref{def: pseudo-self-dual}),
we prove that triangle lemma of Hall envelopes holds 
(Theorem \ref{thm: AFSQV}, Corollary \ref{cor: tri lem for AFSQV}). This provides foundations for studying geometric $R$-matrices of shifted quantum groups.

\subsection{Definitions}

Let $Q=(Q_0,Q_1)$ be any quiver and $\bv,\bd_{\In},\bd_{\Out}\in \bN^{Q_0}$. Denote 
\begin{align*}
    \underline{\bd}=(\bd_{\In},\bd_{\Out}),\quad \underline{\bd}^\bv=(\bd_{\In}+\bv,\bd_{\Out}).
\end{align*}
For simplicity, we consider cyclic stability condition $\theta\in \bQ^{Q_0}$,~i.e.
\begin{equation}
\label{equ on cycc condit}\theta_i<0,\,\,\, \;\forall\, i\in Q_0.\end{equation} 
The space of framed representations (ref.~\eqref{equ rvab}) is
$$R(\mathbf v,\underline{\bd}^{\bv})=R(\mathbf v,\underline{\bd})\oplus 
\bigoplus_{i\in Q_0}\Hom(\bC^{\mathbf v_i},\bC^{\mathbf v_i}),$$ 
which we visualize vertex-wise as
\begin{equation}\label{diag of R'2}
\xymatrix{
&V_i \ar[dr]^{B_i} &\\
V'_i \ar[ur]^{I_i} & \bC^{\bd_{\In,i}} \ar[u]_{A_i} & \bC^{\bd_{\Out,i}}.
}
\end{equation}
In the above diagram $V'_i\cong \bC^{\mathbf v_i}\cong V_i$.
Define GIT quotient and stack quotient by
$$\cM_\theta(\mathbf v,\underline{\bd}^{\bv}):=R(\mathbf v,\underline{\bd}^{\bv})/\!\!/_\theta G\;,\quad \fM(\mathbf v,\underline{\bd}^{\bv}):=[R(\mathbf v,\underline{\bd}^{\bv})/ G].$$ 
The $\theta$-semistable locus $R(\mathbf v,\underline{\bd}^{\bv})^{\theta\emph{-}ss}$ contains the open subscheme 
$\mathring{R}(\mathbf v,\underline{\bd}^{\bv})$ on which the map $I_i$ are isomorphisms for all $i\in Q_i$. $\mathring{R}(\mathbf v,\underline{\bd}^{\bv})$ is $G$-invariant and descends to an open subscheme 
$$\mathring\cM_\theta(\mathbf v,\underline{\bd}^{\bv}):=\mathring{R}(\mathbf v,\underline{\bd}^{\bv})/G\subseteq  \cM_\theta(\mathbf v,\underline{\bd}^{\bv}). $$ 
One has an isomorphism
\begin{align}\label{equ on mvcongr}
    \mathring\cM_\theta(\mathbf v,\underline{\bd}^{\bv})\cong R(\mathbf v,\underline{\bd}),
\end{align}
under which, the flavour symmetry $G'$ which acts on $V'_i$ in \eqref{diag of R'2} becomes the gauge group acting on $R(\mathbf v,\underline{\bd})$.

Consider a torus in the center of the flavour group: 
\begin{equation}\label{equ on Cstar}\bC^*\subseteq  G', \end{equation}
which acts on $V'_i$ in \eqref{diag of R'2} with weight $-1$ for all $i\in Q_0$ \footnote{For other generic stability condition $\zeta\in \bQ^{Q_0}$, we take $\bC^*\subseteq  G'$ which acts on $V'_i$ in \eqref{diag of R'2} with weight $\mathrm{sign}(\zeta_i)$. Then we use this $\bC^*$-action to construct $\widetilde{\bPsi}$.}. Then $\cM_\theta(\mathbf v,\underline{\bd})$ is naturally identified with a connected component in $\cM_\theta(\mathbf v,\underline{\bd}^{\bv})^{\bC^*}$. 

Let $\sT\subseteq  \sF$ be a torus in the flavour group $\sF=\Aut_G(R(\bv,\underline{\bd}))$. Then we fix a $(G\times \sT)$-invariant function 
$$\sw\colon R(\bv,\underline{\bd})\to \bC.$$ 
We denote the descent functions on $\fM(\bv,\underline{\bd})$ and $\cM_\theta(\bv,\underline{\bd})$ still by $\sw$.

Let $\sw'$ be the pullback of $\sw$ along the projection $R(\mathbf v,\underline{\bd}^{\bv})\to R(\mathbf v,\underline{\bd})$ that forgets $\{I_i\}_{i\in Q_0}$. $\sw'$ descends to a $(G'\times \sT)$-invariant function 
$$\sw'\colon \cM_\theta(\mathbf v,\underline{\bd}^{\bv})\to \C$$
By Proposition \ref{prop on stable corr} and the discussions in \S\ref{sec:general flavour}, there exist a cohomological stable envelope $\Stab_+$ for $(\cM_\theta(\mathbf v,\underline{\bd}^{\bv}),\sw',G'\times \sT,\bC^*,+)$ and a $K$-theoretic stable envelope $\Stab^{\mathsf s}_+$ for $(\cM_\theta(\mathbf v,\underline{\bd}^{\bv}),\sw',G'\times \sT,\bC^*,+,\mathsf s)$ with a slope $\mathsf s\in \mathrm{Char}(G)\otimes_{\bZ}\bR$. 

\begin{Definition}\label{def of nona stab}
We define $\widetilde{\bPsi}_H\colon H^{\sT}(\cM_\theta(\mathbf v,\underline{\bd}),\sw)\to H^{\sT}(\mathfrak{M}(\mathbf v,\underline{\bd}),\sw)$ as the composition of the following maps
\begin{align*}
\xymatrix{
H^{\sT}(\cM_\theta(\mathbf v,\underline{\bd}),\sw)\ar[r]^-{1\otimes \id}   & H^{G'\times \sT}(\cM_\theta(\mathbf v,\underline{\bd}),\sw)\ar[r]^-{\Stab_+} & H^{G'\times \sT}(\cM_\theta(\mathbf v,\underline{\bd}^{\bv}),\sw')\ar[d]^-{j^*}\\
& H^{\sT}(\mathfrak{M}(\mathbf v,\underline{\bd}),\sw) & \ar[l]_-{\cong} H^{G'\times \sT}(\mathring\cM_\theta(\mathbf v,\underline{\bd}^{\bv}),\sw') 
}
\end{align*}
Here $1\in H_{G'}(\pt)$ is the unit, and  $j\colon \mathring\cM_\theta(\mathbf v,\underline{\bd}^{\bv})\hookrightarrow \cM_\theta(\mathbf v,\underline{\bd}^{\bv})$ is the open immersion.

Similarly, we define $\widetilde{\bPsi}^{\mathsf s}_K\colon K^{\sT}(\cM_\theta(\mathbf v,\underline{\bd}),\sw)\to K^{\sT}(\mathfrak{M}(\mathbf v,\underline{\bd}),\sw)$ as the composition of the following maps
\begin{align*}
\xymatrix{
K^{\sT}(\cM_\theta(\mathbf v,\underline{\bd}),\sw)\ar[r]^-{1\otimes \id} & K^{G'\times \sT}(\cM_\theta(\mathbf v,\underline{\bd}),\sw)\ar[r]^-{\Stab^{\mathsf s}_+} & K^{G'\times \sT}(\cM_\theta(\mathbf v,\underline{\bd}^{\bv}),\sw')\ar[d]^-{j^*}\\
& K^{\sT}(\mathfrak{M}(\mathbf v,\underline{\bd}),\sw) & \ar[l]_-{\cong} K^{G'\times \sT}(\mathring\cM_\theta(\mathbf v,\underline{\bd}^{\bv}),\sw') 
}
\end{align*}
Here in the first map $1\in K_{G'}(\pt)$ is the unit.

We call $\widetilde{\bPsi}_H$ and $\widetilde{\bPsi}^{\mathsf s}_K$ cohomological and $K$-theoretic \textit{interpolation maps} respectively.
\end{Definition}


\begin{Lemma}\label{lem:res of na stab to stable locus}
Let $\mathbf k:\cM_\theta(\bv,\underline{\bd})\hookrightarrow \fM(\bv,\underline{\bd})$ be the open immersion of stable locus, then 
\begin{align*}
    \mathbf k^*\circ\widetilde{\bPsi}_H=\id,\quad \mathbf k^*\circ\widetilde{\bPsi}^{\mathsf s}_K=\id.
\end{align*}
\end{Lemma}

\begin{proof}
Let $\mathsf q:\mathring\cM_\theta(\mathbf v,\underline{\bd}^\bv)\to \mathfrak{M}(\mathbf v,\underline{\bd})$ be the quotient map. Then 
$$\mathsf q^{-1}(\cM_\theta(\mathbf v,\underline{\bd}))\subseteq  \mathring\cM_\theta(\mathbf v,\underline{\bd}^\bv)$$ 
is the open subscheme of quiver representations which are $\theta$-semistable after forgetting $\{I_i\}_{i\in Q_0}$. Note that 
$$\Attr_+(\cM_\theta(\mathbf v,\underline{\bd}))\subseteq  \cM_\theta(\mathbf v,\underline{\bd}^\bv)$$ is exactly the locus where quiver representations are $\theta$-semistable after forgetting $\{I_i\}_{i\in Q_0}$. It follows that 
\begin{equation}\label{equ on qMattr}\mathsf q^{-1}(\cM_\theta(\mathbf v,\underline{\bd}))=\mathring\cM_\theta(\mathbf v,\underline{\bd}^\bv)\bigcap\Attr_+(\cM_\theta(\mathbf v,\underline{\bd})), 
\end{equation} 
and the restriction of the quotient map to $\mathsf q^{-1}(\cM_\theta(\mathbf v,\underline{\bd}))$: 
$$\mathsf q:\mathsf q^{-1}(\cM_\theta(\mathbf v,\underline{\bd}))\to \cM_\theta(\mathbf v,\underline{\bd})$$ is identified with the attraction map $a:\Attr_+(\cM_\theta(\mathbf v,\underline{\bd}))\to \cM_\theta(\mathbf v,\underline{\bd})$. By a direct calculation, we have 
$$\Stab_+(1\otimes\gamma)|_{\mathsf q^{-1}(\cM_\theta(\mathbf v,\underline{\bd}))}=\mathsf q^*(1\otimes\gamma),$$ 
for all $\gamma\in H^{\sT}(\cM_\theta(\mathbf v,\underline{\bd}),\sw)$. Since the isomorphism $H^{G'\times \sT}(\mathsf q^{-1}(\cM_\theta(\mathbf v,\underline{\bd})),\sw)\cong H^{\sT}(\cM_\theta(\mathbf v,\underline{\bd}),\sw)$ maps $\mathsf q^*(1\otimes\gamma)$ to $\gamma$, we see that $\mathbf k^*\circ \widetilde{\bPsi}_H=\id$. The $K$-theoretic version is proven similarly and we omit the details.
\end{proof}

Hall operations are defined similarly as the symmetric case. Let $\sA\subseteq \sT$ be a subtorus. Choose a homomorphism 
$$\phi:\sA\to G,$$ and denote $R(\bv,\underline{\bd})^{\sA,\phi}$ the fixed locus of $R(\bv,\underline{\bd})$ under the action of $\sA$ via the homomorphism $$(\phi,\id):\sA\to G\times \sA$$ and denote $G^{\phi(\sA)}$ the fixed subgroup of $G$ under the conjugation action of the subgroup $\phi(\sA)$. 

Note that 
$$R(\bv,\underline{\bd})^{\sA,\phi}\cong R(Q_\phi,\bv_\phi,\underline{\bd_\phi}), \quad G^{\phi(\sA)}\cong \prod_{i\in Q_{\phi,0}}\GL(\bv_{\phi,i})$$ for some quiver $Q_\phi$ with gauge dimension vector $\bv_\phi$ and in-coming and out-going framing dimension vectors $\bd_{\phi,\In}$ and $\bd_{\phi,\Out}$ respectively. Define
\begin{align*}
\fM(\bv,\underline{\bd})^{\sA,\phi}:=\left[R(\bv,\underline{\bd})^{\sA,\phi}/G^{\phi(\sA)}\right],\quad \cM_\theta(\bv,\underline{\bd})^{\sA,\phi}:=R(\bv,\underline{\bd})^{\sA,\phi}/\!\!/_\theta G^{\phi(\sA)}.
\end{align*}
In the following discussions, we assume that $\cM_\theta(\bv,\underline{\bd})^{\sA,\phi}$ is nonempty. Then the same argument as Lemma \ref{fix pts of sym quiv var} shows that $\cM_\theta(\bv,\underline{\bd})^{\sA,\phi}$ is a connected component of the $\sA$-fixed locus $\cM_\theta(\bv,\underline{\bd})^\sA$, and we have
\begin{align}\label{fix pt decomp}
    \cM_\theta(\bv,\underline{\bd})^\sA=\bigsqcup_{\phi/\sim} \cM_\theta(\bv,\underline{\bd})^{\sA,\phi},
\end{align}
where the sum is taken for all homomorphisms $\phi:\sA\to G$ modulo the equivalence relation: $\phi_1\sim\phi_2$ if and only if they give the isomorphic $(G\times \sA)$-module structures on $R(\bv,\underline{\bd})$.

Let us fix a chamber $\fC\subseteq  \Lie(\sA)_\bR$, and define 
\begin{align*}
    L(\bv,\underline{\bd})^\phi_{\fC}:=\Attr_\fC(R(\bv,\underline{\bd})^{\sA,\phi})\subseteq  R(\bv,\underline{\bd}).
\end{align*}
Then $L(\bv,\underline{\bd})^\phi_{\fC}$ is a linear subspace of $R(\bv,\underline{\bd})$, and there are maps
\begin{equation*}
\xymatrix{
R(\bv,\underline{\bd})^{\sA,\phi} &\ar[l]_-{q} L(\bv,\underline{\bd})^\phi_{\fC} \ar[r]^{p} & R(\bv,\underline{\bd}), }
\end{equation*}
where $q$ is the attraction map and $p$ is the natural closed immersion. 

Let $P^\phi_\fC:=\Attr_\fC(G^{\phi(\sA)})$ be the parabolic subgroup of $G$, which naturally acts on $L(\bv,\underline{\bd})^\phi_{\fC}$, and $p$ and $q$ in the above diagram are $P^\phi_\fC$-equivariant, where $P^\phi_\fC$-action on $R(\bv,\underline{\bd})$ is via $P^\phi_\fC\to G$ and $P^\phi_\fC$-action on $R(\bv,\underline{\bd})^{\sA,\phi}$ is via attraction map $P^\phi_\fC\to G^{\phi(\sA)}$. Passing to the quotient stack, we obtain the following diagram
\begin{equation*}
\xymatrix{
\fM(\bv,\underline{\bd})^{\sA,\phi} &\ar[l]_-{\fq} \mathfrak L(\bv,\underline{\bd})^\phi_{\fC} \ar[r]^{\fp} & \fM(\bv,\underline{\bd}),
}
\end{equation*}
where $\mathfrak L(\bv,\underline{\bd})^\phi_{\fC}:=[L(\bv,\underline{\bd})^\phi_{\fC}/P^\phi_\fC]$. We note that $\fq$ is a smooth morphism and $\fp$ is a proper morphism.

\begin{Definition}\label{def of hall op_general}
We define the \textit{Hall operations}
\begin{align*}
    &\mathfrak{m}^\phi_{\fC}:=\fp_*\circ\fq^*\colon H^\sT(\fM(\bv,\underline{\bd})^{\sA,\phi},\sw)\longrightarrow H^\sT(\fM(\bv,\underline{\bd}),\sw),\\
    &\mathfrak{m}^\phi_{\fC}:=\fp_*\circ\fq^*\colon K^\sT(\fM(\bv,\underline{\bd})^{\sA,\phi},\sw)\longrightarrow K^\sT(\fM(\bv,\underline{\bd}),\sw).
\end{align*}
\end{Definition}

\begin{Definition}\label{defi of prestab}
Let $\mathsf s\in \mathrm{Char}(G)\otimes_\bZ\bR$ be a slope. We define the \textit{Hall envelopes} 
\begin{align*}
    \HallEnv_\fC\colon H^\sT(\cM_\theta(\bv,\underline{\bd})^{\sA},\sw)\longrightarrow H^\sT(\cM_\theta(\bv,\underline{\bd}),\sw),\\
    \HallEnv^{\mathsf s}_\fC\colon K^\sT(\cM_\theta(\bv,\underline{\bd})^{\sA},\sw)\longrightarrow K^\sT(\cM_\theta(\bv,\underline{\bd}),\sw),
\end{align*}
by 
\begin{align*}
\HallEnv_\fC\big|_{H^\sT(\cM_\theta(\bv,\underline{\bd})^{\sA,\phi},\sw)} &:=\mathbf k^*\circ\mathfrak{m}^\phi_{\fC}\circ\widetilde{\bPsi}_H^\phi,\\
\HallEnv^{\mathsf s}_\fC \big|_{K^\sT(\cM_\theta(\bv,\underline{\bd})^{\sA,\phi},\sw)} &:= \mathbf k^*\circ\mathfrak{m}^\phi_{\fC}\circ\widetilde{\bPsi}_K^{\phi,\mathsf s'}.
\end{align*}
Here $\mathbf k:\cM_\theta(\bv,\underline{\bd})\hookrightarrow \fM(\bv,\underline{\bd})$ is the open immersion of stable locus, $\widetilde{\bPsi}_H^{\phi}, \widetilde{\bPsi}_K^{\phi,\mathsf s'}$ are interpolation maps (Definition \ref{def of nona stab}) for the stack $\fM(\bv,\underline{\bd})^{\sA,\phi}$ with slope 
\begin{align}\label{s' for prestab}
    \mathsf s'=\mathsf s\otimes \det\left(T\fM(\bv,\underline{\bd})^{\sA,\phi\text{-repl}}\right)^{1/2}\in \mathrm{Char}(G^{\phi(\sA)})\otimes_\bZ \bR,
\end{align}
and $T\fM(\bv,\underline{\bd})=R(\bv,\underline{\bd})-\Lie(G)$.

Moreover, we say that $\HallEnv_\fC$ (resp.\,$\HallEnv^{\mathsf s}_\fC$) is \textit{compatible with the Hall operation} if for all $\phi:\sA\to G$ that appears in \eqref{fix pt decomp} the left square (resp.\,right square) below is commutative
\begin{equation*}
\xymatrix{
H^\sT(\fM(\bv,\underline{\bd})^{\sA,\phi},\sw) \ar[rr]^{\mathfrak{m}^\phi_{\fC}} & & H^\sT(\fM(\bv,\underline{\bd}),\sw)\\
H^\sT(\cM_\theta(\bv,\underline{\bd})^{\sA,\phi},\sw) \ar[u]^{\widetilde{\bPsi}_H^\phi} \ar[rr]^{\HallEnv_\fC} & & H^\sT(\cM_\theta(\bv,\underline{\bd}),\sw),
\ar[u]_{\widetilde{\bPsi}_H}}
\,\,\,
\xymatrix{
K^\sT(\fM(\bv,\underline{\bd})^{\sA,\phi},\sw) \ar[rr]^{\mathfrak{m}^\phi_{\fC}} & & K^\sT(\fM(\bv,\underline{\bd}),\sw)\\
K^\sT(\cM_\theta(\bv,\underline{\bd})^{\sA,\phi},\sw) \ar[u]^{\widetilde{\bPsi}_K^{\phi,\mathsf s'}} \ar[rr]^{\HallEnv^{\mathsf s}_\fC} & & K^\sT(\cM_\theta(\bv,\underline{\bd}),\sw). \ar[u]_{\widetilde{\bPsi}^{\mathsf s}_K}
}
\end{equation*}
\end{Definition}
\begin{Remark}\label{rmk: pstab supp and norm}
Hall envelope $\HallEnv_\fC$ (resp. $\HallEnv^{\mathsf s}_\fC$) satisfies axioms (i) and (ii) for stable envelopes in Definition \ref{def of stab coho} (resp.~Definition \ref{def of stab k}). Namely, for arbitrary $\beta\in H^\sT (\cM_\theta(\bv,\underline{\bd})^{\sA,\phi}, \sw)$ and $\gamma \in K^{\sT} (\cM_\theta(\bv,\underline{\bd})^{\sA,\phi}, \sw)$, we have
\begin{enumerate}[(i)]
\setlength{\parskip}{1ex}

\item  $\HallEnv_{\fC}(\beta)$ and $\HallEnv^{\mathsf s}_{\fC}(\gamma)$ are supported on $\Attr^f_\fC (\cM_\theta(\bv,\underline{\bd})^{\sA,\phi})$; 

\item $$\HallEnv_{\fC} (\beta)\big|_{\cM_\theta(\bv,\underline{\bd})^{\sA,\phi}} =  e^\sT \left(N_{\cM_\theta(\bv,\underline{\bd})^{\sA,\phi}/ \cM_\theta(\bv,\underline{\bd})}^-\right) \cdot \beta,  \,\, 
\HallEnv^{\mathsf s}_{\fC} (\gamma)\big|_{\cM_\theta(\bv,\underline{\bd})^{\sA,\phi}} =  e^\sT_K \left(N_{\cM_\theta(\bv,\underline{\bd})^{\sA,\phi}/ \cM_\theta(\bv,\underline{\bd})}^-\right) \cdot \gamma. $$
\end{enumerate}
As a corollary, $\HallEnv_\fC$ and $\HallEnv^{\mathsf s}_\fC$ are triangular with invertible diagonals after localization, in particular they are \textit{invertible} after localization. In Section \ref{sec: pstab vs stab}, we will discuss situations when $\HallEnv_\fC$ (resp.~$\HallEnv^{\mathsf s}_\fC$) satisfies the axiom (iii) in Definition \ref{def of stab coho} (resp.~Definition \ref{def of stab k}).
\end{Remark}
We show that triangle lemma of Hall envelopes follows from the associativity of Hall operations. 
\begin{Lemma}\label{tri lem for compatible prestab}
Let $X=\cM_\theta(\bv,\underline{\bd})$, $\fC'$ be a face of $\fC$, and $\sA'$ be the subtorus of $\sA$ associated with $\fC'$. Suppose that $\HallEnv_{\fC/\fC'}$ (resp. $\HallEnv^{\mathsf s'}_{\fC/\fC'}$) is compatible with the Hall operation, then the left triangle (resp.\,right triangle):
\begin{equation}\label{cd: tri lem}
\xymatrix{
H^\sT(X^\sA,\sw^\sA) \ar[rr]^{\HallEnv_{\fC}} \ar[dr]_{\HallEnv_{\fC/\fC'}} & & H^\sT(X,\sw), \\
 & H^\sT(X^{\sA'},\sw^{\sA'}) \ar[ur]_{\,\,\,\HallEnv_{\fC'}} &
}\quad
\xymatrix{
K^\sT(X^\sA,\sw^\sA) \ar[rr]^{\HallEnv^{\mathsf s}_{\fC}} \ar[dr]_{\HallEnv^{\mathsf s'}_{\fC/\fC'}} & & K^\sT(X,\sw) \\
 & K^\sT(X^{\sA'},\sw^{\sA'}) \ar[ur]_{\,\,\, \HallEnv^{\mathsf s}_{\fC'}} &
}
\end{equation}
commutes, where $\mathsf s'=\mathsf s|_{X^{\sA'}}\otimes\det\left(N^-_{X^{\sA'}/X}\right)^{1/2}$.
\end{Lemma}

\begin{proof}
Fix a homomorphism $\phi:\sA\to G$, and use the same letter $\phi$ to denote its restriction to subtorus $\sA'$. Then the restriction of $\mathsf s'$ on the component $\cM_\theta(\bv,\underline{\bd})^{\sA',\phi}$ can be rewritten as
\begin{align*}
    \mathsf s'=\mathsf s\otimes \det\left(T\fM(\bv,\underline{\bd})^{\sA',\phi\text{-repl}}\right)^{1/2}\in \mathrm{Char}(G^{\phi(\sA')})\otimes_\bZ \bR.
\end{align*}
Let us define
\begin{align*}
    \mathsf s''=\mathsf s\otimes \det\left(T\fM(\bv,\underline{\bd})^{\sA,\phi\text{-repl}}\right)^{1/2}\in \mathrm{Char}(G^{\phi(\sA)})\otimes_\bZ \bR,
\end{align*}
and consider the following diagram on a triangular prism:
\begin{equation*}
\xymatrix{
K^\sT(\fM(\bv,\underline{\bd})^{\sA,\phi},\sw) \ar[rr]^{\mathfrak{m}^\phi_{\fC}}\ar[dr]_{\mathfrak{m}^\phi_{\fC/\fC'}} & & K^\sT(\fM(\bv,\underline{\bd}),\sw) \ar[dd]^{\mathbf k^*} \\
 & K^\sT(\fM(\bv,\underline{\bd})^{\sA',\phi},\sw) \ar[ur]_{\mathfrak{m}^\phi_{\fC'}} &\\
K^\sT(\cM_\theta(\bv,\underline{\bd})^{\sA,\phi},\sw) \ar[rr]^{\HallEnv^{\mathsf s}_{\fC}\qquad\qquad}|-{\phantom{aaa}} \ar[dr]_{\HallEnv^{\mathsf s'}_{\fC/\fC'}} \ar[uu]^{\widetilde{\bPsi}^{\phi,\mathsf s''}_K} & & K^\sT(\cM_\theta(\bv,\underline{\bd}),\sw)  \\
 & K^\sT(\cM_\theta(\bv,\underline{\bd})^{\sA',\phi},\sw) \ar[ur]_{\,\,\, \HallEnv^{\mathsf s}_{\fC'}} \ar[uu]_>>>>>>{\widetilde{\bPsi}^{\phi,\mathsf s'}_K} &
}
\end{equation*}
The right and the backward squares are commutative by the definition of Hall envelope. The left square is commutative by assumption. The upper triangle is commutative because Hall operations are associative. These imply the commutativity of the lower triangle.

The cohomology version is proven similarly and we omit the details.
\end{proof}

We obtain an explicit formula for Hall envelopes when $\sw=0$.
\begin{Proposition}\label{prop: explicit formula PStab w=0}
Assume $\sw=0$, then we have 
\begin{align}\label{explicit formula PStab w=0_coh}
    \HallEnv_\fC([\cM_\theta(\bv,\underline{\bd})^{\sA,\phi}])=\sum_{w\in W/W^\phi}w\left(e^{\sT\times G^{\phi(\sA)}}\left(T\fM(\bv,\underline{\bd})^{\sA,\phi\text{-repl}}\right)\right) \cdot [\cM_\theta(\bv,\underline{\bd})]. 
\end{align}
Here $W$, $W^\phi$ are Weyl groups of $G$, $G^{\phi(\sA)}$ respectively,
$T\fM(\bv,\bd)=R(\bv,\bd)-\Lie(G)$ and $(-)^{\sA, \phi\text{-repl}}$ is the repelling part with respect to homomorphism $\phi\colon \sA\to G$, and $w$ acts on a weight $\mu$ of $G$ by $w(\mu)(g):=\mu(w^{-1}\cdot g\cdot w)$.

Let $\mathsf s\in \mathrm{char}(G)\otimes_\bZ \bR$ be a generic slope and $\chi\in \mathrm{char}(G^{\phi(\sA)})$ such that $\mathsf s\otimes \det\left(T\fM(\bv,\underline{\bd})^{\sA,\phi\text{-repl}}\right)^{1/2}$ is in a sufficiently small neighbourhood of $\chi$, then
\begin{align}\label{explicit formula PStab w=0_k}
    \HallEnv^{\mathsf s}_\fC([\mathcal L_\chi])=\sum_{w\in W/W^\phi}w\left(\chi\cdot e^{\sT\times G^{\phi(\sA)}}_K\left(T\fM(\bv,\underline{\bd})^{\sA,\phi\text{-repl}}\right)\right) \cdot [\mathcal O_{\cM_\theta(\bv,\underline{\bd})}].
\end{align}
Here $\mathcal L_\chi\in \Pic(\cM_\theta(\bv,\underline{\bd})^{\sA,\phi})$ is the descent of the character $\chi$.
\end{Proposition}

\begin{proof}
Let $\Stab_+$ be the same as in Definition \ref{def of nona stab}, then we have 
\begin{align*}
\Stab_+([\cM_\theta(\mathbf v,\underline{\bd})])=[\cM_\theta(\bv,\underline{\bd}^{\bv})]
\end{align*}
since the right-hand-side satisfies the axioms of stable envelope. Replacing $\cM_\theta(\mathbf v,\underline{\bd})$ by $\cM_\theta(\mathbf v,\underline{\bd})^{\sA,\phi}$, we get
\begin{align*}
    \widetilde{\bPsi}^{\phi}_H([\cM_\theta(\mathbf v,\underline{\bd})^{\sA,\phi}])=[\fM(\bv,\underline{\bd})^{\sA,\phi}].
\end{align*}
Similar argument shows that
\begin{align*}
    \widetilde{\bPsi}^{\phi,\mathsf s'}_K([\mathcal L_\chi])=\chi\otimes \mathcal O_{\fM(\bv,\underline{\bd})^{\sA,\phi}}.
\end{align*}
Then \eqref{explicit formula PStab w=0_coh} and \eqref{explicit formula PStab w=0_k} follow from explicit formulas of cohomological \cite[Thm.~2]{CoHA} and $K$-theoretic \cite[Prop.~3.4]{P} Hall operations respectively.
\end{proof}

\subsection{The case of symmetric quiver varieties}

In this subsection, we discuss the case when $Q$ is symmetric with a symmetric framing $\bd_{\In}=\bd_{\Out}$ such that $(G\times \sA)$-action is self-dual, in other words, $\cM_\theta(\bv,\bd)=\cM_\theta(\bv,\underline{\bd})$ is a symmetric quiver variety in Definition \ref{def of sym quiver}, in particular, a symmetric GIT quotient in Definition \ref{def of sym var}. 


\begin{Proposition}\label{prop: compare na stab_sym quiver}
Suppose that $Q$ is symmetric with a symmetric framing $\bd_{\In}=\bd_{\Out}$ such that $(G\times \sA)$-action is self-dual. Then $\widetilde{\bPsi}_H=\bPsi_H$ where the latter is the cohomological nonabelian stable envelope in Definition \ref{na stab coh_sym quot}.  

For a generic slope $\mathsf s\in \mathrm{Char}(G)\otimes_\bZ\bR$, we have $\widetilde{\bPsi}^{\mathsf s}_K={\bPsi}^{\mathsf s}_K$, where the latter is the $K$-theoretic nonabelian stable envelope in Definition \ref{def of na stab k_window}.
\end{Proposition}

\begin{proof}
It is easy to see that the symmetric quiver variety $\cM_\theta(\mathbf v,\mathbf v+\mathbf d)$ is isomorphic to the total space of a vector bundle $E$ on $\cM_\theta(\mathbf v,\underline{\bd}^\bv)$ \footnote{This follows from the fact that $\theta_i<0$ for all $i\in Q_0$. For other generic stability condition $\zeta\in \bQ^{Q_0}$, $\mathrm{Tot}(E)$ is isomorphic to an open subscheme of $\cM_\zeta(\mathbf v,\mathbf v+\mathbf d)$, where $E=\bigoplus_{\zeta_i<0}\Hom(V_i,V'_i)\oplus \bigoplus_{\zeta_i>0}\Hom(V'_i,V_i)$.}, where $E=\bigoplus_{i\in Q_0}\Hom(V_i,V'_i)$. $E$ is repelling with respect to the positive chamber. Then according to Theorem \ref{thm:attr plus repl}, $\Stab_+$ is induced by the correspondence 
$$\left[\overline{\Attr}_+\left(\Delta_{\cM_\theta(\mathbf v,\mathbf d)}\right)\right]\subseteq  \cM_\theta(\mathbf v,\underline{\bd}^\bv)\times \cM_\theta(\mathbf v,\mathbf d).$$ 
So $\widetilde{\bPsi}_H$ is induced by the correspondence given by the irreducible closed substack $$\left[\left(\overline{\Attr}_+\left(\Delta_{\cM_\theta(\mathbf v,\mathbf d)}\right)\bigcap\left(\mathring\cM_\theta(\mathbf v,\underline{\bd}^\bv)\times \cM_\theta(\mathbf v,\mathbf d)\right)\right)/G'\right]\subseteq  \mathfrak{M}(\bv,\bd)\times \cM_\theta(\mathbf v,\mathbf d).$$
Using an equality similar as \eqref{equ on qMattr}, we deduce that  
\begin{equation}\label{equ attplusDelta}\left[\left({\Attr}_+\left(\Delta_{\cM_\theta(\mathbf v,\mathbf d)}\right)\bigcap\left(\mathring\cM_\theta(\mathbf v,\underline{\bd}^\bv)\times \cM_\theta(\mathbf v,\mathbf d)\right)\right)/G'\right]=\Delta_{\cM_\theta(\mathbf v,\mathbf d)}\subseteq  \cM_\theta(\mathbf v,\mathbf d)\times \cM_\theta(\mathbf v,\mathbf d). \end{equation}
Therefore $\widetilde{\bPsi}_H$ is induced by the correspondence given by their closure, i.e. $$\overline{\Delta_{\cM_\theta(\mathbf v,\mathbf d)}}\subseteq  \mathfrak{M}(\bv,\bd)\times \cM_\theta(\mathbf v,\mathbf d).$$
This shows that $\widetilde{\bPsi}_H=\bPsi_H$. The $K$-theoretic version $\widetilde{\bPsi}^{\mathsf s}_K={\bPsi}^{\mathsf s}_K$ follows from Theorem \ref{thm: deg bound na stab k} and Proposition \ref{prop:Psi} below and a direct calculation showing that 
\begin{equation*}\frac{1}{2}\deg_{\sigma_i} e^G_K(R(\bv,\bd)-\Lie(G))+\wt_{\sigma_i}\mathsf s=\deg_{\sigma_i} e^{G}_K\left(N_{S_i/R(\mathbf v,\mathbf d)}\right)+\wt_{\sigma_i}\left(\mathsf s\otimes \left(\det N_{S_i/R(\mathbf v,\mathbf d)}\right)^{1/2}\right). \qedhere  \end{equation*}
\end{proof}


Recall that the unstable locus 
$$R(\mathbf v,\mathbf d)^{\theta\emph{-}u}:=R(\mathbf v,\mathbf d)\setminus R(\mathbf v,\mathbf d)^{\theta\emph{-}ss}$$ admits equivariant Kempf-Ness (KN) stratification by connected locally-closed subvarieties \cite{Kir,DH}. 
The KN stratification is constructed iteratively by selecting a pair $(\sigma_i\in \cochar(G),Z_i:=R(\mathbf v,\mathbf d)^{\sigma_i})$ which maximizes the numerical invariant
\begin{align}
\mu(\sigma):=\frac{(\sigma,\theta)}{|\sigma|}
\end{align}
among those $(\sigma,Z)$ for which $Z$ is not contained in the union of the previously defined strata. Here $|\cdot|$ is a fixed conjugation-invariant norm on the cocharacters of $G$. For our purpose, we fix the norm square $|\sigma|^2$ to be
\begin{align}
    |\sigma|^2:=\sum_{i\in Q_0}|\theta_i| \tr(\sigma_i^2),
\end{align}
where $\sigma_i:\bC^*\to \GL(\mathbf v_i)$ is the $i$-th component of $\sigma$, which is assumed without loss of generality to be diagonal, i.e.
\begin{align*}
    \sigma_i(t)=\diag(t^{\sigma_{i,1}},\cdots,t^{\sigma_{i,\mathbf v_i}}),
\end{align*}
and $\tr(\sigma_i^2)$ is a short-hand notation for $\sum_{j=1}^{\mathbf v_i}\sigma_{i,j}^2$. One defines the open subvariety $Z_i^*\subseteq  Z_i$ to consist of those points not lying on previously defined strata, and $$Y_i:=\Attr_{\sigma_i}(Z_i^*).$$ Then define the new strata $$S_i=G\cdot Y_i.$$

\begin{Proposition}\label{prop:Psi}
Suppose that $Q$ is symmetric with a symmetric framing $\bd_{\In}=\bd_{\Out}$ such that $(G\times \sA)$-action is self-dual. 
Then, for a generic $\mathsf s\in \mathrm{Char}(G)\otimes_\bZ\bR$, $\widetilde{\bPsi}^{\mathsf s}_K$ is the unique $K_\sT(\pt)$-linear map from $K^\sT(\cM_\theta(\mathbf v,\mathbf d),\sw)$ to $K^\sT(\mathfrak{M}(\mathbf v,\mathbf d),\sw)$ which satisfies the following two properties
\begin{itemize}
    \item $\mathbf k^*\circ \widetilde{\bPsi}^{\mathsf s}_K=\id$, where $\mathbf k:\cM_\theta(\bv,\underline{\bd})\hookrightarrow \fM(\bv,\underline{\bd})$ is the open immersion of stable locus,
    \item there is a strict inclusion of intervals
    \begin{align}\label{deg bound on KN strata_K}
        \deg_{\sigma_i} {\mathbf i}_{\sigma_i}^*\widetilde{\bPsi}^{\mathsf s}_K(\gamma)\subsetneq \deg_{\sigma_i} e^{G}_K\left(N_{S_i/R(\mathbf v,\mathbf d)}\right)+\wt_{\sigma_i}\left(\mathsf s\otimes \left(\det N_{S_i/R(\mathbf v,\mathbf d)}\right)^{1/2}\right)
    \end{align}
    for all cocharacters $\sigma_i:\bC^*\to G$ that appear in the KN stratification, and all $\gamma\in K^\sT(\cM_\theta(\mathbf v,\mathbf d),\sw)$. Here $\mathbf i_{\sigma_i}:[Z^*_i/G^{\sigma_i}]\to \mathfrak{M}(\bv,\bd)$ is the natural map between quotient stacks.
\end{itemize}
\end{Proposition}
The proof of this proposition is postponed to the Appendix \ref{sec:proof of prop psi}. 

\begin{Corollary}\label{cor: hall compatible_sym quiver}
Suppose that $Q$ is symmetric with a symmetric framing $\bd_{\In}=\bd_{\Out}$ such that $(G\times \sA)$-action is self-dual. Then for a generic slope $\mathsf s$, we have 
$$\HallEnv_\fC=\Stab_\fC, \quad \HallEnv^{\mathsf s}_\fC=\Stab^{\mathsf s}_\fC. $$
Moreover, $\HallEnv_\fC$ and $\HallEnv^{\mathsf s}_\fC$ are compatible with Hall operations.
\end{Corollary}

\begin{proof}
This follows from Proposition \ref{prop: compare na stab_sym quiver}, Theorem \ref{thm: hall coh_sym quot} and Theorem \ref{thm: hall k_sym quot}. 
\end{proof}

In the below, we simplify the notations by setting $\widetilde{\bPsi}_H=\bPsi_H$ and $\widetilde{\bPsi}^{\mathsf s}_K={\bPsi}^{\mathsf s}_K$ for symmetric quiver $Q$ with a symmetric framing $\bd_{\In}=\bd_{\Out}$ such that $(G\times \sA)$-action is self-dual.

\subsection{Connection to BPS cohomology}\label{sect on bps coh}
In this section, we give a description of cohomological nonabelian stable envelopes on symmetric quiver varieties 
via BPS cohomology of Davison-Meinhardt \cite{DM}. 
This will be used to confirm the prediction of Botta-Davison on `critical stable envelopes' \cite[\S 1.2]{BD} (ref.~Remark \ref{rmk on boda}). 

Consider the Jordan-H\"older morphism 
$$\JH:\fM(\bv,\bd)\to \cM_0(\bv,\bd)$$ sending a representation to the direct sum of the subquotients (simple modules) appearing in its Jordan–H\"older filtration. Let 
$$\IC_{\fM(\bv,\bd)}:=\bQ_{\fM(\bv,\bd)}[\dim \fM(\bv,\bd)]$$ 
be the constant perverse sheaf on $\fM(\bv,\bd)$. It is shown in \cite[Thm.~A]{DM} that the negative degree perverse cohomology of $\JH_*\varphi_\sw \IC_{\fM(\bv,\bd)}$ vanishes, i.e.  
\begin{align*}
    \JH_*\varphi_\sw \IC_{\fM(\bv,\bd)}\in \prescript{p}{}{\D}_c^{\geqslant0}(\cM_0(\bv,\bd)).
\end{align*}
We note that the framed quiver in this paper is related to the unframed quiver considered in \cite{DM} by Crawley-Boevey's procedure \cite{CB}, which results in a $\bC^*$-gerbe 
$$\fM(\bv,\bd)\to [\fM(\bv,\bd)/\bC^*], $$ 
and the original statement of \cite[Thm.~A]{DM} is that 
$$\JH_*\varphi_\sw \IC_{[\fM(\bv,\bd)/\bC^*]}\in \prescript{p}{}{\D}_c^{\geqslant 1}(\cM_0(\bv,\bd)). $$
For a $\sT$-invariant function $\sw$ on $\fM(\bv,\bd)$, one defines the \textit{BPS sheaf}:
\begin{align*}
    \mathcal{BPS}_{\sw}:=\prescript{p}{}{\cH}^0\left(\JH_*\varphi_\sw \IC_{\fM(\bv,\bd)}\right)\in \Perv(\cM_0(\bv,\bd)).
\end{align*}
Note that $\mathcal{BPS}_{\sw}$ is contructed $\sT$-equivariantly, and we define the \textit{BPS cohomology}:
\begin{align*}
H^\sT(\fM(\bv,\bd),\sw)_{\mathrm{BPS}}:=H_\sT(\cM_0(\bv,\bd),\mathcal{BPS}_{\sw}[\dim \fM(\bv,\bd)]).
\end{align*}
We note that $\omega_{\fM(\bv,\bd)}\cong \IC_{\fM(\bv,\bd)}[\dim \fM(\bv,\bd)]$, so there exists a natural $H_\sT(\pt)$-linear map  
\begin{align}\label{BPS embed into total}
    \iota\colon H^\sT(\fM(\bv,\bd),\sw)_{\mathrm{BPS}}\to H^\sT(\fM(\bv,\bd),\sw).
\end{align}
By \cite[Thm.~C]{DM}, the map \eqref{BPS embed into total} is injective.

Consider the Jordan-H\"older morphism 
$$\JH^\theta:\cM_\theta(\bv,\bd)\to \cM_0(\bv,\bd)$$ 
sending a $\theta$-semistable representation to the direct sum of the subquotients (simple modules) appearing in its Jordan–H\"older filtration. It is well-known that $\JH^\theta$ is a proper morphism, and fits into the  commutative diagram
\begin{equation*}
\xymatrix{
\cM_\theta(\bv,\bd) \ar@{^{(}->}[rr]^{j} \ar[dr]_{\JH^\theta} & & \fM(\bv,\bd) \ar[dl]^{\JH}\\
& \cM_0(\bv,\bd), &
}
\end{equation*}
where $j$ is the open immersion of $\theta$-semistable locus. Let 
$$\IC_{\cM_\theta(\bv,\bd)}:=\bQ_{\cM_\theta(\bv,\bd)}[\dim \cM_\theta(\bv,\bd)]$$ be the constant perverse sheaf on $\cM_\theta(\bv,\bd)$. Since $\JH^\theta$ is proper and semismall \cite[Lem.~4.4]{Toda2}\,\footnote{It can be shown that $\JH^\theta$ is small, see \cite[Thm.~1.1]{COZZ3}.}, $\JH^\theta_*\IC_{\cM_\theta(\bv,\bd)}$ is a perverse sheaf on $\cM_0(\bv,\bd)$; thus 
$$\JH^\theta_*\varphi_\sw\IC_{\cM_\theta(\bv,\bd)}\cong \varphi_\sw\JH^\theta_*\IC_{\cM_\theta(\bv,\bd)}$$ 
is also a perverse sheaf. The natural map $\varphi_\sw\IC_{\fM(\bv,\bd)}\to j_*\varphi_\sw\IC_{\cM_\theta(\bv,\bd)}$ induces a map $\JH_*\varphi_\sw \IC_{\fM(\bv,\bd)}\to \JH^\theta_*\varphi_\sw\IC_{\cM_\theta(\bv,\bd)}$. Truncation of the above map to the perverse degree $\leqslant0$ part gives a map:
\begin{align}\label{map Theta}
    \Theta\colon \mathcal{BPS}_\sw\to \JH^\theta_*\varphi_\sw\IC_{\cM_\theta(\bv,\bd)}.
\end{align}
According to \cite[Lem.~4.7]{Toda2}, the map $\Theta$ in \eqref{map Theta} is an isomorphism. We still use $\Theta$ to denote the map induced by taking hypercohomologies on domain and codomain of \eqref{map Theta}. Define
\begin{align}\label{bps coho emb}
    \jmath:=\iota\circ \Theta^{-1}\colon H^\sT(\cM_\theta(\bv,\bd),\sw)\to H^\sT(\fM(\bv,\bd),\sw).
\end{align}

\begin{Proposition}\label{prop:na stab and bps coho}
We have $\jmath=\mathbf\Psi_H$. In particular, $\mathbf\Psi_H\colon H^\sT(\cM_\theta(\bv,\bd),\sw)\to H^\sT(\fM(\bv,\bd),\sw)$ induces an isomorphism $H^\sT(\cM_\theta(\bv,\bd),\sw)\cong H^\sT(\fM(\bv,\bd),\sw)_{\mathrm{BPS}}$.
\end{Proposition}

\begin{proof}
According to the proof of Proposition \ref{prop: compare na stab_sym quiver}, $\mathbf\Psi_H$ is induced by the correspondence $[\overline{\Delta_{\cM_\theta(\mathbf v,\mathbf d)}}]$, 
where $$\overline{\Delta_{\cM_\theta(\mathbf v,\mathbf d)}}\subseteq  \mathfrak{M}(\bv,\bd)\times \cM_\theta(\mathbf v,\mathbf d)$$ is the closure of the diagonal of the $\theta$-semistable locus $\cM_\theta(\bv,\bd)\times \cM_\theta(\mathbf v,\mathbf d)$. Precisely, we regard 
\begin{align*}
[\overline{\Delta_{\cM_\theta(\mathbf v,\mathbf d)}}]&\in H^\sT_{2\dim \cM_\theta(\mathbf v,\mathbf d)}(\mathfrak{M}(\bv,\bd)\times \cM_\theta(\mathbf v,\mathbf d),\sw\boxminus\sw)_{\mathfrak{M}(\bv,\bd)\underset{\cM_0(\mathbf v,\mathbf d)}{\times} \cM_\theta(\mathbf v,\mathbf d)}\text{ \tiny via the canonical map \eqref{can with supp_coh}}\\
\text{\tiny Thom-Sebastiani Thm.}&=H_\sT^{-2\dim \cM_\theta(\mathbf v,\mathbf d)}\left(\mathfrak{M}(\bv,\bd)\underset{\cM_0(\mathbf v,\mathbf d)}{\times} \cM_\theta(\mathbf v,\mathbf d), i^!(\varphi_{\sw}\omega_{\mathfrak{M}(\bv,\bd)}\boxtimes \varphi_{-\sw}\omega_{\cM_\theta(\bv,\bd)})\right),
\end{align*}
where $i:\mathfrak{M}(\bv,\bd)\times_{\cM_0(\mathbf v,\mathbf d)} \cM_\theta(\mathbf v,\mathbf d)\hookrightarrow \mathfrak{M}(\bv,\bd)\times\cM_\theta(\mathbf v,\mathbf d)$ is the natural closed immersion. 

According to \cite[(8.6.4)]{CG}, we have an isomorphism
\begin{align*}
H_\sT^{\bullet}\left(\mathfrak{M}(\bv,\bd)\underset{\cM_0(\mathbf v,\mathbf d)}{\times} \cM_\theta(\mathbf v,\mathbf d), i^!(\varphi_{\sw}\omega_{\mathfrak{M}(\bv,\bd)}\boxtimes \varphi_{-\sw}\omega_{\cM_\theta(\bv,\bd)})\right)
&\cong \Ext^\bullet_{\D^b_\sT(\cM_0(\mathbf v,\mathbf d))}\left(\JH^\theta_*\varphi_{-\sw}\bQ_{\cM_\theta(\bv,\bd)},\JH_*\varphi_{\sw}\omega_{\mathfrak{M}(\bv,\bd)}\right) \\
&\cong \Ext^\bullet_{\D^b_\sT(\cM_0(\mathbf v,\mathbf d))}\left(\JH^\theta_*\varphi_{\sw}\bQ_{\cM_\theta(\bv,\bd)},\JH_*\varphi_{\sw}\omega_{\mathfrak{M}(\bv,\bd)}\right),
\end{align*}
so $[\overline{\Delta_{\cM_\theta(\mathbf v,\mathbf d)}}]$ induces a morphism in $\D^b_\sT(\cM_0(\mathbf v,\mathbf d))$:
\begin{align*}
    [\overline{\Delta_{\cM_\theta(\mathbf v,\mathbf d)}}]\colon \JH^\theta_*\varphi_\sw\omega_{\cM_\theta(\bv,\bd)}\to
    \JH_*\varphi_{\sw}\omega_{\mathfrak{M}(\bv,\bd)}.
\end{align*}
Then according to \cite[Prop.~2.11]{VV2}, critical convolution map $\mathbf\Psi_H$ is obtained by taking hypercohomologies of the above morphism. As we have seen previously, 
$$\JH^\theta_*\varphi_\sw\omega_{\cM_\theta(\bv,\bd)}\in \Perv(\cM_0(\bv,\bd))[\dim \cM_\theta(\bv,\bd)], \quad \JH_*\varphi_{\sw}\omega_{\mathfrak{M}(\bv,\bd)}\in \prescript{p}{}{\D}_c^{\geqslant -\dim \cM_\theta(\bv,\bd)}(\cM_0(\bv,\bd)), $$ 
therefore $[\overline{\Delta_{\cM_\theta(\mathbf v,\mathbf d)}}]$ factors as 
$$[\overline{\Delta_{\cM_\theta(\mathbf v,\mathbf d)}}]\colon \JH^\theta_*\varphi_\sw\omega_{\cM_\theta(\bv,\bd)}\to \mathcal{BPS}_\sw[\dim \cM_\theta(\bv,\bd)]
\xrightarrow{[\overline{\Delta_{\cM_\theta(\mathbf v,\mathbf d)}}]'}   \JH_*\varphi_{\sw}\omega_{\mathfrak{M}(\bv,\bd)}, $$
where $\mathcal{BPS}_\sw[\dim \cM_\theta(\bv,\bd)]=\prescript{p}{}{\tau}^{\leqslant -\dim \cM_\theta(\bv,\bd)}\JH_*\varphi_{\sw}\omega_{\mathfrak{M}(\bv,\bd)}$. 

Composing $[\overline{\Delta_{\cM_\theta(\mathbf v,\mathbf d)}}]$ with the natural map 
$$\JH_*\varphi_{\sw}\omega_{\mathfrak{M}(\bv,\bd)}\to \JH_*j_*\varphi_{\sw}\omega_{\cM_\theta(\bv,\bd)}\cong\JH^\theta_*\varphi_\sw\omega_{\cM_\theta(\bv,\bd)},$$ we get the critical convolution map $[\Delta_{\cM_\theta(\mathbf v,\mathbf d)}]:\JH^\theta_*\varphi_\sw\omega_{\cM_\theta(\bv,\bd)}$ induced by the diagonal of $\cM_\theta(\bv,\bd)\times \cM_\theta(\bv,\bd)$, which is identity map. It follows that $\Theta\circ [\overline{\Delta_{\cM_\theta(\mathbf v,\mathbf d)}}]'=\id$. Then by the definition of $\jmath$, we have $\jmath=\mathbf\Psi_H$.
\end{proof}

\begin{Remark}\label{rmk on boda}
Combining Corollary \ref{cor: hall compatible_sym quiver} with Proposition \ref{prop:na stab and bps coho}, we obtain the following commutative diagram
\begin{equation*}
\xymatrix{
H^\sT(\fM(\bv,\bd)^{\sA,\phi},\sw) \ar[rr]^{\mathfrak{m}^\phi_{\fC}} & & H^\sT(\fM(\bv,\bd),\sw)\\
H^\sT(\cM_\theta(\bv,\bd)^{\sA,\phi},\sw) \ar[u]^{\jmath^\phi} \ar[rr]^{\Stab_\fC} & & H^\sT(\cM_\theta(\bv,\bd),\sw),
\ar[u]_{\jmath}
}
\end{equation*}
where $\jmath^\phi$ is the map \eqref{bps coho emb} for the stack $\fM(\bv,\bd)^{\sA,\phi}$. Suppose that $\sA\cong (\bC^*)^l$ is the framing torus which acts on the framing vector space $D$ by setting $D=\sum_{i=1}^l a_iD_i$ such that $K_\sA(\pt)\cong\bQ[a_i^\pm]_{i=1}^l$. Let $\fC$ be the chamber $a_1<a_2<\cdots<a_l$, then $\Stab_\fC$ agrees with the ``critical stable envelope'' $\mathrm{CStab}$ defined in \cite[\S 1.2]{BD}.
\end{Remark}

\subsection{The case of tripled quivers with canonical cubic potentials}

When $Q$ is the tripled quiver of another quiver $Q'$ with cubic potential $\sw=\sum_{i\in Q_0}\tr(\varepsilon_i\mu_i)$, 
where $\varepsilon_i$ is the loop at $i$-th node and $\mu_i$ is the $i$-th component of the moment map $\mu\colon R(\overline{Q'},\bv,\bd)\to \Lie(G)^\vee$, dimensional reduction gives isomorphisms 
(Example \ref{ex: doubled vs tripled}, Remark \ref{rmk on dim red is}):
\begin{equation}\label{equ on psi dim red}
\begin{split}
    &\delta_H\colon H^\sT(\cN_\theta(\bv,\bd))\cong H^\sT(\cM_\theta(\bv,\bd),\sw),\quad \delta_H\colon H^\sT(\fN(\bv,\bd))\cong H^\sT(\fM(\bv,\bd),\sw),\\
    &\delta_K\colon K^\sT(\cN_\theta(\bv,\bd))\cong K^\sT(\cM_\theta(\bv,\bd),\sw),\quad \delta_K\colon K^\sT(\fN(\bv,\bd))\cong K^\sT(\fM(\bv,\bd),\sw).
\end{split}
\end{equation}
We define the following maps:
\begin{equation}\label{equ on phi relation}
\begin{split}
    &\Psi_H:=\delta_H^{-1}\circ\bPsi_H\circ\delta_H\colon H^\sT(\cN_\theta(\bv,\bd))\to H^\sT(\fN(\bv,\bd)),\\ 
    &\Psi^{\mathsf s}_K:=\delta_K^{-1}\circ\bPsi^{\mathsf s}_K\circ\delta_K\colon K^\sT(\cN_\theta(\bv,\bd))\to K^\sT(\fN(\bv,\bd)).
\end{split}
\end{equation}
Here the Nakajima quiver variety $\cN_\theta(\bv,\bd)$, the preprojective stack $\fN(\bv,\bd)$ fits into open and closed immersion:  
$$\cN_\theta(\bv,\bd)=\mu^{-1}(0)/\!\!/_\theta G\hookrightarrow \fN(\bv,\bd)=[\mu^{-1}(0)/G] \stackrel{\pmb{\iota}}{\hookrightarrow}  \fR(\bv,\bd):=[R(\overline{Q'},\bv,\bd)/G].$$
In \cite[\S 2.1.8]{AO1}, a nonabelian stable envelope\,\footnote{What we call $\overline{\Psi}^{\mathsf s}_K(\alpha)$ here is denoted by $\mathbf s_\alpha$ in \cite[\S 2.1.8]{AO1}.}
$$\overline{\Psi}^{\mathsf s}_K\colon K^\sT(\cN_\theta(\bv,\bd))\to K^\sT(\fR(\bv,\bd))$$ is constructed as the composition of the following maps
\begin{align*}
\xymatrix{
K^{\sT}(\cN_\theta(\mathbf v,\mathbf d))\ar[r]^-{1\otimes \id} & K^{G'\times \sT}(\cN_\theta(\mathbf v,\mathbf d))\ar[r]^-{\Stab^{\mathsf s'}_+} & K^{G'\times \sT}(\cN_\theta(\mathbf v,\mathbf v+\mathbf d))\ar[d]^-{j^*}\\
& K^{\sT}(\fR(\mathbf v,\mathbf d)) & \ar[l]_-{\cong} K^{G'\times \sT}(\mathring\cN_\theta(\mathbf v,\mathbf v+\mathbf d)),
}
\end{align*}
where $\cN_\theta(\mathbf v,\mathbf v+\mathbf d)$ is the Nakajima quiver variety with the quiver data visualized vertex-wise as
\begin{equation*}
\xymatrix{
&V_i \ar@<2pt>[dl]^{J_i} \ar@<2pt>[d]^{B_i} \\
V'_i \ar@<2pt>[ur]^{I_i} & D_i \ar@<2pt>[u]^{A_i}, 
}
\end{equation*}
and $j:\mathring\cN_\theta(\mathbf v,\mathbf v+\mathbf d)\subseteq  \cN_\theta(\mathbf v,\mathbf v+\mathbf d)$ is the open immersion of the locus on which the maps $\{I_i\}_{i\in Q_0}$ are isomorphisms, $\mathsf s'$ is related to $\mathsf s$ by
\begin{align}\label{slope shift}
    \mathsf s'=\mathsf s\otimes \det\left(T^{1/2}_{\fN(\bv,\bv+\bd)}\right)^{-1/2},\quad \text{where }T^{1/2}_{\fN(\bv,\bv+\bd)}=T^{1/2}_{\fN(\bv,\bd)}+\sum_{i\in Q_0}\hbar^{-1}\Hom(V_i,V'_i)\in K_G(\pt).
\end{align}
The shift in slope comes from the difference between degree conditions in Definition \ref{def of stab k} and in \cite[\S9.1]{Oko} (see Remark \ref{rk Oko normalizer}). We fix an orientation of $Q'$ so that 
\begin{align*}
    T^{1/2}_{\fN(\bv,\bd)}=\sum_{a\in Q'_1}\Hom(V_{t(a)},V_{h(a)})+\sum_{i\in Q_0}(\Hom(D_i,V_i)-\End(V_i))\in K_G(\pt).
\end{align*}
Repeating the above procedure with $K$-theory replaced by Borel-Moore homology, we obtain a cohomological nonabelian stable envelope 
$$\overline{\Psi}_H\colon H^\sT(\cN_\theta(\bv,\bd))\to H^\sT(\fR(\bv,\bd)). $$
\begin{Proposition}\label{prop:compare na stab nak}
We have the following equations of maps:
\begin{align}
    \overline{\Psi}_H=\pmb{\iota}_*\circ \Psi_H,\quad \overline{\Psi}^{\mathsf s}_K=\pmb{\iota}_*\circ \Psi^{\mathsf s\otimes \pmb{\delta}}_K, \,\,\, 
    \mathrm{where}\,\, \pmb{\delta}=\det\left(T^{1/2}_{\fN(\bv,\bd)}\right)^{-1/2}.
\end{align}

\end{Proposition}

\begin{proof}
Notice that $\cM_\theta(\bv,\bv+\bd)$ is the total space of the vector bundle $E=\bigoplus_{i\in Q_0}\Hom(V_i,V'_i)$ over $\cM_\theta(\bv,\bv+\bd,\bd)$, and $E$ is repelling with respect to the $+$ chamber. Let $\widetilde{\sw}=\sum_{i\in Q_0}\tr(\varepsilon_i\widetilde{\mu}_i)$ be a function on $\cM_\theta(\bv,\bv+\bd)$, where $\widetilde{\mu}_i\colon R(\overline{Q'},\bv,\bv+\bd)\to 
\Lie(G)^\vee$ is the moment map. Then $\widetilde{\sw}$ is $(\sT\times G')$-invariant and $\widetilde{\sw}\big|_{\cM_\theta(\bv,\bv+\bd)}=\sw'$. By Proposition \ref{prop:vb and stab_repl}, the pushforward along zero section map gives the following commutative diagram
\begin{equation*}
\xymatrix{
K^{G'\times \sT}(\cM_\theta(\mathbf v,\mathbf d),\sw) \ar[r]^-{\Stab^{\mathsf s'}_+} & K^{G'\times \sT}(\cM_\theta(\mathbf v,\mathbf v+\mathbf d),\widetilde\sw)\\
K^{G'\times \sT}(\cM_\theta(\mathbf v,\mathbf d),\sw) \ar[u]^{\id} \ar[r]^-{\Stab^{\mathsf s''}_+} & K^{G'\times \sT}(\cM_\theta(\mathbf v,\underline{\bd}^\bv),\sw') \ar[u]_{\overline{\pmb{\iota}}_*},
}
\end{equation*}
where $\mathsf s''=\mathsf s'\otimes (\det E)^{1/2}$ and $\overline{\pmb{\iota}}:\cM_\theta(\mathbf v,\underline{\bd}^\bv)\hookrightarrow \cM_\theta(\mathbf v,\mathbf v+\mathbf d)$ is the zero section map. Using \eqref{slope shift}, we can rewrite $\mathsf s''=\mathsf s\otimes\pmb{\delta}$. The pushforward map is compatible with dimensional reduction, i.e.\,the following diagram is commutative
\begin{align*}
\xymatrix{
  K^{G'\times \sT}(\mathring\cM_\theta(\mathbf v,\underline{\bd}^\bv),\sw') \ar[r]^{\overline{\pmb{\iota}}_*} & K^{G'\times \sT}(\mathring\cM_\theta(\mathbf v,\mathbf v+\mathbf d),\widetilde\sw) \\
K^{\sT}(\fN(\bv,\bd))  \ar[r]^{\pmb{\iota}_*} \ar[u]^{\delta_K } & K^{\sT}(\fR(\mathbf v,\mathbf d)) \ar[u]_{\delta_K}.
} 
\end{align*}
By Theorem \ref{dim red and stab_K}, we have the following commutative diagram
\begin{align*}
\xymatrix{
  K^{G'\times \sT}(\cM_\theta(\mathbf v,\mathbf d),\sw) \ar[r]^-{\Stab^{\mathsf s'}_+} & K^{G'\times \sT}(\cM_\theta(\mathbf v,\mathbf v+\mathbf d),\widetilde\sw) \\
K^{G'\times \sT}(\cN_\theta(\mathbf v,\mathbf d))  \ar[r]^-{\Stab^{\mathsf s'}_+} \ar[u]^{\delta_K}  & K^{G'\times \sT}(\cN_\theta(\mathbf v,\mathbf v+\mathbf d)) \ar[u]_{\delta_K}.
}
\end{align*}
Here we set normalizer $\mathcal E$ equal to the structure sheaf. Combining the above three commutative diagrams, we obtain the equation $\overline{\Psi}^{\mathsf s}_K=\pmb{\iota}_*\circ \Psi^{\mathsf s\otimes \pmb{\delta}}_K$. The other equation can be proven similarly and we omit the details. 
\end{proof}

In the coming remark, we explain that the above results reproduce previous works 
on comparing stable envelopes and Hall operations on Nakajima quiver varieties. 
\begin{Remark}\label{rmk on cpr with b}
As a special case of Corollary \ref{cor: hall compatible_sym quiver}, Hall envelopes for tripled quiver $\widetilde{Q'}$ with standard cubic potential $\sw$ are compatible with Hall operations. Dimensional reduction of the Hall compatibility diagrams in Definition \ref{defi of prestab} gives the following commutative diagrams:
\begin{equation}\label{dim red hall op vs stab}  
\xymatrix{
H^\sT(\fN(\bv,\bd)^{\sA,\phi}) \ar[rr]^{\mathfrak{m}^\phi_{\Pi,\fC}\circ\, \varepsilon} & & H^\sT(\fN(\bv,\bd))\\
H^\sT(\cN_\theta(\bv,\bd)^{\sA,\phi}) \ar[u]^{\Psi_H^\phi} \ar[rr]^{\mathbf{Stab}_\fC} & & H^\sT(\cN_\theta(\bv,\bd)),
\ar[u]_{\Psi_H}
}\qquad
\xymatrix{
K^\sT(\fN(\bv,\bd)^{\sA,\phi}) \ar[rr]^{\mathfrak{m}^\phi_{\Pi,\fC}\circ\, \mathcal E\cdot} & & K^\sT(\fN(\bv,\bd))\\
K^\sT(\cN_\theta(\bv,\bd)^{\sA,\phi}) \ar[u]^{\Psi_K^{\phi,\mathsf s\otimes\pmb{\delta}_\phi}} \ar[rr]^{\mathbf{Stab}^{\mathsf s}_\fC} & & K^\sT(\cN_\theta(\bv,\bd)).\ar[u]_{\Psi^{\mathsf s\otimes\pmb{\delta}}_K}
}
\end{equation}
Below we explain notations in above.
The Nakajima quiver variety is an open substack of the preprojective stack:
\begin{equation}\label{equ on prepro st}\cN_\theta(\bv,\bd)=\mu^{-1}(0)/\!\!/_\theta G\hookrightarrow \fN(\bv,\bd)=[\mu^{-1}(0)/G], \end{equation}
and their torus fixed loci are 
$$\cN_\theta(\bv,\bd)^{\sA,\phi}=\mu^{-1}(0)^{\sA,\phi}/\!\!/_\theta G^{\phi(\sA)}, \quad \fN(\bv,\bd)^{\sA,\phi}=[\mu^{-1}(0)^{\sA,\phi}/G^{\phi(\sA)}].$$ 
Let $T^{1/2}$ be a polarization of $T^\vir\fN(\bv,\bd)$\,\footnote{$T^\vir\fN(\bv,\bd)=\sum_{a\in Q'_1}\left(\Hom(V_{t(a)},V_{h(a)})+\hbar^{-1}\Hom(V_{h(a)},V_{t(a)})\right)-\sum_{i\in Q'_0}(1+\hbar^{-1})\End(V_i).$},~i.e.~$T^\vir\fN(\bv,\bd)=T^{1/2}+\hbar^{-1}(T^{1/2})^\vee$. 
$\mathbf{Stab}_\fC$ is the cohomological stable envelope with polarization $T^{1/2}$ defined in \cite[\S3]{MO}, which is related to our Definition \ref{def of normalizer coho} (applied 
to Nakajima quiver variety $\cN_\theta(\bv,\bd)$ with zero potential) by 
$$\mathbf{Stab}_\fC=\Stab_{\fC,\varepsilon},\quad \mathrm{where}\,\, \varepsilon=(-1)^{\rk (T^{1/2})^{\sA,\phi\text{-attr}}}. $$ 
$\mathbf{Stab}^{\mathsf s}_\fC$ is the $K$-theoretic stable envelope with slope $\mathsf s$ and polarization $T^{1/2}$ defined in \cite[\S9.1]{Oko} which is related to our Definition \ref{def of normalizer k} (applied to Nakajima quiver variety $\cN_\theta(\bv,\bd)$ with zero potential)  
by a shift in slope and a different normalization (see Remark \ref{rk Oko normalizer}). 
The shifts in slope of $K$-theoretic nonabelian stable envelopes are 
\begin{align*}
    \pmb{\delta}=\det(T^{1/2}),\quad \pmb{\delta}_\phi=\det((T^{1/2})^{\sA,\phi}).
\end{align*} 
$\mathfrak{m}^\phi_{\Pi,\fC}$ in \eqref{dim red hall op vs stab} is the Hall operation for preprojective stack which is the dimensional reduction of $\mathfrak{m}^\phi_\fC$ \cite{YZ1,YZ2}. $\mathcal E\cdot$ is the multiplication by the $K$-theory class
\begin{align*}
    \mathcal{E}=\frac{(-1)^{\rk(T^{1/2})^{\sA,\phi\text{-attr}}}}{\det(\sqrt{\hbar}(T^{1/2})^{\sA,\phi\text{-attr}})}.
\end{align*}
In \cite[Thm.~4.3]{YZ1}, it is shown that the preprojective Hall operation $\mathfrak{m}^\phi_{\Pi,\fC}$ is  compatible with a (twisted) ordinary Hall operation $\mathfrak{m}^\phi_{\fC}$ on the smooth stack $\fR(\bv,\bd)$ in the sense that the following diagrams are commutative
\begin{equation}\label{preproj hall vs tw hall op}
\xymatrix{
H^\sT(\fR(\bv,\bd)^{\sA,\phi}) \ar[rr]^{\widetilde{\mathfrak{m}}^\phi_\fC} & & H^\sT(\fR(\bv,\bd))\\
H^\sT(\fN(\bv,\bd)^{\sA,\phi}) \ar[u]^{\pmb{\iota}_*} \ar[rr]^{\mathfrak{m}^\phi_{\Pi,\fC}} & & H^\sT(\fN(\bv,\bd)),
\ar[u]_{\pmb{\iota}_*}
}\qquad
\xymatrix{
K^\sT(\fR(\bv,\bd)^{\sA,\phi}) \ar[rr]^{\widetilde{\mathfrak{m}}^\phi_\fC} & & K^\sT(\fR(\bv,\bd))\\
K^\sT(\fN(\bv,\bd)^{\sA,\phi}) \ar[u]^{\pmb{\iota}_*} \ar[rr]^{\mathfrak{m}^\phi_{\Pi,\fC}} & & K^\sT(\fN(\bv,\bd)),\ar[u]_{\pmb{\iota}_*}
}
\end{equation}
where
\begin{equation}\label{equ on m twist}
\widetilde{\mathfrak{m}}^\phi_\fC=\mathfrak{m}^\phi_\fC\circ e^{\sT\times G^{\phi(\sA)}}_{\bullet}(\hbar^{-1}\Lie(G)^{\sA,\phi\text{-attr}})\cdot\,,
\end{equation}
for $\bullet=$ empty or $K$ depending on whether it is cohomological or $K$-theoretic. 
Combining Proposition \ref{prop:compare na stab nak} with diagrams \eqref{dim red hall op vs stab} and \eqref{preproj hall vs tw hall op}, we obtain the following commutative diagrams:
\begin{equation}\label{dim red hall op vs stab 2}
\xymatrix{
H^\sT(\fR(\bv,\bd)^{\sA,\phi}) \ar[rr]^{\widetilde{\mathfrak{m}}^\phi_\fC\circ\,\varepsilon} & & H^\sT(\fR(\bv,\bd))\\
H^\sT(\cN_\theta(\bv,\bd)^{\sA,\phi}) \ar[u]^{\overline{\Psi}_H^\phi} \ar[rr]^{\mathbf{Stab}_\fC} & & H^\sT(\cN_\theta(\bv,\bd)),
\ar[u]_{\overline{\Psi}_H}
}\qquad
\xymatrix{
K^\sT(\fR(\bv,\bd)^{\sA,\phi}) \ar[rr]^{\widetilde{\mathfrak{m}}^\phi_\fC\circ\, \mathcal E\cdot} & & K^\sT(\fR(\bv,\bd))\\
K^\sT(\cN_\theta(\bv,\bd)^{\sA,\phi}) \ar[u]^{\overline{\Psi}_K^{\phi,\mathsf s}} \ar[rr]^{\mathbf{Stab}^{\mathsf s}_\fC} & & K^\sT(\cN_\theta(\bv,\bd)),\ar[u]_{\overline{\Psi}^{\mathsf s}_K}
}
\end{equation}
where $\overline{\Psi}_H, \overline{\Psi}^{\mathsf s}_K$ are constructed in \cite{AO1}. 

When $\sA$ is a framing torus, the left square in \eqref{dim red hall op vs stab 2} reproduces 
a result of Botta \cite[Thm.\,5.6]{Bot}, where his proof uses the method of abelianizations \cite{Sh}. Botta-Davison \cite{BD} has another proof of Botta's result by showing that ${\mathfrak{m}}^\phi_{\Pi,\fC}\circ {\Psi}_H^\phi$ and ${\Psi}_H\circ\Stab_\fC$ are induced by the same correspondence. 

Moreover, the left square in \eqref{dim red hall op vs stab 2} can be used to deduce an inductive formula of stable envelopes for Nakajima quiver varieties \cite[Thm.~5.12]{Bot}, which reduces the computations of stable envelopes to the computations of $\overline{\Psi}_H$ in the framing rank one cases. 

\end{Remark}

As an application, we deduce explicit formulas of stable envelopes for the Nakajima quiver variety $\cN_\theta(\bv,\bd)$ using Remark \ref{rmk on cpr with b} under certain assumptions on the fixed loci. Note that 
\begin{equation}\label{equ on rqA}R(\overline{Q'},\bv,\bd)^{\sA,\phi}=R(\overline{Q'_\phi},\bv_\phi,\bd_\phi) \end{equation} for a certain quiver $Q'_\phi$ (possibly disconnected) with gauge dimension vector $\bv_\phi$ and framing dimension vector $\bd_\phi$. 

The following result, in the case when $Q'$ is an affine type A quiver and $\sA=$ framing torus $\times$ $\bC^*$ that resolves the loop (i.e.~\eqref{equ on rqA} is a finite type A quiver variety), 
recovers cohomological and $K$-theoretic limits of \cite[Thm.~1]{Din} and also \cite[Prop.~7]{AO1}.


\begin{Corollary}\label{cor: explicit formula nak}
Assume that every connected component of $(Q'_\phi,\bv_\phi,\bd_\phi)$ is a Dynkin quiver with minuscule framing,~i.e.~there is a node $i_0$ that corresponds to a minuscule fundamental weight of the Lie algebra associated with the Dynkin quiver, such that $\bd_{\phi,i_0}=1$, $\bd_{\phi,i}=0$ if $i\neq i_0$. 

Then 
in $\sT$-localized homology $H^\sT(\fN(\bv,\bd))_{\loc}:=H^\sT(\fN(\bv,\bd))\otimes_{\bQ[\Lie(\sT)]}\mathrm{Frac}(\bQ[\Lie(\sT)])$, we have 
\begin{align}\label{explicit formula nak_coh}
\Psi_H\circ\mathbf{Stab}_\fC([\cN_\theta(\bv,\bd)^{\sA,\phi}])= 
\sum_{w\in W/W^\phi}w\left(e^{\sT\times G^{\phi(\sA)}}\left((T^{1/2})^{\sA,\phi\text{-repl}}+\hbar(T^{1/2})^{\sA,\phi\text{-attr}}\right)\right) \cdot [\fN(\bv,\bd)]^{\vir}\:,
\end{align}
where $\fN(\bv,\bd)$ \eqref{equ on prepro st} has a quasi-smooth stack structure and $[\fN(\bv,\bd)]^{\vir}$ denotes its virtual class. In particular,
\begin{align*}
\mathbf{Stab}_\fC([\cN_\theta(\bv,\bd)^{\sA,\phi}])= 
\sum_{w\in W/W^\phi}w\left(e^{\sT\times G^{\phi(\sA)}}\left((T^{1/2})^{\sA,\phi\text{-repl}}+\hbar(T^{1/2})^{\sA,\phi\text{-attr}}\right)\right) \cdot [\cN_\theta(\bv,\bd)]\:.
\end{align*}
Let $\mathsf s$ be a generic slope and $\chi\in \mathrm{char}(G^{\phi(\sA)})$ such that $\mathsf s\otimes \pmb{\delta}_\phi$ is in a sufficiently small neighbourhood of $\chi$, then 
in $\sT$-localized $K$-theory $K^\sT(\fN(\bv,\bd))_{\loc}:=K^\sT(\fN(\bv,\bd))\otimes_{\bQ[\sT]}\mathrm{Frac}(\bQ[\sT])$, we have 
\begin{align}\label{explicit formula nak_k}
\Psi^{\mathsf s\otimes\pmb{\delta}}_K\circ\mathbf{Stab}^{\mathsf s}_\fC([\mathcal L_\chi])=\sqrt{\hbar}^{\rk (T^{1/2})^{\sA,\phi\text{-attr}}}\cdot
\sum_{w\in W/W^\phi}w\left(\chi\cdot e^{\sT\times G^{\phi(\sA)}}_K\left((T^{1/2})^{\sA,\phi\text{-repl}}+\hbar(T^{1/2})^{\sA,\phi\text{-attr}}\right)\right) \cdot 
[\mathcal O^{\vir}_{\fN(\bv,\bd)}]\:,
\end{align}
where $\mathcal L_\chi\in \Pic(\cN_\theta(\bv,\bd)^{\sA,\phi})$ is the descent of the character $\chi$, and 
$[\mathcal O^{\vir}_{\fN(\bv,\bd)}]$ denotes the virtual structure sheaf.  

In particular,
\begin{align*}
\mathbf{Stab}^{\mathsf s}_\fC([\mathcal L_\chi])=\sqrt{\hbar}^{\rk (T^{1/2})^{\sA,\phi\text{-attr}}}\cdot
\sum_{w\in W/W^\phi}w\left(\chi\cdot e^{\sT\times G^{\phi(\sA)}}_K\left((T^{1/2})^{\sA,\phi\text{-repl}}+\hbar(T^{1/2})^{\sA,\phi\text{-attr}}\right)\right) \cdot 
[\mathcal O_{\cN_\theta(\bv,\bd)}]\:.
\end{align*}
\end{Corollary}

\begin{proof}
By Lemma \ref{lem:na stab minuscule ADE} below, we have 
\begin{align*}
\Psi_H([\cN_\theta(\bv,\bd)^{\sA,\phi}])=[\fN(\bv,\bd)^{\sA,\phi}]^{\vir},\quad \Psi^{\phi,\mathsf s\otimes\pmb{\delta}_\phi}_K([\mathcal L_\chi])=[\chi\otimes \mathcal O^{\vir}_{\fN(\bv,\bd)^{\sA,\phi}}].
\end{align*}
By the commutativity of diagrams \eqref{dim red hall op vs stab}, we get 
\begin{align*}
\Psi_H\circ\mathbf{Stab}_\fC([\cN_\theta(\bv,\bd)^{\sA,\phi}])=\mathfrak{m}^\phi_{\Pi,\fC}\left(\epsilon\cdot[\fN(\bv,\bd)^{\sA,\phi}]^{\vir}\right),\quad \Psi^{\mathsf s\otimes\pmb{\delta}}_K\circ\mathbf{Stab}^{\mathsf s}_\fC([\mathcal L_\chi])=\mathfrak{m}^\phi_{\Pi,\fC}\left(\mathcal E\cdot[\chi\otimes \mathcal O^{\vir}_{\fN(\bv,\bd)^{\sA,\phi}}]\right).
\end{align*}
Since $$R(\overline{Q'},\bv,\bd)=M\oplus \hbar^{-1}M^\vee, \,\, \mathrm{where}\,\,\, 
M=\bigoplus_{a\in Q'_1}\Hom(V_{t(a)},V_{h(a)})\oplus\bigoplus_{i\in Q'_0}\Hom(D_i,V_i), $$ 
we have inclusion $R(\overline{Q'},\bv,\bd)^\sT\subseteq  M$. Since $M\subseteq  \mu^{-1}(0)$, we see that $R(\overline{Q'},\bv,\bd)^\sT\subseteq  \mu^{-1}(0)$. Therefore the pushforward maps 
$$\pmb{\iota}_*\colon H^\sT(\fN(\bv,\bd))\to H^\sT(\fR(\bv,\bd)), \quad \pmb{\iota}_*\colon K^\sT(\fN(\bv,\bd))\to K^\sT(\fR(\bv,\bd))$$ 
are isomorphisms after $\sT$-localization. Note that 
$$\pmb{\iota}_*([\fN(\bv,\bd)]^{\vir})=e^{\sT\times G}(\hbar^{-1}\Lie(G)^\vee)\cdot[\fR(\bv,\bd)], \quad \pmb{\iota}_*([\mathcal O^{\vir}_\fN(\bv,\bd)])=e^{\sT\times G}_K(\hbar^{-1}\Lie(G)^\vee)\cdot[\mathcal O_\fR(\bv,\bd)]. $$ 
Then it follows from the compatibility \eqref{preproj hall vs tw hall op} that 
\begin{align*}
&\quad\, \Psi_H\circ\mathbf{Stab}_\fC([\cN_\theta(\bv,\bd)^{\sA,\phi}])=\pmb{\iota}_*^{-1}\circ\widetilde{\mathfrak{m}}^\phi_{\fC}\left(\epsilon\cdot e^{\sT\times G^{\phi(\sA)}}(\hbar^{-1}\Lie(G^{\phi(\sA)})^\vee)\cdot [\fR(\bv,\bd)^{\sA,\phi}]\right) \\
&=\pmb{\iota}_*^{-1}\sum_{w\in W/W^\phi}w\left(\epsilon\cdot e^{\sT\times G^{\phi(\sA)}}\left(\hbar^{-1}\Lie(G^{\phi(\sA)})^\vee+\hbar^{-1}\Lie(G)^{\sA,\phi\text{-attr}}+T\fR(\bv,\bd)^{\sA,\phi\text{-repl}}\right) \right)\cdot [\fR(\bv,\bd)]\\
&=\sum_{w\in W/W^\phi}w\left(\epsilon\cdot e^{\sT\times G^{\phi(\sA)}}\left(\hbar^{-1}\Lie(G^{\phi(\sA)})^\vee+\hbar^{-1}\Lie(G)^{\sA,\phi\text{-attr}}+T\fR(\bv,\bd)^{\sA,\phi\text{-repl}}\right)\right) \cdot\frac{[\fN(\bv,\bd)]^{\vir}}{e^{\sT\times G}(\hbar^{-1}\Lie(G)^\vee)}\\
&=\sum_{w\in W/W^\phi}w\left(e^{\sT\times G^{\phi(\sA)}}\left((T^{1/2})^{\sA,\phi\text{-repl}}+\hbar(T^{1/2})^{\sA,\phi\text{-attr}}\right)\right) \cdot [\fN(\bv,\bd)]^{\vir}.
\end{align*}
Here in the second equality, we use \eqref{equ on m twist} and the explicit formula of Hall operation as \eqref{explicit formula w=0_coh}.

Similarly in the $K$-theory, we have 
\begin{align*}
&\quad\, \Psi^{\mathsf s\otimes\pmb{\delta}}_K\circ\mathbf{Stab}^{\mathsf s}_\fC([\mathcal L_\chi])=\pmb{\iota}_*^{-1}\circ\widetilde{\mathfrak{m}}^\phi_{\fC}\left(\epsilon\cdot e^{\sT\times G^{\phi(\sA)}}_K(\hbar^{-1}\Lie(G^{\phi(\sA)})^\vee)\cdot [\mathcal O_{\fR(\bv,\bd)^{\sA,\phi}}]\right) \\
&=\pmb{\iota}_*^{-1}\sum_{w\in W/W^\phi}w\left(\chi\cdot\mathcal E\cdot e^{\sT\times G^{\phi(\sA)}}_K\left(\hbar^{-1}\Lie(G^{\phi(\sA)})^\vee+\hbar^{-1}\Lie(G)^{\sA,\phi\text{-attr}}+T\fR(\bv,\bd)^{\sA,\phi\text{-repl}}\right)\right) \cdot [\mathcal O_{\fR(\bv,\bd)}]\\
&=\sum_{w\in W/W^\phi}w\left(\chi\cdot\mathcal E\cdot e^{\sT\times G^{\phi(\sA)}}_K\left(\hbar^{-1}\Lie(G^{\phi(\sA)})^\vee+\hbar^{-1}\Lie(G)^{\sA,\phi\text{-attr}}+T\fR(\bv,\bd)^{\sA,\phi\text{-repl}}\right)\right) \cdot\frac{[\mathcal O^{\vir}_{\fN(\bv,\bd)}]}{e^{\sT\times G}_K(\hbar^{-1}\Lie(G)^\vee)}\\
&=\sqrt{\hbar}^{\rk (T^{1/2})^{\sA,\phi\text{-attr}}}\cdot\sum_{w\in W/W^\phi}w\left(\chi\cdot e^{\sT\times G^{\phi(\sA)}}_K\left((T^{1/2})^{\sA,\phi\text{-repl}}+\hbar(T^{1/2})^{\sA,\phi\text{-attr}}\right)\right) \cdot [\mathcal O^{\vir}_{\fN(\bv,\bd)}]. \qedhere
\end{align*}
\end{proof}
In the above proof, we use the following lemma, whose proof will be given in Appendix \ref{sec on proof of lem:na stab minuscule ADE}.  
\begin{Lemma}\label{lem:na stab minuscule ADE}
Let $\Gamma$ be a Dynkin quiver with gauge dimension vector $\br$ and framing dimension vector $\be$ such that $\be$ is minuscule. Assume that $\cN_\theta(\br,\be)$ is nonempty. Then
\begin{align}\label{na stab minuscule ADE_coh}
    \Psi_H([\cN_\theta(\br,\be)])=[\fN(\br,\be)]^{\vir}.
\end{align}
Let $G=\prod_{i\in \Gamma_0}\GL(\br_i)$ denote the gauge group, and $\mathsf s\in \mathrm{char}(G)\otimes_\bZ \bR$ be a generic slope taken in a sufficiently small neighbourhood of $\chi\in \mathrm{char}(G)$, then 
\begin{align}\label{na stab minuscule ADE_k}
    \Psi^{\mathsf s}_K([\mathcal L_\chi])=[\chi\otimes \mathcal O^{\vir}_{\fN(\br,\be)}],
\end{align}
where $\mathcal L_\chi\in \Pic(\cN_\theta(\br,\be))$ is the descent of the character $\chi$.
\end{Lemma}

\begin{Remark}
We note that $\cN_\theta(\br,\be)$ is a point for minuscule $\be$, so $\cN_\theta(\bv,\bd)^{\sA,\phi}$ in Corollary \ref{cor: explicit formula nak} is also a point. Then \eqref{explicit formula nak_coh} and \eqref{explicit formula nak_k} completely determine $\mathbf{Stab}_\fC$ and $\mathbf{Stab}^{\mathsf s}_\fC$ in this situation.
\end{Remark}

\subsection{Hall operation compatible Hall envelopes}\label{sect on hall comp with hall}
We show for a class of asymmetric quiver varieties, where the asymmetric part comes only from the framing, their Hall envelopes 
are compatible with Hall operations.

\begin{Definition}\label{def: symmetrization}
Let $Q$ be a symmetric quiver and $\bv,\bd_{\In},\bd_{\Out}\in \bN^{Q_0}$ be dimension vectors. The \textit{symmetrization} of $R(\bv,\underline{\bd})$ is defined to be $R(\bv,\mathbf c)$ \eqref{equ on sym qu} where $\mathbf c_i=\max\{\bd_{\In,i},\bd_{\Out,i}\}$. We identify $R(\bv,\underline{\bd})$ as a $G$-subrepresentation of 
\begin{align*}
    R(\bv,\mathbf c)=R(\bv,\underline{\bd})\oplus \bigoplus_{\begin{subarray}{c}i\in Q_0 \\ \mathrm{s.t.}\, \bd_{\In,i}>\bd_{\Out,i}   \end{subarray}}
     \Hom(\bC^{\bv_i},\bC^{\bd_{\In,i}-\bd_{\Out,i}})\oplus \bigoplus_{\begin{subarray}{c}i\in Q_0 \\ \mathrm{s.t.}\, \bd_{\In,i}<\bd_{\Out,i}   \end{subarray}} \Hom(\bC^{\bd_{\Out,i}-\bd_{\In,i}},\bC^{\bv_i}).
\end{align*}
\end{Definition}

\begin{Definition}\label{def: pseudo-self-dual}
Let $Q$ be a symmetric quiver and $\bv,\bd_{\In},\bd_{\Out}\in \bN^{Q_0}$ be dimension vectors and $\sA$ be a subtorus of flavour group $\Aut_G R(\bv,\underline{\bd})$. We say that the $\sA$-action on $R(\bv,\underline{\bd})$ is \textit{pseudo-self-dual} if it extends to a self-dual $\sA$-action on the symmetrization $R(\bv,\mathbf c)$. 
We say that the $\sA$-action on $\cM_\theta(\bv,\underline{\bd})$ is \textit{pseudo-self-dual} if it is induced from a pseudo-self-dual action on $R(\bv,\underline{\bd})$.
\end{Definition}

Here is the main theorem of this section. 
\begin{Theorem}\label{thm: AFSQV}
Let $Q$ be a symmetric quiver and $\bv,\bd_{\In},\bd_{\Out}\in \bN^{Q_0}$ be dimension vectors such that $\bd_{\In,i}\geqslant \bd_{\Out,i}$ for all $i\in Q_0$, choose cyclic stability $\theta$
\eqref{equ on cycc condit}. Let $\sT\subseteq  \Aut_G R(\bv,\underline{\bd})$ be a torus in the flavour group 
and $\sA$ be a pseudo-self-dual subtorus of $\sT$, $\fC$ be a chamber associated with $\sA$ action on $\cM_\theta(\bv,\underline{\bd})$, 
and $\mathsf s$ be a generic slope. 

Then Hall envelopes $\HallEnv_\fC$ and $\HallEnv^{\mathsf s}_\fC$ are compatible with Hall operations. 
\end{Theorem}
By Theorem \ref{thm: AFSQV} and Lemma \ref{tri lem for compatible prestab}, we get the triangle lemma for the above asymmetric 
quiver varieties. 
\begin{Corollary}\label{cor: tri lem for AFSQV}
In the setting of Theorem \ref{thm: AFSQV} and let $X=\cM_\theta(\bv,\underline{\bd})$, then both diagrams in \eqref{cd: tri lem} commute.
\end{Corollary}

\subsubsection{Proof of Theorem \ref{thm: AFSQV}, cohomology case}

We begin with the following lemma.
In below, we omit the canonical map in Lemma \ref{can map induce stab} and simply write 
$$[\overline{\Attr}_\fC\left(\Delta_{(-)}\right)]:=\can\circ\,[\overline{\Attr}_\fC\left(\Delta_{(-)}\right)].$$

\begin{Lemma}\label{lem: prestab corr}
In the setting of Theorem \ref{thm: AFSQV}, the Hall envelope 
$$\HallEnv_\fC\colon H^\sT(\cM_\theta(\bv,\underline{\bd})^{\sA},\sw)\longrightarrow H^\sT(\cM_\theta(\bv,\underline{\bd}),\sw)$$ is induced by the correspondence
\begin{align*}
    \left[\overline{\Attr}_\fC(\Delta_{\cM_\theta(\bv,\underline{\bd})^{\sA}})\right]\in H^\sT(\cM_\theta(\bv,\underline{\bd})\times \cM_\theta(\bv,\underline{\bd})^{\sA},\sw\boxminus\sw)_{\Attr^f_\fC}\;.
\end{align*}
\end{Lemma}

\begin{proof}
Consider the vector bundle $E=E_+\oplus E_-$ on $\cM_\theta(\mathbf v,\underline{\bd}^{\bv})$, where the fibers of $E_+$ and $E_-$ are
\begin{equation}\label{bundle added to AFSQV}
\bigoplus_{\begin{subarray}{c}i\in Q_0 \\ \mathrm{s.t.}\, \bd_{\In,i}>\bd_{\Out,i}   \end{subarray}}
\Hom(V_i,\bC^{\bd_{\In,i}-\bd_{\Out,i}})\;,\quad \bigoplus_{i\in Q_0}\Hom(V_i,V'_i)\;.
\end{equation}
Then $\mathrm{Tot}(E)$ is isomorphic to the symmetric quiver variety $\cM_\theta(\mathbf v,\bv+\bd_{\In})$.
Note that
$E_+$ (resp.~$E_-$) is attracting (resp.~repelling) with respect to the positive chamber under the $\bC^*$-action in \eqref{equ on Cstar}. Then by 
Propositions \ref{prop:vb and stab_attr_corr}, \ref{prop:vb and stab_repl_corr}, the $\Stab_+$ (in Definition \ref{def of nona stab}) is induced by the correspondence 
$$\left[\overline{\Attr}_+\left(\Delta_{\cM_\theta(\mathbf v,\underline{\bd})}\right)\right]\subseteq  \cM_\theta(\mathbf v,\underline{\bd}^{\bv})\times \cM_\theta(\mathbf v,\underline{\bd}),$$ 
where the closure is taken inside $\cM_\theta(\mathbf v,\underline{\bd}^{\bv})\times \cM_\theta(\mathbf v,\underline{\bd})$.

So $\widetilde{\bPsi}_H$ (in Definition \ref{def of nona stab}) is induced by the correspondence given by the irreducible closed substack $$\left[\left(\overline{\Attr}_+\left(\Delta_{\cM_\theta(\mathbf v,\underline{\bd})}\right)\bigcap\left(\mathring\cM_\theta(\mathbf v,\underline{\bd}^{\bv})\times \cM_\theta(\mathbf v,\underline{\bd})\right)\right)/G'\right]\subseteq  \mathfrak{M}(\bv,\underline{\bd})\times \cM_\theta(\mathbf v,\underline{\bd}).$$
As in \eqref{equ attplusDelta}, we have 
$$\left[\left(\Attr_+\left(\Delta_{\cM_\theta(\mathbf v,\underline{\bd})}\right)\bigcap\left(\mathring\cM_\theta(\mathbf v,\underline{\bd}^{\bv})\times \cM_\theta(\mathbf v,\underline{\bd})\right)\right)/G'\right]=\Delta_{\cM_\theta(\mathbf v,\underline{\bd})}\subseteq  \cM_\theta(\mathbf v,\underline{\bd})\times \cM_\theta(\mathbf v,\underline{\bd}).$$ 
Therefore $\widetilde{\bPsi}_H$ is induced by the correspondence given by the closure of the diagonal of stable locus $\cM_\theta(\mathbf v,\underline{\bd})$ inside the stack $\mathfrak{M}(\bv,\underline{\bd})\times \cM_\theta(\mathbf v,\underline{\bd})$, i.e. $$\overline{\Delta_{\cM_\theta(\mathbf v,\underline{\bd})}}\subseteq  \mathfrak{M}(\bv,\underline{\bd})\times \cM_\theta(\mathbf v,\underline{\bd}).$$
Let $\phi:\sA\to G$ be a homomorphism that appears in \eqref{fix pt decomp}, then  $\widetilde{\bPsi}_H^{\phi}$ is similarly induced by the correspondence 
\begin{align*}
    [\overline{\Delta_{\cM_\theta(\mathbf v,\underline{\bd})^{\sA,\phi}}}].
\end{align*}
The same argument as Proposition \ref{prop: compare na stab_sym quiver} shows that $\mathfrak{m}^\phi_{\fC}\circ \widetilde{\bPsi}_H^\phi$ is induced by the correspondence $$[Z]\in H^\sT(\fM(\bv,\underline{\bd})\times \cM_\theta(\bv,\underline{\bd})^{\sA,\phi},\sw\boxminus\sw)_{Z}, $$
where $Z=\overline{\Attr}_\fC\left(\Delta_{\cM_\theta(\bv,\underline{\bd})^{\sA,\phi}}\right)$ is the closure of $\Attr_\fC\left(\Delta_{\cM_\theta(\bv,\underline{\bd})^{\sA,\phi}}\right)$ inside the stack $\fM(\bv,\underline{\bd})\times \cM_\theta(\bv,\underline{\bd})^{\sA,\phi}$ (here attracting set is taken inside the stable locus 
$\cM_\theta(\bv,\underline{\bd})\times \cM_\theta(\bv,\underline{\bd})^{\sA,\phi}$). 

It follows that $\HallEnv_\fC=\mathbf k^*\circ \mathfrak{m}^\phi_{\fC}\circ \widetilde{\bPsi}_H^\phi$ is induced by the correspondence $\left[Z\big|_{\cM_\theta(\bv,\underline{\bd})\times \cM_\theta(\bv,\underline{\bd})^{\sA,\phi}}\right]$, which is exactly $\left[\overline{\Attr}_\fC(\Delta_{\cM_\theta(\bv,\underline{\bd})^{\sA,\phi}})\right]$,
 where the closure is taken in $\cM_\theta(\bv,\underline{\bd})\times \cM_\theta(\bv,\underline{\bd})^{\sA,\phi}$
\end{proof}

Let $\bC^*_t\subseteq  G'$ be a torus in the center of the flavour group which acts on $V'_i$ in \eqref{diag of R'2} with weight $-1$ for all $i\in Q_0$, then $\widetilde{\sA}:=\sA\times \bC^*_t$ naturally acts on $R(\bv,\underline{\bd}^{\bv})$. Let 
$$K_{\bC^*_t}(\pt)=\bQ[t^\pm], \quad K_{\sA}(\pt)=\bQ[a_i^\pm]_{i=1}^r, $$ and consider the chamber $\widetilde{\fC}$ in $\Lie(\widetilde{\sA})_\bR$ such that $\fC_+=\{a_i=0\}_{i=1}^r\times\{t>0\}$ is a face of $\widetilde{\fC}$ and $\widetilde{\fC}/\fC_+=\fC$. Combining Lemma \ref{lem: prestab corr} with Proposition \ref{tri lem for unsym fram} and Theorem \ref{tri lem for coh stab}, we see that the composition 
\begin{align*}
    \Stab_+\circ \HallEnv_\fC\colon H^{\sT\times G'}(\cM_\theta(\bv,\underline{\bd})^{\sA,\phi},\sw)\to H^{\sT\times G'}(\cM_\theta(\bv,\underline{\bd}^{\bv}),\sw)
\end{align*}
is induced by the correspondence $$\left[\overline{\Attr}_{\widetilde{\fC}}\left(\Delta_{\cM_\theta(\bv,\underline{\bd})^{\sA,\phi}}\right)\right],$$ which is the closure of the attracting set $\Attr_{\widetilde{\fC}}\left(\Delta_{\cM_\theta(\bv,\underline{\bd})^{\sA,\phi}}\right)$ inside the variety $\cM_\theta(\mathbf v,\underline{\bd}^{\bv})\times \cM_\theta(\bv,\underline{\bd})^{\sA,\phi}$.

Therefore the map $\widetilde{\bPsi}_H\circ \HallEnv_\fC$ is induced by the correspondence $$[S]\in H^\sT(\fM(\bv,\underline{\bd})\times \cM_\theta(\bv,\underline{\bd})^{\sA,\phi},\sw\boxminus\sw)_{S},$$
where
\begin{align*}
S=\left[\frac{\overline{\Attr}_{\widetilde{\fC}}\left(\Delta_{\cM_\theta(\bv,\underline{\bd})^{\sA,\phi}}\right)\bigcap \left(\mathring\cM_\theta(\mathbf v,\underline{\bd}^{\bv})\times \cM_\theta(\bv,\underline{\bd})^{\sA,\phi}\right)}{G'}\right]
\end{align*}
is an irreducible closed substack of $\fM(\bv,\underline{\bd})\times \cM_\theta(\bv,\underline{\bd})^{\sA,\phi}$. We have 
\begin{align*}
\Attr_{\widetilde{\fC}}\left(\Delta_{\cM_\theta(\bv,\underline{\bd})^{\sA,\phi}}\right)\bigcap \left(\mathring\cM_\theta(\mathbf v,\underline{\bd}^{\bv})\times \cM_\theta(\bv,\underline{\bd})^{\sA,\phi}\right)=\left(\mathsf{q}\times \id_{ \cM_\theta(\bv,\underline{\bd})^{\sA,\phi}}\right)^{-1}\left(\Attr_\fC\left(\Delta_{\cM_\theta(\bv,\underline{\bd})^{\sA,\phi}}\right)\right),
\end{align*}
where $\mathsf q:\mathring\cM_\theta(\mathbf v,\underline{\bd}^{\bv})\to \mathfrak{M}(\mathbf v,\underline{\bd})$ is the quotient map and $\Attr_\fC\left(\Delta_{\cM_\theta(\bv,\underline{\bd})^{\sA,\phi}}\right)$ is the attracting set of $\Delta_{\cM_\theta(\bv,\underline{\bd})^{\sA,\phi}}$ in $\cM_\theta(\bv,\underline{\bd})\times \cM_\theta(\bv,\underline{\bd})^{\sA,\phi}$, regarded as a locally closed substack in $\fM(\bv,\underline{\bd})\times \cM_\theta(\bv,\underline{\bd})^{\sA,\phi}$. It follows that $S$ contains $\Attr_\fC\left(\Delta_{\cM_\theta(\bv,\underline{\bd})^{\sA,\phi}}\right)$ as an open substack. 
By the closedness and irreducibility of $S$,  $$S=\overline{\Attr}_\fC\left(\Delta_{\cM_\theta(\bv,\underline{\bd})^{\sA,\phi}}\right)=Z.$$ 
This is exactly the correspondence that induces $\mathfrak{m}^\phi_{\fC}\circ \widetilde{\bPsi}_H^\phi$ (by the proof of Lemma \ref{lem: prestab corr}),
and we are done. 


\subsubsection{Proof of Theorem \ref{thm: AFSQV}, K-theory case} By Lemma \ref{lem:res of na stab to stable locus}, $\mathbf k^*\circ\widetilde{\bPsi}^{\mathsf s}_K=\id$, it suffices to show that 
$$\mathfrak{m}^\phi_\fC\circ\widetilde{\bPsi}^{\phi,\mathsf s'}_K(\gamma)\in \im(\widetilde{\bPsi}^{\mathsf s}_K)$$ for all $\gamma\in K^\sT(\cM_\theta(\bv,\underline{\bd})^{\sA,\phi},\sw)$. 

Recall the vector bundle $E_+$ \eqref{bundle added to AFSQV} on $\cM_\theta(\bv,\underline{\bd}^{\bv})$. Denote
\begin{align*}
    \underline{\bd}_{\In}:=(\bd_{\In},\bd_{\In}),\quad \underline{\bd}_{\In}^\bv:=(\bv+\bd_{\In},\bd_{\In}).
\end{align*}
$\mathrm{Tot}(E_+)$ is isomorphic to $\cM_\theta(\bv,\underline{\bd}_{\In}^{\bv})$ which is used to define nonabelian stable envelope $\widetilde{\bPsi}^{\mathsf s}_{K}$ for symmetric quiver variety $\cM_\theta(\bv,\bd_{\In})$. 
Restricting to the open locus \eqref{equ on mvcongr} (where $\bd=\bd_{\In}$) and \eqref{equ on mvcongr},
and taking quotient by $G$, the bundle projection becomes  
$$\pi:\fM(\bv,\bd_{\In})\to \fM(\bv,\underline{\bd}). $$  
By Proposition \ref{prop:vb and stab_attr}, we have commutative diagrams:
\begin{equation*}
\xymatrix{
K^\sT(\fM(\bv,\underline{\bd}),\sw\circ\pi) \ar[r]^-{\pi^*}_-{\cong} &  K^\sT(\fM(\bv,\bd_{\In}),\sw\circ\pi)  \\
K^\sT(\cM_\theta(\bv,\underline{\bd}),\sw\circ\pi) \ar[u]^{\widetilde{\bPsi}^{\mathsf s}_{K}} \ar[r]^-{\pi^*}_-{\cong} & K^\sT(\cM_\theta(\bv,\bd_{\In}),\sw\circ\pi) \ar[u]_{\widetilde{\bPsi}^{\mathsf s}_{K,\text{sym}}},
}
\,\,\,
\xymatrix{
K^\sT(\fM(\bv,\bd_{\In})^{\sA,\phi},\sw\circ\pi) \ar[r]^-{i^{\sA *}}_-{\cong} & K^\sT(\fM(\bv,\underline{\bd})^{\sA,\phi},\sw\circ\pi) \\
K^\sT(\cM_\theta(\bv,\bd_{\In})^{\sA,\phi},\sw\circ\pi) \ar[u]^{\widetilde{\bPsi}^{\phi,\mathsf s'}_{K,\text{sym}}} \ar[r]^-{i^{\sA *}}_-{\cong} & K^\sT(\cM_\theta(\bv,\underline{\bd})^{\sA,\phi},\sw\circ\pi) \ar[u]_{\widetilde{\bPsi}^{\phi,\mathsf s'}_{K}},
}
\end{equation*}
where $i^\sA:\fM(\bv,\underline{\bd})^{\sA,\phi} \hookrightarrow \fM(\bv,\bd_{\In})^{\sA,\phi}$ is the zero section of the fixed locus $\pi^\sA$. The right diagram follows from 
the torus fixed version of the left diagram and $i^{\sA*}=(\pi^{\sA*})^{-1}$.
Here ``sym'' in the subscript means the nonabelian stable envelopes for symmetric quiver varieties. Then it remains to show that for all $\gamma\in K^\sT(\cM_\theta(\bv,\bd_{\In})^{\sA,\phi},\sw)$, we have 
$$\pi^*\circ \mathfrak{m}^\phi_\fC\circ i^{\sA*}\circ\widetilde{\bPsi}^{\phi,\mathsf s'}_{K,\text{sym}}(\gamma)\in \im(\widetilde{\bPsi}^{\mathsf s}_{K,\text{sym}}).$$
By Proposition \ref{prop:Psi}, it is enough to prove the following degree condition
\begin{itemize}
    \item there is a strict inclusion of intervals
    \begin{align*}
        \deg_{\sigma_i} {\mathbf i}_{\sigma_i}^*\circ(\pi^*\circ \mathfrak{m}^\phi_\fC\circ i^{\sA*}\circ\widetilde{\bPsi}^{\phi,\mathsf s'}_{K,\text{sym}}(\gamma))
        \subsetneq \deg_{\sigma_i} \hat e^{G}_K\left(N_{S_i/R(\mathbf v,\bd_{\In})}\right)+\wt_{\sigma_i}\left(\mathsf s\right)
    \end{align*}
    for all cocharacters $\sigma_i:\bC^*\to G$ that appear in the KN stratification, and all $\gamma\in K^\sT(\cM_\theta(\bv,\bd_{\In})^{\sA,\phi},\sw)$, 
    where $\hat e^G_K$ is defined in \eqref{hat e}. 
\end{itemize}
Similarly as the proof of Theorem \ref{thm: hall k_sym quot}, 
it remains to show the following analogy of \eqref{deg bound 2}: 
\begin{equation}\label{deg bound asym 2}
\begin{split}
    &\lim_{t\to 0,\infty} \frac{\mathbf j^{\phi*}_{w^{-1}(\sigma)}\circ\widetilde{\bPsi}_K^{\phi,\mathsf s'}(\gamma)}{\sqrt{\hat e^{(G^{\phi(\sA)})^{w^{-1}(\sigma)}}_K\left(R(\bv,\bd_{\In})^{\sA,\phi}-\Lie(G^{\phi(\sA)})\right)}\cdot\hat e^{(G^{\phi(\sA)})^{w^{-1}(\sigma)}}_K\left(E_+^{\sA,\phi\text{-repl}}\right)\cdot w^{-1}(\sigma)^*(\mathsf s')}\\
    & \text{ exist for all $w\in W^\sigma\backslash W/W^\phi$.}
\end{split}
\end{equation}
Note that by the attracting condition, we have
\begin{align*}
    \lim_{t\to 0,\infty} \frac{1}{\hat e^{(G^{\phi(\sA)})^{w^{-1}(\sigma)}}_K\left(E_+^{\sA,\phi\text{-repl}}\right)}\; \text{ exist for all $w\in W^\sigma\backslash W/W^\phi$},
\end{align*}
and by Theorem \ref{thm: deg bound na stab k}, we have
\begin{align*}
    \deg_{w^{-1}(\sigma)} \mathbf j^{\phi*}_{w^{-1}(\sigma)}\widetilde{\bPsi}^{\phi,\mathsf s'}_K(\gamma)\subsetneq  \frac{1}{2}\deg_{w^{-1}(\sigma)} e^{G^{\phi(\sA)}}_K(R(\bv,\bd_{\In})^{\sA,\phi}-\Lie(G^{\phi(\sA)}))+\wt_{w^{-1}(\sigma)}\mathsf s',
\end{align*}
then \eqref{deg bound asym 2} follows from the above two degree bounds. 

This finishes the proof of Theorem \ref{thm: AFSQV} in the $K$-theory case.

\subsection{Hall v.s. stable envelopes}\label{sec: pstab vs stab}
In this section we compare Hall envelopes with stable envelopes. 

For cohomology, the following result is a consequence of Lemma \ref{lem: prestab corr} and Theorem \ref{thm:attr plus repl}.
\begin{Proposition}\label{prop: stab=prestab coh}
In the setting of Theorem \ref{thm: AFSQV}, and assume that the bundle 
\begin{equation*}
\bigoplus_{\begin{subarray}{c}i\in Q_0 \\ \mathrm{s.t.}\, \bd_{\In,i}>\bd_{\Out,i}   \end{subarray}}
\Hom(V_i,\bC^{\bd_{\In,i}-\bd_{\Out,i}})
\end{equation*}
in \eqref{bundle added to AFSQV} can be written as a direct sum of attracting and repelling subbundles in chamber $\fC$ (as in Theorem \ref{thm:attr plus repl}), 
then cohomological stable envelope exists and
\begin{align*}
    \Stab_\fC=\HallEnv_\fC.
\end{align*}
\end{Proposition}

The $K$-theory counterpart of Proposition \ref{prop: stab=prestab coh} is more complicated, as we do not have a $K$-theory analogy of Theorem \ref{thm:attr plus repl}.

\begin{Proposition}\label{prop: stab=prestab k}
In the setting of Theorem \ref{thm: AFSQV}, consider one of the following situations
\begin{enumerate}
\item the bundle $\bigoplus_{\begin{subarray}{c}i\in Q_0 \\ \mathrm{s.t.}\, \bd_{\In,i}>\bd_{\Out,i}   \end{subarray}}
\Hom(V_i,\bC^{\bd_{\In,i}-\bd_{\Out,i}})$
in \eqref{bundle added to AFSQV} is attracting in chamber $\fC$,
\item $\sA$ is a minuscule framing torus, i.e. $\sA\cong \bC^*$ which acts on framing vector space with weight $0$ and $1$.
\end{enumerate}
then for a generic slope $\mathsf s$, $K$-theoretic stable envelope exists and
\begin{align*}
    \Stab^{\mathsf s}_\fC=\HallEnv^{\mathsf s}_\fC.
\end{align*}
\end{Proposition}

\begin{proof}
It is enough to show that $\HallEnv^{\mathsf s}_\fC$ satisfies three axioms of stable envelopes. The support axiom (Definition \ref{def of stab k}(i)) and the normalization axiom (Definition \ref{def of stab k}(ii)) are checked in the same way as Theorem \ref{thm: hall k_sym quot} and we shall not repeat here. As in \eqref{bundle added to AFSQV}, we denote
\begin{align*}
    F=\bigoplus_{\begin{subarray}{c}i\in Q_0 \\ \mathrm{s.t.}\, \bd_{\In,i}>\bd_{\Out,i}   \end{subarray}}\Hom(V_i,\bC^{\bd_{\In,i}-\bd_{\Out,i}}).
\end{align*}
Similar as the proof of Theorem \ref{thm: hall k_sym quot}, the degree axiom is examined by showing the following analogy of \eqref{deg bound 4}
\begin{equation}\label{deg bound asym 4}
\begin{split}
&\lim \frac{\phi^*(\mathsf s')\cdot \mathbf j^{\phi*}_{\sA,w^{-1}({\phi'})}\circ\widetilde{\bPsi}_K^{\phi,\mathsf s'}(\gamma)}{w^{-1}({\phi'})^*(\mathsf s')\cdot \sqrt{\hat e^{\sT\times (G^{\phi(\sA)})^{w^{-1}(\phi'(\sA))}}_K\left(R(\bv,\bd_{\In})^{\sA,\phi}-\Lie(G^{\phi(\sA)})\right)}}\times\frac{\hat e^{\sT\times (G^{\phi(\sA)})^{w^{-1}(\phi'(\sA))}}_K\left(F^{\sA,w^{-1}(\phi')\text{-repl}}\right)}{\hat e^{\sT\times (G^{\phi(\sA)})^{w^{-1}(\phi'(\sA))}}_K\left(F^{\sA,\phi\text{-repl}}\right)}\\
&\text{exist for all $w\in W^{\phi'}\backslash W/W^\phi$ and all directions of $\mathsf a\to \infty$,}
\end{split}
\end{equation}
where $\mathsf s'$ is as \eqref{s' for prestab}, $\phi'\colon \sA\to G$ is a homomorphism such that $\cM_\theta(\bv,\underline{\bd})^{\sA,\phi'}\bigcap \Attr_\fC^f(\cM_\theta(\bv,\underline{\bd})^{\sA,\phi})$ is nonempty.

In situation (1), the repelling part of $F$ is trivial by assumption, and \eqref{deg bound asym 4} is the same as \eqref{deg bound 4} which has been proven in Theorem \ref{thm: hall k_sym quot}. 

In situation (2), let us write decomposition of framing vector spaces into $\sA$-weight spaces:
\begin{align*}
    \underline{\bd}=\underline{\bd}^{(0)}\oplus \underline{\bd}^{(1)},
\end{align*}
where the superscripts denote the $\sA$-weight $\in \{0,1\}$. Since $\sA$-action is pseudo-symmetric by assumption, we have $\bd^{(0)}_{\In,i}\geqslant\bd^{(0)}_{\Out,i}$ and $\bd^{(1)}_{\In,i}\geqslant\bd^{(1)}_{\Out,i}$ for all $i\in Q_0$. We can write $\phi$ and $\phi'$ fixed components as
\begin{align*}
     \cM_\theta(\bv,\underline{\bd})^{\sA,\phi}=\cM_{\theta}(\bv_{\phi}^{(0)},\underline{\bd}^{(0)})\times \cM_{\theta}(\bv_{\phi}^{(1)},\underline{\bd}^{(1)}),\\
     \cM_\theta(\bv,\underline{\bd})^{\sA,\phi'}=\cM_{\theta}(\bv_{\phi'}^{(0)},\underline{\bd}^{(0)})\times \cM_{\theta}(\bv_{\phi'}^{(1)},\underline{\bd}^{(1)}),
\end{align*}
for certain decompositions $\bv=\bv^{(0)}_\phi+\bv^{(1)}_\phi=\bv^{(0)}_{\phi'}+\bv^{(1)}_{\phi'}$. Then we have
\begin{align*}
F^{\sA,\phi\text{-repl}}=\bigoplus_{\begin{subarray}{c}i\in Q_0 \\ \mathrm{s.t.}\, \bd^{(0)}_{\In,i}>\bd^{(0)}_{\Out,i}   \end{subarray}}
\Hom(V^{(1)}_{\phi,i},\bC^{\bd^{(0)}_{\In,i}-\bd^{(0)}_{\Out,i}}),\quad F^{\sA,\phi'\text{-repl}}=\bigoplus_{\begin{subarray}{c}i\in Q_0 \\ \mathrm{s.t.}\, \bd^{(1)}_{\In,i}>\bd^{(1)}_{\Out,i}   \end{subarray}}
\Hom(V^{(1)}_{\phi',i},\bC^{\bd^{(1)}_{\In,i}-\bd^{(1)}_{\Out,i}}).
\end{align*}
It follows that
\begin{align*}
\hat e^{\sT\times (G^{\phi(\sA)})^{w^{-1}(\phi'(\sA))}}_K\left(F^{\sA,\phi\text{-repl}}\right)=O(\mathsf a^{\pm\frac{n}{2}}),\quad \hat e^{\sT\times (G^{\phi(\sA)})^{w^{-1}(\phi'(\sA))}}_K\left(F^{\sA,w^{-1}(\phi')\text{-repl}}\right)=O(\mathsf a^{\pm\frac{n'}{2}})
\end{align*}
as $\mathsf a^{\pm}\to \infty$, where 
\begin{align*}
    n=\sum_{i\in Q_0}\bv_{\phi}^{(1)}\cdot(\bd^{(0)}_{\In,i}-\bd^{(0)}_{\Out,i}),\quad n'=\sum_{i\in Q_0}\bv_{\phi'}^{(1)}\cdot(\bd^{(0)}_{\In,i}-\bd^{(0)}_{\Out,i}).
\end{align*}
Note that $\cM_\theta(\bv,\underline{\bd})^{\sA,\phi'}\bigcap \Attr_\fC^f(\cM_\theta(\bv,\underline{\bd})^{\sA,\phi})\neq \emptyset$ implies $\cM_\theta(\bv,\underline{\bd})^{\sA,\phi}\leq\cM_\theta(\bv,\underline{\bd})^{\sA,\phi'}$ in the ample partial order (Remark \ref{partial order by line bundle}), and the latter is equivalent to
\begin{align*}
    \bv_{\phi,i}^{(1)}\geqslant \bv_{\phi',i}^{(1)}, \,\,\, \forall\, i\in Q_0,
\end{align*}
In particular, $n\geqslant n'$, and it follows that
\begin{align*}
    \lim_{\mathsf a^{\pm}\to \infty}\frac{\hat e^{\sT\times (G^{\phi(\sA)})^{w^{-1}(\phi'(\sA))}}_K\left(F^{\sA,w^{-1}(\phi')\text{-repl}}\right)}{\hat e^{\sT\times (G^{\phi(\sA)})^{w^{-1}(\phi'(\sA))}}_K\left(F^{\sA,\phi\text{-repl}}\right)}\;\text{exist}.
\end{align*}
Then \eqref{deg bound asym 4} follows from \eqref{deg bound 4}.
\end{proof}

There are examples where Hall envelopes are not equal to stable envelopes and do not satisfy the triangle lemma.

\begin{Example}\label{ex on fail of tri lem}
Consider the quiver $Q$ consisting of a single node with no arrows, and take $\underline{\bd}=(1,2)$. Let 
$$\cM(n \underline{\bd})=\bigsqcup_{k}\cM_{\theta}(k,n \underline{\bd})$$ 
be the disjoint union of quiver varieties, where we fix $\theta=-1$. Note that $\cM_{\theta}(k,n\underline{\bd})$ is a vector bundle on $\Gr(k,n)$. Let $\bC^*_{\hbar}$ act on 
$\cM(n\underline{\bd})$ by scaling the out-going framing with weight $\hbar^{-1}$, and $\sA=\bC^*_{a_1}\times\bC^*_{a_2}\times\bC^*_{a_3}$ act on 
$$X=\cM(3\underline{\bd}), $$ by assigning weights to the framings: $a_1\underline{\bd}+a_2\underline{\bd}+a_3\underline{\bd}$. Set $\sT=\sA\times \bC^*_\hbar$. Then the equivariant $K$-theory of the $\sA$-fixed loci is generated by the tensor product of structure sheaves:
\begin{align*}
    K^{\sT}(X^{\sA})=\bigoplus_{(k_1,k_2,k_3)\in \{0,1\}^3}K_{\sT}(\pt)\cdot[\mathcal O_{\cM_{\theta}(k_1,\underline{\bd})}\boxtimes\mathcal O_{\cM_{\theta}(k_2,\underline{\bd})}\boxtimes\mathcal O_{\cM_{\theta}(k_3,\underline{\bd})}]\:.
\end{align*}
Fix a generic slope $\mathsf s\in \bR\cong \Pic(\cM_{\theta}(k,3\underline{\bd}))\otimes_\bZ \bR$. 

\begin{Proposition}\label{prop on 3 ex}
For $(X,\sT,\sA)$ in the above example, we have:
\begin{enumerate}
    \item There exists a chamber $\fC$ of $\sA$ with a face $\fC'$ such that the triangle lemma \eqref{cd: tri lem} fails for the $K$-theoretic Hall envelope. 
    \item There exists a chamber $\fC$ of $\sA$ such that $K$-theoretic stable envelope $\Stab_{\fC}^{\mathsf s}$ does not exist.
    \item There exists a one-dimensional torus $\sigma:\bC^*\to \sA$ such that $\Stab_{\sigma}^{\mathsf s}\neq \HallEnv_{\sigma}^{\mathsf s}$, where $\HallEnv_{\sigma}^{\mathsf s}$ is the Hall envelope defined using the image of $\sigma$ in the positive chamber.
\end{enumerate}
\end{Proposition}

\begin{proof}
Suppose that the triangle lemma \eqref{cd: tri lem} always holds for any chamber $\fC$ and face $\fC'$, then by \cite{COZZ2}, the following Yang-Baxter equation 
\begin{align}\label{YBE to check}
R^{\mathsf s+\pmb{\mu}^{(3)}_+}_{12}(a_1/a_2)R^{\mathsf s+\pmb{\mu}^{(2)}_-}_{13}(a_1/a_3)R^{\mathsf s+\pmb{\mu}^{(1)}_+}_{23}(a_2/a_3)=R^{\mathsf s+\pmb{\mu}^{(1)}_-}_{23}(a_2/a_3)R^{\mathsf s+\pmb{\mu}^{(2)}_+}_{13}(a_1/a_3)R^{\mathsf s+\pmb{\mu}^{(3)}_-}_{12}(a_1/a_2)
\end{align}
should hold, where $\pmb\mu^{(i)}_+=\frac{1-k_i}{2}$, $\pmb\mu^{(i)}_-=\frac{k_i-2}{2}$, with $R$-matrix:
\begin{align*}
R^{\mathsf s}_{ij}(a_i/a_j)=(\HallEnv_{a_i<a_j}^{\mathsf s})^{-1}\HallEnv_{a_i>a_j}^{\mathsf s}.
\end{align*}
Denote $[ab]:=[\mathcal O_{\cM_{\theta}(a,\underline{\bd})}\boxtimes\mathcal O_{\cM_{\theta}(b,\underline{\bd})}]$. In the basis $[00],[01], [10],[11]$, the Hall envelopes can be explicitly computed using \eqref{explicit formula PStab w=0_k} and we get
\begin{equation}\label{hallenv dominant ex}
\begin{split}
\HallEnv_{a_1>a_2}^{\mathsf s}=
\begin{pmatrix}
1 & & & \\
& 1-a_1/a_2 & (a_2/a_1)^{\lfloor\mathsf s-1/2\rfloor}(1-\hbar)^2 &\\
 & 0 & (1-\hbar a_1/a_2)^2 &\\
& & & (1-\hbar a_1/a_2)^2
\end{pmatrix},\\
\HallEnv_{a_1<a_2}^{\mathsf s}=
\begin{pmatrix}
1 & & & \\
&(1-\hbar a_2/a_1)^2 & 0 &\\
& (a_1/a_2)^{\lfloor\mathsf s-1/2\rfloor}(1-\hbar)^2 & 1-a_2/a_1 &\\
& & & (1-\hbar a_2/a_1)^2
\end{pmatrix}.
\end{split}
\end{equation}
It is a straightforward to check that $R^{\mathsf s}_{ij}$ does not satisfy the Yang-Baxter equation \eqref{YBE to check}, and this proves (1).

We notice that the Hall envelopes \eqref{hallenv dominant ex} satisfy the axioms in Definition \ref{def of stab k} therefore we have
\begin{align*}
    \HallEnv_{a_1\gtrless a_2}^{\mathsf s}=\Stab_{a_1\gtrless a_2}^{\mathsf s}.
\end{align*}
In particular, $R$-matrix is given by $R^{\mathsf s}_{ij}(a_i/a_j)=(\Stab_{a_i<a_j}^{\mathsf s})^{-1}\Stab_{a_i>a_j}^{\mathsf s}$. The failure of Yang-Baxter equation implies that there exists a chamber $\fC$ of $\sA$ such that $\Stab_{\fC}^{\mathsf s}$ does not exist (otherwise the triangle lemma automatically holds by Lemma \ref{triangle lemma for K} and YBE should also hold). This proves (2).

Finally, suppose that $\Stab_{\sigma}^{\mathsf s}=\HallEnv_{\sigma}^{\mathsf s}$ for every one-dimensional torus $\sigma:\bC^*\to \sA$. Suppose that $\fC=\{a_i>a_j>a_k\}$ and $\fC'=\{a_i=a_j>a_k\}$ are a pair of chambers and faces such that triangle lemma fails for the $K$-theoretic Hall envelope as in (1). According to the proof of Lemma \ref{triangle lemma for K}, there exists an one-dimensional torus $\xi:\bC^*\to \sA$ in the chamber $\fC$, such that
\begin{align*}
    \Stab_{\xi}^{\mathsf s}=\Stab_{a_i=a_j>a_k}^{\mathsf s}\circ \Stab_{a_i>a_j}^{\mathsf s'}.
\end{align*}
We have $\Stab_{a_i=a_j> a_k}^{\mathsf s}=\HallEnv_{a_i=a_j> a_k}^{\mathsf s}$ since they are given by one-dimensional torus. 
We also have $$\Stab_{a_i>a_j}^{\mathsf s'}=\HallEnv_{a_i>a_j}^{\mathsf s'}$$ by the above explicit computation. Note that the construction of Hall envelope can be replaced by an arbitrary one-dimensional torus in the chamber $\fC$, and result remains the same, in particular $\HallEnv_{\fC}^{\mathsf s}=\HallEnv_{\xi}^{\mathsf s}$. Thus we have 
\begin{align*}
    \HallEnv_{\fC}^{\mathsf s}=\HallEnv_{\xi}^{\mathsf s}= \Stab_{\xi}^{\mathsf s}=\HallEnv_{a_i=a_j>a_k}^{\mathsf s}\circ \HallEnv_{a_i>a_j}^{\mathsf s'}.
\end{align*}
This contracts with the choice of $\fC$ and $\fC'$. This prove (3).
\end{proof}

\end{Example}






\appendix

\section{Pullback isomorphisms along attraction maps}
In this section, we take $(Y,\sT,\sA,\sw)$ as in Setting \ref{setting of stab}. Assume moreover that $F=Y^\sA$ is connected, and $X=\Attr_\fC(F)$. Consider the attraction map 
\begin{align*}
f\colon X\to F, \quad x\mapsto \lim_{t\to 0}\sigma(t)\cdot x \quad \mathrm{for}\,\, \mathrm{a}\,\, \sigma\in \fC. 
\end{align*} 
By \cite{BB}, $f$ is an affine fibration and does not depend on the choice of $\sigma$. Since $\sw$ is $\sA$-invariant, we have 
\begin{align*}
\sw=f^*(\sw^\sA), 
\end{align*} 
so there are smooth pullbacks 
\begin{align*}
    f^*\colon H^\sT(F,\sw^\sA)\to H^\sT(X,\sw), \quad f^*\colon K^\sT(F,\sw^\sA)\to K^\sT(X,\sw). 
\end{align*}
By \eqref{pb is iso on coho}, the first map is an isomorphism. As we do not find a reference on the comparison of the second map, we show below the following.
\begin{Proposition}\label{pullback iso along attr}
In the above situation, $f^*\colon K^\sT(F,\sw^\sA)\to K^\sT(X,\sw)$ is an isomorphism.
\end{Proposition}

\begin{proof}
By \cite[Amplification 3.18]{Hal}, we have a SOD for the derived category of coherent sheaves:
\begin{align*}
    \D^b([X/\sT])=\left\langle \cdots,f^*\D^b([F/\sT])_{-1}, f^*\D^b([F/\sT])_0, f^*\D^b([F/\sT])_1,\cdots\right\rangle, 
\end{align*}
where
$$\left\langle \cdots,\D^b([F/\sT])_{-1}, \D^b([F/\sT])_0, \D^b([F/\sT])_1,\cdots\right\rangle$$ is an SOD of $\D^b([F/\sT])$, and $f^*$ maps $\D^b([F/\sT])_{v}$ fully faithfully onto its image in $\D^b([X/\sT])$ for all $v\in\bZ$. Here $\D^b([F/\sT])_{v}$ is the full subcategory of $\D^b([F/\sT])$ whose objects are complexes with $\sigma$-weight equals to $v$, where $\sigma$ is a fixed cocharacter in $\fC$.

This induces an SOD for the matrix factorization category (e.g.~\cite[Prop.~2.1]{P}):
\begin{align*}
    \mathrm{MF}([X/\sT],\sw)=\left\langle \cdots,f^*\mathrm{MF}\left([F/\sT],\sw^\sA\right)_{-1}, f^*\mathrm{MF}\left([F/\sT],\sw^\sA\right)_0, 
    f^*\mathrm{MF}\left([F/\sT],\sw^\sA\right)_1,\cdots\right\rangle,
\end{align*}
where 
$$\left\langle \cdots,\mathrm{MF}\left([F/\sT],\sw^\sA\right)_{-1}, \mathrm{MF}\left([F/\sT],\sw^\sA\right)_0, \mathrm{MF}\left([F/\sT],\sw^\sA\right)_1,\cdots\right\rangle$$ is an SOD of $\mathrm{MF}\left([F/\sT],\sw^\sA\right)$, and $f^*$ maps $\mathrm{MF}\left([F/\sT],\sw^\sA\right)_{v}$ fully faithfully onto its image in $\mathrm{MF}\left([X/\sT],\sw\right)$ for all $v\in\bZ$. By passing to the Grothendieck group, we get isomorphism $f^*\colon K^\sT(F,\sw^\sA)\cong K^\sT(X,\sw)$.
\end{proof}

\section{Excess intersection formula of critical cohomology and \texorpdfstring{$K$}{K}-theory}\label{app on excess formula}

Consider the following Cartesian diagram of smooth and connected varieties:
\begin{equation}\label{cd:excess diagram}
\xymatrix{
X' \ar[r]^{f'} \ar[d]_{h} \ar@{}[dr]|{\Box} & Y' \ar[d]^{g} \\
X \ar[r]^f & Y,
}
\end{equation}
where $f$ is a closed immersion. 
The natural map $N_{f'}\to h^*N_f$ between normal bundles is an embedding of vector bundles. We define the \textit{excess bundle}: 
$$E:=h^*N_f/N_{f'}. $$
Suppose that $\sw\colon Y\to \bC$ is a regular function on $Y$. We have proper pushforward
$$f_*\colon H(X,\varphi_{\sw\circ f}\omega_X)\to H(Y,\varphi_{\sw}\omega_Y), \quad f'_*\colon H(X',\varphi_{\sw\circ f\circ h}\omega_{X'})\to H(Y',\varphi_{\sw\circ g}\omega_{Y'}), $$ 
and~l.c.i.~pullback
$$g^*\colon H(Y,\varphi_{\sw}\omega_Y)\to H(Y',\varphi_{\sw\circ g}\omega_{Y'}[-2d_g]), \quad h^*\colon H(X,\varphi_{\sw\circ f}\omega_X)\to H(X',\varphi_{\sw\circ f\circ h}\omega_{X'}[-2d_h]). $$ 
Here $d_g=\dim Y'-\dim Y$ and $d_h=\dim X'-\dim X$.

\begin{Proposition}[Excess intersection formula for critical cohomology]\label{excess_coh}
Notations as above, we have 
\begin{align*}
    g^*\circ f_*=f'_*\circ e(E)\cdot h^*,
\end{align*}
where $e(E)$ is the Euler class of the excess bundle $E$.
\end{Proposition}

\begin{proof}
Note that $g=\pr_Y\circ \Gamma_g$, where $\Gamma_g\colon Y'\to Y'\times Y$ is the graph of $g$ and $\pr_Y\colon Y'\times Y\to Y$ is the projection to the second component. 
They fit into the following Cartesian diagram:
\begin{equation}\label{diag on graph trick}
\xymatrix{
X' \ar[r]^{f'} \ar[d]_{f'\times h} \ar@{}[dr]|{\Box} & Y' \ar[d]^{\Gamma_g} \\
Y'\times X \ar[r]^{1_{Y'}\times f} \ar[d]_{\pr_X} \ar@{}[dr]|{\Box} & Y'\times Y \ar[d]^{\pr_Y}\\
X \ar[r]^f & Y.
}
\end{equation}
In particular, we have compatibility of proper pushforward and flat pullback: 
\begin{align*}
    \pr_Y^*\circ f_*=(1_{Y'}\times f)_*\circ \pr_X^*.
\end{align*}
Thom-Sebastiani isomorphism gives
$$H(Y'\times Y,\varphi_{\sw\circ\pr_Y}\omega_{Y'\times Y})\cong H(Y',\omega_{Y'})\otimes H(Y,\varphi_{\sw}\omega_Y). $$ 
Note that $\pr_Y^*\colon H(Y,\varphi_{\sw}\omega_Y)\to H(Y'\times Y,\varphi_{\sw\circ\pr_Y}\omega_{Y'\times Y})$ equals to $\gamma\mapsto [Y']\otimes \gamma$, where $[Y']\in H^{-2\dim_{Y'}}(Y',\omega_{Y'})$ is the fundamental cycle, and similarly $\pr_X^*(-)=[Y']\otimes (-)$.

Suppose that the proposition holds for the upper square of the above diagram, i.e. 
$$\Gamma_g^*\circ (1_{Y'}\times f)_*=f'_*\circ e(E)\cdot (f'\times h)^*,$$ 
then we have
\begin{align*}
    g^*\circ f_*=\Gamma_g^*\circ \pr_Y^*\circ f_*=\Gamma_g^*\circ (1_{Y'}\times f)_*\circ \pr_X^*=f'_*\circ e(E)\cdot (f'\times h)^*\circ \pr_X^*=f'_*\circ e(E)\cdot h^*,
\end{align*}
where we use $N_{1_{Y'}\times f}\cong \pr_X^*N_f$ in the third equality.
So it is enough to show the claim for the upper square. As $\Gamma_g$ is a closed immersion, so we are left to show the proposition holds for \eqref{cd:excess diagram}
when $g$ is a closed immersion. 

Applying the composition of adjunctions $f_*f^!\to \mathrm{id}$, $\mathrm{id}\to g_*g^*$ to $\omega_{Y}$ (here these are operations on complexes of sheaves), we obtain  
a morphism in the derived category $D^b_c(Y)$ of constructible complexes: 
\begin{equation}\label{equ on fgom}f_*\omega_X\to g_*\omega_{Y'}[-2d_g], \end{equation}
where we use $g^*\omega_Y\cong \omega_{Y'}[-2d_g]$. Since $f_*\omega_X$ is supported on $X$, the map \eqref{equ on fgom} factors as
\begin{align*}
    f_*\omega_X\to f_*f^!g_*\omega_{Y'}[-2d_g]\to g_*\omega_{Y'}[-2d_g],
\end{align*}
where $f_*\omega_X\to f_*f^!g_*\omega_{Y'}[-2d_g]$ is obtained by applying $f_*$ to map 
$$c:\omega_X\to  h_*\omega_{X'}[-2d_g]\cong h_*f'^!\omega_{Y'}[-2d_g]\cong f^!g_*\omega_{Y'}[-2d_g]. $$ 
Using the adjoint pair $(h^*, h_*)$, the map $c$ factors as 
$$\omega_X\to h_*h^*\omega_X\xrightarrow{h_*(\tilde c)}  h_*\omega_{X'}[-2d_g], $$ 
where $\tilde c:h^*\omega_X\to \omega_{X'}[-2d_g]$ is given by applying the natural transformation 
\begin{equation}\label{equ on hff'g}h^*f^!\to f'^!g^* \end{equation} to $\omega_Y$. Here the transformation is obtained by the adjunction of the map $f^!\to f^!g_*g^*\cong h_*f'^!g^*$. 

Using smoothness, we have $\omega_{X}=\mathbb{Q}_X[2\dim X]$, $\omega_{X'}=\mathbb{Q}_{X'}[2\dim X']$, so 
the map $\tilde c$ can be identified with multiplication by an element $\alpha\in H^{2d_h-2d_g}(X')$. 

In summary, we obtain the following commutative diagram in $D^b_c(Y)$:
\begin{equation}
\xymatrixcolsep{4pc}
\xymatrix{
& \omega_Y \ar[dr] & \\
f_*\omega_X \ar[r] \ar[ur] \ar[d] & f_*f^!g_*\omega_{Y'}[-2d_g] \ar[r] & g_*\omega_{Y'}[-2d_g]\\
f_*h_*h^*\omega_X \ar[r]^{f_*h_*\alpha\cdot \quad\,\,} 
 & f_*h_*\omega_{X'}[-2d_g]\ar[r]^{\cong} \ar[u]^{\cong} & g_*f'_*\omega_{X'}[-2d_g]. \ar[u]
}
\end{equation}
We note that $g^*\circ f_*$ (here $g^*$ denotes the Gysin pullback) is induced by applying $\mathbf R\Gamma(Y,-)$ to the path 
$$\varphi_\sw f_*\omega_X\to \varphi_\sw\omega_Y\to \varphi_\sw g_*\omega_{Y'}[-2d_g], $$ 
followed by natural transformation 
\begin{equation}\label{equ on vagy'}\varphi_\sw g_*\omega_{Y'}[-2d_g]\xrightarrow{\cong} g_* \varphi_{\sw\circ g} \omega_{Y'}[-2d_g], \end{equation}
and $f'_*\circ \alpha\cdot h^*$ (here $h^*$ denotes the Gysin pullback) is induced by applying $\mathbf R\Gamma(Y,-)$ to the path 
$$\varphi_\sw f_*\omega_X\to \varphi_\sw f_*h_*h^*\omega_X\to \varphi_\sw f_*h_*\omega_{X'}[-2d_g]\xrightarrow{\cong} \varphi_\sw g_*f'_*\omega_{X'}[-2d_g] \to \varphi_\sw g_*\omega_{Y'}[-2d_g],$$ followed by the same transformation \eqref{equ on vagy'}.
Therefore we get the equation 
\begin{align*}
    g^*\circ f_*=f'_*\circ \alpha\cdot h^*.
\end{align*}
It remains to show that $\alpha=e(E)$.  
Recall that the `restriction with supports' (ref.~\cite[\S 8.3.21]{CG}): 
$$f^!\colon H^\BM_{*}(Y')\to H^\BM_{*+2d_f}(X'),$$ 
which is defined by applying $\mathbf R\Gamma(Y',-)$ to the map 
$$\omega_{Y'}\cong g^!\omega_Y\to g^!f_*f^*\omega_Y\cong f'_*h^!\omega_X[-2d_f]\cong f'_*\omega_{X'}[-2d_f]. $$ 
The map $\omega_{Y'}\to  f'_*\omega_{X'}[-2d_f]$ factors into 
$$\omega_{Y'}\to f'_*f'^*\omega_{Y'}\xrightarrow{f'_*(c')}  f'_*\omega_{X'}[-2d_f], $$ 
where $c'$ is the natural transformation $f'^*g^!\omega_Y\to h^!f^*\omega_Y$, which is given by applying 
\eqref{equ on hff'g} to $\mathbb{Q}_Y$ and then applying Verdier duality functor $\bD_{X'}$. 
Since $c'$ is obtained by $\widetilde{c}$ by applying Verdier duality functor $\bD_{X'}$, $c'$ can also be identified with multiplication by $\alpha$. 
Then we get an equation on maps
between Borel-Moore homologies:
\begin{align}\label{equ on f!f!}
    f^!=\alpha\cdot f'^*\colon H^\BM_{*}(Y')\to H^\BM_{*+2d_f}(X'),
\end{align}
where $f'^*\colon H^\BM_{*}(Y')\to H^\BM_{*+2d_{f'}}(X')$ is the Gysin pullback. 

Next we determine $\alpha$ by using the excess intersection formula in Chow homology and the cycle map. 
In Chow homology, we have \cite[Thm.~6.3]{Ful}:
\begin{align}\label{ex for in chow}
    f^!=e(E)\cap f'^*\colon A_*(Y')\to A_{*+d_f}(X'),
\end{align}
where $A_*(-)$ is the Chow homology, $f^!$ is the refined Gysin pullback, and $f'^*$ is the Gysin pullback. It is known that the `restriction with supports'  is compatible with refined Gysin map \cite[\S 8.3.21, \S 2.6.21]{CG}, i.e. 
\begin{align*}
    f^!_\BM\circ \mathrm{cyc}=\mathrm{cyc}\circ f^!_\CH\colon A_*(Y')\to H^\BM_{2*+2d_f}(X'),
\end{align*}
where $\cyc\colon A_*\to H^{\BM}_{2*}$ is the cycle map, and $f^!_\BM, f^!_\CH$ denote Gysin pullbacks in BM and Chow homologies. 

Since Gysin pullback is a special case of `restriction with supports', we also have 
\begin{equation}\label{equ on f!ch and bm}f'^!_\BM\circ \mathrm{cyc}=\mathrm{cyc}\circ f'^!_\CH. \end{equation}
Let $[Y']_\CH$ ($[Y']_\BM$) be the fundamental cycle in BM homology (Chow homology) respectively, then  
\begin{equation}\label{equ on cpt cho and bm}\cyc([Y']_\CH)=[Y']_\BM. \end{equation}
Combining \eqref{equ on f!f!}, \eqref{ex for in chow}, \eqref{equ on f!ch and bm}, \eqref{equ on cpt cho and bm}, we obtain
\begin{align*}
    \alpha\cdot [X']_\BM=\alpha\cdot f'^!_\BM ([Y']_\BM)=f^!_\BM\circ \cyc([Y']_\CH)=\cyc\circ f^!_\CH([Y']_\CH)=\cyc(e(E)\cap [X']_\CH)=e(E)\cdot [X']_\BM.
\end{align*}
By Poincar\'e duality, the map $H^{*}(X')\ni \beta\mapsto \beta\cdot [X']_\BM\in H^\BM_{2\dim X'-*}(X')$ is an isomorphism; thus $\alpha=e(E)$.
\end{proof}

\begin{Remark}[Equivariant version of Proposition \ref{excess_coh}]\label{equiv excess_coh}
Suppose that the maps in the diagram \eqref{cd:excess diagram} are equivariant morphisms between $G$-varieties, where $G$ is an algebraic group, and assume that $\sw$ is $G$-invariant. Then we claim that the equivariant version of Proposition \ref{excess_coh} holds, i.e.
\begin{align}\label{eq:equiv excess_coh}
    g^*\circ f_*=f'_*\circ e^G(E)\cdot h^*,
\end{align}
where $e^G(E)$ is the $G$-equivariant Euler class of excess bundle $E$. To prove this claim, we consider the $n$-acyclic free $G$-variety $M_n$ as in the \cite[\S 3.1]{BL} (see also \cite[\S 2.4]{Dav}), and for every $G$-variety $V$ we define 
$$V_n:=(V\times M_n)/G.$$ Then we get the following Cartesian diagram of smooth connected varieties by applying $(-\times M_n)/G$ to \eqref{cd:excess diagram}
\begin{equation*}
\xymatrix{
X'_n \ar[r]^{f_n'} \ar[d]_{h_n} \ar@{}[dr]|{\Box} & Y_n' \ar[d]^{g_n} \\
X_n \ar[r]^{f_n} & Y_n
}
\end{equation*}
By Proposition \ref{excess_coh} we have 
\begin{align}\label{eq:excess_coh n-th step}
g_n^*\circ f_{n*}=f'_{n*}\circ e(E_n)\cdot h_n^*,
\end{align}
where $$E_n=h_n^*N_{f_n}/N_{f'_n}\cong (E\times M_n)/G$$ is the excess bundle. It is known that there exists $G$-equivariant closed embedding $M_n\hookrightarrow M_{n+1}$. As in \cite[\S 2.4]{Dav}, the equivariant critical cohomology is isomorphic to the inductive limit,~i.e.
\begin{align*}
H_G(Y,\varphi_\sw\,\omega_Y)\cong \lim_{\substack{\longrightarrow\\ n\to \infty}}H(Y_n,\varphi_{\sw_n}\omega_{Y_n}),
\end{align*}
where $\sw_n$ is the descent of $\sw\circ\pr_Y\colon Y\times M_n\to \bC$ from $Y\times M_n$ to $Y_n$ since $\sw$ is $G$-invariant. And the functors between equivariant critical 
cohomologies are inductive limit of their finite counterparts, e.g.
\begin{align*}
    f_*=\lim_{\substack{\longrightarrow\\ n\to \infty}}f_{n*}, \quad g^*=\lim_{\substack{\longrightarrow\\ n\to \infty}}g_{n}^*,\quad e^G(E)\cdot =\lim_{\substack{\longrightarrow\\ n\to \infty}} e(E_n)\cdot.
\end{align*}
Then \eqref{eq:equiv excess_coh} follows from the inductive limit of \eqref{eq:excess_coh n-th step}. 
\end{Remark}

There is also a $K$-theoretic counterpart of Proposition \ref{excess_coh}. Given diagram \eqref{cd:excess diagram}, we have proper pushforward 
$$f_*\colon K(X,\sw\circ f)\to K(Y,\sw), \quad f'_*\colon K(X',\sw\circ f\circ h)\to K(Y',\sw\circ g), $$ 
and~l.c.i.~pullback 
$$g^*\colon K(Y,\sw)\to K(Y',\sw\circ g), \quad h^*\colon K(X,\sw\circ f)\to K(X',\sw\circ f\circ h). $$ 

\begin{Proposition}[Excess intersection formula for critical $K$-theory]\label{excess_K}
Notations as above, we have 
\begin{align*}
    g^*\circ f_*=f'_*\circ e_K(E)\cdot h^*,
\end{align*}
where $e_K(E):=\bigwedge^{*} E^\vee$ is the $K$-theoretic Euler class of the excess bundle $E$.
\end{Proposition}
\begin{proof}
As in the proof of Proposition \ref{excess_coh}, by the graph construction \eqref{diag on graph trick}, we may assume $g$ is a closed immersion. 

We have a commutative diagram of derived schemes
\begin{equation}\label{diag on many xder}
\xymatrix{
X'  \ar[r]_{i \,\,\,} \ar[rd]_{h}   \ar@/^1.2pc/[rr]^{f'} & X'^{\der} \ar[r]_{\,\, \,\, f'^\der} \ar[d]^{h^{\der}}^{ }    & Y'   \ar[d]^{g} \\
& X \ar[r]^{f} & Y,
} 
\end{equation}
where the middle square is the homotopy pullback diagram. 
We denote $Z^{\der}_Y:=Z^{\der}(\sw)\hookrightarrow Y$ to be the derived zero locus of $\sw$ and similarly $Z^{\der}_X:=Z^{\der}(\sw\circ f)\hookrightarrow X$, $Z^{\der}_{Y'}$, $Z^{\der}_{X'}$ for the derived zero loci of pullback functions. They fit into commutative diagram of derived schemes (here we use same notations for maps as in the previous diagram):
\begin{equation}\label{diag on many zder}
\xymatrix{
Z^{\der}_{X'}  \ar[r]_{i \,\,\,} \ar[rd]_{h}  \ar@/^1.2pc/[rr]^{f'} & Z^{\der}_{X'^{\der}} \ar[r]_{\,\, \,\, f'^\der} \ar[d]^{h^{\der}}^{ }   & Z^{\der}_{Y'}  \ar[d]^{g}   \\
& Z^{\der}_X \ar[r]^{f} & Z^{\der}_{Y}.
} 
\end{equation}
As the canonical map commutes with proper pushforward and Gysin pullback, we have commutative diagrams 
$$
\xymatrix{
K(Z^{\der}_X) \ar[r]^{g^*\circ f_*} \ar@{->>}[d]_{\can}   & K(Z^{\der}_{Y'})   \ar@{->>}[d]^{\can}  \\
K(X,\sw\circ f) \ar[r]^{g^*\circ f_*} & K(Y',\sw\circ g),
} 
\quad \quad
\xymatrix{
K(Z^{\der}_X) \ar[rr]^{f'_*\circ e_K(E)\cdot h^*} \ar@{->>}[d]_{\can}  &  & K(Z^{\der}_{Y'})   \ar@{->>}[d]^{\can}  \\
K(X,\sw\circ f) \ar[rr]^{f'_*\circ e_K(E)\cdot h^*}  & & K(Y',\sw\circ g).
} 
$$
Then the claimed equality follows from Lemma \ref{lem on excess k} below. 
\end{proof}
\begin{Lemma}\label{lem on excess k}
When $g\colon Y'\to Y$ above is a closed immersion, we have an equality of maps
$$    g^*\circ f_*=f'_*\circ e_K(E)\cdot h^*\colon K(Z^{\der}_X)\to K(Z^{\der}_{Y'}). $$
\end{Lemma}
\begin{proof}
For a Noetherian derived scheme $Z^\der$ with inclusion of its classical truncation
$$i\colon Z:=t_0(Z^\der)\to Z^\der, $$
and $\calF \in \D^{\mathrm{b}}_{\mathrm{coh}}(Z^{\der})$, we have 
$$\sum_p (-1)^p i_*[\calH^p(\calF)]=[\calF]\in K(Z^\der). $$
Here by the nil-invariance of $G$-theory \cite[Cor.~3.4]{Kh}, we have 
\begin{equation}\label{equ on nilinv}i_*\colon K(Z)\cong K(Z^\der). \end{equation}
By base change in diagram \eqref{diag on many zder}, for any $\calF\in \D^{\mathrm{b}}_{\mathrm{coh}}(Z^{\der}_X)$, we have equalities in $K$-theory: 
\begin{align}\label{equ on bc gf}
[g^* f_*\calF]=[\textbf{L}g^* \circ \dR f_*(\calF)]=[\dR f'^\der_*\circ \textbf{L} h^{\der *} (\calF)]=\sum_n(-1)^n [(t_0(f'))_*\calH^n (\textbf{L} h^{\der *} \calF)]\in K(Z_{Y'}).
\end{align}
Here $t_0(f')\colon Z^{}_{X'}:=t_0\left(Z^{\der}_{X'}\right)\to Z^{}_{Y'}:=t_0\left(Z^{\der}_{Y'}\right)$ denotes the classical truncation of $f'$, 
and we have used the fact that classical truncation of $i$ is an isomorphism, and we identify the $G$-theory of $Z^{}_{Y'}$ and $Z^{\der}_{Y'}$ by \eqref{equ on nilinv}.
 
Using diagrams \eqref{diag on many xder}, \eqref{diag on many zder}, we have 
\begin{align}
\label{equ on hder}
\textbf{L} h^{\der *} \calF=\calF\otimes^{\textbf{L}}_{\mathcal O_{Z_X^\der}}\mathcal O_{Z_{X'^\der}^\der}= \calF\otimes^{\textbf{L}}_{\mathcal O_{X}}\mathcal O_{X'^\der}.
\end{align}
Here in the first equality we omit $h^{\der }_*$ since $h^\der$ is an affine morphism so we can identify $\mathcal O_{Z_{X'^\der}^\der}$ with its direct image in $Z_X^\der$ and identify $\D^{\mathrm{b}}_{\mathrm{coh}}\left(Z_{X'^\der}^\der\right)$ with $\D^{\mathrm{b}}_{\mathrm{coh}}\left(h^{\der}_*\mathcal O_{Z_{X'^\der}^\der}\right)$. In the second equality, we use the identification 
$$\mathcal O_{Z_X^\der}\otimes^{\textbf{L}}_{\mathcal O_{X}}\mathcal O_{X'^\der}\cong \mathcal O_{Z_{X'^\der}^\der}, $$ 
and we regard $\calF\otimes^{\textbf{L}}_{\mathcal O_{X}}\mathcal O_{X'^\der}$ as an object in 
$\D^{\mathrm{b}}_{\mathrm{coh}}\left(h^{\der}_*\mathcal O_{Z_{X'^\der}^\der}\right)$.

By a direct calculation, cohomology sheaves of $\mathcal O_{X'^\der}$ (viewed as $\mathcal O_{X'}$-modules) are given by  
\begin{align}\label{equ on ox com}
    \calH^i (\mathcal O_{X'^\der})=\begin{cases}
        {\bigwedge}_{\oO_{X'}}^{-i}(E^\vee), & \text{if}-\rk E\leqslant i\leqslant 0,\\
       \quad \quad  0, &\text{ otherwise}.
    \end{cases}
\end{align}
Note also the spectral sequence
\begin{align}\label{equ on sp se}
    \calH^q(\mathcal F\otimes^\textbf{L}_{\mathcal O_{X}}\calH^p(\mathcal O_{X'^\der})) \Longrightarrow \calH^{p+q} (\calF\otimes^\textbf{L}_{\mathcal O_{X}}\mathcal O_{X'^\der}),
\end{align}
where  
\begin{align}\label{equ on pq calc}
\mathcal F\otimes^\textbf{L}_{\mathcal O_{X}}\calH^p(\mathcal O_{X'^\der})&\cong  
\mathcal F\otimes^\textbf{L}_{\mathcal O_{X}}\wedge_{\oO_{X'}}^{-p}E^\vee \\ \nonumber
&\cong \mathcal F\otimes^\textbf{L}_{\mathcal O_{X}}\mathcal O_{X'}\otimes^\textbf{L}_{\mathcal O_{X'}}\wedge_{\oO_{X'}}^{-p}E^\vee \\ \nonumber
&\cong \textbf{L}h^*\mathcal F\otimes^\textbf{L}_{\mathcal O_{X'}}\wedge_{\oO_{X'}}^{-p}E^\vee \\ \nonumber
&\cong \textbf{L}h^*\mathcal F\otimes^\textbf{L}_{\mathcal O_{Z^{\der}_{X'}}}\wedge_{\oO_{Z^{\der}_{X'}}}^{-p}E^\vee. \nonumber
\end{align}
Combining \eqref{equ on bc gf}, \eqref{equ on hder}, \eqref{equ on ox com}, \eqref{equ on sp se}, \eqref{equ on pq calc}, we obtain 
\begin{align*}[g^* f_*\calF]&=\sum_{p,q\in \bZ}(-1)^{p+q}\left[t_0(f')_*\calH^q\left(\textbf{L}h^*\mathcal F\otimes^\textbf{L}_{\mathcal O_{Z^{\der}_{X'}}}\wedge_{\oO_{Z^{\der}_{X'}}}^{-p}E^\vee\right)\right] \\
&=\sum_{p,q\in \bZ}(-1)^{p+q}\left[t_0(f')_*\left(\calH^q\left(\textbf{L}h^*\mathcal F\right)\otimes_{\mathcal O_{Z^{\der}_{X'}}}\wedge_{\oO_{Z^{\der}_{X'}}}^{-p}E^\vee\right)\right] \\
&= \left[\dR f'_*\left(\left(\textbf{L}h^*\mathcal F\right)\otimes_{\mathcal O_{Z^{\der}_{X'}}}\wedge_{\oO_{Z^{\der}_{X'}}}^{\bullet}E^\vee\right)\right] \\
&=f'_*(e_K(E)\cdot h^*(\mathcal{F})),  
\end{align*}
where the third equality uses \eqref{equ on nilinv}. 
\end{proof}

\begin{Remark}[Equivariant version of Proposition \ref{excess_K}]\label{equiv excess_K}
Suppose that the maps in the diagram \eqref{cd:excess diagram} are equivariant morphisms between $G$-varieties, where $G$ is an algebraic group, and assume that $\sw$ is $G$-invariant. Then the equivariant version of Proposition \ref{excess_K} holds, i.e.
\begin{align}\label{eq:equiv excess_K}
    g^*\circ f_*=f'_*\circ e^G_K(E)\cdot h^*,
\end{align}
where $e^G_K(E):=\bigwedge^{*} E^\vee$ is the $G$-equivariant $K$-theoretic Euler class of excess bundle $E$. This is because the spectral sequence in the proof of Proposition \ref{excess_K} is $G$-equivariant if $\mathcal F$ is $G$-equivariant.
\end{Remark}

\begin{Remark}\label{disconnected X'}
In the case when $X'$ is disconnected, the Propositions \ref{excess_coh} and \ref{excess_K} and Remarks \ref{equiv excess_coh} and \ref{equiv excess_K} still hold. In this case, $e(E)\cdot$ is the component-wise multiplication.
\end{Remark}

\section{A deformed dimensional reduction for critical cohomology}

Let $Y$ be a smooth and connected variety and $f_1,\ldots,f_n,\phi$ are regular functions on $Y$. Consider the product and a function on it:  
$$X:=Y\times \bA^n, \quad \sw:=\phi+\sum_{i=1}^nf_i x_i,$$ 
where we omit the pullback of functions from $Y$ to $X$, 
$\{x_i\}_{i=1}^n$ are linear coordinates of $\bA^n$. Denote the common zero locus of $f_1,\ldots,f_n$ by  
$$Z=Z(\{f_i\}_{i=1}^n). $$
Let $\pi\colon X\to Y$ be the projection, and $j\colon S:=\pi^{-1}(Z)\hookrightarrow X$ the inclusion. We have natural morphism in $\D^b_c(Y)$:
\begin{align}\label{nat trans def dim red}
    \pi_*j_*\varphi_{\sw|_S}j^!\omega_X\longrightarrow \pi_*\varphi_{\sw}\omega_X\:.
\end{align}

\begin{Theorem}\label{thm: def dim red reg sec}
Assume that $Z$ is smooth of codimension $n$ in $Y$, then \eqref{nat trans def dim red} is an isomorphism.
\end{Theorem}
We first prove the case when $n=1$.
\begin{Lemma}\label{lem: def dim red reg sec n=1}
Theorem \ref{thm: def dim red reg sec} is true when $n=1$.
\end{Lemma}

\begin{proof}
We write $f:=f_1$, $\sw=\phi+fx$. Since $\pi_*\varphi_{\sw}\omega_X$ is supported on $\pi(\Crit(\sw))\subseteq  Z$ and $\pi_*j_*\varphi_{\sw|_S}j^!\omega_X$ is obviously supported on $Z$, it is enough to show that for an arbitrary $p\in Z$, and an arbitrary analytic open neighbourhood $W$ of $p$ in $Y$, there exists an analytic open neighbourhood $V$ of $p$ in $W$ such that the induced map between cohomologies
\begin{align*}
    H^*\left(V,(\pi_*j_*\varphi_{\sw|_S}j^!\omega_X)\big|_V\right)\longrightarrow H^*\left(V,(\pi_*\varphi_{\sw}\omega_X)\big|_V\right)
\end{align*}
is an isomorphism. Equivalently by base change, we need to find $V$ such that the pushforward map 
\begin{align*}
    H^*\left(\pi^{-1}(V)\cap S,\sw\big|_{\pi^{-1}(V)\cap S}\right)\longrightarrow H^*\left(\pi^{-1}(V),\sw\big|_{\pi^{-1}(V)}\right)
\end{align*}
of critical cohomology is isomorphism. 

By assumption, $f\colon Y\to \bA^1$ is smooth in a Zariski open neighbourhood of $Z$. For any $p\in Z$, we can take an analytic open neighbourhood $V$ of $p$ in $Y$ 
such that $f|_V\colon V \to \bA^1 $ is of form 
$$V=V_0\times B\to B, \quad (v,z)\mapsto z, $$
where $B$ is the open disk $\{|z|<\epsilon\}\subseteq  \bA^1$ for some $0<\epsilon\ll 1$. 
Then we have 
$$V\cap Z=V_0\times \{0\}, $$ 
and 
$$\sw|_{\pi^{-1}(V)}=\phi|_V+zx. $$ 
Let $\phi_0:=\phi|_{V\cap Z}=\phi|_{V_0\times \{0\}}$ and we regard it as a holomorphic function on $V$:  
$$\phi_0\colon V\to \C, $$
by pulling it back along the projection $V\to V_0$. Let 
$$\delta\phi:=\phi|_{V}-\phi_0, $$  
then we have $\delta\phi|_{z=0}=0$ by definition. Therefore $\frac{\delta\phi}{z}\colon V\to \C$ is a well-defined holomorphic function. 

Consider an isomorphism between complex manifolds: 
$$\pi^{-1}(V)=V_0\times B\times \bC_{ x}\cong V_0\times B\times \bC_{\tilde x}, \quad (v,z,x)\mapsto \left(v,z,\tilde x:=x+\frac{\delta\phi}{z}\right). $$ 
In this new coordinate system, we have 
$$\pi^{-1}(V)\cap S\cong V_0\times \{0\}\times \bC_{\tilde x},\quad\mathrm{and} \quad \sw|_{\pi^{-1}(V)}=\phi_0+z\tilde x, $$ 
where we have separated the variables, where $\phi_0$ depends only on coordinate $v$.

Combining Thom-Sebastiani isomorphism and a direct computation of $H^*(B\times \bC_{\tilde x},z\tilde x)$, we get
\begin{align*}
H^*\left(\pi^{-1}(V),\sw\big|_{\pi^{-1}(V)}\right)&\cong H^*(V_0\times B \times \bC_{\tilde x},\phi_0+z\tilde x) \\ 
 &\cong H^*( V_0,\phi_0)\otimes H^*(B\times \bC_{\tilde x},z\tilde x)  \\
    &\cong H^*( V_0,\phi_0)\otimes H^{\BM}_{-*}( \{z=0\}\times \bC_{\tilde x})   \\ 
    &\cong H^*(V_0\times \{z=0\}\times \bC_{\tilde x} ,\phi_0) \\ 
    &\cong H^*\left(\pi^{-1}(V)\cap S,\sw\big|_{\pi^{-1}(V)\cap S}\right).
\end{align*}
Note that the isomorphism is induced by the pushforward map for critical cohomology. 
Therefore we are done.
\end{proof}

\begin{proof}[Proof of Theorem \ref{thm: def dim red reg sec}]
By assumption, the map 
$$(f_1,\ldots,f_n)\colon Y\to \bA^n$$ is smooth in a Zariski open neighbourhood of $Z$. Since $\pi_*\varphi_{\sw}\omega_X$ is supported on $\pi(\Crit(\sw))\subseteq  Z$, we can replace $Y$ by an open neighbourhood $U$ of $Z$. By shrinking $U$, we may assume $(f_1,\ldots,f_n)\colon U\to \bA^n$ is a smooth map. 

Let $Z_{\leqslant k}:=Z(\{f_i\}_{i=1}^k)$ be the zero locus, then we have inclusions of smooth varieties: 
$$Z=Z_{\leqslant n}\subseteq  Z_{\leqslant n-1}\subseteq \cdots\subseteq  Z_{\leqslant 1}\subseteq  U. $$ 
Let $S_{\leqslant k}:=\pi^{-1}(Z_{\leqslant k})$ (where $\pi\colon U\times \bA^n\to U$ is projection) and set $S_{\leqslant 0}:=U\times \bA^n$. Denote the closed immersions:
$$j_{k}\colon S_{\leqslant k}\hookrightarrow S_{\leqslant k-1}, \quad j_{\leqslant k}\colon S_{\leqslant k}\hookrightarrow U\times \bA^n. $$
Then the natural morphism \eqref{nat trans def dim red} factors into a sequence of natural morphisms:
\begin{align*}
    \pi_*j_*\varphi_{\sw|_S}j^!\omega_X=\pi_*j_{\leqslant n*}\varphi_{\sw|_{S_{\leqslant n}}}j_{\leqslant n}^!\omega_X\to \pi_*j_{\leqslant n-1*}\varphi_{\sw|_{S_{\leqslant n-1}}}j_{\leqslant n-1}^!\omega_X\to\cdots\to\pi_*j_{\leqslant 1*}\varphi_{\sw|_{S_{\leqslant 1}}}j_{\leqslant 1}^!\omega_X\to \pi_*\varphi_{\sw}\omega_X\:,
\end{align*}
where the intermediate arrow 
$$\pi_*j_{\leqslant k*}\varphi_{\sw|_{S_{\leqslant k}}}j_{\leqslant k}^!\omega_{X}\to \pi_*j_{\leqslant k-1*}\varphi_{\sw|_{S_{\leqslant k-1}}}j_{\leqslant k-1}^!\omega_{X}
$$ is induced by applying $(Z_{\leqslant k-1}\hookrightarrow U)_*$ to the natural map
\begin{align*}
(S_{\leqslant k-1}\to Z_{\leqslant k-1})_*j_{k*}\varphi_{\sw|_{S_{\leqslant k}}}j_{k}^!\omega_{S_{\leqslant k-1}}\to (S_{\leqslant k-1}\to Z_{\leqslant k-1})_*\varphi_{\sw|_{S_{\leqslant k-1}}}\omega_{S_{\leqslant k-1}}\:,
\end{align*}
which is an isomorphism by Lemma \ref{lem: def dim red reg sec n=1}. Thus \eqref{nat trans def dim red} is an isomorphism.
\end{proof}

\section{A dimensional reduction for critical \texorpdfstring{$K$}{K}-theory}\label{app on k dim red}

In the appendix, we prove a dimensional reduction for critical $K$-theories.
\begin{Setting}
Let $G$ be a complex linear reductive group with a direct sum of finite dimensional linear $G$-representations 
$$W=V\oplus U.$$ 
Let $\pi\colon W\to V$ be the projection with a $G$-equivariant section $s\in \Gamma(V,\underline{U}^\vee)$ of the dual bundle $\underline{U}^\vee:=V\times U^\vee\to V$ of $\pi$. Let $Z^{\der}(s)$ be the derived zero locus of $s$, 
which fits into the following diagram 
\begin{equation}
\label{diag on zs dim red}
\xymatrix{
 \pi^{-1}(Z^{\der}(s))  \ar[r]^{\quad \quad \iota} \ar[d]^{ } \ar@{}[dr]|{\Box}  & W \ar[d]^{\pi}  \\
Z^{\der}(s) \ar[r]^{ } & V.
} 
\end{equation}
We define a $G$-invariant regular function
\begin{equation}\label{equ of w from dim red}\sw:=\langle e,s \rangle\colon W\to \mathbb{C}, \end{equation}
where $e$ is the coordinate of the fiber of $\pi$.

Choose a stability condition such that the semistable locus equals the stable locus:
$$W^{ss}=W^{s}\neq \emptyset, $$ 
with unstable locus denoted by $W^{u}:=W\setminus W^{ss}$. 
\end{Setting}

By \cite[Thm.~2.10]{Hal}, there is a semi-orthogonal decomposition (SOD): 
$$\langle \D^b_{[W^u/G]}([W/G])_{<0}, \,\mathbb{G},\, \D^b_{[W^u/G]}([W/G])_{\geqslant 0}
 \rangle =\D^b([W/G]). $$
Moreover, under the restriction to the open locus: 
$$\D^b([W/G])\twoheadrightarrow  \D^b(W^s/G), $$
we have an equivalence of categories: 
$$ \mathbb{G}\cong \D^b(W^s/G).$$
This induces an SOD for matrix factorization categories of \eqref{equ of w from dim red} (e.g.~\cite[Prop.~2.1]{P}): 
\begin{equation}\label{equ on sod}\langle \mathrm{MF}_{[W^u/G]}([W/G],\sw)_{<0},\mathbb{G}_{\sw},\mathrm{MF}_{[W^u/G]}([W/G],\sw)_{\geqslant 0}\rangle 
=\mathrm{MF}([W/G],\sw) \twoheadrightarrow \mathrm{MF}(W^s/G,\sw),    \end{equation}
and an equivalence of categories:  
$$\Phi\colon \mathbb{G}_{\sw}\cong \mathrm{MF}(W^s/G,\sw).$$
\begin{Proposition}\label{prop on dim red for kgp}
The natural inclusion $\iota$ in \eqref{diag on zs dim red} induces an isomorphism of $K$-groups: 
\begin{equation}\label{equ on ito}\iota_*\colon K((\pi^{-1}(Z^{\der}(s)))^s/G) \xrightarrow{\cong} K(W^s/G,\sw). \end{equation}
\end{Proposition}
\begin{proof}
We have the following commutative diagram 
$$
\xymatrix{
K(\mathbb{G}_{\sw}) \ar@{^(->}[r]  \ar@/^1.7pc/[rr]^{\Phi}_{\cong} & K([W/G],\sw)   \ar@{->>}[r] & K(W^s/G,\sw) \\
& K([\pi^{-1}Z^{\der}(s)/G])  \ar[u]_{\iota_*}  \ar@{->>}[r]  &  K((\pi^{-1}Z^{\der}(s))^s/G)  \ar[u]_{\iota_*} \\ 
& K([Z^{\der}(s)/G])   \ar[u]_{\pi^*}^{\cong} \ar@/^4pc/[uu]^{\cong}, &  
} 
$$
where $\iota_*\circ \pi^*$ is an isomorphism by the dimensional reduction result \cite[Prop.~2.5 \& Cor.~3.13]{Toda}. 
The commutativity of the diagram implies that the map \eqref{equ on ito} is surjective, and the composition 
$$K(\mathbb{G}_{\sw}) \hookrightarrow K([W/G],\sw)      \stackrel{(\iota_*)^{-1}}{\cong} K([\pi^{-1}Z^{\der}(s)/G]) \twoheadrightarrow K((\pi^{-1}Z^{\der}(s))^s/G)  \xrightarrow{\iota_*} K(W^s/G,\sw) $$
is the isomorphism given by $\Phi$. We are left to check the injectivity of the map \eqref{equ on ito}, which is further reduced to 
check the surjectivity of the composition 
$$K(\mathbb{G}_{\sw}) \hookrightarrow K([W/G],\sw)   \stackrel{(\iota_*)^{-1}}{\cong} K([\pi^{-1}Z^{\der}(s)/G]) 
\twoheadrightarrow K((\pi^{-1}Z^{\der}(s))^s/G).  $$
By the SOD \eqref{equ on sod}, it is enough to check that any $\mathcal{P}^{\bullet}\in \mathrm{MF}_{[W^u/G]}([W/G],\sw)$
goes to zero under the composition 
\begin{equation}\label{equ on KG}  K([W/G],\sw)   \stackrel{(\iota_*)^{-1}}{\cong} K([\pi^{-1}Z^{\der}(s)/G]) 
\twoheadrightarrow K((\pi^{-1}Z^{\der}(s))^s/G). \end{equation}
Here that the first map $(\iota_*)^{-1}$ of \eqref{equ on KG} is given by the tensor product 
$$[\mathcal{P}^{\bullet}\otimes \mathrm{Kos}(\tau,s)^\vee]\in K([\pi^{-1}Z^{\der}(s)/G]), $$
with the Koszul factorization 
\begin{align}\label{Kos fact}
\mathrm{Kos}(\tau,s):=\left(\xymatrix{
\bigwedge^{\mathrm{odd}}\pi^*U^\vee
\ar@/^/[r]^{\,\,\, s\wedge+\iota_{\tau}}&\bigwedge^{\mathrm{even}}\pi^*U^\vee  \ar@/^/[l]^{\,\,\, s\wedge+\iota_{\tau}}
}\right), 
\end{align}
where $\tau$ is the tautological section of $\pi^*U\to U$. By definition, $\mathcal{P}^{\bullet}$ is acyclic when restricted to the stable locus; therefore $\mathcal{P}^{\bullet}$ and $[\mathcal{P}^{\bullet}\otimes \mathrm{Kos}(\tau,s)^\vee]$ go to zero under the map \eqref{equ on KG}. 
\end{proof}
\begin{Remark}\label{rmk on teqiuv}
One can extend the above result to the equivariant setting in the presence of a torus action. 
\end{Remark}

\section{Proof of Proposition \ref{prop:Psi}}\label{sec:proof of prop psi}

Our strategy of proving Proposition \ref{prop:Psi} is to deduce it from the axiom (iii) of the abelian stable envelope (Definition  \ref{def of stab k}). The difficulty is that the nonabelian stable envelope is constructed from an abelian stable envelope with respect to a $\bC^*$-action which is in general different from the $\sigma_i$ that appears in the KN stratification. We resolve this issue by showing that the $\Stab^{\mathsf s}_+$ that appears in the construction of $\widetilde{\bPsi}^{\mathsf s}_K$ is essentially the same as the stable envelope with respect to the two dimensional torus $\bC^*\times \bC^*_{\sigma_i}$ (see Lemma \ref{lem:larger chamber stab}). The main geometric fact that we use in the proof is to identify the KN strata $S_i$ with an open subset in a $\bC^*\times \bC^*_{\sigma_i}$-attracting set (Lemma \ref{lem:more explicit KN strata}). We deduce it from a preliminary result that identifies $S_i$ with an open subset in a $\bC^*$-attracting set (Proposition \ref{prop:explicit KN strata}).

\subsection{KN strata of the unstable locus}
We give an explicit description of the KN strata. We first define
\begin{align}\label{sigma}
    \Sigma:=\{\mathbf u\in \bN^{Q_0}\colon \mathbf u_i\leqslant\mathbf v_i,\,\forall\, i\in Q_0\},
\end{align}
together with the partial order $\le$ on $\Sigma$ given by $\mathbf u\le \mathbf u'$ if and only if $\mathbf u_i\leqslant \mathbf u'_i$ for all $i\in Q_0$. For $\mathbf u\in \Sigma$, we define cocharacter $\sigma(\mathbf u)$ of $G$ by
\begin{align}\label{sigma(u)}
    \sigma(\mathbf u)=(-\omega_{i,\mathbf u_i})_{i\in Q_0}\,,
\end{align}
where $\omega_{i,\mathbf u_i}$ is the $\mathbf u_i$-th fundamental cocharacter of $\GL(\mathbf v_i)$. It is straightforward to see that if $\mathbf u,\mathbf u'\in \Sigma$ then $\mathbf u\le \mathbf u'$ implies $\mu(\sigma(\mathbf u))\le \mu(\sigma(\mathbf u'))$.
We keep using the notation 
\begin{align*}
    \underline{\bd}^\bv=(\bd+\bv,\bd).
\end{align*}

\begin{Remark}\label{rmk:Sigma_theta}
For $\mathbf u\in \Sigma$, we associate to it the following connected component in $\cM_\theta(\mathbf v,\underline{\bd}^\bv)^{\bC^*}$:
\begin{align}\label{F(u)}
    F_{\mathbf u}:=\cM_\theta(\mathbf u,\underline{\mathbf{0}}^{\bv})\times \cM_\theta(\mathbf v-\mathbf u,\mathbf d).
\end{align}
Here $\cM_\theta(\mathbf u,\underline{\mathbf{0}}^{\bv})$ is the quiver variety given by the GIT quotient
\begin{align*}
    \cM_\theta(\mathbf u,\underline{\mathbf{0}}^{\bv})=\left(\bigoplus_{a\in Q_1}\Hom(\bC^{\mathbf u_{t(a)}},\bC^{\mathbf u_{h(a)}})\oplus \bigoplus_{i\in Q_0}\Hom(\bC^{\mathbf v_{i}},\bC^{\mathbf u_{i}})\right)/\!\!/_\theta \prod_{i\in Q_0}\GL(\mathbf u_{i})
\end{align*}
We note that $F_{\mathbf u}$ is nonempty if and only if $R(\mathbf v-\mathbf u,\mathbf d)^{\theta\emph{-}ss}$ is nonempty, because $\cM_\theta(\mathbf u,\underline{\mathbf{0}}^{\bv})$ is always nonempty.

Let us also define 
\begin{align}\label{sigma_theta}
    \Sigma_\theta:=\{\mathbf u\in \Sigma\colon R(\mathbf v-\mathbf u,\mathbf d)^{\theta\emph{-}ss}\neq \emptyset\},
\end{align}
then $F_{\mathbf u}\neq \emptyset\Longleftrightarrow \mathbf u\in \Sigma_\theta$. Consider the ample line bundle $\mathcal L_{\theta}:=\bigotimes_{i\in Q_0}(\det\mathcal V_i)^{ -\theta_i}$ on $\cM_\theta(\mathbf v,\underline{\bd}^\bv)$, where $\cV_i$ is the tautological vector bundle corresponding to $i$-th node. $\mathcal L_{\theta}$ induces the ample partial order $\le$ on $\Fix_{\bC^*}(\cM_\theta(\mathbf v,\underline{\bd}^\bv))$ as in the Remark \ref{partial order by line bundle}. Then for $\mathbf u,\mathbf u'\in \Sigma_\theta$, we have the following implication
\begin{align*}
    \mathbf u\le \mathbf u'\Longrightarrow F_{\mathbf u}\le F_{\mathbf u'}.
\end{align*}
Let us refine $\le$ to a total order on $\Sigma$ such that its restriction to $\Sigma_\theta$ refines the ample partial order.
\end{Remark}

\begin{Lemma}\label{lem:M' covered by attr}
We have the decomposition $$\cM_\theta(\mathbf v,\underline{\bd}^\bv)=\bigcup_{\mathbf u\in \Sigma_\theta}\Attr_+(F_{\mathbf u}).$$
\end{Lemma}

\begin{proof}
Since the $\bC^*$-weights in the ring $\bC[R(\mathbf v,\underline{\bd}^\bv)]^{G}$ of invariant functions  are nonpositive, the $\bC^*$-action on the affine quotient $\cM'_0(\mathbf v,\underline{\bd}^\bv)$ is attracting. Since the canonical map $\cM_\theta(\mathbf v,\underline{\bd}^\bv)\to \cM'_0(\mathbf v,\underline{\bd}^\bv)$ is proper, we obtain the decomposition 
$$\cM_\theta(\mathbf v,\underline{\bd}^\bv)=\bigcup_{F\in \Fix_{\bC^*}(\cM_\theta(\mathbf v,\underline{\bd}^\bv))}\Attr_+(F). $$ 
We note that any $F\in \Fix_{\bC^*}(\cM_\theta(\mathbf v,\underline{\bd}^\bv))$ is of form $F_{\mathbf u}$ for some $\mathbf u\in \Sigma_\theta$ and vice versa, then the lemma follows.
\end{proof}

\begin{Proposition}\label{prop:explicit KN strata}
Up to $G$-conjugation and a positive scalar multiple, the cocharacters that appear in KN stratification of $R(\mathbf v,\mathbf d)^{\theta\emph{-}u}$ are given by cocharacters in the set
\begin{align}
    \{\sigma(\mathbf u):\mathbf u\in \Sigma_\theta,\;\mathbf u\neq \mathbf 0\}.
\end{align}
For the above $\sigma(\mathbf u)$, the associated strata $S_{\sigma(\mathbf u)}$ is given by $$S_{\sigma(\mathbf u)}=\Attr_+(F_{\mathbf u})\cap \mathring\cM_\theta(\mathbf v,\underline{\bd}^\bv),$$ 
where $F_{\mathbf u}$ is the connected component in $\cM_\theta(\mathbf v,\underline{\bd}^\bv)^{\bC^*}$ given by \eqref{F(u)}. 

Moreover, if $\mathbf u\in \Sigma_\theta$, then we have
\begin{align*}
    \overline{S_{\sigma(\mathbf u)}}\setminus S_{\sigma(\mathbf u)}\subseteq  \bigcup_{\mathbf u'\in \Sigma_\theta,\mathbf u'>\mathbf u} S_{\sigma(\mathbf u')},
\end{align*}
if $\mathbf u\notin \Sigma_\theta$, then we have
\begin{align*}
    R(\mathbf v,\mathbf d)^{\sigma(\mathbf u)}\subseteq  \bigcup_{\mathbf u'\in \Sigma_\theta,\mathbf u'>\mathbf u} S_{\sigma(\mathbf u')}.
\end{align*}
\end{Proposition}

\begin{proof}
Using Cauchy–Schwarz inequality \footnote{Namely, we have $\left(\sum_{i\in Q_0}|\theta_i|\tr(\sigma_i^2)\right)\left(\sum_{i\in Q_0}|\theta_i|\right)\geqslant \left(\sum_{i\in Q_0}|\theta_i|\sum_{j=1}^{\mathbf v_i}|\sigma_{i,j}|\right)^2\geqslant (\sigma,\theta)^2$, which implies that $|\mu(\sigma)|\leqslant \sqrt{\sum_{i\in Q_0}|\theta_i|}$. Moreover the equality holds if and only if $\sigma_{i,j}=\sigma_{i',j'}$ for all possible $i,i',j,j'$. Taking sign into consideration, we see that the maximal value of $\mu(\sigma)$ is achieved if and only if $\sigma_{i,j}=-1$ for all possible $i,j$, up to $G$-conjugation and a positive scalar multiple.}, one can see that the invariant $\mu(\sigma)$ is maximized by taking $\sigma=\sigma(\mathbf v)$, up to $G$-conjugation and a positive scalar multiple. Note that $\sigma(\mathbf v)$ is nothing but the $\bC^*\subseteq  G$ that we previously considered to construct $\widetilde{\bPsi}^{\mathsf s}_K$. 
In particular, the image of $\sigma(\mathbf v)$ is in the center of $G$, so $\Attr_{\sigma(\mathbf v)}(R(\mathbf v,\mathbf d)^{\sigma(\mathbf v)})$ is invariant under the $G$-action. Moreover $R(\mathbf v,\mathbf d)^{\sigma(\mathbf v)}\cong R(\mathbf v,\mathbf 0)\times R(\mathbf 0,\mathbf d)$, and obviously we have $R(\mathbf 0,\mathbf d)=R(\mathbf 0,\mathbf d)^{\theta\emph{-}ss}=\pt$. It follows from Lemma \ref{lem KN strata vs attr} below that the minimal KN stratum is given by 
$$S_{\sigma(\mathbf v)}=\Attr_{\sigma(\mathbf v)}\left(R(\mathbf v,\mathbf d)^{\sigma(\mathbf v)}\right)=\Attr_+(F_{\mathbf v})\cap \mathring\cM_\theta(\mathbf v,\underline{\bd}^\bv).$$
By minimality, we have $\overline{S_{\sigma(\mathbf v)}}=S_{\sigma(\mathbf v)}$.

To proceed, we introduce a conjugation-invariant map: 
$$\mathsf{Supp}:\cochar(G)\to \Sigma, \quad \mathsf{Supp}(\lambda)_i=\rk\lambda'_i, $$  
where $\lambda'$ is $G$-conjugate to $\lambda$ and $\lambda'_i=\diag(\lambda'_{i,j})_{j=1}^{\mathbf v_i}$ is a cocharacter of $\GL(\mathbf v_i)$ in the standard basis. Then it is easy to see that if $\lambda\in \cochar(G)$ and $\mathbf u=\mathsf{Supp}(\lambda)$, then
\begin{align}\label{lambda fix vs sigma(u) fix}
    R(\mathbf v,\mathbf d)^{\lambda}\subseteq  G\cdot R(\mathbf v,\mathbf d)^{\sigma(\mathbf u)}.
\end{align}
Moreover, the Cauchy–Schwarz inequality shows that the numerical invariant
\begin{align*}
    \mu(\lambda)=\frac{(\lambda,\theta)}{|\lambda|}\bigg|_{\mathsf{Supp}(\lambda)=\mathbf u}
\end{align*}
is maximized by $\lambda=\sigma(\mathbf u)$ up to a positive scalar multiple and $G$-conjugation. Then we claim that any $\sigma_i$ that appears in the KN stratification, up to a positive scalar multiple and $G$-conjugation, belongs to the set
\begin{align*}
    \{\sigma(\mathbf u):\mathbf u\in \Sigma,\;\mathbf u\neq \mathbf 0\}.
\end{align*}
In fact, $R(\mathbf v,\mathbf d)^{\sigma_i}$ is not contained in previously defined strata by the construction of KN strata. Let $\mathbf u_i:=\mathsf{Supp}(\sigma_i)$, then \eqref{lambda fix vs sigma(u) fix} implies that $R(\mathbf v,\mathbf d)^{\sigma(\mathbf u_i)}$ is not contained in previously defined strata. Since $\sigma_i$ maximize $\mu(\sigma)$ among those $\sigma$ such that $R(\mathbf v,\mathbf d)^{\sigma}$ is not contained in previously defined strata, we must have $\sigma_i=\sigma(\mathbf u_i)$ up to a positive scalar multiple and $G$-conjugation.

It remains to show that
\begin{enumerate}
    \item $\sigma(\mathbf u)$ appears in KN stratification if and only if $\mathbf u\in \Sigma_\theta$, 
    \item if $\mathbf u\in \Sigma_\theta$, then $S_{\sigma(\mathbf u)}=\Attr_+(F_{\mathbf u})\cap \mathring\cM_\theta(\mathbf v,\underline{\bd}^\bv)$,
    \item  if $\mathbf u\in \Sigma_\theta$, then we have
\begin{align*}
    \overline{S_{\sigma(\mathbf u)}}\setminus S_{\sigma(\mathbf u)}\subseteq  \bigcup_{\mathbf u'\in \Sigma_\theta,\mathbf u'>\mathbf u} S_{\sigma(\mathbf u')},
\end{align*}
\item  if $\mathbf u\notin \Sigma_\theta$, then we have
\begin{align*}
    R(\mathbf v,\mathbf d)^{\sigma(\mathbf u)}\subseteq  \bigcup_{\mathbf u'\in \Sigma_\theta,\mathbf u'>\mathbf u} S_{\sigma(\mathbf u')}.
\end{align*}
\end{enumerate}
We prove the above four statements by induction downwards along the total order $\le$ on $\Sigma$ defined in Remark \ref{rmk:Sigma_theta}. The maximal case $\mathbf u=\mathbf v$ has been proven previously. Let $\mathbf u\in \Sigma, \mathbf u\neq \mathbf 0$. Suppose that (1)-(4) have been proven for all $\mathbf u'\in \Sigma$ such that $\mathbf u'> \mathbf u$. Consider the following union of KN strata
\begin{align*}
    X_{\mathbf u}:=\bigcup_{\mathbf u'\in \Sigma_\theta,\mathbf u'> \mathbf u}S_{\sigma(\mathbf u')}.
\end{align*}
We claim that $X_{\mathbf u}$ is closed. In fact, by the induction hypothesis, we have $\overline{S_{\sigma(\mathbf u')}}\subseteq  X_{\mathbf u}$ for all $\mathbf u'\in \Sigma_\theta,\mathbf u'> \mathbf u$. Then $X_{\mathbf u}=\bigcup_{\mathbf u'\in \Sigma_\theta,\mathbf u'> \mathbf u}\overline{S_{\sigma(\mathbf u')}}$ is closed. 

Suppose that $\mathbf u\notin \Sigma_\theta$, then according to Lemma \ref{lem KN strata vs attr} we have the inclusion $$R(\mathbf v,\mathbf d)^{\sigma(\mathbf u)}\subseteq  \bigcup_{\mathbf u'> \mathbf u}G\cdot \Attr_{\sigma(\mathbf u')}\left(R(\mathbf v,\mathbf d)^{\sigma(\mathbf u')}\right).$$ By induction hypothesis, we have 
\begin{align*}
    \bigcup_{\mathbf u'> \mathbf u}G\cdot \Attr_{\sigma(\mathbf u')}\left(R(\mathbf v,\mathbf d)^{\sigma(\mathbf u')}\right)\subseteq  \bigcup_{\mathbf u''\in \Sigma_\theta,\mathbf u''> \mathbf u} S_{\sigma(\mathbf u'')}=X_{\mathbf u}.
\end{align*}
Here we have used the fact that $S_{\sigma(\mathbf u'')}$ is $G$-invariant so it is invariant under taking attracting set and $G$-action. It follows that $R(\mathbf v,\mathbf d)^{\sigma(\mathbf u)}\subseteq  X_{\mathbf u}$. Thus $\sigma(\mathbf u)$ does not appear in KN stratification. This proves (4) and part of (1).

Suppose that $\mathbf u\in \Sigma_\theta$, then according to Lemma \ref{lem KN strata vs attr} we have 
\begin{align*}
    G\cdot \Attr_{\sigma(\mathbf u)}\left(R(\mathbf u,\mathbf 0)\times R(\mathbf v-\mathbf u,\mathbf d)^{\theta\emph{-}ss}\right)=\Attr_+(F_{\mathbf u})\cap \mathring\cM_\theta(\mathbf v,\underline{\bd}^\bv).
\end{align*}
In particular, $R(\mathbf v,\mathbf d)^{\sigma(\mathbf u)}\cap \left(\Attr_+(F_{\mathbf u})\cap \mathring\cM_\theta(\mathbf v,\underline{\bd}^\bv)\right)\neq \emptyset$. We note that by induction hypothesis, we have
\begin{align*}
    X_{\mathbf u}=\left(\bigcup_{\mathbf u'\in \Sigma_\theta,\mathbf u'> \mathbf u}\Attr_+(F_{\mathbf u'})\right)\cap \mathring\cM_\theta(\mathbf v,\underline{\bd}^\bv).
\end{align*}
Since $F_{\mathbf u}\cap F_{\mathbf u'}=\emptyset$ if $\mathbf u\neq \mathbf u'$, we deduce that $R(\mathbf u,\mathbf 0)\times R(\mathbf v-\mathbf u,\mathbf d)^{\theta\emph{-}ss}$ is an nonempty open subset in $R(\mathbf v,\mathbf d)^{\sigma(\mathbf u)}$, and $(R(\mathbf u,\mathbf 0)\times R(\mathbf v-\mathbf u,\mathbf d)^{\theta\emph{-}ss})\cap X_{\mathbf u}=\emptyset$. Then $\sigma(\mathbf u)$ appears in KN stratification because $\sigma(\mathbf u)$ maximizes $\mu(\sigma)$ among those $\sigma\in \{\sigma(\mathbf u'):\mathbf u'\in \Sigma,\;\mathbf u'\le \mathbf u\}$. This finishes the proof of (1). 

Denote $S'_{\sigma(\mathbf u)}:=\Attr_+(F_{\mathbf u})\cap \mathring\cM_\theta(\mathbf v,\underline{\bd}^\bv)$, then we have 
\begin{align}\label{S contains S'}
    S_{\sigma(\mathbf u)}=G\cdot \Attr_{\sigma(\mathbf u)}\left(R(\mathbf v,\mathbf d)^{\sigma(\mathbf u)}\setminus X_{\mathbf u}\right)\supseteq G\cdot \Attr_{\sigma(\mathbf u)}\left(R(\mathbf u,\mathbf 0)\times R(\mathbf v-\mathbf u,\mathbf d)^{\theta\emph{-}ss}\right)=S'_{\sigma(\mathbf u)}
\end{align}
$S'_{\sigma(\mathbf u)}$ is dense in $S_{\sigma(\mathbf u)}$ because $R(\mathbf u,\mathbf 0)\times R(\mathbf v-\mathbf u,\mathbf d)^{\theta\emph{-}ss}$ is open and dense in $R(\mathbf v,\mathbf d)^{\sigma(\mathbf u)}$. Therefore we have
\begin{align*}
    \overline{S_{\sigma(\mathbf u)}}\setminus S'_{\sigma(\mathbf u)}=\overline{S'_{\sigma(\mathbf u)}}\setminus S'_{\sigma(\mathbf u)}&\subseteq  \left(\overline{\Attr_+(F_{\mathbf u})}\cap \mathring\cM_\theta(\mathbf v,\underline{\bd}^\bv)\right)\setminus \left(\Attr_+(F_{\mathbf u})\cap \mathring\cM_\theta(\mathbf v,\underline{\bd}^\bv)\right)\\
    \text{\tiny by definition of $\preceq$ }&\subseteq  \left(\bigcup_{F'\in \Fix_{\bC^*}(\cM_\theta(\mathbf v,\underline{\bd}^\bv)),F_{\sigma(\mathbf u)}\prec F'}\Attr_+(F')\right)\bigcap \mathring\cM_\theta(\mathbf v,\underline{\bd}^\bv)\\
    \text{\tiny since $\le$ refines $\preceq$ }&\subseteq  \left(\bigcup_{F'\in \Fix_{\bC^*}(\cM_\theta(\mathbf v,\underline{\bd}^\bv)),F_{\sigma(\mathbf u)}< F'}\Attr_+(F')\right)\bigcap \mathring\cM_\theta(\mathbf v,\underline{\bd}^\bv)\\
    \text{\tiny by Remark \ref{rmk:Sigma_theta} }&= \left(\bigcup_{\mathbf u'\in \Sigma_\theta,\mathbf u'> \mathbf u}\Attr_+(F_{\mathbf u'})\right)\bigcap \mathring\cM_\theta(\mathbf v,\underline{\bd}^\bv)\\
    &=X_{\mathbf u}.
\end{align*}
It follows that
\begin{align}\label{S' contains S}
    S_{\sigma(\mathbf u)}\subseteq  \overline{S_{\sigma(\mathbf u)}}\setminus X_{\mathbf u}\subseteq  \overline{S_{\sigma(\mathbf u)}}\setminus (\overline{S_{\sigma(\mathbf u)}}\setminus S'_{\sigma(\mathbf u)})=S'_{\sigma(\mathbf u)}.
\end{align}
The inclusions \eqref{S contains S'} and \eqref{S' contains S} imply that $S_{\sigma(\mathbf u)}=S'_{\sigma(\mathbf u)}$. This proves (2). Finally (3) follows from (2) because we just show that $\overline{S_{\sigma(\mathbf u)}}\setminus S'_{\sigma(\mathbf u)}\subseteq  X_{\mathbf u}$. This finishes the induction step.
\end{proof}

Let $\mathbf u\in \Sigma$ be a nonzero element. Note that 
$$R(\mathbf v,\mathbf d)^{\sigma(\mathbf u)}\cong R(\mathbf u,\mathbf 0)\times R(\mathbf v-\mathbf u,\mathbf d).$$
In view of the isomorphism \eqref{equ on mvcongr}, $R(\mathbf u,\mathbf 0)\times R(\mathbf v-\mathbf u,\mathbf d)^{\theta\emph{-}ss}$ is naturally identified with a (possibly empty) locally-closed subvariety of $\mathring\cM_\theta(\mathbf v,\underline{\bd}^\bv)$.

\begin{Lemma}\label{lem KN strata vs attr}
Let $\mathbf u\in \Sigma$ be a nonzero element.
\begin{itemize}
    \item If $\mathbf u\notin \Sigma_\theta$, then we have 
    \begin{align*}
        R(\mathbf v,\mathbf d)^{\sigma(\mathbf u)}\subseteq  \bigcup_{\mathbf u'> \mathbf u}G\cdot \Attr_{\sigma(\mathbf u')}\left(R(\mathbf v,\mathbf d)^{\sigma(\mathbf u')}\right).
    \end{align*}
    \item If $\mathbf u\in \Sigma_\theta$, then we have
    \begin{align*}
    G\cdot \Attr_{\sigma(\mathbf u)}\left(R(\mathbf u,\mathbf 0)\times R(\mathbf v-\mathbf u,\mathbf d)^{\theta\emph{-}ss}\right)=\Attr_+(F_{\mathbf u})\cap \mathring\cM_\theta(\mathbf v,\underline{\bd}^\bv),
    \end{align*}
    in particular, the right-hand-side is nonempty.
\end{itemize}
\end{Lemma}

\begin{proof}
If $\mathbf u\notin \Sigma_\theta$, then by definition of $\Sigma_\theta$ every quiver representation in $R(\mathbf v-\mathbf u,\mathbf d)$ is $\theta$-unstable. Let $(H_i)_{i\in Q_0}\in R(\mathbf v-\mathbf u,\mathbf d)$ be a quiver representation, and take its Harder-Narasimhan filtration 
$$0=H^{(0)}\subseteq  H^{(1)}\subseteq \cdots \subseteq  H^{(n)}=H$$ with respect to the $\theta$-stability,~i.e.~$H^{(i+1)}/H^{(i)}$ is $\theta$-semistable for all possible $i$ and its $\theta$-slope decreases strictly as $i$ increases. By assumption $H^{(1)}\neq H$, so the $\theta$-slope of $H^{(1)}$ is greater than the the $\theta$-slope of $H$, and the latter is zero. Since we assume $\theta_i<0$ for all $i\in Q_0$, this implies that $H^{(1)}$ contains the framing node $\infty$ in the Crawley-Boevey quiver associated with $(Q,\bd)$,~i.e.~$H^{(1)}\in R(\bv-\bu',\bd)^{\theta\emph{-}ss}$ for some $\bu'>\bu$. It follows that $H/H^{(1)}\in R(\bu'-\bu,\mathbf 0)$. By decomposing $H$ into an extension of $H/H^{(1)}$ by $H^{(1)}$, we see that
\begin{align*}
    H\in G_{\bv-\bu}\cdot\Attr_{\sigma(\bu'-\bu)}\left(R(\mathbf v-\mathbf u,\mathbf d)^{\sigma(\bu'-\bu)}\right),
\end{align*}
where $G_{\bv-\bu}=\prod_{i\in Q_0}\GL(\bv_i-\bu_i)$ is the gauge group. Since $H$ is arbitrary, we have 
\begin{align*}
    R(\mathbf v-\mathbf u,\mathbf d)\subseteq   \bigcup_{\mathbf u'> \mathbf u}G_{\bv-\bu}\cdot\Attr_{\sigma(\bu'-\bu)}\left(R(\mathbf v-\mathbf u,\mathbf d)^{\sigma(\bu'-\bu)}\right).
\end{align*}
Since $R(\mathbf v,\mathbf d)^{\sigma(\mathbf u)}\cong R(\mathbf u,\mathbf 0)\times R(\mathbf v-\mathbf u,\mathbf d)$, it follows that
\begin{align*}
        R(\mathbf v,\mathbf d)^{\sigma(\mathbf u)}\subseteq  \bigcup_{\mathbf u'> \mathbf u}G\cdot \Attr_{\sigma(\mathbf u')}\left(R(\mathbf v,\mathbf d)^{\sigma(\mathbf u')}\right).
\end{align*}
If $\mathbf u\in \Sigma_\theta$, then it is enough to show that
\begin{align}\label{LHS in RHS}
    \Attr_{\sigma(\mathbf u)}\left(R(\mathbf u,\mathbf 0)\times R(\mathbf v-\mathbf u,\mathbf d)^{\theta\emph{-}ss}\right)\subseteq  \Attr_+(F_{\mathbf u})
\end{align}
and 
\begin{align}\label{RHS in LHS}
    G\cdot \Attr_{\sigma(\mathbf u)}\left(R(\mathbf u,\mathbf 0)\times R(\mathbf v-\mathbf u,\mathbf d)^{\theta\emph{-}ss}\right)\supseteq\Attr_+(F_{\mathbf u})\cap \mathring\cM_\theta(\mathbf v,\underline{\bd}^\bv).
\end{align}
In view of the isomorphism \eqref{equ on mvcongr}, $R(\mathbf u,\mathbf 0)\times R(\mathbf v-\mathbf u,\mathbf d)^{\theta\emph{-}ss}$ is isomorphic to the moduli space of representations of the following quiver:
\begin{equation}\label{fixed pts quiver rep}
\xymatrix{
M_i  & & N_i \ar@<2pt>[d]^{B_i}\\
\bC^{\bu_i} \ar[u]^{I'_i}  & \bC^{\bv_i-\bu_i} \ar[ur]^{I''_i} & D_i \ar@<2pt>[u]^{A_i}
}
\end{equation}
where it is required that $I'_i$ and $I''_i$ are isomorphisms for all $i\in Q_0$, and the RHS part is $\theta$-semistable after forgetting $I''$. Taking attracting set with respect to $\sigma(\bu)$ action amounts to taking an extension of the LHS of \eqref{fixed pts quiver rep} by the RHS of \eqref{fixed pts quiver rep}, i.e. $\Attr_{\sigma(\mathbf u)}(R(\mathbf u,\mathbf 0)\times R(\mathbf v-\mathbf u,\mathbf d)^{\theta\emph{-}ss})$ is the moduli space of the following quiver
\begin{equation}\label{attr quiver rep}
\xymatrix{
M_i \ar@{.>}[drr] \ar@{.>}[rr] & & N_i \ar@<2pt>[d]^{B_i}\\
\bC^{\bu_i} \ar[u]^{I'_i} \ar@/^0.8pc/@{.>}[urr] & \bC^{\bv_i-\bu_i} \ar[ur]^{I''_i} & D_i \ar@<2pt>[u]^{A_i}
}
\end{equation}
with the same conditions as above. Here dotted lines are new linear maps that appear in extensions between representations. Similar argument shows that $\Attr_+(F_{\mathbf u})$ is the moduli space of the quiver
\begin{equation}\label{attr quiver rep 2}
\xymatrix{
M_i \ar@{.>}[drr] \ar@{.>}[rr] & & N_i \ar@<2pt>[d]^{B_i}\\
\bC^{\bv_i} \ar[u] \ar@{.>}[urr] &  & D_i \ar@<2pt>[u]^{A_i}
}
\end{equation}
with the condition that rigid lines parts are $\theta$-semistable. Writing $\bC^{\bv_i}=\bC^{\bu_i}\oplus\bC^{\bv_i-\bu_i}$ and comparing stability conditions between \eqref{attr quiver rep} and \eqref{attr quiver rep 2}, we obtain \eqref{LHS in RHS}.

To prove \eqref{RHS in LHS}, we notice that $\Attr_+(F_{\mathbf u})\cap \mathring\cM_\theta(\mathbf v,\underline{\bd}^\bv)$ is the moduli space of the quiver \eqref{attr quiver rep 2} with the condition that rigid lines parts are $\theta$-semistable, and the 
\begin{equation*}
\xymatrix{
M_i  &  N_i \\
\bC^{\bv_i} \ar[u] \ar@{.>}[ur] & 
}
\end{equation*}
part induces isomorphisms $\bC^{\bv_i}\cong M_i\oplus N_i$. Given a representation of the quiver \eqref{attr quiver rep 2} with the above conditions satisfied, let $K_i:=\ker(\bC^{\bv_i}\to M_i)$ and $P_i:=\ker(\bC^{\bv_i}\to N_i)$. Then by the assumption we have $\bC^{\bv_i}=P_i\oplus K_i$, and the induced maps $P_i\to M_i, K_i\to N_i$ are isomorphisms. We rewrite the quiver diagram \eqref{attr quiver rep 2} with $\bC^{\bv_i}$ replaced by $P_i\oplus K_i$, and obtain the following quiver
\begin{equation}\label{attr quiver rep 3}
\xymatrix{
M_i \ar@{.>}[drr] \ar@{.>}[rr] & & N_i \ar@<2pt>[d]^{B_i}\\
P_i \ar[u]^{I'_i}  & K_i \ar[ur]^{I''_i} & D_i \ar@<2pt>[u]^{A_i}
}
\end{equation}
The stability conditions in \eqref{attr quiver rep 3} are such that $I'_i, I''_i$ are isomorphisms and the part formed by $N_i,D_i$ and arrows among them is $\theta$-semistable. Up to a choice of isomorphisms $P_i\cong \bC^{\bu_i}, K_i\cong \bC^{\bv_i-\bu_i}$ (which can be accomplished by a $G'$ action), \eqref{attr quiver rep 3} is identified with a special case of \eqref{attr quiver rep}. This proves \eqref{RHS in LHS}.
\end{proof}

\subsection{Finish the proof}

We have given an explicit description of the KN stratification in terms of attracting sets of $\cM_\theta(\bv,\underline{\bd}^\bv)$ in Proposition \ref{prop:explicit KN strata}. For our purpose, we shall need the following enhancement of Proposition \ref{prop:explicit KN strata}.

Let $\sA=(\bC^*)^2$ act on $\cM_\theta(\bv,\underline{\bd}^\bv)$ via $(\sigma(\bu),\sigma(\bv)):\bC^*\times\bC^*\to G'$. Let $(t_1,t_2)\in \Lie(\sA)_{\bR}$ be the coordinate, then it is easy to see that the root hyperplanes are $t_1=0$, $t_2=0$, and $t_1+t_2=0$. Define $\fC$ to be the chamber $t_1>0,t_2>0$. Write the $\sigma(\bu)$-fixed components of $F_\bu$ as 
\begin{align*}
    F_{\mathbf u}^{\sigma(\bu)}=\bigsqcup_{\bu'\in \Sigma,\bu'\le \bu}F_{\bu,\bu'},\;\text{ where }\; F_{\bu,\bu'}=\cM_\theta(\bu',\underline{\mathbf{0}}^{\bu})\times \cM_\theta(\bu-\bu',\underline{\mathbf{0}}^{\bv-\bu})\times \cM_\theta(\bv-\bu,\bd).
\end{align*}
Then it is straightforward to see that
\begin{align*}
    \cM_\theta(\bv,\underline{\bd}^\bv)^\sA=\bigsqcup_{\bu,\bu'\in \Sigma,\bu'\le \bu}F_{\bu,\bu'}.
\end{align*}

\begin{Lemma}\label{lem:more explicit KN strata}
If $\bu\in \Sigma_\theta$, then we have
\begin{align*}
    \Attr_+(F_{\mathbf u})\cap \mathring\cM_\theta(\mathbf v,\underline{\bd}^\bv)=\Attr_\fC(F_{\bu,\bu})\cap \mathring\cM_\theta(\mathbf v,\underline{\bd}^\bv).
\end{align*}
In particular, $S_{\sigma(\mathbf u)}=\Attr_\fC(F_{\bu,\bu})\cap \mathring\cM_\theta(\mathbf v,\underline{\bd}^\bv)$.
\end{Lemma}

\begin{proof}
The ``+'' chamber is given by $t_1=0,t_2>0$, which is on the boundary of $\fC$. Let $\fC/+$ be the chamber $t_1>0$ of the torus $\sA/\bC^*_{\sigma(\bv)}$ action on $F_{\bu}$, then the same argument as Lemma \ref{lem:M' covered by attr} shows that $F_\bu=\bigcup_{\bu'\in \Sigma,\bu'\le \bu}\Attr_{\fC/+}(F_{\bu,\bu'})$. Since for any $\bu'\in \Sigma,\bu'\le \bu$ we have $\Attr_\fC(F_{\bu,\bu'})=\Attr_+(\Attr_{\fC/+}(F_{\bu,\bu'}))$, it follows that we have decomposition
\begin{align*}
    \Attr_+(F_{\mathbf u})=\bigcup_{\bu'\in \Sigma,\bu'\le \bu}\Attr_{\fC}(F_{\bu,\bu'}).
\end{align*}
Then the lemma follows from the following claim: if $\bu'<\bu$ then $\Attr_{\fC}(F_{\bu,\bu'})\cap \mathring\cM_\theta(\mathbf v,\underline{\bd}^\bv)=\emptyset$. In fact, $\Attr_{\fC}(F_{\bu,\bu'})$ is the moduli space of quiver representations $V\in\cM_\theta(\bv,\underline{\bd}^\bv)$ which admits filtration by subrepresentations $V^{(1)}\subseteq  V^{(2)}\subseteq  V$ such that 
\begin{align*}
    V/V^{(2)}\in \cM_\theta(\bu',\underline{\mathbf{0}}^{\bu}),\quad V^{(2)}/V^{(1)}\in \cM_\theta(\bu-\bu',\underline{\mathbf{0}}^{\bv-\bu}),\quad V^{(1)}\in \cM_\theta(\bv-\bu,\bd).
\end{align*}
The linear maps $\{I_i\}_{i\in Q_0}$ in \eqref{diag of R'2} respect the filtration, in particular it is isomorphism if and only if it is isomorphism after passing to associated graded vector spaces. If $\bu'<\bu$, then there exists $i\in Q_0$ such that the source and the target of $I_i$ have different ranks, so $I_i$ can not be isomorphism. This proves the claim.
\end{proof}

\begin{Lemma}\label{lem:larger chamber stab}
There exists $K$-theoretic stable envelope $\Stab^{\mathsf s}_\fC$ for $(\cM_\theta(\mathbf v,\underline{\bd}^\bv),\sw',G'^{\sigma(\bu)}\times \sT,\sA,\fC,\mathsf s)$, and we have
\begin{align}\label{larger chamber stab_K}
    \Stab^{\mathsf s}_\fC\circ\, \mathrm{res}^{G'}_{G'^{\sigma(\bu)}}=\mathrm{res}^{G'}_{G'^{\sigma(\bu)}}\circ\Stab^{\mathsf s}_+.
\end{align}
\end{Lemma}

\begin{proof}
Consider the vector bundle $E=\bigoplus_{i\in Q_0}\Hom(\cV_i,V'_i)$ on $\cM_\theta(\bv,\bv+\bd)$. It is easy to see that $E$ is repelling with respect to the chamber $\fC$, and $\mathrm{Tot}(E)\cong \cM_\theta(\bv,\bv+\bd)$ is a symmetric quiver variety. We note that the $\bC^*$-fixed component $\cM_\theta(\bv,\bd)$ is fixed by bigger torus $\sA$. Moreover, the restriction of $E$ to the $\sA$-fixed component $\cM_\theta(\bv,\bd)$ is moving. Then
the existence of $\Stab^{\mathsf s}_\fC$ follows from Lemma \ref{lem:repl vb induce stab_K}. 
The equation \eqref{larger chamber stab_K} follows from Remark \ref{rmk:repl vb induce stab_K triangle} and the commutative diagram \eqref{change of group_coh}.
\end{proof}

Now we have all the ingredients of proving Proposition \ref{prop:Psi}. 


By Lemma \ref{lem:larger chamber stab}, Proposition \ref{prop:explicit KN strata} and Lemma \ref{lem:more explicit KN strata}, we have
\begin{align*}
    \mathbf i_{\sigma(\bu)}^*\widetilde{\bPsi}^{\mathsf s}_K(\gamma)=\Stab^{\mathsf s}_\fC(1\otimes\gamma)|_{\Attr_\fC(F_{\bu,\bu})\cap \mathring{\cM}'_\theta(\bv,\bv+\bd)}.
\end{align*}
By the degree axiom of stable envelope, we have
\begin{align*}
    \deg_\sA \Stab^{\mathsf s}_\fC(1\otimes\gamma)|_{\Attr_\fC(F_{\bu,\bu})}&\subsetneq  \deg_\sA e^{G^{\sigma(\bu)}}_K\left(N^-_{F_{\bu,\bu}/\cM_\theta(\bv,\bv+\bd)}\right)+\wt_\sA\left(\mathsf s\otimes \left(\det N^-_{F_{\bu,\bu}/\cM_\theta(\bv,\bv+\bd)}\right)^{1/2}\right)\\
    &=\deg_\sA e^{G}_K\left(N_{S_{\sigma(\bu)}/R(\bv,\bd)}\right)+\wt_\sA\left(\mathsf s\otimes \left(\det N_{S_i/R(\mathbf v,\mathbf d)}\right)^{1/2}\right).
\end{align*}
Projection of the above inclusion of polytopes to the cocharacter lattice of subtorus $\bC^*_{\sigma(\bu)}$ gives rise to the desired degree bound:
\begin{align*}
    \deg_{\sigma(\bu)}\mathbf i_{\sigma(\bu)}^*\widetilde{\bPsi}^{\mathsf s}_K(\gamma)&\subseteq  \deg_{\sigma(\bu)} \Stab_\fC(1\otimes\gamma)|_{\Attr_\fC(F_{\bu,\bu})}\\
    &\subseteq  \deg_{\sigma(\bu)} e^{G}_K\left(N_{S_{\sigma(\bu)}/R(\bv,\bd)}\right)+\wt_{\sigma(\bu)}\left(\mathsf s\otimes \left(\det N_{S_i/R(\mathbf v,\mathbf d)}\right)^{1/2}\right).
\end{align*}
Since $\mathsf s$ is generic, the above inclusion must be strict by Remark \eqref{rmk: generic slope}. Finally, the uniqueness can be proven in the same way as Proposition \ref{uniqueness of K stab}, and we omit the details.

\section{Proof of Lemma \ref{lem:na stab minuscule ADE}}\label{sec on proof of lem:na stab minuscule ADE}

We prove here the $K$-theory version \eqref{na stab minuscule ADE_k} (the proof of the cohomology version \eqref{na stab minuscule ADE_coh} is similar and we leave readers to check details).
Recall that $\Psi^{\mathsf s}_K$ is by definition the dimensional reduction of the nonabelian stable envelope $\bPsi^{\mathsf s}_K$ for the symmetric quiver variety $\cM_\theta(\br,\be)$ associated with the tripled quiver $\widetilde{\Gamma}$ with potential $\sw=\sum_{i\in \Gamma_0}\tr(\varepsilon_i\mu_i)$, where $\{\varepsilon_i\}_{i\in \Gamma_0}$ is the set of edge loops and $\mu\colon R(\overline{\Gamma},\br,\be)\to \mathfrak{g}^\vee$ is the moment map.

We claim that it is enough to show that 
\begin{align*}
    \can ([\chi\otimes\mathcal O^\vir_{Z(\mu)}])\in K^\sT(\fM(\br,\be),\sw)
\end{align*}
satisfies the degree condition 
\begin{align}\label{deg bound to be proven}
        \deg_{\sigma(\bu)} {\mathbf i}_{\sigma(\bu)}^* \can ([\chi\otimes\mathcal O^\vir_{Z(\mu)}])\subseteq  \frac{1}{2}\deg_{\sigma(\bu)} e^{G}_K\left(R(\overline{\Gamma},\br,\be)-\mathfrak{g}\right)+\wt_{\sigma(\bu)}\left(\mathsf s\right),
\end{align}
for all cocharacters in $\{\sigma(\bu):\bu\in \Sigma_\theta,\:\bu \neq \mathbf{0}\}$ with $\sigma(\bu)$ defined in \eqref{sigma(u)} and $\Sigma_\theta$ the following (see \eqref{sigma} and \eqref{sigma_theta}):
\begin{align*}
     \Sigma_\theta:=\{\mathbf u\in \bN^{\Gamma_0}\colon \mathbf u_i\leqslant\br_i,\,\forall\, i\in \Gamma_0,\,\text{ and } \cM_\theta(\br-\bu,\be)\neq \emptyset\}.
\end{align*}
In \eqref{deg bound to be proven}, ${\mathbf i}_{\sigma(\bu)}:[Z^*_{\sigma(\bu)}/G^{\sigma(\bu)}]\to \fM(\br,\be)=[R(\widetilde{\Gamma},\br,\be)/G]$ is the natural map  \eqref{equ on na map zig}, and we define $Z(\mu)\hookrightarrow \fM(\br,\be)$ by the following Cartesian diagrams
\begin{equation}\label{diag which defines zmu}
\xymatrix{
Z(\mu)  \ar[r] \ar[d]_{ }  \ar@{}[dr]|{\Box}& [R(\widetilde{\Gamma},\br,\be)/G]=\fM(\br,\be)  \ar[d]^{ } \\
\fN(\br,\be)=[\mu^{-1}(0)/G]  \ar[r] \ar[d]_{ } \ar@{}[dr]|{\Box} & [R(\overline{\Gamma},\br,\be)/G]  \ar[d]^{\mu} \\
[0/G] \ar@{^{(}->}[r]^{0} & [\fg^\vee/G],
}
\end{equation}
where $0$ is the inclusion of zero, $\mu$ is the moment map of double quiver representation and the right vertical map in the upper square is the 
vector bundle map by forgetting loops.
We endow $Z(\mu)$ with virtual structure sheaf 
$$\mathcal O^\vir_{Z(\mu)}:=0^!\mathcal O_{\fM(\br,\be)},$$ 
where $0^!$ is the refined Gysin pullback. 
The canonical map \eqref{equ on can map k with sp} $\can \colon K^\sT(Z(\sw))\to K^\sT(\fM(\br,\be),\sw)$ is the proper pushforward and note that $Z(\mu)\subseteq  Z(\sw)$.
 
Let $\delta_K$ be the dimensional reduction map given by pullback followed by the pushforward (see \eqref{def of deltak}, also \eqref{equ on psi dim red} in the stacky case). 
By base change on \eqref{diag which defines zmu}, we know 
\begin{equation}\label{equ on deltakchi}\delta_K([\chi\otimes\mathcal O^\vir_{\fN(\br,\be)}])=\can ([\chi\otimes\mathcal O^\vir_{Z(\mu)}]). \end{equation}
Restricting to the stable locus, we have 
$$\mathbf k^*\can ([\chi\otimes\mathcal O^\vir_{Z(\mu)}])=\can ([\chi\otimes\mathcal O^\vir_{Z(\mu)}\big|_{\cM_\theta(\br,\be)}])=\delta_K([\mathcal L_\chi]). $$
Since $\mathsf s$ is generic, the inclusion in \eqref{deg bound to be proven} must be strict if it holds. Combining with the characterization of nonabelian stable envelopes (Proposition \ref{prop:Psi}), \eqref{deg bound to be proven} implies that 
\begin{equation}\label{equ on deltakchi2}\can ([\chi\otimes\mathcal O^\vir_{Z(\mu)}])=\bPsi^{\mathsf s}_K\circ\delta_K([\mathcal L_\chi]). \end{equation}
Eqn.~\eqref{equ on phi relation},~\eqref{equ on deltakchi} and \eqref{equ on deltakchi2} imply that
\begin{align*}
    \Psi^{\mathsf s}_K([\mathcal L_\chi])=\delta_K^{-1}\circ\bPsi^{\mathsf s}_K\circ \delta_K([\mathcal L_\chi])=[\chi\otimes\mathcal O^\vir_{\fN(\br,\be)}], 
\end{align*}
which is \eqref{na stab minuscule ADE_k}. This proves the claim. We are therefore left to show \eqref{deg bound to be proven}. 


Now suppose that $\bu\in \Sigma_\theta\setminus\{\mathbf 0\}$. Since the canonical map commutes with Gysin pullbacks, we have $${\mathbf i}_{\sigma(\bu)}^* \can ([\chi\otimes\mathcal O^\vir_{Z(\mu)}])=\can ({\mathbf i}_{\sigma(\bu)}^* [\chi\otimes\mathcal O^\vir_{Z(\mu)}])=\can ([\chi\otimes {\mathbf i}_{\sigma(\bu)}^* \mathcal O^\vir_{Z(\mu)}]).$$
The naturality of refined Gysin pullbacks implies that 
\begin{align*}
    {\mathbf i}_{\sigma(\bu)}^* \mathcal O^\vir_{Z(\mu)}={\mathbf i}_{\sigma(\bu)}^* 0^!\mathcal O_{\fM(\br,\be)}=0^!{\mathbf i}_{\sigma(\bu)}^* \mathcal O_{\fM(\br,\be)}=0^!\mathcal O_{[Z^*_{\sigma(\bu)}/G^{\sigma(\bu)}]}.
\end{align*}
Note that there is a commutative diagram
\begin{equation*}
\xymatrix{
Z^*_{\sigma(\bu)} \ar@{^{(}->}[r] \ar[d]_{\mu} & R(\widetilde{\Gamma},\br,\be) \ar[d]^{\mu} \\
(\fg^{\vee})^{\sigma(\bu)} \ar@{^{(}->}[r] & \fg^\vee,
}
\end{equation*}
with horizontal arrows being natural closed immersion and vertical arrows being the composition of the right vertical maps
in \eqref{diag which defines zmu} (by abuse of notations, we denote it also as $\mu$). Then we have 
\begin{align*}
    0^!\mathcal O_{[Z^*_{\sigma(\bu)}/G^{\sigma(\bu)}]}&=\left([0/G^{\sigma(\bu)}]\hookrightarrow [(\fg^{\vee})^{\sigma(\bu)}/G^{\sigma(\bu)}]\right)^!\circ 
    \left([(\fg^{\vee})^{\sigma(\bu)}/G^{\sigma(\bu)}]\hookrightarrow[\fg^{\vee}/G^{\sigma(\bu)}]\right)^!\mathcal O_{[Z^*_{\sigma(\bu)}/G^{\sigma(\bu)}]}\\
    &=e^{G^{\sigma(\bu)}}_K\left((\fg^{\vee})^{\sigma(\bu)\text{-moving}}\right)\cdot \left([0/G^{\sigma(\bu)}]\hookrightarrow [(\fg^{\vee})^{\sigma(\bu)}/G^{\sigma(\bu)}]\right)^!\mathcal O_{[Z^*_{\sigma(\bu)}/G^{\sigma(\bu)}]}.
\end{align*}
To summarize, we arrive at
\begin{align*}
    {\mathbf i}_{\sigma(\bu)}^* \can ([\chi\otimes\mathcal O^\vir_{Z(\mu)}])=\chi\cdot e^{G^{\sigma(\bu)}}_K\left((\fg^{\vee})^{\sigma(\bu)\text{-moving}}\right)\cdot\can\left(\left([0/G^{\sigma(\bu)}]\hookrightarrow [(\fg^{\vee})^{\sigma(\bu)}/G^{\sigma(\bu)}]\right)^!\mathcal O_{[Z^*_{\sigma(\bu)}/G^{\sigma(\bu)}]}\right).
\end{align*}
Note that 
$$\deg_{\sigma(\bu)}\left([0/G^{\sigma(\bu)}]\hookrightarrow [(\fg^{\vee})^{\sigma(\bu)}/G^{\sigma(\bu)}]\right)^!\mathcal O_{[Z^*_{\sigma(\bu)}/G^{\sigma(\bu)}]}=0. $$ 
Since the map $\can\colon K^{\sT\times G^{\sigma(\bu)}}(Z(\sw)\cap Z^*_{\sigma(\bu)})\to K^{\sT\times G^{\sigma(\bu)}}(Z^*_{\sigma(\bu)},\sw)$ does not enlarge the $\sigma(\bu)$ degree, we have 
\begin{align*}
    \deg_{\sigma(\bu)}{\mathbf i}_{\sigma(\bu)}^* \can ([\chi\otimes\mathcal O^\vir_{Z(\mu)}])\subseteq \deg_{\sigma(\bu)}\left(\chi\cdot e^{G^{\sigma(\bu)}}_K\left((\fg^{\vee})^{\sigma(\bu)\text{-moving}}\right)\right).
\end{align*}
Then it is enough to show that
\begin{equation}\label{deg bound 5}
\lim_{t\to 0,\infty} \frac{\chi\cdot e^{G^{\sigma(\bu)}}_K((\fg^{\vee})^{\sigma(\bu)\text{-moving}})}{\mathsf s\cdot \sqrt{\hat e^{G^{\sigma(\bu)}}_K\left((R(\br,\be)-\fg)^{\sigma(\bu)\text{-moving}}\right)}}\,\,\text{ exist},
\end{equation}
where $t$ is the equivariant parameter in $K_{\bC^*}(\pt)=\bQ[t^{\pm}]$. 

By the assumption on $\mathsf s$, we have $\chi/\mathsf s=t^{\epsilon}$ for $0<|\epsilon|\ll 1$. Straightforward computation leads to
\begin{align*}
    \frac{\chi\cdot e^{G^{\sigma(\bu)}}_K((\fg^{\vee})^{\sigma(\bu)\text{-moving}})}{\mathsf s\cdot \sqrt{\hat e^{G^{\sigma(\bu)}}_K\left((R(\br,\be)-\fg)^{\sigma(\bu)\text{-moving}}\right)}}\sim O\left(t^{\epsilon\pm \mathsf d}\right)\:\text{ as }\:t^{\pm 1}\to \infty,\:\text{ where }\:\mathsf d=\frac{1}{2}\bu\cdot(\mathsf C\br-\mathsf C\bu-\be).
\end{align*}
In the above, $\mathsf C$ is the Cartan matrix of $\Gamma$, and $\bu\cdot\bv:=\sum_{i\in \Gamma_0}\bu_i\bv_i$ is the standard inner product. Then \eqref{deg bound 5} is equivalent to the following claim:
\begin{itemize}
    \item if $\bu\in \Sigma_\theta\setminus\{\mathbf 0\}$, then $\bu\cdot(\be-\mathsf C\br+\mathsf C\bu)>0$.
\end{itemize}
According to \cite[Thm.\,10.2]{Nak2}, the nonemptyness of $\cN_\theta(\br,\be)$ implies that $\lambda_{\be}-\alpha_\br$ is a weight that appears in the irreducible module $V(\lambda_\be)$ of $\fg_\Gamma$ with highest weight $\lambda_\be$. Here $\fg_\Gamma$ is the Lie algebra whose Dynkin diagram is $\Gamma$ (not to be confused with $\fg$ which is the Lie algebra of the gauge group). Let $\{\alpha_i\}_{i\in \Gamma_0}$ be simple roots of $\fg_\Gamma$ and $\{\lambda_i\}_{i\in \Gamma_0}$ be fundamental weights of $\fg_\Gamma$, then 
$$\lambda_{\be}:=\sum_{i\in \Gamma_0}\be_i\lambda_i, \quad \alpha_{\br}:=\sum_{i\in \Gamma_0}\br_i\alpha_i.$$ 
Note that $\lambda_{\be}$ is a minuscule dominant weight of $\fg_\Gamma$, so every weight $\lambda$ that appears in $V(\lambda_\be)$ satisfies the minuscule condition:
\begin{align*}
    (\lambda,\alpha)\in \{0,\pm 1\},\;\forall\:\text{root }\alpha.
\end{align*}
In particular, we have
\begin{align*}
    \bu\cdot(\be-\mathsf C\br+\mathsf C\bu)=(\lambda_{\be}-\alpha_\br+\alpha_{\bu},\alpha_\bu)=(\lambda_{\be}-\alpha_\br,\alpha_\bu)+2\geqslant 1.
\end{align*}
This finishes the proof. 


\end{document}